\documentclass[a4paper,twoside]{ociamthesis}

\usepackage[utf8]{inputenc}
\usepackage[english]{babel}
\usepackage{textgreek}

\usepackage{breakurl}
\usepackage{hyperref}

\usepackage[
    style=alphabetic,
    sorting=nyt,
    backend=biber,
    maxbibnames=99,
    doi=true,
    isbn=false,
    url=true
]{biblatex}

\DeclareSourcemap{
  \maps[datatype=bibtex]{
    \map{
      \step[fieldsource=doi,final]
      \step[fieldset=url,null]
      \step[fieldset=urldate,null]
    }
    \map{
      \step[fieldsource=eprint,final]
      \step[fieldset=url,null]
      \step[fieldset=urldate,null]
    }
  }
}

\DeclareBibliographyCategory{enumpapers}

\newcommand{\enumcite}[1]{%
  \nocite{#1}
  \addtocategory{enumpapers}{#1}%
  \defbibcheck{key#1}{
    \iffieldequalstr{entrykey}{#1}
      {}
      {\skipentry}}%
  \printbibliography[heading=none,check=key#1]%
}

\usepackage{amsthm, amssymb, amsmath}

\usepackage{setspace}

\usepackage[frozencache]{minted}
\setminted{frame=lines,fontseries=mono,fontsize=\footnotesize}
\setmintedinline{fontsize=auto}
\BeforeBeginEnvironment{minted}{\vspace{-16pt}}
\AfterEndEnvironment{minted}{\vspace{-8pt}}

\newcommand{\py}[1]{{\color{brown}\mintinline{python}{#1}}}

\usepackage{caption}
\usepackage{subcaption}

\usepackage{graphicx}
\graphicspath{ {./img/} }

\usepackage{tikz}
\usepackage{tikz-cd}
\usepackage{src/tikzit}

\tikzstyle{node}=[minimum size=0.3cm]
\tikzstyle{Z}=[fill={rgb,255:red,230; green,254; blue,230}, draw={rgb,255: red,61; green,77; blue,61}, shape=circle]
\tikzstyle{X}=[fill={rgb,255:red,255; green,135; blue,136}, draw={rgb,255: red,102; green,54; blue,54}, shape=circle]
\tikzstyle{Y}=[fill=zxblue, draw=zxdblue, shape=circle]
\tikzstyle{Z_big}=[fill={rgb,255:red,230; green,254; blue,230}, draw={rgb,255: red,61; green,77; blue,61}, shape=circle, minimum width=1.6em, font={\small}]
\tikzstyle{X_big}=[fill={rgb,255:red,255; green,135; blue,136}, draw={rgb,255: red,102; green,54; blue,54}, shape=circle, minimum width=1.6em, font={\small}]
\tikzstyle{Y_big}=[fill=zxblue, draw=zxdblue, shape=circle, minimum width=1.6em, font={\small}]
\tikzstyle{Z_tri}=[fill={rgb,255:red,230; green,254; blue,230}, draw={rgb,255: red,61; green,77; blue,61}, regular polygon, regular polygon sides=3, draw, shape border rotate=0, inner sep=0pt, minimum width=15pt, line width=0.75]
\tikzstyle{X_tri}=[fill={rgb,255:red,255; green,135; blue,136}, draw={rgb,255: red,102; green,54; blue,54}, regular polygon, regular polygon sides=3, draw, shape border rotate=0, inner sep=0pt, minimum width=15pt, line width=0.75]
\tikzstyle{Z_tri_inv}=[fill={rgb,255:red,230; green,254; blue,230}, draw={rgb,255: red,61; green,77; blue,61}, regular polygon, regular polygon sides=3, draw, shape border rotate=180, inner sep=0pt, minimum width=15pt, line width=0.75, font={\footnotesize}]
\tikzstyle{X_tri_inv}=[fill={rgb,255:red,255; green,135; blue,136}, draw={rgb,255: red,102; green,54; blue,54}, regular polygon, regular polygon sides=3, draw, shape border rotate=180, inner sep=0pt, minimum width=15pt, line width=0.75]
\tikzstyle{Z_tri_l}=[fill={rgb,255:red,230; green,254; blue,230}, draw={rgb,255: red,61; green,77; blue,61}, regular polygon, regular polygon sides=3, draw, shape border rotate=90, inner sep=0pt, minimum width=15pt, line width=0.75]
\tikzstyle{X_tri_l}=[fill={rgb,255:red,255; green,135; blue,136}, draw={rgb,255: red,102; green,54; blue,54}, regular polygon, regular polygon sides=3, draw, shape border rotate=90, inner sep=0pt, minimum width=15pt, line width=0.75]
\tikzstyle{H}=[fill=yellow, draw=black, shape=rectangle, minimum width=2mm, minimum height=2mm]
\tikzstyle{Y_box}=[draw=black, shape=rectangle, minimum width=2mm, minimum height=2mm, tikzit fill={rgb,255: red,255; green,128; blue,0}]
\tikzstyle{Z_med}=[fill={rgb,255:red,230; green,254; blue,230}, draw={rgb,255: red,61; green,77; blue,61}, shape=circle, minimum width=1.1em, font={\footnotesize}]
\tikzstyle{X_med}=[fill={rgb,255:red,255; green,135; blue,136}, draw={rgb,255: red,102; green,54; blue,54}, shape=circle, minimum width=1.1em, font={\footnotesize}]
\tikzstyle{Y_med}=[fill=zxblue, draw=zxdblue, font={\footnotesize}, minimum width=1.1em, font={\footnotesize}, shape=circle]
\tikzstyle{ZP}=[fill={rgb,255:red,230; green,254; blue,230}, draw={rgb,255: red,61; green,77; blue,61}, regular polygon, regular polygon sides=3, shape border rotate=30]
\tikzstyle{ZM}=[fill={rgb,255:red,230; green,254; blue,230}, draw={rgb,255: red,61; green,77; blue,61}, regular polygon, regular polygon sides=3, shape border rotate=-30]
\tikzstyle{XP}=[fill={rgb,255:red,255; green,135; blue,136}, draw={rgb,255: red,102; green,54; blue,54}, regular polygon, regular polygon sides=3, shape border rotate=30]
\tikzstyle{XM}=[fill={rgb,255:red,255; green,135; blue,136}, draw={rgb,255: red,102; green,54; blue,54}, regular polygon, regular polygon sides=3, shape border rotate=-30]
\tikzstyle{small_box}=[fill=white, draw=black, shape=rectangle, minimum width=1.5cm, minimum height=1.5cm, font={\footnotesize}]
\tikzstyle{med_rectangle}=[fill=white, draw=black, shape=rectangle, minimum width=2.8cm, minimum height=1.5cm, font={\footnotesize}]
\tikzstyle{big_rectangle}=[fill=white, draw=black, shape=rectangle, minimum width=4.5cm, minimum height=1.5cm, font={\footnotesize}]
\tikzstyle{smol_box}=[fill=white, draw=black, shape=rectangle, minimum width=1cm, minimum height=1cm]
\tikzstyle{poo}=[minimum height=0.7cm, minimum width=0.7cm, path picture={\node at (path picture bounding box.center) {\includegraphics[width=0.7cm] {figures/poo}};}]
\tikzstyle{Z_long}=[fill={rgb,255:red,230; green,254; blue,230}, draw={rgb,255: red,61; green,77; blue,61}, shape=rectangle, rounded corners=0.25cm, minimum height=0.5cm, inner sep=0.25em, font={\scriptsize}]
\tikzstyle{X_long}=[fill={rgb,255:red,255; green,135; blue,136}, draw={rgb,255: red,102; green,54; blue,54}, shape=rectangle, rounded corners=0.25cm, minimum height=0.5cm, inner sep=0.25em, font={\scriptsize}]
\tikzstyle{med_rectv}=[fill=white, draw=black, shape=rectangle, minimum width=1.1cm, minimum height=2.4cm, font={\scriptsize}, inner sep=0.2cm]
\tikzstyle{Z dot}=[fill={rgb,255: red,0; green,127; blue,0}, draw=black, shape=circle, minimum width=1.5mm]
\tikzstyle{X dot}=[fill={rgb,255:red,255; green,21; blue,0}, draw=black, shape=circle, minimum width=1.5mm]
\tikzstyle{Z phase dot}=[fill={rgb,255: red,0; green,127; blue,0}, draw=black, shape=circle, font={\footnotesize}]
\tikzstyle{X phase dot}=[fill={rgb,255:red,255; green,21; blue,0}, draw=black, shape=circle, font={\footnotesize}]
\tikzstyle{ZYa}=[draw=black, shape=rectangle, rectangle split, rectangle split parts=2, rectangle split horizontal, rectangle split part fill={zxgreen, zxblue}, rectangle split draw splits=false, minimum height=2mm, font={\tiny}]
\tikzstyle{YZ}=[draw=black, shape=rectangle, rectangle split, rectangle split parts=2, rectangle split horizontal, rectangle split part fill={zxblue, zxgreen}, rectangle split draw splits=false, minimum height=2mm, font={\tiny}]
\tikzstyle{XYa}=[draw=black, shape=rectangle, rectangle split, rectangle split parts=2, rectangle split horizontal, rectangle split part fill={zxred, zxblue}, rectangle split draw splits=false, minimum height=2mm, font={\tiny}]
\tikzstyle{YX}=[draw=black, shape=rectangle, rectangle split, rectangle split parts=2, rectangle split horizontal, rectangle split part fill={zxblue, zxred}, rectangle split draw splits=false, minimum height=2mm, font={\tiny}]
\tikzstyle{tiny_box}=[fill=white, draw=black, shape=rectangle, minimum width=1cm, minimum height=1cm, font={\footnotesize}]
\tikzstyle{scalar}=[fill=white, draw=black, shape=diamond, font={\scriptsize}]

\tikzstyle{dashs}=[-, dashed, line width=0.15mm]
\tikzstyle{thick}=[-, line width=0.5mm]
\tikzstyle{arrow}=[->]
\tikzstyle{invisible}=[-, draw=none]
\tikzstyle{functor}=[-, fill={rgb,255: red,240; green,240; blue,240}]
\tikzstyle{boxedge}=[-, fill=white]

\usepackage[perpage]{footmisc}

\usepackage{ragged2e}
\usepackage{epigraph}
\newcommand{\justepigraph}[2]{
\epigraph{\footnotesize \justifying #1}{#2}}

\DeclareMathOperator*{\argmin}{arg\,min}

\newcommand{\injects}{\hookrightarrow}
\newcommand{\sub}{\subseteq}
\newcommand{\s}{\enspace}

\newcommand{\eval}[1]{[ \! [ #1 ] \! ]}

\newcommand{\bra}[1]{\langle#1|}
\newcommand{\ket}[1]{|#1\rangle}
\newcommand{\braket}[2]{\langle#1|#2\rangle}

\newcommand{\dom}{\mathtt{dom}}
\newcommand{\cod}{\mathtt{cod}}
\newcommand{\ttcap}{\mathtt{cap}}
\newcommand{\ttcup}{\mathtt{cup}}
\newcommand{\id}{\mathtt{id}}
\newcommand{\then}{\mathtt{then}}
\newcommand{\tensor}{\mathtt{tensor}}
\newcommand{\len}{\mathtt{len}}
\newcommand{\boxes}{\mathtt{boxes}}
\newcommand{\spider}{\mathtt{spider}}
\newcommand{\shift}{\mathtt{shift}}
\newcommand{\merge}{\mathtt{merge}}
\newcommand{\unit}{\mathtt{unit}}
\newcommand{\ttsplit}{\mathtt{split}}
\newcommand{\counit}{\mathtt{counit}}

\renewcommand{\S}{\mathbb{S}}
\newcommand{\B}{\mathbb{B}}
\newcommand{\N}{\mathbb{N}}
\newcommand{\Z}{\mathbb{Z}}
\newcommand{\R}{\mathbb{R}}
\newcommand{\C}{\mathbb{C}}
\newcommand{\G}{\mathbf{G}}
\newcommand{\D}{\mathbb{D}}
\newcommand{\F}{\mathbb{F}}

\def\fcmp{\mathbin{\raise 0.6ex\hbox{\oalign{\hfil$\scriptscriptstyle      \mathrm{o}$\hfil\cr\hfil$\scriptscriptstyle\mathrm{9}$\hfil}}}}

\newcommand{\downmapsto}{\rotatebox[origin=c]{-90}{$\mapsto$}\mkern2mu}

\newtheorem{definition}{Definition}[section]
\newtheorem{proposition}[definition]{Proposition}
\newtheorem{theorem}[definition]{Theorem}

\newtheorem{lemma}[definition]{Lemma}
\newtheorem{corollary}[definition]{Corollary}
\newtheorem{example}[definition]{Example}
\newtheorem{remark}[definition]{Remark}
\newtheorem{python}[definition]{Listing}

\title{Category Theory for Quantum\\Natural Language Processing}
\author{Alexis TOUMI} \college{Wolfson College} \degree{Doctor of Philosophy}
\degreedate{Trinity 2022}

\begin{document}

\maketitle
\begin{abstract}
This thesis introduces quantum natural language processing (QNLP) models based on a simple yet powerful analogy between computational linguistics and quantum mechanics: grammar as entanglement.
The grammatical structure of text and sentences connects the meaning of words in the same way that entanglement structure connects the states of quantum systems.
Category theory allows to make this language-to-qubit analogy formal: it is a monoidal functor from grammar to vector spaces.
We turn this abstract analogy into a concrete algorithm that translates the grammatical structure onto the architecture of parameterised quantum circuits.
We then use a hybrid classical-quantum algorithm to train the model so that evaluating the circuits computes the meaning of sentences in data-driven tasks.

The implementation of QNLP models motivated the development of DisCoPy (Distributional Compositional Python), the toolkit for applied category theory of which the first chapter gives a comprehensive overview.
String diagrams are the core data structure of DisCoPy, they allow to reason about computation at a high level of abstraction.
We show how they can encode both grammatical structures and quantum circuits, but also logical formulae, neural networks or arbitrary Python code.
Monoidal functors allow to translate these abstract diagrams into concrete computation, interfacing with optimised task-specific libraries.

The second chapter uses DisCopy to implement QNLP models as parameterised functors from grammar to quantum circuits.
It gives a first proof-of-concept for the more general concept of functorial learning: generalising machine learning from functions to functors by learning from diagram-like data.
In order to learn optimal functor parameters via gradient descent, we introduce the notion of diagrammatic differentiation: a graphical calculus for computing the gradients of parameterised diagrams.

\end{abstract}
\tableofcontents


\chapter*{Introduction}
\addcontentsline{toc}{chapter}{Introduction}
\markboth{Introduction}{Introduction}


\section*{What are quantum computers good for?}
\addcontentsline{toc}{section}{What are quantum computers good for?}

\justepigraph{
Nature isn't classical, dammit, and if you want to make a simulation of nature, you'd better make it quantum mechanical, and by golly it's a wonderful problem, because it doesn't look so easy.
}{\textit{Simulating Physics with Computers}, Feynman (1981)}

Quantum computers harness the principles of quantum theory such as superposition and entanglement to solve information-processing tasks.
In the last 42 years, quantum computing has gone from theoretical speculations to the implementation of machines that can solve problems beyond what is possible with classical means.
This section will sketch a brief and biased history of the field and of its future challenges.

In 1980, Benioff~\cite{Benioff80} takes the abstract definition of a computer and makes it physical: he designs a quantum mechanical system whose time evolution encodes the computation steps of a given Turing machine.
In retrospect, this may be taken as the first proof that quantum mechanics can simulate classical computers.
The same year, Manin~\cite{Manin80} looks at the opposite direction: he argues that it should take exponential time for a classical computer to simulate a generic quantum system.
Feynman~\cite{Feynman82,Feynman85} comes to the same conclusion and suggests a way to simulate quantum mechanics much more efficiently: building a quantum computer!

So what are quantum computers good for? Feynman's intuition gives us a first, trivial answer: at least quantum computers could simulate quantum mechanics efficiently. Deutsch~\cite{Deutsch85a} makes the question formal by defining quantum Turing machines and the circuit model.
Deutsch and Jozsa~\cite{DeutschJozsa92} design the first quantum algorithm and prove that it solves \emph{some} problem exponentially faster than any classical \emph{deterministic} algorithm.\footnote
{A classical \emph{randomised} algorithm solves the problem in constant time with high probability.}
Simon~\cite{Simon94} improves on their result by designing a problem that a quantum computer can solve exponentially faster than any classical algorithm.
Deutsch-Jozsa and Simon relied on oracles\footnote{An oracle is a black box that allows a Turing machine to solve a certain problem in one step.}
and promises\footnote{The input is promised to satisfy a certain property, which may be hard to check.} and their problems have little practical use.
However, they inspired Shor's algorithm~\cite{Shor94} for prime factorisation and discrete logarithm. These two problems are believed to require exponential time for a classical computer and their hardness is at the basis of the public-key cryptography schemes currently used on the internet.

In 1997, Grover provides another application for quantum computers: ``searching for a needle in a haystack''~\cite{Grover97}.
Formally, given a function $f : X \to \{0, 1\}$ and the promise that there is a unique $x \in X$ with $f(x) = 1$, Grover's algorithm finds $x$ in $O(\sqrt{|X|})$ steps, quadratically faster than the optimal $O(|X|)$ classical algorithm.
Grover's algorithm may be used to brute-force symmetric cryptographic keys twice bigger than what is possible classically~\cite{BernsteinEtAl09}.
It can also be used to obtain quadratic speedups for the exhaustive search involved in the solution of NP-hard problems such as constraint satisfaction~\cite{Ambainis04}.
Independently, Bennett et al.~\cite{BennettEtAl97} prove that Grover's algorithm is in fact optimal, adding evidence to the conjecture that quantum computers cannot solve these NP-hard problems in polynomial time.
Chuang et al.~\cite{ChuangEtAl98} give the first experimental demonstration of a quantum algorithm, running Grover's algorithm on two qubits.

Shor's and Grover's discovery of the first real-world applications sparked a considerable interest in quantum computing.
The core of these two algorithms has then been abstracted away in terms of two
subroutines: phase estimation~\cite{Kitaev95} and amplitude
amplification~\cite{BrassardEtAl02}, respectively.
Making use of both these subroutines, the HHL\footnote{Named after its discoverers Harrow, Hassidim and Lloyd.} algorithm~\cite{HarrowEtAl09} tackles one of the most ubiquitous problems in scientific computing: solving systems of linear equations.
Given a matrix $A \in \mathbb{R}^{n \times n}$ and a vector ${b} \in \mathbb{R}^{n}$, we want to find a vector ${x}$ such that $A {x} = {b}$.
Under some assumptions on the sparsity and the condition number of $A$, HHL finds (an approximation of) $x$ in time logarithmic in $n$ when a classical algorithm would take quadratic time simply to read the entries of $A$.
This initiated a new wave of enthusiasm for quantum computing with the promise of exponential speedups for machine learning tasks such as regression~\cite{WiebeEtAl12}, clustering~\cite{LloydEtAl13}, classification~\cite{RebentrostEtAl14}, dimensionality reduction~\cite{LloydEtAl14} and recommendation~\cite{KerenidisPrakash16}.
The narrative is appealing: machine learning is about finding patterns in large amounts of data represented as high-dimensional vectors and tensors, which is precisely what quantum computers are good at.
The argument can be formalised in terms of complexity theory: HHL is \texttt{BQP}-complete\footnote
{A \texttt{BQP}-complete problem is one that is polynomial-time  equivalent to the circuit model, the hardest problem that a quantum computer can solve with bounded error in polynomial time.}
hence if there is an exponential advantage for quantum algorithms at all there must be one for HHL.

However, the exponential speedup of HHL comes with some caveats, thoroughly analysed by Aaronson~\cite{Aaronson15}.
Two of these challenges are common to many quantum algorithms:
1) the efficient encoding of classical data into quantum states and
2) the efficient extraction of classical data via quantum measurements.
Indeed, what HHL really takes as input is not a vector ${b}$ but a quantum state $\ket{b} = \sum_{i=1}^n b_i \ket{i}$ called its amplitude encoding.
Either the input vector ${b}$ has enough structure that we can describe it with a simple, explicit formula.
This is the case for example in the calculation of electromagnetic scattering cross-sections~\cite{CladerEtAl13}.
Or we assume that our classical data has been loaded onto a quantum random-access memory (qRAM) that can prepare the state in logarithmic time~\cite{GiovannettiEtAl08}.
Not only is qRAM a daunting challenge from an engineering point of view, in some cases it also requires too much error correction for the state preparation to be efficient \cite{ArunachalamEtAl15}.
Symmetrically, the output of HHL is not the solution vector ${x}$ itself but a quantum state $\ket{x}$ from which we can measure some observable $\bra{x} M \ket{x}$.
If preparing the state $\ket{b}$ requires a number of gates exponential in the number of qubits, or if we need exponentially many measurements of $\ket{x}$ to compute our classical output, then the quantum speedup disappears.

Shor, Grover and HHL all assume \emph{fault-tolerant} quantum computers~\cite{Shor96}.
Indeed, any machine we can build will be subject to noise when performing quantum operations, errors are inevitable: we need an error correcting code that can correct these errors faster than they appear.
This is the content of the \emph{quantum threshold theorem}~\cite{AharonovBen-Or08} which proves the possibility of fault-tolerant quantum computing given physical error rates below a certain threshold.
One noteworthy example of such a quantum error correction scheme is Kitaev's toric code~\cite{Kitaev03} and the general idea of topological quantum computation~\cite{FreedmanEtAl03} which offers the long-term hope for a quantum computer that is fault-tolerant ``by its physical nature''.
However this hope relies on the existence of quasi-particles called Majorana zero-modes, which as of 2021 has yet to be experimentally demonstrated~\cite{Ball21}.

The road to large-scale fault-tolerant quantum computing will most likely be a long one.
So in the meantime, what can we do with the noisy intermediate-scale quantum machines we have available today, in the so-called NISQ era~\cite{Preskill18}?
Most answers involve a hybrid classical-quantum approach where a classical algorithm is used to optimise the preparation of quantum states~\cite{McCleanEtAl16}.
Prominent examples include the quantum approximate optimisation algorithm (QAOA~\cite{FarhiEtAl14}) for combinatorial problems such as maximum cut and the variational quantum eigensolver (VQE~\cite{PeruzzoEtAl14}) for approximating the ground state of chemical systems.
These variational algorithms depend on the choice of a parameterised quantum circuit called the \emph{ansatz}, based on the structure of the problem and the resources available.
Some families of ansätze such as instantaneous quantum polynomial-time (IQP) circuits are believed to be hard to simulate classically even at constant depth \cite{ShepherdBremner09}, opening the door to potentially near-term NISQ speedups.

Although the hybrid approach first appeared in the context of machine learning~\cite{BangEtAl08}, the idea of using parameterised quantum circuits as machine learning models went mostly unnoticed for a decade~\cite{BenedettiEtAl19}.
It was rediscovered under the name of quantum neural networks~\cite{FarhiNeven18} then implemented on two-qubits~\cite{HavlicekEtAl19}, generating a new wave of attention for quantum machine learning.
The idea is straightforward: 1) encode the input vector ${x} \in \mathbb{R}^n$ as a quantum state $\ket{\phi_{x}}$ via the ansatz of our choice, 2) initialise a random vector of parameters ${\theta} \in \mathbb{R}^d$ and encode it as a measurement $M_\theta$, again via some choice of ansatz 3) take the probability $y = \bra{\phi({x})} M_\theta \ket{\phi({x})}$ as the prediction of the model.
A classical algorithm then uses this quantum prediction as a subroutine to find the optimal parameters $\theta$ in some data-driven task such as classification.

One of the many challenges on the way to solving real-world problems with parameterised quantum circuits is the existence of \emph{barren plateaus}~\cite{McCleanEtAl18}:
with random circuits as ansatz, the probability of non-zero gradients is exponentially small in the number of qubits and our classical optimisation gets lost in a flat landscape.
One cannot help but notice the striking similarity with the vanishing gradient
problem for classical neural networks, formulated twenty years earlier~\cite{Hochreiter98}.
Barren plateaus do not appear in circuits with enough structure such as quantum convolutional networks~\cite{PesahEtAl21}, they can also be mitigated by structured initialisation strategies~\cite{GrantEtAl19}.
Another direction is to avoid gradients altogether and use \emph{kernel methods}~\cite{SchuldKilloran19}:
instead of learning a measurement $M_\theta$, we use our NISQ device to estimate the distance $|\braket{\phi_{x'}}{\phi_x}|^2$ between pairs of input vectors $x, x' \in \mathbb{R}^n$ embedded in the high-dimensional Hilbert space of our ansatz.
We then use a classical support vector machine to find the optimal hyperplane that separates our data, with theoretical guarantees to learn quantum models at least as good as the variational approach~\cite{Schuld21}.

Random quantum circuits may be unsuitable for machine learning, but they play a crucial role in the quest for \emph{quantum advantage}, the experimental demonstration of a quantum computer solving a task that cannot be solved by classical means in any reasonable time.
We are back to Feynman's original intuition: sampling from a random quantum circuit is the perfect candidate for such a task.
The end of 2019 saw the first claim of such an advantage with a 53-qubit computer~\cite{AruteEtAl19}.
The claim was almost immediately contested by a classical simulation of 54 qubits in two and a half days~\cite{PednaultEtAl19} then in five minutes~\cite{YongEtAl21}.
Zhong et al.~\cite{ZhongEtAl20} made a new claim with a 76-photon linear optical quantum computer followed by another with a 66-qubit computer~\cite{WuEtAl21,ZhuEtAl21}.
They estimate that a classical simulation of the sampling task they completed in a couple of hours would take at least ten thousand years.

Now that quantum computers are being demonstrated to compute something beyond classical, the question remains: can they compute something \emph{useful}?


\section*{Why should we make NLP quantum?}
\addcontentsline{toc}{section}{Why should we make NLP quantum?}

\justepigraph{
A girl operator typed out on a keyboard the following Russian text in English characters: ``Mi pyeryedayem mislyi posryedstvom ryechi''.
\linebreak
The machine printed a translation almost simultaneously: ``We transmit thoughts by means of speech.''
\linebreak
The operator did not know Russian.
}{
\textit{New York Times} (8th January 1954)
}

The previous section hinted at the fact that quantum computing cannot
simply solve any problem faster.
There needs to be some structure that a quantum computer can exploit:
its own structure in the case of physics simulation or the group-theoretic structure of cryptographic protocols in Shor's algorithm.

So why should we expect quantum computers to be any good at natural language processing (NLP)?
This section will argue that natural language shares a common structure with quantum theory, in the form of two linguistic principles: \emph{compositionality}
and \emph{distributionality}.

We start our history of artificial intelligence (AI) in 1950 with a philosophical question from Turing~\cite{Turing50}: ``Can machines think?'' reformulated in terms of a game, now known as the Turing test, in which a machine tries to convince a human interrogator that it is human too.
In order to put human and machine on an equal footing, Turing suggests to let them communicate only via written language: his thought experiment actually defined an NLP task.
Only four years later, NLP goes from philosophical speculation to experimental demonstration: the IBM 701 computer successfully translated sentences from Russian to English such as ``They produce alcohol out of potatoes.''~\cite{Hutchins04}.
With only six grammatical rules and a 250-word vocabulary taken from organic chemistry and other general topics, this first experiment generated a great deal of public attention and the overly-optimistic prediction that machine translation would be an accomplished task in ``five, perhaps three'' years.

Two years later, Chomsky~\cite{Chomsky56, Chomsky57} proposes a hierarchy of models for natural language syntax which hints at why NLP would not be solved so fast.
In the most expressive model, which he argues is the most appropriate for studying natural language, the parsing problem is in fact Turing-complete.
Let alone machine translation, merely deciding whether a given sequence of words is grammatical can go beyond the power of any physical computer.
Chomsky's parsing problem is a linguistic reinterpretation of an older problem from Thue~\cite{Thue14}, now known as the \emph{word problem for monoids}\footnote
{Historically, Thue, Markov and Post were working with \emph{semigroups}, i.e. unitless monoids.
} and proved undecidable by Post~\cite{Post47} and Markov~\cite{Markov47} independently.
This reveals a three-way connection between theoretical linguistics, computer science and abstract algebra which will pervade much of this thesis.
But if we are interested in solving practical NLP problems, why should we care about such abstract constructions as formal grammars?

Most NLP tasks of interest involve natural language \emph{semantics}: we want machines to compute the \emph{meaning} of sentences.
Given the grammatical structure of a sentence, we can compute its meaning as a function of the meanings of its words.
This is known as the \emph{principle of compositionality}, usually attributed to Frege.\footnote
{Compositionality does not appear in any of Frege's published work~\cite{Pelletier01}.
Frege did state what is known as the \emph{context principle}:
``it is enough if the sentence as whole has meaning; thereby also its parts obtain their meanings''.
This can be taken as a kind of dual to compositionality: the meanings of the words are functions of the meaning of the sentence.}
It was already implicit in Boole's \emph{laws of thought}~\cite{Boole54} and then made explicit by Carnap~\cite{Carnap47}.
Montague~\cite{Montague70,Montague70a,Montague73} then formalised this principle as a \emph{homomorphism} from the algebra of syntax (i.e. grammar) to that of semantics (i.e. logic).
He applied compositionality to linguistics for the first time, arguing that there is ``no important theoretical difference between natural languages and the artificial languages of logicians''.
Compositionality became the basis of the symbolic approach to NLP, also known as \emph{good old-fashioned AI} (GOFAI)~\cite{Haugeland89}.
Word meanings are first encoded in a machine-readable format, then the machine can compose them to answer complex questions.
This approach culminated in 2011 with IBM Watson defeating a human champion at \emph{Jeopardy!}~\cite{LallyFodor11}.

The same year, Apple deploy their virtual assistant in the pocket of millions of users, soon followed by internet giants Amazon and Google.
While Siri, Alexa and their competitors have made NLP mainstream, none of them make any explicit use of formal grammars.
Instead of the complex grammatical analysis and knowledge representation of expert systems like Watson, the AI of these next-generation NLP machines is powered by deep neural networks and machine learning of big data.
Although their architecture got increasingly complex, these neural networks implement a simple statistical concept: \emph{language models}, i.e. probability distributions over sequences of words.
Instead of the compositionality of symbolic AI, these statistical methods rely on another linguistic principle, \emph{distributionality}: words with similar distributions have similar meanings.

This principle may be traced back to Wittgenstein's \emph{Philosophical Investigations}: ``the meaning of a word is its use in the language''~\cite{Wittgenstein53}, usually shortened into the slogan \emph{meaning is use}.
It was then formulated in the context of computational linguistics by Harris~\cite{Harris54}, Weaver~\cite{Weaver55} and Firth~\cite{Firth57}, who coined the famous quotation: ``You shall know a word by the company it keeps!''
Before deep neural networks took over, the standard way to formalise distributionality had been \emph{vector space models}~\cite{SaltonEtAl75}.
We have a set of $N$ words appearing in a set of $M$ documents and we simply count how many times each word appears in each document to get a $M \times N$ matrix.
We normalise it with a weighting scheme like tf-idf (term frequency by inverse document frequency), factorise it (via e.g. singular value decomposition or non-negative matrix factorisation) and we're done!
The columns of the matrix encode the meanings of words, taking their inner product yields a measure of word similarity which can then be used in tasks such as classification or clustering.
This method has the advantage of simplicity and it works surprisingly well in a wide range of applications from spam detection to movie recommendation~\cite{TurneyPantel10}.
Its main limitation is that a sentence is represented not as a sequence but as a \emph{bag of words}, the word vectors will be the same whether the corpus contained ``dog bites man'' or ``man bites dog''.
A standard way to fix this is to compute vectors not for words in isolation but for $n$-grams, windows of $n$ consecutive words for some fixed size $n$.
However the fix has its own limits: if $n$ is too small we cannot detect any long-range correlations, if it is too big then the matrix is so sparse that we cannot detect anything at all.

In contrast, the recurrent neural networks (RNNs) of Rumelhart, Hinton and Williams~\cite{RumelhartEtAl86} are inherently sequential and their internal state can encode arbitrarily long-range correlations.
At each step, the network processes the next word in a sequence and updates its internal state.
This internal memory can then be used to predict the rest of the sequence, or fed as input to another network e.g. for translation into another language.
Once the obstacles to training were overcome (such as the vanishing gradients mentioned above), RNN architectures such as long short-term memory (LSTM)~\cite{HochreiterSchmidhuber97} set records in a variety of NLP tasks such as language modeling~\cite{SutskeverEtAl11}, speech recognition~\cite{GravesEtAl13} and machine translation~\cite{SutskeverEtAl14}.
The purely sequential approach of RNNs turned out to be limited: when the network is done reading, the information from the first word has to propagate through the entire text before it can be translated.
Bidirectional RNNs~\cite{SchusterPaliwal97} fix this issue by reading both left-to-right and right-to-left.
Nonetheless, it is somewhat unsatisfactory from a cognitive perspective (humans manage to understand text without reading backward, why should a machine do that?) and also harder to use in online settings where words need to be processed one at a time.

Attention mechanisms provide a much more elegant solution: instead of assuming that the ``company'' of a word is its immediate left and right neighbourhood, we let the neural network itself learn which words are relevant to which.
First introduced as a way to boost the performance of RNNs on translation tasks~\cite{BahdanauEtAl15}, attention has then become the basis of the \emph{transformer model}~\cite{VaswaniEtAl17}: a stack of attention mechanisms which process sequences without recurrence altogether.
Starting with BERT~\cite{DevlinEtAl19}, transformers have replaced RNNs as the state-of-the-art NLP model, culminating with the GPT-3 language generator authoring its own article in \emph{The Guardian}~\cite{GPT-320}:
``A robot wrote this entire article. Are you scared yet, human?''

Indeed \emph{why} should we be scared?
Because we are ignorant of \emph{how} the robot wrote the article and we cannot explain what in its billions of parameters made it write the way it did.
Transformers and neural networks in general are \emph{black boxes}: we can probe the way they map inputs to outputs, but if we look at the terabytes of weights in between, we find no interpretation of the mapping.
Moreover without explainability there can be no fairness: if we cannot explain how its decisions are made, we can hardly prevent the network from reproducing the discriminations present both in the datasets and in the assumptions of the data scientist.
We argue that explainable AI requires to make the distributional black boxes transparent by endowing them with a compositional structure: we need \emph{compositional distributional} (DisCo) models that reconcile symbolic GOFAI with deep learning.

DisCo models have their roots in neuropsychology rather than AI.
Indeed, they first appeared as models of the brain rather than architectures of learning machines.
In their seminal work~\cite{McCullochPitts43}, McCullogh and Pitts give the first formal definition of neural networks and show how their ``all-or-nothing'' behaviour\footnote
{A neuron's response is either maximal or zero, regardless of the stimulus strength.}
allow them to encode a fragment of propositional logic.
Hebb~\cite{Hebb49} then introduced the first biological mechanism to explain learning and structured perception: ``neurons that fire together, wire together''.
These computational models of the brain became the basis of \emph{connectionism}~\cite{Smolensky87,Smolensky88} and the \emph{neurosymbolic}~\cite{Hilario97} approach to AI: high-level symbolic reasoning emerges from low-level neural networks.
An influential example is Smolensky's \emph{tensor product representation}~\cite{Smolensky90}, where discrete structures such as lists and trees are embedded into the tensor product of two vector spaces, one for variables and one for values.
Concretely, a list $x_1, \dots, x_n$ of $n$ vectors of dimension $d$ is represented as a tensor $\sum_{i \leq n} \ket{i} \otimes x_i \in \mathbb{R}^n \otimes \mathbb{R}^d$.
Smolensky~\cite{Smolensky90} is also the first to make the analogy between the distributional representations of compositional structures in AI and the group representations of quantum physics.
He argues that symbolic structures embed in neural networks in the same way that the symmetries of particles embed in their state space: via \emph{representation theory}, a precursor of \emph{category theory} which we discuss in the next section.

Clark and Pulman~\cite{ClarkPulman07} propose to apply this tensor product representation to NLP, but they note its main weakness: lists of different lengths do not live in the same space, which makes it impossible to compare sentences with different grammatical structures.
The categorical compositional distributional (DisCoCat) models of Clark, Coecke and Sadrzadeh~\cite{ClarkEtAl08,ClarkEtAl10} overcome this issue by taking the analogy with quantum one step further.
Word meanings and grammatical structure are to linguistics what quantum states and entanglement structure are to physics.
DisCoCat word meanings live in vector spaces and they compose with tensor products: the states of quantum theory do too.
Grammar tells you how words are connected and how information flows in a sentence and in the same way, entanglement connects quantum states and tells you how information flows in a complex quantum system.
This analogy allows to borrow well-established mathematical tools from quantum theory, and it was implemented on classical hardware with some empirical success on small-scale tasks such as sentence comparison~\cite{GrefenstetteEtAl11} and word sense disambiguation~\cite{GrefenstetteSadrzadeh11,KartsaklisEtAl13}.
However representing the meaning of sentences as quantum processes comes at a price: they can be exponentially hard to simulate classically.

If DisCoCat models are intractable for classical computers, why not use a quantum computer instead?
Zeng and Coecke~\cite{ZengCoecke16} answered this question with the first quantum natural language processing (QNLP) algorithm\footnote
{We exclude previous algorithms that are inspired by quantum theory but run on classical computers such as the frameworks of Chen~\cite{Chen02} and Blacoe et al.~\cite{BlacoeEtAl13}.}
and the proof of a quadratic speedup on a sentence classification task.
Wieber et al.~\cite{WiebeEtAl19} later defined a QNLP algorithm based on a generalisation of the tensor product representation and proved it is \texttt{BQP}-complete: if any quantum algorithm has an exponential advantage, then in principle there must be one for QNLP.
However promising they may be, both algorithms assume fault-tolerance and they are at least as far away from solving real-world problems as Grover and HHL.

This is where the work presented in this thesis comes in: we show it is possible to implement DisCoCat models on the machines available today.
The author and collaborators~\cite{MeichanetzidisEtAl20a,CoeckeEtAl20} introduced the first NISQ-friendly framework for QNLP by translating DisCoCat models into variational quantum algorithms.
We then implemented this framework and demonstrated the first QNLP experiment on a toy question-answering task~\cite{MeichanetzidisEtAl20} and more recent experiments showed empirical success on a larger-scale classification task~\cite{LorenzEtAl21}.
Our framework was later applied to machine translation~\cite{AbbaszadeEtAl21,VicenteNieto21}, word-sense disambiguation~\cite{Hoffmann21} and even to generative music~\cite{MirandaEtAl21}.
Future experiments will have to demonstrate that QNLP is more than a mere analogy and that it can achieve \emph{quantum advantage on a useful task}.
But before we can discuss our implementation in detail, we have to make the DisCoCat analogy formal.


\section*{How can category theory help?}
\addcontentsline{toc}{section}{How can category theory help?}

\justepigraph{
I should still hope to create a kind of \emph{universal symbolistic} (\emph{spécieuse générale}) in which all truths of reason would be reduced to a kind of calculus.
}{\emph{Letter to Nicolas Remond}, Leibniz (1714)}

``Every sufficiently good analogy is yearning to become a functor''~\cite{Baez06} and we will see that the analogy behind DisCoCat models is indeed a functor.
Coecke et al.~\cite{CoeckeEtAl13} make a meta-analogy between their models of natural language and \emph{topological quantum field theories} (TQFTs).
Intuitively, there is an analogy between regions of spacetime and quantum processes: both can be composed either in sequence or in parallel.
TQFTs formalise this analogy: they assign a quantum system to each region of space and a quantum process to each region of spacetime, in a way that respects sequential and parallel composition.
In the same structure-preserving way, DisCoCat models assign a vector space to each grammatical type and a linear map to each grammatical derivation.
Both TQFTs and DisCoCat can be given a one-sentence definition in terms of category theory: they are examples of \emph{functors into the category of vector spaces}.

How can the same piece of general abstract nonsense (category theory's nickname) apply to both quantum gravity and natural language processing?
And how can this nonsense be of any help in the implementation of QNLP algorithms?
This section will answer with a history of category theory and its applications to quantum physics and computational linguistics, from an abstract framework for meta-mathematics to a concrete toolbox for NLP on quantum hardware.
First, a short philosophical digression on the etymology of the words ``functor'' and ``category'' shall bring some light to their divergent meanings in mathematics and linguistics.

The word ``functor'' first appears in Carnap's \emph{Logical syntax of language}~\cite{Carnap37} to describe what would be called a \emph{function symbol} in a modern textbook on first-order logic.
He introduces them as a way to reduce the laws of empirical sciences like physics to the pure syntax of his formal logic, taking the example of a \emph{temperature functor} $T$ such that $T(3) = 5$ means ``the temperature at position 3 is 5''\footnote
{MacLane~\cite{MacLane38} would later remark that Carnap's formal language cannot express the coordinate system for positions, nor the scale in which temperature is measured.}.
This meaning has then drifted to become synonymous with \emph{function words} such as ``such'', ``as'', ``with'', etc. These words do not refer to anything in the world but serve as the grammatical glue between the \emph{lexical words} that describe things and actions.
They represent less than one thousandth of our vocabulary but nearly half of the words we speak~\cite{ChungPennebaker07}.

Categories (from the ancient Greek \emph{\foreignlanguage{greek}{κατηγορία}}, ``that which can be said'') have a much older philosophical tradition.
In his \emph{Categories}, Aristotle first makes the distinction between the simple forms of speech (the things that are ``said without any combination'' such as ``man'' or ``arguing'') and the composite ones such as ``a man argued''.
He then classifies the simple, atomic things into ten categories: ``each signifies either substance or quantity or qualification or a relative or where or when or being-in-a-position or having or doing or being-affected''.
A common explanation~\cite{Ryle37} for how Aristotle arrived at such a list is that it comes from the possible \emph{types of questions}: the answer to ``What is it?'' has to be a substance, the answer to ``How much?'' a quantity, etc.
Although he was using language as a tool, his system of categories aims at classifying things in the world, not forms of speech: it was meant as an \emph{ontology}, not a grammar.
In his \emph{Critique of Pure Reason}~\cite{Kant81}, Kant revisits Aristotle's system to classify not the world, but the mind: he defines categories of understanding rather than categories of being.
The idea that every object (whether in the world or in the mind) is an object of a certain type has then become foundational in mathematical logic and Russell's \emph{theory of types}~\cite{Russell03}.
The same idea has also had a great influence in linguistics and especially in the \emph{categorial grammar} tradition initiated by Ajdukiewicz~\cite{Ajdukiewicz35} and Bar-Hillel~\cite{Bar-Hillel53,Bar-Hillel54}, where categories have now become synonymous with \emph{grammatical types}.
As we shall see in section~\ref{subsection:lambek-discocat}, the key innovation from Aristotelian categories to categorial grammars is that the grammatical types now come with some structure: we can compose \emph{atomic categories} together to form complex types, confusingly called \emph{functor categories}.

Independently of their use in linguistics, a series of papers from Eilenberg and MacLane~\cite{EilenbergMacLane42,EilenbergMacLane42a,EilenbergMacLane45} gave categories and functors their current mathematical definition.
Inspired by Aristotle's categories of things and Kant's categories of thoughts, they defined categories as types of \emph{mathematical structures}: sets, groups, spaces, etc.
Their great insight was to focus not on the content of the objects (elements, points, etc.) but on the composition of the \emph{arrows} between them: functions, homomorphisms, continuous maps, etc.
Applying the same insight to categories themselves, what really matters are the arrows between them: \emph{functors}, maps from one category to another that preserve the form of arrows.\footnote
{We can play the same game again: what matters are not so much the functors themselves but the \emph{natural transformations} between them, which is what category theory was originally meant to define.
To keep playing that game is to fall in the rabbit hole of infinity category theory~\cite{RiehlVerity16}.}
A prototypical example is Poincaré's construction of the fundamental group of a topological space~\cite{Poincare95}, which can be defined as a functor from the category of (pointed) topological spaces to that of groups: every continuous map between spaces induces a homomorphism between their fundamental groups, in a way that respects composition and identity.
Thus, the abstraction of category theory allowed to formalise the analogies between topology and algebra, proving results about one using methods from the other.
It was then used as a tool for the foundation of algebraic geometry by the school of Grothendieck~\cite{GrothendieckDieudonne60}, which brought the analogy between geometric shapes and algebraic equations to a new level of abstraction and led to the development of \emph{topos theory}.

The establishment of category theory as an independent discipline and as a foundation for mathematics owes much to the work of Lawvere.
His influential Ph.D. thesis~\cite{Lawvere63} on \emph{functorial semantics} set up a framework for model theory where logical theories are categories and their models are functors.
He then undertook the axiomatisation of the category of sets~\cite{Lawvere64} and the category of categories~\cite{Lawvere66}.
The resulting notion of elementary topos~\cite{Lawvere70a} subsumed Grothendieck's definition and emphasised the foundational concept of \emph{adjunction}~\cite{Lawvere69a,Lawvere70}.
``Adjoint functors arise everywhere'' became the slogan of MacLane's classic textbook \emph{Categories for the working mathematician}~\cite{MacLane71}.
Lambek~\cite{Lambek68,Lambek69,Lambek72} used the related notion of \emph{cartesian closed categories} to extend the Curry-Howard correspondance between logic and computation into a trinity with category theory: proofs and programs are arrows, logical formulae and data types are objects.
The discovery of this three-fold connection resulted in a wide range of applications of category theory to theoretical computer science, surveyed in Scott~\cite{Scott00}.

This unification of mathematics, logic and computer science has been followed by a program for the categorical foundations for physics, initiated by Lawvere's topos-theoretic treatment of classical dynamics~\cite{Lawvere79} and continuum physics~\cite{LawvereSchanuel86} with Schanuel.
As we mentioned at the start of this section, the work of Atiyah~\cite{Atiyah88}, Baez and Dolan~\cite{BaezDolan95} on TQFTs showed categories and functors to be essential tools in the grand unification project of quantum gravity~\cite{Baez06}.
This now quaternary analogy between physics, mathematics, logic and computation was popularised by Baez and Stay in their \emph{Rosetta Stone}~\cite{BaezStay10}.
On more concrete grounds, this connection between category theory and quantum physics appeared in Selinger's proposal of a quantum programming language~\cite{Selinger04} and the development of a quantum lambda calculus~\cite{VanTonder04,SelingerValiron06,SelingerEtAl09}.
The same insight blossomed in the school of \emph{categorical quantum mechanics} (CQM) led by Abramsky and Coecke~\cite{AbramskyCoecke04}, where quantum processes are arrows in \emph{compact closed categories}.
This approach culminated in the \emph{ZX calculus} of Coecke and Duncan~\cite{CoeckeDuncan08,CoeckeDuncan11}, a categorical axiomatisation which was proved complete for qubit quantum computing~\cite{JeandelEtAl18a,HadzihasanovicEtAl18}
with applications including error correction \cite{ChancellorEtAl18,GidneyFowler19}, circuit optimisation~\cite{KissingerVanDeWetering20,DuncanEtAl20,DeBeaudrapEtAl20}, compilation \cite{CowtanEtAl20,DeGriendDuncan20} and extraction \cite{BackensEtAl21}.

In quantum computing as well, adjunction is fundamental: it underlies the definition of entanglement and the proof of correctness for the \emph{teleportation protocol}.
Back in 2004 when Coecke first presented this result at the McGill category theory seminar, Lambek immediately pointed out the analogy with his \emph{pregroup grammars}~\cite{Lambek99,Lambek01} where adjunction is the only grammatical rule\footnote
{See \cite{Coecke21} for a first-hand account of this story and a praise of Jim Lambek.}.
Half a century beforehand, the \emph{Lambek calculus}~\cite{Lambek58,Lambek59,Lambek61} revealed an analogy between the derivations in categorial grammars and proof trees in mathematical logic.
He then extended this analogy in \emph{Categorial and categorical grammar}~\cite{Lambek88} where he showed that these grammatical derivations are in fact arrows in \emph{closed monoidal categories} and proposed to cast Montague semantics as a topos-valued functor.
Later, he argued not ``that categories should play a role in linguistics, but rather that they already do''~\cite{Lambek99b}.
Indeed, Hotz~\cite{Hotz66} had already proved that Chomsky's generative grammars were \emph{free monoidal categories}, although his original German article was never translated to English.
The idea of using functors as semantics had appeared implicitly in Knuth~\cite{Knuth68a} in the context-free case and was made explicit by Benson~\cite{Benson70a} for unrestricted grammars.
From this categorical formulation of linguistics, Lambek~\cite{Lambek10} first suggested the analogy between linguistics and physics which is the basis of this thesis: \emph{pregroup reductions as quantum processes}.

It is remarkable that Lambek could foresee QNLP without \emph{string diagrams}\footnote
{String diagrams do not appear in any of Lambek's published work.
Instead, he either uses lines of equations, proof trees or ``underlinks'' for pregroup adjunctions~\cite{Lambek08}.
He admits ``not having had the patience to absorb'' the topological definition of Joyal-Street string diagrams~\cite{Lambek10}.
}, probably the most powerful tool in the hands of the applied category theorist.
They first appeared in another article from Hotz~\cite{Hotz65} as a formalisation of the diagrams commonly used in electronics.
Penrose~\cite{Penrose71} then used the same notation as an informal shortcut for tedious tensor calculations, and later applied it to relativity theory with Rindler~\cite{PenroseRindler84}.
Joyal and Street~\cite{JoyalStreet88,JoyalStreet91,JoyalStreet95} gave the first topological definition of string diagrams and characterised them as the arrows of free monoidal categories.
A generalisation of string diagrams called \emph{proof nets} were introduced by Girard~\cite{Girard87} as a way to free the proofs of his \emph{linear logic} from ``the bureaucracy of syntax'', they were then applied to the Lambek calculus~\cite{Roorda92} and to its multimodal extensions~\cite{MootPuite02}.

At first a piece of mathematical folklore that was hand-drawn on blackboards and rarely included in publications, string diagrams were published at a much bigger scale with the advent of typesetting tools like \LaTeX \ and Ti\emph{k}Z.
Selinger's survey~\cite{Selinger10}, makes the hierarchy of categorical structures (symmetric, compact closed, etc.) correspond to a hierarchy of graphical gadgets (swaps, wire bending, etc.).
In \emph{Picturing Quantum Processes}~\cite{CoeckeKissinger17}, Coecke and Kissinger introduce quantum theory with over a thousand diagrams.
And the list of applications keeps growing:
electronics~\cite{BaezFong15} and chemistry~\cite{BaezPollard17},
control theory~\cite{BaezErbele14} and concurrency~\cite{BonchiEtAl14a},
databases~\cite{BonchiEtAl18} and knowledge representation \cite{Patterson17},
Bayesian inference~\cite{CoeckeSpekkens12,ChoJacobs19} and causality~\cite{KissingerUijlen19},
cognition~\cite{BoltEtAl17} and game theory~\cite{GhaniEtAl18},
functional programming \cite{Riley18} and machine learning~\cite{FongEtAl17}.

If they are a great tool for writing scientific papers, string diagrams can also be a powerful data structure for developing software applications:
quantomatic~\cite{KissingerZamdzhiev15} and its successor PyZX~\cite{KissingerVanDeWetering19} perform automatic rewriting of diagrams in the ZX calculus,
globular~\cite{BarEtAl18} and its successor homotopy.io~\cite{ReutterVicary19} are proof assistants for higher category theory,
cartographer~\cite{SobocinskiEtAl19} and catlab~\cite{PattersonEtAl21} implement diagrams in symmetric monoidal categories, which are also implicit in the circuit data structure of the t$|$ket$\rangle$ compiler~\cite{SivarajahEtAl20}.
String diagrams are the main data structure of our QNLP algorithms: we translate the diagrams of sentences into diagrams of quantum circuits.
As none of the existing category theory software was flexible enough, we had to implement our own: DisCoPy~\cite{FeliceEtAl20}, a Python library for computing with functors and diagrams in monoidal categories.
DisCoPy then became the engine underlying lambeq~\cite{KartsaklisEtAl21}, a high-level library for experimental QNLP.
Although its development was driven by the implementation of DisCoCat models on quantum computers, DisCoPy was designed as a general-purpose toolkit for applied category theory.
It is freely available\footnote{\url{https://github.com/oxford-quantum-group/discopy}} (as in free beer and in free speech), reliable (with 100\% code coverage) and extensively documented\footnote{\url{https://discopy.readthedocs.io/}}.

In conclusion, category theory can really be a \emph{theory of anything}: from algebraic geometry and quantum gravity to natural language processing.
There is a striking analogy between category theory and string diagrams as a universal graphical language and the \emph{characteristica universalis} and \emph{calculus ratiocinator} dreamt by Leibniz three hundred years ago, a formal language and computational framework that would be able to express all of mathematics, science and philosophy.
Indeed, not only can categories be tools for the working mathematicians and scientists, they can also be of help to the philosophers.
In the footsteps of Grassmann's \emph{Ausdehnungslehre}~\cite{Grassmann44} and his project of an algebraic formalisation of Hegel, Lawvere~\cite{Lawvere89,Lawvere91,Lawvere92,Lawvere96} set out to formulate Hegelian dialectics in terms of adjunctions.
This led to the ongoing effort of Schreiber, Corfield and their collaborators on the nLab~\cite{SchreiberEtAl21} to translate \emph{Wissenschaft Der Logik}~\cite{Hegel12} in terms of category theory.
Not only can it accommodate the absolute idealism of Hegel, category theory can also deal with the pragmatism of Peirce~\cite{Peirce06},
who developed first-order logic independently of Frege using what was later recognised as the first string diagrams~\cite{BradyTrimble98,BradyTrimble00,MelliesZeilberger16,HaydonSobocinski20}.
String diagrams have also been used to model Wittgenstein's language games as functors from a grammar to a category of games~\cite{HedgesLewis18}.
In recent work~\cite{FeliceEtAl20a}, we applied these functorial language games to question answering, going from philosophy to NLP via category theory.


\section*{Contributions}
\addcontentsline{toc}{section}{Contributions}

The first chapter is an extended version of the DisCoPy paper~\cite{FeliceEtAl20a}.
It emerged from a dialectic teacher-student collaboration with Giovanni de Felice: implementing our own category theory library was a way to teach him Python programming.
Bob Coecke then added the capital letters to the name of DisCoPy.
We list the contributions of each section.

\begin{enumerate}
\item We\footnote
{The ``we'' of this section refers to the author of this thesis.
Although we believe that science is collaboration and that the notion of personal contribution is obsolete, it is in fact required by university regulations: ``Where some part of the thesis is not solely the work of the candidate or has been carried out in collaboration with one or more persons, the candidate shall submit a clear statement of the extent of his or her own contribution.''}
give an introduction to elementary category theory for the Python programmer which is at the same time an introduction to object-oriented programming for the applied category theorist.
This includes an implementation of:
\begin{itemize}
    \item the category $\mathbf{Pyth}$ with Python types as objects and functions as arrows (listing~\ref{listing:Function}),
    \item the category $\mathbf{Mat}_\S$ with natural numbers as objects and matrices with entries in a rig $\S$ as arrows (listing~\ref{listing:matrix}),
    \item free categories (listing~\ref{listing:cat.py}) with quantum circuits as example (\ref{example:Circuit}),
    \item the category $\mathbf{Cat}$ with categories as objects and functors as arrows (listing~\ref{listing:Functor}),
    \item quotient categories (section~\ref{subsection:quotient-categories}),
    \item categories with a dagger structure, i.e. an identity-on-objects contravariant involutive endofunctor, and categories enriched in commutative monoids (section~\ref{subsection:dagger-sums-bubbles}),
    \item categories with bubbles, i.e. arbitrary unary operators on homsets, with the example of neural networks (\ref{example:neural-net}) and propositional logic (\ref{example:propositional-logic}).
\end{itemize}

\item We give an elementary definition of string diagrams for monoidal categories.
Our construction decomposes the free monoidal category construction into three basic steps: 1) a layer endofunctor on the category of monoidal signatures, 2) the free premonoidal category as a free category of layers and 3) the free monoidal category as a quotient by interchangers.
To the best of our knowledge, this \emph{premonoidal approach} had been relegated to mathematical folklore: it was known by those who knew it, yet it never appeared in print.
This includes:
\begin{itemize}
    \item $\mathbf{Pyth}$ with lists of types as objects and tupling as tensor (listing~\ref{listing:monoidal.Function}),
    \item $\mathbf{Tensor}_\S \simeq \mathbf{Mat}_\S$ with lists of natural numbers as objects and Kronecker product as tensor (listing~\ref{listing:tensor}),
    \item free monoidal categories (listing~\ref{listing:monoidal.Diagram}) with quantum circuits as example (\ref{example:circuit-diagrams}),
    \item quotient monoidal categories (listing~\ref{subsection:quotient-monoidal}) with quantum circuit optimisation as example (\ref{example:simplify-circuits}),
    \item monoidal categories with daggers, sums and bubbles (section~\ref{subsection:monoidal-daggers-sums-bubbles}) with the example of post-processed quantum circuits (\ref{example:postprocessed-circuit}) and first-order logic à la Peirce (\ref{example:monoidal-formula}).
\end{itemize}
DisCoPy uses a \emph{point-free} or \emph{tacit programming} style where diagrams are described only by composition and tensor.
We discuss how to go from tacit to explicit programming, defining diagrams using the standard syntax for Python functions (section~\ref{subsection:tacit-to-explicit}).

\item We prove the equivalence between our elementary definition of diagrams in terms of list of layers and the topological definition in terms of \emph{labeled generic progressive plane graphs}.
One side of this equivalence underlies the drawing algorithm of DisCoPy, the other side is the basis of a prototype for an automatic diagram recognition algorithm.
We then discuss how to extend this to non-generic, non-progressive, non-planar, non-graph-like diagrams, which opens the door to the next section.

\item We describe our object-oriented implementation of monoidal categories with extra structure.
The hierarchy of categorical structures (monoidal, closed, rigid, etc.) is encoded in a hierarchy of Python classes and an inheritance mechanism implements the free-forgetful adjunctions between them.
This includes an implementation of:
\begin{itemize}
\item free rigid categories, for which we introduce the \emph{snake removal} algorithm to compute normal forms (section~\ref{subsection:rigid}),
\item the syntax for diagrams in free braided, symmetric, tortile and compact-closed categories (section~\ref{subsection:symmetric}),
\item the syntax for diagrams in free hypergraph categories, i.e. with coherent special commutative Frobenius algebras on each object (section~\ref{subsection:hypergraph}),
\item the syntax for diagrams in free cartesian and cocartesian diagrams (section~\ref{subsection:cartesian}) with $\mathbf{Pyth}$ as an example of a \emph{rig category} with two monoidal structures (listing~\ref{listing:python-co-cartesian}),
\item the free biproduct completion as the category of matrices with arrows as entries (section~\ref{subsection:biproducts}), taking quantum measurements as example (\ref{example:biproduct-measurement}),
\item the syntax for diagrams in closed monoidal categories (section~\ref{subsection:closed}) with currying of functions in $\mathbf{Pyth}$ as example (\ref{example:closed-function}),
\item an implementation of $\mathbf{Pyth}$ as a traced cartesian and cocartesian category (listing~\ref{listing:traced-python}) and $\mathbf{Mat}_\B \simeq \mathbf{FinRel}$ as a traced biproduct category (listing~\ref{listing:traced-matrix}).
\end{itemize}

\item We discuss the relationship between our premonoidal approach and the existing graph-based data structures for diagrams in symmetric monoidal categories.
This includes:
\begin{itemize}
\item a comparison between our definition of \emph{premonoidal diagrams} as lists of layers and the free premonoidal category as a state construction over a monoidal category (section~\ref{subsection:state-construction}),
\item an implementation of \emph{hypergraph diagrams}, i.e. the arrows of free hypergraph categories, and the subcategories of compact, traced and symmetric diagrams (section~\ref{subsection:hypergraph-vs-premonoidal}),
\item an implementation of free sesquicategories (i.e. 2-categories without interchangers) with \emph{coloured diagrams} as 2-cells (listing~\ref{listing:free-sesquicategory}),
\item an implementation of $\mathbf{Cat}$ as a sesquicategory with (not-necessarily-natural) transformations as 2-cells (listing~\ref{listing:Transformation}),
\item an implementation of free monoidal 2-categories with diagrams as 1-cells and rewrites as 2-cells (listing~\ref{listing:free-monoidal-2-category}).
\end{itemize}
\end{enumerate}
The second chapter deals with QNLP, building on joint work with Bob Coecke, Giovanni de Felice and Konstantinos Meichanetzidis \cite{MeichanetzidisEtAl20,CoeckeEtAl20,MeichanetzidisEtAl20a}.
\begin{itemize}
\item Section~\ref{section:NLP} gives a short introduction to formal grammars and ambiguity (\ref{subsection:chomsky}), the Lambek calculus, Montague semantics and DisCocat models (\ref{subsection:lambek-discocat}).
We conclude with a discussion of previous work on anaphora and the quantum complexity of language (\ref{section:anaphora}).
\item Section~\ref{section:discocat-qnlp} defines QNLP models as functors from grammar to quantum circuits and show that any DisCoCat model can be implemented in this way.
We discuss our implementation of classical-quantum channels and mixed quantum circuits (\ref{section:mixed-circuits}) and the use of our snake removal algorithm to reduce both the number of qubits and the amount of post-selection required for QNLP models (\ref{subsection:snake-removal}).
\item We review previous implementations of DisCoCat models and study their relationship with \emph{knowledge graph embeddings} (\ref{subsection:kge}) and  hybrid classical-quantum algorithm to train QNLP models on a question-answering task (\ref{subsection:vqqa}).
The underlying idea of \emph{functorial learning}, i.e. learning structure-preserving functors from diagram-like data, provides a theoretical framework for machine learning on structured data.
\end{itemize}
The last section has been published in joint work with Richie Yeung and Giovanni de Felice~\cite{ToumiEtAl21}.
It introduces \emph{diagrammatic differentiation}, a graphical calculus for computing the gradients of parameterised diagrams which applies to the training of QNLP models but also to functorial learning in general.
\begin{itemize}
\item In section~\ref{2-dual-diagrams}, we generalise the dual number construction from rings to monoidal categories. Dual diagrams are formal sums of a string diagram (the real part) and its derivative with respect to some parameter (the epsilon part).
We use bubbles to encode differentiation of diagrams and express the standard rules of calculus (linearity, product, chain) entirely in terms of diagrams.
\item In section~\ref{2b-differentiating-zx}, we study diagrammatic differentiation for the ZX calculus.
This allows to compute the gradients of linear maps with respect to phase parameters.
\item In section~\ref{3-dual-circuits}, we look at the diagrammatic differentiation of mixed quantum circuits, this yields a definition of the parameter-shift rules used in quantum machine learning.
\item In section~\ref{4-bubbles}, we define the gradient of diagrams with bubbles in terms of the chain rule. This allows to differentiate quantum circuits with neural networks as classical post-processing.
\end{itemize}

\pagebreak

\section*{Publications}
\addcontentsline{toc}{section}{Publications}

The material presented in this thesis builds on the following publications.
\begin{itemize}[label={}]
\item \enumcite{FeliceEtAl19}\vspace{-10pt}
\item \enumcite{MeichanetzidisEtAl20a}\vspace{-10pt}
\item \enumcite{FeliceEtAl20a}\vspace{-10pt}
\item \enumcite{CoeckeEtAl20}\vspace{-10pt}
\item \enumcite{MeichanetzidisEtAl20}\vspace{-10pt}
\item \enumcite{KartsaklisEtAl21}\vspace{-10pt}
\item \enumcite{ToumiEtAl21}
\item \enumcite{ToumiKoziell-Pipe21}\vspace{5pt}
\end{itemize}\vspace{10pt}
During his DPhil, the author has also published the following articles.
\begin{itemize}[label={}]
\item \enumcite{BorosEtAl19}\vspace{-10pt}
\item \enumcite{FeliceEtAl20}\vspace{-10pt}
\item \enumcite{ShieblerEtAl20}\vspace{-10pt}
\item \enumcite{CoeckeEtAl21}\vspace{-10pt}
\item \enumcite{McPheatEtAl21}\vspace{-10pt}
\end{itemize}

\pagebreak

\section*{Outreach}
\addcontentsline{toc}{section}{Outreach}

The content of this thesis has also been the subject of science popularisation aimed at a wide audience.

\begin{itemize}
\item A blog post summarising our first experiment and two podcasts with long discussions on the topic.
\end{itemize}
\begin{itemize}[label={}]
\item \enumcite{CoeckeEtAl20b}\vspace{-10pt}
\item \enumcite{FuturatiPodcast21}\vspace{-10pt}
\item \enumcite{MachineLearningStreetTalk21}\vspace{-10pt}
\end{itemize}

\begin{itemize}
\item Two invited lectures at the \emph{Compositional Systems and Methods} group in TalTech, Estonia.
Lecture notes are available as Jupyter~\cite{KluyverEtAl16} notebooks.
\end{itemize}
\begin{itemize}[label={}]
\item \enumcite{ToumiFelice21}\vspace{-10pt}
\item \enumcite{ToumiFelice21a}\vspace{-10pt}
\end{itemize}

\begin{itemize}
\item A presentation at an educational event for programmers and data scientists.
\end{itemize}
\begin{itemize}[label={}]
\item \enumcite{PyD20}
\end{itemize}

\begin{itemize}
\item A hackathon where students implemented QNLP experiments with DisCoPy.
\end{itemize}
\begin{itemize}[label={}]
\item \enumcite{Molina21}\vspace{-10pt}
\end{itemize}

\begin{itemize}
\item A press release explaining QNLP in plain English.
\end{itemize}
\begin{itemize}[label={}]
\item \enumcite{TQI20}\vspace{-10pt}
\end{itemize}

\begin{itemize}
\item Press releases introducing lambeq~\cite{KartsaklisEtAl21} to a business audience.
\end{itemize}
\begin{itemize}[label={}]
\item \enumcite{HPC21}\vspace{-10pt}
\item \enumcite{Smith-Goodson21}\vspace{-10pt}
\end{itemize}


\chapter{DisCoPy: Python for the applied category theorist}\label{chapter:discopy}

Python has become the programming language of choice for most applications in both natural language processing (e.g. Stanford NLP~\cite{ManningEtAl14}, NLTK~\cite{LoperBird02} and SpaCy~\cite{HonnibalMontani17}) and quantum computing (with development kits like Qiskit~\cite{Cross18} and PennyLane~\cite{BergholmEtAl20} and interfaces to compilers like pytket~\cite{SivarajahEtAl20}).
Thus, it was the obvious choice of language for an implementation of QNLP.
However, unlike functional programming languages like Haskell, Python has little support for category theory.
Indeed, before the release of DisCoPy, the only existing Python framework for category theory was a module of SymPy~\cite{MeurerEtAl17} that can draw commutative diagrams in finite categories.
Hence, the first step in implementing QNLP was to develop our own framework for applied category theory in Python: DisCoPy.
Its main feature are the drawing of string diagrams (e.g. the grammatical structure of sentences) and the application of functors (e.g. to quantum circuits, either executed on quantum hardware or classically simulated).

String diagrams have become the lingua franca of applied category theory.
However, the definitions one can find in the literature usually fall into one of two extremes: either definitions by general abstract nonsense or definitions by example and appeal to intuition.
On one side of the spectrum, the standard technical reference has become the \emph{Geometry of tensor calculus}~\cite{JoyalStreet91} where Joyal and Street define string diagrams as equivalence classes of labeled topological graphs embedded in the plane and then characterise them as the arrows of free monoidal categories.
On the other, \emph{Picturing quantum processes}~\cite{CoeckeKissinger17} contains over a thousand string diagrams but their formal definition as well as any mention of category theory are relegated to mere appendices.

The aims of this chapter are three-fold: 1) it gives an overview of the DisCoPy package and its design principles, 2) it introduces elementary category theory to the Python programmer and 3) it introduces object-oriented programming to the applied category theorist.
The first section introduces categories and functors with no mathematical prerequisites apart from sets and monoids.
The second section introduces monoidal categories, defining string diagrams from first principles.
The third section defines the drawing and reading algorithms for string diagrams, which arise as the two sides of the equivalence between the premonoidal and the topological definitions.
The fourth section introduces monoidal categories with extra structure and the inheritence mechanism which implements this hierarchy of structure.
The fifth section gives the category theoretic foundations for our definition of diagrams, which we call the premonoidal approach, it discusses the relationship between this approach and the exisiting graph-based data structures for diagrams in symmetric monoidal categories.

\paragraph{Too long; didn't read.}
We provide a brief summary for the reader who wishes to skip the category theory of this chapter and go straight to the QNLP of chapter~\ref{chapter-2:qnlp}.
String diagrams are defined with respect to a \emph{monoidal signature} $\Sigma$: a set of \emph{objects} $\Sigma_0$, a set of \emph{boxes} $\Sigma_1$ and a pair of functions $\dom, \cod : \Sigma_1 \to \Sigma_0^\star$ assigning input and output \emph{types} (i.e. lists of objects) to each box.
A \emph{layer} is a triple $(u, f, v) \in L(\Sigma) = \Sigma_0^\star \times \Sigma_1 \times \Sigma_0^\star$ of a box $f \in \Sigma_1$ with types $u \in \Sigma_0^\star$ on the left and $v \in \Sigma_0^\star$ on the right, where we define $\dom(x, f, y) = x \dom(f) y$ and $\cod(x, f, y) = x \cod(f) y$.
Finally, a \emph{diagram} $d$ is given by a domain $\dom(d) \in \Sigma_0^\star$ and a list of layers $d_1, \dots, d_n \in L(\Sigma)$ such that $\dom(d_1) = \dom(d)$ and $\dom(d_{i + 1}) = \cod(d_i)$ for $i \leq n$.

Given a type $x \in \Sigma_0^\star$, we define the \emph{identity} diagram $\id(x)$ with $\dom(\id(x)) = x$ and an empty list of layers.
Given two diagrams $d$ and $d'$ with $\cod(d) = \dom(d')$, we define their \emph{composition} $d \fcmp d'$ by concatenating of their layers.
Given a diagram $d$ and a type $x \in \Sigma_0^\star$, we define the left and right \emph{whiskering} $x \otimes d$ and $d \otimes x$ by concatenating $x$ to the left and right of each layer in $d$.
Every diagram can be written in terms of boxes, identity, composition and whiskering.
Given two diagrams $d$ and $d'$, we can define their \emph{tensor} as $d \otimes d' = d \otimes \dom(d') \ \fcmp \ \cod(d) \otimes d'$.
The \emph{interchanger} relation induced by $d \otimes \dom(d') \ \fcmp \ \cod(d) \otimes d' \s \sim \s \dom(d) \otimes d' \ \fcmp \ d \otimes \cod(d')$ relates this biased definition of tensor to the one in the opposite direction.

A \emph{monoidal category}\footnote
{What we call a monoidal category technically is a \emph{coloured PRO}, see remark~\ref{remark:coloured-PRO}.} is a monoidal signature $C$ together with an identity function $\id : C_0 \to C_1$, a (partial) composition function $\then : C_1 \times C_1 \to C_1$ and a tensor function $\tensor : C_1 \times C_1 \to C_1$ subject to some axioms spelled out in section~\ref{section:monoidal}.
Diagrams up to interchanger are the \emph{free monoidal category} $C_\Sigma$ generated by the signature $\Sigma$ (see section~\ref{subsection:free-monoidal}).
In practice, this means that we can define a \emph{monoidal functor} $F : C_\Sigma \to D$ as a \emph{signature homorphism} $F : \Sigma \to D$, i.e. a pair of functions $F_0 : \Sigma_0 \to D_0$ and $F_1 : \Sigma_1 \to D_1$ compatible with $\dom$ and $\cod$.
Intuitively, once we have specified the interpretation of each object and each box, the interpretation of every diagram is fixed.
If we take $\Sigma$ to encode the rules of our grammar and $D$ to be a monoidal category of quantum circuits, we get our definition of QNLP models: they are monoidal functors $F : C_\Sigma \to D$.


\section{Categories in Python} \label{section:cat}

What are categories and how can they be useful to the Python programmer?
This section will answer this question by taking the standard mathematical definitions and breaking them into \emph{data}, which can be translated into Python code, and \emph{axioms}, which cannot be formally verified in Python, but can be translated into test cases.
The data for a category is given by a tuple $C = (C_0, C_1, \dom, \cod, \id, \then)$ where:
\begin{itemize}
\item $C_0$ and $C_1$ are classes of \emph{objects} and \emph{arrows} respectively,
\item $\dom, \cod : C_1 \to C_0$ are functions called \emph{domain} and \emph{codomain},
\item $\id : C_0 \to C_1$ is a function called \emph{identity},
\item $\then : C_1 \times C_1 \to C_1$ is a partial function called \emph{composition}, denoted by $(\fcmp)$.
\end{itemize}
Given two objects $x, y \in C_0$, the set\footnote
{We assume this is a set rather than a proper class, i.e. we work with \emph{locally small} categories.}
$C(x, y) = \{f \in C_1 \s \vert \s \dom(f), \cod(f) = x, y \}$ is called a \emph{homset} and we write $f : x \to y$ whenever $f \in C(x, y)$.
We denote the composition $\then(f, g)$ by $f \fcmp g$, translated to \py{f >> g} or \py{g << f} in Python.
The axioms for the category $C$ are the following:
\begin{itemize}
\item $\id(x) : x \to x$ for all objects $x \in C_0$,
\item for all arrows $f, g \in C_1$, the composition $f \fcmp g$ is defined iff $\cod(f) = \dom(g)$, moreover we have $f \fcmp g : \dom(f) \to \cod(g)$,
\item $\id(\dom(f)) \fcmp f = f = f \fcmp \id(\cod(f))$ for all arrows $f \in C_1$,
\item $f \fcmp (g \fcmp h) = (f \fcmp g) \fcmp h$ whenever either side is defined for $f, g, h \in C_1$.
\end{itemize}

Note that we play with the overloaded meaning of the word \emph{class}: we use it to mean both a mathematical collection that need not be a set, and a Python class with its methods and attributes.
Reading it in the latter sense, $\dom$ and $\cod$ are \emph{attributes} of the arrow class, $\then$ is a \emph{method}, $\id$ is a \emph{static method}.
Thus, the Python implementation of a category is nothing but a pair of classes \py{Ob} and \py{Arrow} for objects and arrows, together with four methods $\dom, \cod, \then$ and $\id$ .
The \py{Category} class is nothing but a named tuple with two attributes \py{ob} and \py{ar} for its object and arrow class respectively.
The axioms can be implemented as (necessarily non-exhaustive) software tests, however Python has no formal semantics so there is no hope to formally verify them.

The data for a \emph{functor} $F : C \to D$ between two categories $C$ and $D$ is given by a pair of overloaded functions $F : C_0 \to D_0$ and $F : C_1 \to D_1$ such that:
\begin{itemize}
    \item $F(\dom(f)) = \dom(F(f))$ and $F(\cod(f)) = \cod(F(f))$ for all $f \in C_1$,
    \item $F(\id(x)) = \id(F(x))$ and $F(f \fcmp g) = F(f) \fcmp F(g)$ for all $x \in C_0$ and $f, g \in C_1$.
\end{itemize}
Thus, implementing a functor in Python amounts to implementing a class with the magic method \py{__call__} of the appropriate type, and then implementing software tests to check that the axioms hold.

The data for a \emph{transformation} $\alpha : F \to G$ between two parallel functors $F, G : C \to D$ is given by a function from objects $x \in C_0$ to components $\alpha(x) : F(x) \to G(x)$ in $D$.
A \emph{natural transformation} is one where $\alpha(x) \fcmp G(f) = F(f) \fcmp \alpha(y)$ for all arrows $f : x \to y$ in $C$.
Again, implementing a transformation amounts to implementing a class with a \py{__call__} method of the appropriate type, checking that a transformation is natural cannot be done formally in Python.
The class templates listed below summarise the required data for categories, functors and transformations.

\begin{python}\label{listing:abstract classes}
{\normalfont Class templates for categories, functors and transformations.}

\begin{minted}{python}
from __future__ import annotations
from dataclasses import dataclass
from typing import overload

class Ob: ...

class Arrow:
    dom: Ob
    cod: Ob

    @staticmethod
    def id(x: Ob) -> Arrow: ...
    def then(self, other: Arrow) -> Arrow: ...

@dataclass
class Category:
    ob: type = Ob
    ar: type = Arrow

class Functor:
    dom: Category
    cod: Category

    @overload
    def __call__(self, x: Ob) -> Ob: ...

    @overload
    def __call__(self, f: Arrow) -> Arrow: ...

class Transformation:
    dom: Functor
    cod: Functor

    def __call__(self, x: Ob) -> Arrow: ...
\end{minted}
\end{python}

\begin{remark}
Throughout the thesis we use the postponed evaluation of annotations introduced in Python \py{3.7}~\cite{Langa17}.
Python cannot statically check that arrow composition is well-typed as this would require some form of dependent types, the best we can do is raise an \py{AssertionError} at runtime.
\end{remark}

\begin{example}
When the class of objects and arrows are in fact sets, $C$ is called a \emph{small category}.
For example, the category $\mathbf{FinSet}$ has the set of all finite sets as objects and the set of all functions between them as arrows.
This time equality of functions between finite sets is decidable, so we can write unit tests that check that the axioms hold on specific examples.
\end{example}

\begin{example}
When the class of objects and arrows are finite sets, we can draw the category as a directed multigraph with objects as nodes and arrows as edges, together with the list of equations between paths.
A functor $F : C \to D$ from such a finite category $C$ is called a \emph{commutative diagram} in $D$.
One commutative diagram can state a large number of equations, which can be read by \emph{diagram chasing}.
\end{example}

\begin{example}
A monoid is the same as a category with one object, i.e. every arrow (element) can be composed with (multiplied by) every other.
A preorder, i.e. a set with a reflexive transitive relation, is the same as a category with at most one arrow $x \leq y$ between any two objects $x$ and $y$.
Functors between monoids are the same as homomorphisms, functors between preorders are monotone functions.
\end{example}

\begin{example}
Just about any class of mathematical structures will be the objects of a category with the transformations between them as arrows: the category $\mathbf{Set}$ of sets and functions, the category $\mathbf{Mon}$ of monoids and homomorphisms, the category $\mathbf{Preord}$ of preorders and monotone functions, the category $\mathbf{Cat}$ of small categories and functors, etc.
There are embedding (i.e. injective on objects and arrows) functors from $\mathbf{Mon}$ and $\mathbf{Preord}$ to $\mathbf{Cat}$, i.e. preorders and monoids form a subcategory of $\mathbf{Cat}$.
There is a functor from $\mathbf{Cat}$ to $\mathbf{Preord}$ called the preorder collapse which sends a category $C$ to the preorder given by $x \leq y$ iff there is an arrow $f \in C(x, y)$, i.e. we forget the difference between parallel arrows.
There is a faithful (i.e. injective on homsets) functor $U : \mathbf{Mon} \to \mathbf{Set}$ called the \emph{forgetful functor} which sends monoids to their underlying set and homomorphisms to functions.
\end{example}

\begin{example}
In the same way that there is a set $Y^X$ of functions $X \to Y$ for any two sets $X$ and $Y$, for any two categories $C$ and $D$ there is a category $D^C$ with functors $C \to D$ as objects and natural transformations as arrows.
\end{example}

\begin{example}\label{ex:python categories}
We can define the category $\mathbf{Pyth}$ with objects the class of all Python types and arrows the class of all Python functions.
Domain and codomain could be extracted from type annotations, but instead we implement a class \py{Function} with attributes \py{inside: Callable} as well as \py{dom: Ty} and \py{cod: Ty}.
Identity and composition is given by \py{lambda x: x} and \py{lambda f, g:} \py{lambda x:} \py{g(f(x)))}.

As discussed by Milewski~\cite{Milewski14} in the case of Haskell, endofunctors $\mathbf{Pyth} \to \mathbf{Pyth}$ can be thought of as data containers.
For example, we can define a $\mathbf{List}$ functor which sends a type \py{t} to \py{List[t]} and a function \py{f} to \py{lambda *xs: map(f, xs)}.
There is a natural transformation $\eta : \mathbf{Id} \to \mathbf{List}$ from the obvious identity functor, implemented by the built-in function \py{id}.
Its components send objects \py{x : t} of any type \py{t} to the singleton list \py{[x] : List[t]}.
\end{example}

\begin{remark}\label{remark:category-Pyth}
It's not entirely clear what we mean by equality of Python function and hence we can ask whether $\mathbf{Pyth}$ even is a category at all.
We can define a notion of \emph{contextual equivalence}, an instance of Leibniz's \emph{identity of indiscernibles}: two functions are equal if they are interchangeable in all observable contexts.
However, the associativity and unitality axioms may fail in the presence of non-terminating programs.
In the same informal way as in ``platonic'' $\mathbf{Hask}$~\cite{HaskellWikiContributors12}, we may think of a ``platonic'' $\mathbf{Pyth}$, a subset of Python functions where we exclude non-termination and where the axioms of categories hold.
See Danielsson et al. \emph{Fast and loose reasoning is morally correct}~\cite{DanielssonEtAl06} for a theoretical justification of such informal reasoning.
\end{remark}

\begin{python}
{\normalfont Syntactic sugar for composition.}

\py{Composable} implements the syntactic sugar \py{>>} and \py{<<} for composition in diagrammatic order and its opposite, as well as \py{n * self} for $n$-fold composition.
It is an abstract class, i.e. it is subclassed but never instantiated.
The higher-order function \py{inductive} takes a binary method and makes it $n$-ary by using a left fold, we use it as a decorator.

\begin{minted}{python}
class Composable:
    __rshift__ = __llshift__ = lambda self, other: self.then(other)
    __lshift__ = __lrshift__ = lambda self, other: other.then(self)

def inductive(method):
    def result(self, *others):
        if not others: return self
        if len(others) == 1: return method(self, others[0])
        if len(others) > 1: return result(method(self, others[0]), *others[1:])
    return result
\end{minted}
\end{python}

\begin{python}\label{listing:Function}
{\normalfont Implementation of the category $\mathbf{Pyth}$ with \py{type} as objects and \py{Function} as arrows.}

\begin{minted}{python}
from typing import Callable

@dataclass
class Function(Composable):
    inside: Callable
    dom: type
    cod: type

    @staticmethod
    def id(dom: type) -> Function:
        return Function(lambda x: x, dom, dom)

    @inductive
    def then(self, other: Function) -> Function:
        assert self.cod == other.dom
        return Function(lambda xs: other(self(xs)), self.dom, other.cod)

    def __call__(self, x):
        return self.inside(x)
\end{minted}
\end{python}

\begin{example}
The following commutative diagram denotes a functor $3 \to \mathbf{Pyth}$ from the finite category $3$ with three objects $\{ 0, 1, 2 \}$ and three non-identity arrows $f : 0 \to 1, g : 1 \to 2$ and $h : 0 \to 2$, with the only non-trivial composition $f \fcmp g = h$.
\[ \begin{tikzcd}
\py{int}
\ar{rrrrrr}{\py{lambda n: n * (n - 1) // 2}}
\ar{rrrd}[']{\py{range}}
& & & & & & \py{int} \\
& & & \py{Iterable}
\ar{urrr}[']{\py{sum}} & & &
\end{tikzcd}
\]
It is read as the equation \py{sum(range(n)) = n * (n - 1) // 2}.
\end{example}

\begin{example}
The category $\mathbf{Mat}_\S$ has natural numbers as objects and $n \times m$ matrices with values in $\S$ as arrows $n \to m$.
The identity and composition are given by the identity matrix and matrix multiplication respectively.
In order for matrix multiplication to be well-defined and for $\mathbf{Mat}_\S$ to be a category, the scalars $\S$ should have at least the structure of a \emph{rig} (a riNg without Negatives, also called a \emph{semiring}): a pair of monoids $(\S, +, 0)$ and $(\S, \times, 1)$ with the first one commutative and the second a homomorphism for the first, i.e. $a \times 0 = 0 = 0 \times a$ and $(a + b) \times (c + d) = a c + a d + b c + b d$.

The category $\mathbf{Mat}_\C$ is equivalent to the category of finite dimensional vector spaces and linear maps.
When the scalars are Booleans with disjunction and conjunction as addition and multiplication, the category $\mathbf{Mat}_\B$ is equivalent to the category of finite sets and relations.
There is a faithful functor $\mathbf{FinSet} \to \mathbf{Mat}_\B$ which sends finite sets to their cardinality and functions to their graph.
\end{example}

\begin{python}\label{listing:matrix}
{\normalfont Implementation of $\mathbf{Mat}_\S$ with \py{int} as objects and \py{Matrix[dtype]} as arrows.}

\begin{minted}{python}
from typing import Number

class Matrix(Composable):
    dtype = int

    dom: int
    cod: int
    inside: list[list[dtype]]

    def __class_getitem__(cls, dtype: type):
        class C(cls): pass
        C.dtype = dtype
        C.__name__ = C.__qualname__ = "{}[{}]".format(
            cls.__name__, dtype.__name__)
        return C

    def __init__(self, inside: list[list[Number]], dom: int, cod: int):
        assert len(inside) == dom and all(len(row) == cod for row in inside)
        self.inside, self.dom, self.cod =\
            [list(map(self.dtype, row)) for row in inside], dom, cod

    def __eq__(self, other):
        if not isinstance(other, Matrix):
            return self.dom == self.cod == 1 and self.inside[0][0] == other
        return (self.dtype, self.inside, self.dom, self.cod)\
            == (other.dtype, other.inside, other.dom, other.cod)

    @classmethod
    def id(cls, x: int) -> Matrix:
        return cls([[i == j for i in range(x)] for j in range(x)], x, x)

    @inductive
    def then(self, other: Matrix) -> Matrix:
        assert self.dtype == other.dtype and self.cod == other.dom
        inside = [[sum(
            self.inside[i][j] * other.inside[j][k] for j in range(other.dom))
            for k in range(other.cod)] for i in range(self.dom)]
        return type(self)(inside, self.dom, other.cod)

    def __getitem__(self, key: int | tuple[int, ...]) -> Matrix:
        key = key if isinstance(key, tuple) else (key, )
        inside = [[self.inside[i][j] for i in key for j in range(self.cod)]]\
            if len(key) == 1 else [[self.inside[i][j] for i, j in [key]]]
        dom, cod = 1, self.cod if len(key) == 1 else 1
        return type(self)(inside, dom, cod)

for converter in (bool, int, float, complex):
    def method(self):  # Downcasting a 1 by 1 Matrix to a scalar.
        assert self.dom == self.cod == 1
        return converter(self.inside[0][0])
    setattr(Matrix, "__{}__".format(converter.__name__), method)
\end{minted}

Subscriptable types such as \py{list[list[int]]} implemented by the magic method \py{__class_getitem__} are a new feature of Python \py{3.10}~\cite{Levkivskyi17}.
By default, we fix \py{Matrix = Matrix[int]}.
We can get Boolean, real and complex matrices with \py{Matrix[bool]}, \py{Matrix[float]} and \py{Matrix[complex]} respectively.
Note that this implementation is not meant to be efficient, rather it helps in making the thesis self-contained.
As we will mention in section~\ref{subsection:hypergraph-vs-premonoidal}, DisCoPy interfaces with NumPy~\cite{VanDerWaltEtAl11} for efficient matrix multiplication.
\end{python}

\begin{example}
The category $\mathbf{Circ}$ has natural numbers as objects and $n$-qubit quantum circuits as arrows $n \to n$.
There is a functor $\mathtt{eval} : \mathbf{Circ} \to \mathbf{Mat}_\C$ which sends $n$ qubits to $2^n$ dimensions and evaluates each circuit to its unitary matrix.
\end{example}

\subsection{Free categories}\label{subsection:free-categories}

The main principles behind the implementation of DisCoPy follow from the concept of a \emph{free object}.
Let's start from a simple example.
Given a set $X$, we can construct a monoid $X^\star$ with underlying set $\coprod_{n \in \N} X^n$ the set of all finite lists with elements in $X$.
The associative multiplication is given by list concatenation $X^m \times X^n \to X^{m + n}$ and the unit is given by the empty list denoted $1 \in X^0$.
Given a function $f : X \to Y$, we can construct a homomorphism $f^\star : X^\star \to Y^\star$ defined by element-wise application of $f$ (this is what the built-in \py{map} does in Python).
We can easily check that $(f \fcmp g)^\star = f^\star \fcmp g^\star$ and $(\id_X)^\star = \id_{X^\star}$.
Thus, we have defined a functor $F : \mathbf{Set} \to \mathbf{Mon}$.

Why is this functor so special? Because it is the \emph{left adjoint} to the forgetful functor $U : \mathbf{Mon} \to \mathbf{Set}$.
An \emph{adjunction} $F \dashv U$ between two functors $F : C \to D$ and $U : D \to C$ is a pair of natural transformations $\eta : \id_C \to F \fcmp U$ and $\epsilon : U \fcmp F \to \id_D$ called the \emph{unit} and \emph{counit} respectively.
In the case of lists, we already mentioned the unit in example~\ref{ex:python categories}: it is the function that sends every object to a singleton list.
For a monoid $M$, the counit $\epsilon(M) : F(U(M)) \to M$ is the monoid homomorphism that takes lists of elements in $M$ and multiplies them.
We can easily check that these two transformations are indeed natural, thus we get that \emph{lists are free monoids}.
This may be taken as a mathematical explanation for why lists are so ubiquitous in programming.
Another equivalent definition of adjunction is in terms of an isomorphism $C(x, U(y)) \simeq D(F(x), y)$ which is natural\footnote
{The isomorphism $C(x, U(y)) \simeq D(F(x), y)$ is natural in $x$ if it is a natural transformation between the two functors $C(-, U(y)), D(F(-), y) : C \to \mathbf{Set}$.}
in $x \in C_0$ and $y \in D_0$.
In the adjunction for lists, functions $X \to U(M)$ from a set $X$ to the underlying set of a monoid $M$ are in a natural one-to-one correspondence with monoid homomorphisms $X^\star \to M$.
To define a homomorphism from a free monoid, it is sufficient to define the image of each generating element.

Now we want to play the same game with categories instead of monoids.
We can define a forgetful functor $U : \mathbf{Cat} \to \mathbf{Set}$ which sends a small category $C$ to its set of objects $C_0$, and its left adjoint $F : \mathbf{Set} \to \mathbf{Cat}$ which sends a set to the \emph{discrete category} with its elements as objects and only identity arrows.
However, this is a rather boring construction because forgetting the arrows of a categories is too much: the forgetful functor $U$ is not faithful.
Instead, we need to replace the category of sets with the category of \emph{signatures}.
The data for a signature is given by a tuple $\Sigma = (\Sigma_0, \Sigma_1, \dom, \cod)$ where:
\begin{itemize}
    \item $\Sigma_0$ is a set of \emph{generating objects},
    \item $\Sigma_1$ is a set of \emph{generating arrows}, which we will also call \emph{boxes},
    \item $\dom, \cod : \Sigma_1 \to \Sigma_0$ are the domain and codomain.
\end{itemize}
A morphism of signatures $f : \Sigma \to \Sigma'$ is a pair of overloaded functions $f : \Sigma_0 \to \Sigma'_0$ and $f : \Sigma_1 \to \Sigma'_1$ such that $f \fcmp \dom = \dom \fcmp f$ and $f \fcmp \cod = \cod \fcmp f$.
Thus, signatures and their morphisms form a category $\mathbf{Sig}$ and there is a faithful functor $U : \mathbf{Cat} \to \mathbf{Sig}$ which sends a category to its underlying signature: it forgets the identity and composition.
Signatures may be thought of as directed multigraphs \emph{with an attitude}~\cite{NLab}.
Given a signature $\Sigma$, we can define a category $F(\Sigma)$ with nodes as objects and \emph{paths as arrows}.
More precisely, an arrow $f : x \to y$ is given by a length $n \in \N$ and a list $f_1, \dots, f_n \in \Sigma_1$ with $\dom(f_1) = x$, $\cod(f_n) = y$ and $\cod(f_i) = \dom(f_{i + 1})$ for all $i < n$.
Given a morphism of signatures $f : \Sigma \to \Sigma'$, we get a functor $F(f) : F(\Sigma) \to F(\Sigma')$ relabeling boxes in $\Sigma$ by boxes in $\Sigma'$.
Thus, we have defined a functor $F : \mathbf{Sig} \to \mathbf{Cat}$, it remains to show that it indeed forms an adjunction $F \dashv U$.
This is very similar to the monoid case: the unit sends a box in a signature to the path of just itself, the counit sends a path of arrows in a category to their composition.
Equivalently, we have a natural isomorphism $\mathbf{Cat}(F(\Sigma), C) \simeq \mathbf{Sig}(\Sigma, U(C))$: to define a functor $F(\Sigma) \to C$ from a free category is the same as to define a morphism of signatures $\Sigma \to U(C)$.

If lists are such fundamental data structures because they are free monoids, we argue that the arrows of free categories should be just as fundamental: they capture the basic notion of \emph{data pipelines}.
Free categories are implemented in the most basic module of DisCoPy, \py{discopy.cat}, which is sketched in listing~\ref{listing:cat.py}.

\begin{python}~\label{listing:cat.py}
{\normalfont Implementation of the free category $F(\Sigma)$ with $\Sigma_0 = \py{Ob}$ and $\Sigma_1 = \py{Box}$.}

\begin{minted}{python}
@dataclass
class Ob:
    name: str
    __str__ = lambda self: self.name

@dataclass
class Arrow(Composable):
    inside: tuple[Box, ...]
    dom: Ob
    cod: Ob

    @classmethod
    def cast(cls, old: Arrow) -> Arrow:
        return old if isinstance(old, cls) else cls(old.inside, old.dom, old.cod)

    @classmethod
    def id(cls, x: Ob) -> Arrow:
        return cls.cast(Arrow((), x, x))

    def then(self, *others: Arrow) -> Arrow:
        for f, g in zip((self, ) + others, others): assert f.cod == g.dom
        dom, cod = self.dom, others[-1].cod if others else self.cod
        inside = self.inside + sum([other.inside for other in others], ())
        return self.cast(Arrow(inside, dom, cod))

    __len__ = lambda self: len(self.inside)
    __str__ = lambda self: ' >> '.join(map(str, self.inside))\
        if self.inside else '{}.id({})'.format(type(self).__name__, self.dom)

class Box(Arrow):
    def __init__(self, name: str, dom: Ob, cod: Ob):
        self.name = name
        super().__init__((self, ), dom, cod)

    def __eq__(self, other):
        if isinstance(other, Box):
            return (self.name, self.dom, self.cod)\
                == (other.name, other.dom, other.cod)
        return isinstance(other, Arrow) and other.inside == (self, )

    __hash__ = lambda self: hash(repr(self))
    __str__ = lambda self: self.name
    cast = Arrow.cast
\end{minted}
\end{python}

The classes \py{Ob} and \py{Arrow} for objects and arrows are implemented in a straightforward way, using the \py{dataclass} decorator to avoid the bureaucracy of defining initialisation, equality, etc.
We define the method \py{__str__} so that \py{eval(str(f)) == f} for all \py{f: Arrow}, provided that the names of all objects and boxes are in scope.
The attribute \py{inside} holds the list of generating arrows, which we store as an immutable \py{tuple} rather than a mutable \py{list}.
The method \py{Arrow.then} accepts any number of arrows \py{others}, which will prove useful when defining functors.
The \py{Box} class requires more attention: a box \py{f = Box('f', x, y)} is an arrow with the list of just itself as boxes, i.e. \py{f.inside == (f, )}.
For the axiom \py{f >> f.id(y)} \py{== f} \py{== f.id(x) >> f} to hold, we need to make sure that \py{f == Arrow((f, ), x, y)}, i.e. a box is equal to the arrow with just itself as boxes.
The main subtlety in the implementation is the class method \py{cast} which takes an \py{old: Arrow} as input and returns a new member of a given \py{cls}, subclass of \py{Arrow}.
This allows the composition of arrows in a subclass to remain within the subclass, without having to rewrite the method \py{then}.
This means we need to make \py{Arrow.id} a \py{classmethod} as well so that it can call \py{cast} and return an arrow of the appropriate subclass.
We also need to fix \py{Box.cast = Arrow.cast}: when we compose a box then an arrow, we want to return a new arrow object, not a box.

\begin{example}\label{example:Circuit}
We can define \py{Circuit} as a subclass of \py{Arrow} and \py{Gate} as a subclass of \py{Circuit} and \py{Box} defined by a name and a number of qubits.

\begin{minted}{python}
class Circuit(Arrow): pass

class Gate(Box, Circuit):
    cast = Circuit.cast

Id = Circuit.id(Ob('1'))
X, Y, Z, H = [Gate(name, Ob('1'), Ob('1')) for name in "XYZH"]

assert (X >> Y) >> Z == X >> (Y >> Z) and X >> Id == X == Id >> X
assert isinstance(Id, Circuit) and isinstance(X >> Y, Circuit)
\end{minted}
\end{example}

The \py{Functor} class listed in \ref{listing:Functor} has two mappings \py{ob} and \py{ar} as attributes, from objects to objects and from boxes to arrows respectively.
The domain of the functor is implicitly defined as the free category generated by the domain of the \py{ob} and \py{ar} mappings.
The optional arguments \py{dom} and \py{cod} allow to define functors with arbitrary categories as domain and codomain, a \py{Category} is nothing but a pair of types for its objects and arrows.
For now we only use \py{cod} to define the image of identity arrows, otherwise the (co)domain of the functor is defined implicitly as the (co)domain of the \py{ob} and \py{ar} mappings.

We have chosen to implement functors in terms of Python \py{dict} rather than functions mainly because the syntax looked better for small examples.
However, nothing prevents us from making the most of Python's \emph{duck typing}: if it quacks like a \py{dict} and if it has a \py{__getitem__} method, we can use it to define functors like a \py{dict}.
Thus, we can define functors with domains that are not finitely generated, such as the identity functor or more concretely the evaluation functor for quantum gates parameterised by a continuous angle.
An equivalent solution is to subclass \py{Functor} and override its \py{__call__} method directly.
The only downside is that we cannot print, save or check equality for such functors, we can only apply them to objects and arrows.

\begin{python}~\label{listing:Functor}
{\normalfont Implementation of $\mathbf{Cat}$ with \py{Category} as objects and \py{Functor} as arrows.}

\begin{minted}{python}
class DictOrCallable:
    def __class_getitem__(_, source, target):
        return dict[source, target] | Callable[[source], target]

@dataclass
class FakeDict:
    inside: Callable
    __getitem__ = lambda self, key: self.inside(key)

class Functor:
    ob: dictOrCallable[Ob, Ob]
    ar: dictOrCallable[Box, Ar]
    dom: Category = Category(Ob, Arrow)
    cod: Category = Category(Ob, Arrow)

    def __init__(self, ob, ar, dom=None, cod=None):
        dom, cod = dom or type(self).dom, cod or type(self).cod
        ob = ob if hasattr(ob, "__getitem__") else FakeDict(ob)
        ar = ar if hasattr(ar, "__getitem__") else FakeDict(ar)
        self.ob, self.ar, self.dom, self.cod = ob, ar, dom, cod

    def __call__(self, other: Ob | Arrow) -> Ob | Arrow:
        if isinstance(other, Ob):
            return self.ob[other]
        if isinstance(other, Box):
            result = self.ar[other]
            if isinstance(result, self.cod.ar): return result
            # This allows some nice syntactic sugar for the ar mapping.
            return self.cod.ar(result, self(other.dom), self(other.cod))
        if isinstance(other, Arrow):
            base_case = self.cod.ar.id(self(other.dom))
            return base_case.then(*[self(box) for box in other.inside])
        raise TypeError

    @classmethod
    def id(cls, x: Category) -> Functor:
        return cls(lambda obj: obj, lambda box: box, dom=x, cod=x)

    @inductive
    def then(self: Functor, other: Functor) -> Functor:
        assert self.cod == other.dom
        ob, ar = lambda x: other.ob[self.ob[x]], lambda f: other.ar[self.ar[f]]
        return type(self)(ob, ar, self.dom, other.cod)
\end{minted}
\end{python}

\begin{example}
A typical DisCoPy script starts by defining objects and boxes:
\begin{minted}{python}
x, y, z = map(Ob, "xyz")
f, g, h = Box('f', x, y), Box('g', y, z), Box('h', z, x)
\end{minted}
We can define a simple relabeling functor from the free category to itself:
\begin{minted}{python}
F = Functor(
    ob={x: y, y: z, z: x},
    ar={f: g, g: h, h: f})
assert F(f >> g >> h) == F(f) >> F(g) >> F(h) == g >> h >> f
\end{minted}
We can interpret our arrows as Python functions:
\begin{minted}{python}
G = Functor(
    ob={x: int, y: Iterable, z: int},
    ar={f: range, g: sum, h: lambda n: n * (n - 1) // 2},
    cod=Category(type, Function))
assert G(f >> g)(42) == G(h)(42) == 861
\end{minted}
We can interpret our arrows as matrices:
\begin{minted}{python}
H = Functor(
    ob={x: 1, y: 2, z: 2},
    ar={f: [[0, 1]], g: [[0, 1], [1, 0]], h: [[1], [0]]},
    cod=Category(int, Matrix))
assert H(f >> g) == H(h).transpose()
\end{minted}
We can even define functors into $\mathbf{Cat}$, i.e. interpret arrows as functors:
\begin{minted}{python}
I = Functor(
    ob={x: Category(Ob, Arrow), y: Category(Ob, Arrow), z: Category(int, Matrix)},
    ar={f: F, g: H},
    cod=Category(Category, Functor))
assert I(f >> g)(h) == H(F(h)) == H(f)
\end{minted}
\end{example}

\subsection{Quotient categories}\label{subsection:quotient-categories}

After free objects, another concept behind DisCoPy is that of a \emph{quotient object}.
Again, let's start with the example of a monoid $M$.
Suppose we're given a binary relation $R \sub M \times M$, then we can construct a quotient monoid $M / R$ with underlying set the equivalence classes of the smallest congruence generated by $R$.
That is, the smallest relation $(\sim_R) \sub M \times M$ such that:
\begin{itemize}
\item $x \sim_R y$ for all $(x, y) \in R$,
\item $x \sim_R x$ and if $x \sim_R y$ and $y \sim_R z$ then $x \sim_R z$,
\item if $x \sim_R x'$ and $y \sim_R y'$ then $x \times y \sim_R x' \times y'$.
\end{itemize}
The first point says that $R \sub (\sim_R)$.
The second says that $(\sim_R)$ is an equivalence relation.
The third says that $(\sim_R)$ is closed under products, it is equivalent to the substitution axiom: if $x \sim_R y$ then $a x b \sim_R a y b$ for all $a, b \in M$.
Explicitly, the congruence $(\sim_R)$ can be constructed in two steps: first, we define the \emph{rewriting relation} $(\to_R) \sub M \times M$ where $a x b \to_R a y b$ for all $(x, y) \in R$ and $a, b \in M$.
Second, we define $(\sim_R)$ as the \emph{symmetric, reflexive, transitive closure} of the rewriting relation, i.e. two elements $x, y \in M$ are equal in $M / R$ iff they are in the same connected component of the undirected graph induced by $(\to_R) \sub M \times M$.
Now there is a homomorphism $q : M \to M / R$ which sends monoid elements to their equivalence class with the following property: for any homomorphism $f : M \to N$ with $x \sim_R y$ implies $f(x) = f(y)$, there is a unique $f' : M / R \to N$ with $f = q \fcmp f'$.
Intuitively, a homomorphism from a quotient $M / R$ is nothing but a homomorphism from $M$ which respects the axioms $R$.
Up to isomorphism, we can construct any monoid $M$ as the quotient $X^\star / R$ of a free monoid $X^\star$: take $X = U(M)$ and $R = \{ (x y, z) \in X^\star \times X^\star \s \vert \s x \times y = z \in M \}$.

The pair $(X, R \sub X^\star \times X^\star)$ of a set of generating elements $X$ and a binary relation $R$ on its free monoid is called a \emph{presentation} of the monoid $M \simeq X^\star / R$.
Arguably, the most fundamental computational problem is the \emph{word problem for monoids}: given a presentation $(X, R)$ and a pair of lists $x, y \in X^\star$, decide whether $x = y$ in $X^\star / R$.
As mentioned in the introduction, it was shown to be equivalent to Turing's halting problem, and thus undecidable, by Post~\cite{Post47} and Markov~\cite{Markov47}.
The proof is straightforward: we can encode the tape alphabet and the states of a Turing machine in the set $X$ and its transition table into the relation $R$, then whether the machine halts reduces to deciding $x = y$ for $x$ and $y$ the initial and accepting configurations respectively: a proof of equality corresponds precisely to a run of the Turing machine.

The case of quotient categories is similar, only we need to take care of objects now.
Given a category $C$ and a family of binary relations $\{ R_{x,y} \sub C(x, y) \times C(x, y) \}_{x, y \in C_0}$, we can construct a quotient category $C / R$ with equivalence classes as arrows.
There is a functor $Q : C \to C / R$ sending each arrow to its equivalence class, and for any functor $F : C \to D$ with $(f, g) \in R_{x, y}$ implies $F(f) = F(g)$,
there is a unique $F' : C / R \to D$ with $F = Q \fcmp F'$.
Intuitively, a functor from a quotient category $C / R$ is nothing but a functor from $C$ which respects the axioms $R$.
Again, any small category $C$ is isomorphic to the quotient $F(\Sigma) / R$ of a free category $F(\Sigma)$: take $\Sigma = U(C)$ and $R = \{ (f \fcmp g, h) \in F(\Sigma) \times F(\Sigma) \s \vert \s f \fcmp g = h \in C \}$.
The pair $(\Sigma, R \sub \coprod_{x, y \in \Sigma_0} \Sigma(x, y) \times \Sigma(x, y))$ is called a presentation of the category $C \simeq F(\Sigma) / R$.
Since monoids are just categories with one object, the word problem for categories will be just as undecidable as for monoids.

What does it mean to implement a quotient category in Python?
Since presentations of categories are as expressive as Turing machines, we might as well avoid solving the halting problem and just use a Python function to define equality of arrows.
Implementing a quotient category is nothing but implementing a free category and an equality function that respects the axioms of a congruence.
One straightforward way is to define equality of arrows $f, g$ in a free category $F(\Sigma)$ to be the equality of their interpretation $\eval{f} = \eval{g}$ under a functor $\eval{-} : F(\Sigma) \to D$ into a concrete category $D$ where equality is decidable.
Another method is to define a \emph{normal form} method which takes an arrow and returns the representative of its equivalence class, then identity of arrow is identity of their normal forms.

\begin{example} \label{example:1-qubit-presentation}
Take the signature $\Sigma$ with one object $\Sigma_0 = \{ 1 \}$ and four arrows $\Sigma_1  = \{ Z, X, H, -1 \}$ for the Z, X and Hadamard gate and the global $(-1)$ phase.
Let's define the relation $R$ induced by:
\begin{itemize}
    \item $H \fcmp X = Z \fcmp H$ and $Z \fcmp X = (-1) \fcmp X \fcmp Z$,
    \item $f \fcmp f = \id(1)$ and $f \fcmp (-1) = (-1) \fcmp f$ for all $f \in \Sigma_1$.
\end{itemize}
The quotient $F(\Sigma) / R$ is a subcategory of the category $\mathbf{Circ}$ of quantum circuits, it is isomorphic to the quotient induced by the interpretation $\eval{-} : F(\Sigma) \to \mathbf{Mat}_\C$.
Suppose we're given a functor $\mathtt{cost} : F(\Sigma) \to \R^+$, we can define the normal form of a circuit $f$ to be the representative of its equivalence class with the lowest cost.
Thus, deciding equality of circuits reduces to solving circuit optimisation perfectly.
\end{example}

\subsection{Daggers, sums and bubbles}\label{subsection:dagger-sums-bubbles}

We conclude this section by discussing three extra pieces of implementation beyond the basics of category theory: daggers, sums and bubbles.
A \emph{dagger} for a category $C$ can be thought of as a kind of time-reversal for arrows.
More precisely, a dagger is a contravariant endofunctor $\dagger : C \to C^{op}$, i.e. from the category to its opposite with $\dom$ and $\cod$ swapped, which is the identity on objects and an involution, i.e. $(\dagger) \fcmp (\dagger) = \id_\C$.
A $\dagger$-functor is a functor between $\dagger$-categories that commutes with the dagger, thus we get a category $\dagger-\mathbf{Cat}$.
The free $\dagger$-category is constructed as follows.
Define the functor $\dagger : \mathbf{Sig} \to \mathbf{Sig}$ which sends a signature $\Sigma$ to $\dagger(\Sigma)$ with
$\dagger(\Sigma)_0 = \Sigma_0$ and $\dagger(\Sigma)_1 = \{ -1, 1 \} \times \Sigma_1$ with $\dom(b, f) = \cod(f)$ if $b = -1$ else $\dom(f)$ and symmetrically for $\cod$.
Then the free dagger category is the quotient category $F(\dagger(\Sigma)) / R$ for the congruence generated by $(1, f) \fcmp (-1, f) \to_R \id(\dom(f))$ and $(-1, f) \fcmp (1, f) \to_R \id(f.cod)$.

\begin{example}
The conjugate transpose defines a dagger on the category $\mathbf{Mat}_\C$, the adjoint defines a dagger on the category $\mathbf{Circ}$ and the evaluation $\mathbf{Circ} \to \mathbf{Mat}_\C$ is a $\dagger$-functor.
By extension, there is a dagger structure on $\mathbf{Mat}_\S$ for each rig anti-homomorphism $\dagger : \S \to \S$, i.e. a homomorphism for the commutative addition and an anti-homomorphism for the (non-necessarily commutative) product $\dagger(a \times b) = \dagger(b) \times \dagger(a)$.
Thus, when $\S$ is a commutative rig such as the Boolean, $\mathbf{Mat}_\S$ is automatically a $\dagger$-category with the transpose as dagger and the identity as conjugation.
\end{example}

DisCoPy implements free $\dagger$-categories by adding an attribute \py{is_dagger: bool} to boxes and a method \py{Arrow.dagger}, shortened to the postfix operator \py{[::-1]}, which reverses the order of boxes and negates \py{is_dagger} elementwise.
The normal form is computable in linear time but it has not been implemented yet.
In order to implement the syntactic sugar \py{f[::-1] == f.dagger()}, we need to override the \py{__getitem__} method.
In general, DisCoPy defines indexing \py{f[i]} and slicing \py{f[start:stop:step]} so that \py{f[key].inside == f.inside[key]} for any \py{key: int} and any \py{key: slice} with \py{key.step in (-1, 1, None)}.

\begin{python}
{\normalfont Implementation of free $\dagger$-categories and $\dagger$-functors.}
\begin{minted}{python}
class Arrow(cat.Arrow):
    def dagger(self):
        inside = tuple(box.dagger() for box in self.inside[::-1])
        return self.cast(Arrow(inside, self.cod, self.dom))

    def __getitem__(self, key: int | slice) -> Arrow:
        if isinstance(key, slice):
            if key.step == -1:
                inside = tuple(box.dagger() for box in self.inside[key])
                return self.cast(Arrow(inside, self.cod, self.dom))
            if (key.step or 1) != 1:
                raise IndexError
            inside = self.inside[key]
            if not inside:
                if (key.start or 0) >= len(self):
                    return self.id(self.cod)
                if (key.start or 0) <= -len(self):
                    return self.id(self.dom)
                return self.id(self.inside[key.start or 0].dom)
            return self.cast(Arrow(inside, inside[0].dom, inside[-1].cod))
        return self.inside[key]

class Box(cat.Box, Arrow):
    cast = Arrow.cast

    def __init__(self, name: str, dom: Ob, cod: Ob, is_dagger=False):
        self.is_dagger = is_dagger; cat.Box.__init__(self, name, dom, cod)

    def dagger(self):
        return type(self)(self.name, self.cod, self.dom, not self.is_dagger)

class Functor(cat.Functor):
    dom = cod = Category(Ob, Arrow)

    def __call__(self, other):
        if isinstance(other, Box) and other.is_dagger:
            return self(other.dagger()).dagger()
        return super().__call__(other)
\end{minted}
\end{python}

\begin{example}
We can show the dagger is indeed a contravariant endofunctor.

\begin{minted}{python}
x, y, z = map(Ob, "xyz")
f, g = Box('f', x, y), Box('g', y, z)

assert Arrow.id(x)[::-1] == Arrow.id(x)
assert (f >> g)[::-1] == g[::-1] >> f[::-1]
\end{minted}
\end{example}

\begin{python}
{\normalfont Implementation of $\mathbf{Mat}_\S$ as a $\dagger$-category.}

\begin{minted}{python}
def transpose(self: Matrix) -> Matrix:
    inside = [[self[j][i] for j in range(self.dom)] for i in range(self.cod)]
    return type(self)(inside, self.cod, self.dom)

def map(self: Matrix, func: Callable[[Number], Number]) -> Matrix:
    inside = [list(map(func, row)) for row in self.inside]
    return type(self)(inside, self.dom, self.cod)

Matrix.transpose, Matrix.map = transpose, map
Matrix.conjugate = lambda self: self.map(lambda x: x.conjugate())
Matrix.dagger = lambda self: self.conjugate().transpose()
\end{minted}
\end{python}

\begin{example}
We can implement a simulator for $1$-qubit circuits as a $\dagger$-functor.

\begin{minted}{python}
Circuit.eval = lambda self: Functor(
    ob={Ob('1'): 2},
    ar={X: [[0, 1], [1, 0]],
        Y: [[0, -1j], [1j, 0]],
        Z: [[1, 0], [0, -1]],
        H: [[x / sqrt(2) for x in row] for row in [[1, 1], [1, -1]]]},
    cod=Category(int, Matrix[complex]))(self)
\end{minted}

We can check that every circuit is \emph{unitary}, i.e. its dagger is also its inverse.

\begin{minted}{python}
for c in [X, Y, Z, H, X >> Y >> Z >> H]:
    assert (c >> c[::-1]).eval()\
        == Matrix[complex].id(2)\
        == (c[::-1] >> c).eval()
\end{minted}

We can check the equations given in the presentation of example~\ref{example:1-qubit-presentation}.

\begin{minted}{python}
assert (Z >> H).eval() == (H >> X).eval()
assert (Z >> X).eval() == (X >> Z).eval().map(lambda x: -x)
for gate in [H, Z, X]: assert (gate >> gate).eval() == Matrix[complex].id(2)
\end{minted}
\end{example}

A category $C$ is \emph{commutative-monoid-enriched} (CM-enriched) when it comes equipped with a commutative monoid $(+, 0)$ on each homset $C(x, y)$ such that $f \fcmp 0 = 0 = 0 \fcmp f$ and $(f + f') \fcmp (g + g') = f \fcmp g + f \fcmp g' + f' \fcmp g + f' \fcmp g'$ for all arrows $f, g, f', g'$.
A functor $F : C \to D$ between CM-enriched categories is CM-enriched when $F(0) = 0$ and $F(f + g) = F(f) + F(g)$.
For example, the category $\mathbf{Mat}_\S$ is CM-enriched with elementwise addition of matrices.
A commutative-monoid-enriched category with one object is precisely a rig.
Given a signature $\Sigma$, we construct the free CM-enriched category $F^+(\Sigma)$ by taking the free commutative monoid over each homset of $F(\Sigma)$, i.e. arrows $f : x \to y$ in $F^+(\Sigma)$ are \emph{bags} (also called \emph{multisets}) of arrows $f_i : x \to y$ in $F(\Sigma)$.

In DisCoPy, free CM-enriched categories are implemented by \py{Sum}, a subclass of \py{Box} with an attribute \py{terms: list[Arrow]} as well as its own \py{cast} method, which turns an arrow into the sum of just itself.
It is attached to the arrow with \py{Arrow.sum = Sum}, we also override \py{Arrow.then} so that we have \py{f >> (g + h)} \py{== Sum.cast(f) >> (g + h)} for any arrow \py{f}, i.e. .
the composition of an arrow with a sum is the sum of the compositions with its terms.
We define equality so that \py{f == Sum.cast(f)}, equality of bags of terms is implemented as equality of lists sorted by an arbitrary ordering.
DisCoPy functors are commutative-monoid-enriched, i.e. a formal sum of arrows can be interpreted as a concrete sum of matrices.

\begin{python}
{\normalfont Implementation of free sum-enriched categories and functors.}

\begin{minted}{python}
class Arrow(cat.Arrow):
    def __eq__(self, other):
        return other.terms == (self, ) if isinstance(other, Sum)\
            else super().__eq__(other)

    def then(self, other: Arrow) -> Arrow:
        return self.sum.cast(self).then(other) if isinstance(other, Sum)\
            else super().then(other)

    @classmethod
    def zero(cls, dom: Ob, cod: Ob) -> Arrow: return cls.sum((), dom, cod)

    __add__ = lambda self, other: self.sum.cast(self) + other
    __lt__ = lambda self, other: hash(self) < hash(other)  # An arbitrary order.

class Sum(cat.Box, Arrow):
    def __init__(self, terms: tuple[Arrow, ...], dom: Ob, cod: Ob):
        assert all(f.dom == dom and f.cod == cod for f in terms)
        self.terms, name = terms, "Sum({}, {}, [{}])".format(
            dom, cod, ", ".join(map(str, terms)))
        cat.Box.__init__(self, name, dom, cod)

    def __eq__(self, other):
        if isinstance(other, Sum):
            return (self.dom, self.cod, sorted(self.terms))\
                == (other.dom, other.cod, sorted(other.terms))
        return self.terms == (other, )

    def __add__(self, other):
        if not isinstance(other, Sum): return self + self.cast(other)
        return Sum(self.terms + other.terms, self.dom, self.cod)

    @classmethod
    def cast(cls, old: cat.Arrow) -> Sum:
        return old if isinstance(old, cls) else cls((old, ), old.dom, old.cod)

    @inductive
    def then(self, other):
        terms = tuple(f.then(g) for f in self.terms for g in self.cast(other).terms)
        return type(self)(terms, self.dom, other.cod)

    def dagger(self):
        return type(self)(tuple(f.dagger() for f in self.terms), self.cod, self.dom)

    id = lambda x: Sum.cast(Arrow.id(x))

Arrow.sum = Sum

class Functor(cat.Functor):
    dom = cod = Category(Ob, Arrow)

    def __call__(self, other):
        if isinstance(other, Sum):
            unit = self.cod.ar.zero(self(other.dom), self(other.cod))
            return sum([self(f) for f in other.terms], unit)
        return super().__call__(other)
\end{minted}
\end{python}

\begin{python}
{\normalfont Implementation of $\mathbf{Mat}_\S$ as a CM-enriched category.}

\begin{minted}{python}
class Matrix:
    ...
    def __add__(self, other: Matrix) -> Matrix:
        inside = [[x + y for x, y in zip(u, v)]
                  for u, v in zip(self.inside, other.inside)]
        return type(self)(inside, self.dom, self.cod)

    def __radd__(self, other: Number) -> Matrix:
        # We can add a scalar matrix and a number.
        assert self.dom == self.cod == 1
        return self.inside[0][0] + other

    @classmethod
    def zero(cls, dom: int, cod: int) -> Matrix:
        return cls([[0 for _ in range(cod)] for _ in range(dom)], dom, cod)
\end{minted}
\end{python}

We define a \emph{bubble} on a category $C$ as a pair of unary operators $\beta_\dom, \beta_\cod : C_0 \to C_0$ on objects and a unary operator between homsets $\beta : C(x, y) \to C(\beta_\dom(x), \beta_\cod(y))$ for pairs of objects $x, y \in C_0$.
Given a signature $\Sigma$ and a pair $\beta_\dom, \beta_\cod : C_0 \to C_0$, we construct the free category with bubbles $F(\Sigma^\beta)$ by induction on the maximum level $n$ of bubble nesting: take the signature $\Sigma^\beta = \bigcup_{n \in \N} \Sigma^\beta_n$
for $\Sigma^\beta_0 = \Sigma$ and $\Sigma^\beta_{n + 1} = \Sigma + \{ \beta(f) \ \vert \ f \in F(\Sigma^\beta_n) \}$.
That is, box in $\Sigma^\beta$ is a box in $\Sigma^\beta_{n}$ for some $n \in \N$.
A box in $\Sigma^\beta_{n}$ is either a box in $\Sigma$ or an arrow $f : x \to y$ in $F(\Sigma^\beta_{n - 1})$ that we have put inside a bubble $\beta(f) : \beta_\dom(x) \to \beta_\cod(y)$.

\begin{remark}
Bubbles are called \emph{unary operator on homsets} (uooh) in \cite{HaydonSobocinski20} where they are used to encode the sep lines of Peirce's existential graphs, see example~\ref{example:monoidal-formula}.
As we will discuss in section~\ref{section:diag-diff}, they first appeared in Penrose and Rindler \cite{PenroseRindler84} as an informal notation for the covariant derivative.
\end{remark}

\begin{example}
The functorial boxes of Melli\`es \cite{Mellies06} can be thought of as
well-behaved bubbles, i.e. such that the composition of bubbles is the
bubble of the composition. Indeed, a functor $F : \mathbf{C} \to \mathbf{D}$
between two categories $\mathbf{C}$ and $\mathbf{D}$ defines a bubble on the
subcategory of their coproduct $\mathbf{C} \coprod \mathbf{D}$ spanned by $\mathbf{C}$.
\end{example}

\begin{example}\label{example:endofunctor-bubbles}
Any endofunctor $\beta : C \to C$ also defines a bubble, thus we can define a bubble-preserving functor $F(U(C)^\beta) \to C$ which interprets bubbles as functor application.
Any functor between two categories $C$ and $D$ defines a bubble on their disjoint union $C + D$ (i.e. with objects $C_0 + D_0$ and arrows $C_1 + D_1$).
These functor bubbles have also been called \emph{functorial boxes}~\cite{Mellies06}.
\end{example}

\begin{example}
An \emph{exponential rig} is a one-object CM-enriched category $\S$ with a bubble $\exp : \S \to \S$ which is a homomorphism from sum to product, i.e. $\exp(a + b) = \exp(a) \exp(b)$ and $\exp(0) = 1$.
Any rig $\S$ is an exponetial rig by taking $\exp(a) = 1$ for all $a \in \S$.
Non-trivial examples include complex numbers as well as the Boolean rig with negation.
\end{example}

\begin{example}
Matrix exponential is a bubble on the subcategory of square matrices, with the property that $\exp(f + g) = \exp(f) \fcmp \exp(g)$ whenever $f \fcmp g = g \fcmp f$.
Also, any function $\S \to \S$ yields a bubble on $\mathbf{Mat}_\S$ given by element-wise application.
For example, we can define a bubble on the category $\mathbf{Mat}_\B$ of Boolean matrices which sends each matrix $f$ to its entrywise negation $\bar{f}$.
\end{example}

DisCoPy implements free bubbles with \py{Bubble}, a subclass of \py{Box} which we attach to the arrow class with \py{Arrow.bubble = Bubble}.
\py{Bubble} has attributes \py{diagram: Arrow} as well as optional arguments \py{dom: Ob}, \py{cod: Ob} and \py{method: str}.
DisCoPy functors interpret bubbles as the application of \py{method} in the codomain category, by default we send bubbles to bubbles.
The resulting syntax is strictly more expressive than that of free categories alone.
For example, element-wise negation cannot be expressed as a composition: there is no matrix $N : x \to x$ in $\mathbf{Mat}_\B$ such that $N \fcmp f = \bar{f}$ for all $f : x \to y$.
This is also the case for the element-wise application of any non-linear function such as the rectified linear units (ReLU) used in machine learning.
As we will discuss in section~\ref{section:diag-diff}, differentiation of parameterised matrices cannot be expressed as a composition either, but it is a unary operator between homsets, i.e. a bubble.

\begin{python}
{\normalfont Implementation of free categories with bubbles and their functors.}

\begin{minted}{python}
class Bubble(Box):
    method = "bubble"

    def __init__(self, diagram: Arrow, dom=None, cod=None, **params):
        self.diagram = diagram
        self.method = params.pop("method", type(self).method)
        name = "Bubble({}, {}, {})".format(diagram, dom, cod)
        dom, cod = dom or diagram.dom, cod or diagram.cod
        super().__init__(name, dom, cod, **params)

    def dagger(self):
        return type(self)(
            self.diagram, self.dom, self.cod, is_dagger=not self.is_dagger)

    @property
    def is_id_on_objects(self):
        return self.dom == self.diagram.dom and self.cod == self.diagram.cod

Arrow.bubble = lambda self, **kwargs: Bubble(self, **kwargs)

class Functor(cat.Functor):
    def __call__(self, other):
        if isinstance(other, Bubble):
            method = getattr(self.cod.ar, other.method)
            if other.is_id_on_objects:
                return method(self(other.diagram))
            return method(self(other.diagram), self(other.dom), self(other.cod))
        return super().__call__(other)
\end{minted}
\end{python}

\begin{example}\label{example:neural-net}
We can encode the architecture of a neural network as an arrow with sums and bubbles, encoding vector addition and non-linear activation function respectively.
The evaluation of the neural network on some input vector for some parameters is given by the application of a sum-and-bubble-preserving functor into $\mathbf{Mat}_\R$.
The hyper-parameters (i.e. the number of neurons at each layer) are given by the image of the functor on objects.

\begin{minted}{python}
x, y, z = map(Ob, "xyz")

Matrix.ReLU = lambda self: self.map(lambda x: max(x, 0))

class Network(Arrow):
    pass

class ReLU(Bubble, Network):
    method = "ReLU"
    cast = Network.cast

class Box(cat.Box, Network):
    cast = Network.cast

vector, bias = Box('vector', x, y), Box('bias', x, x)
ones, weights = Box('ones', x, y), Box('weights', y, z)
network = ReLU((vector + (bias >> ones)) >> weights)

F = Functor(
    ob={x: 1, y: 4, z: 2},
    ar={vector: [[1.2, -2.3, 3.4, -4.5]],
        bias: [[-3.14]], ones: [[1, 1, 1, 1]],
        weights: [[5.6, -6.7], [7.8, -8.9],
                  [9.0, -0.1], [2.3, -3.4]]},
    dom=Category(Ob, Network),
    cod=Category(int, Matrix[float]))

assert F(network) == F(vector).map(lambda x: x + F(bias))\
                              .then(F(weights)).map(lambda x: max(0, x))
\end{minted}
\end{example}

\begin{example}\label{example:propositional-logic}
We can implement propositional logic with boxes as propositions, composition as conjunction, sum as disjunction and bubble as negation.
The evaluation of a formula in a model corresponds to the application of a sum-and-bubble-preserving functor into $\mathbf{Mat}_\B(1, 1)$.

\begin{minted}{python}
Matrix._not = lambda self: self.map(lambda x: not x)

class Formula(Arrow): pass

class Not(Bubble, Formula):
    method = "_not"

class Proposition(Box, Formula):
    def __init__(self, name):
        Box.__init__(self, name, Ob('x'), Ob('x'))

Not.cast = Proposition.cast = Formula.cast

def model(data: dict[Proposition, bool]):
    return Functor(ob={Ob('x'): 1}, ar={p: [[data[p]]] for p in data},
                   dom=Category(Ob, Formula), cod=Category(int, Matrix[bool]))

p, q = map(Proposition, "pq")
p_implies_q = Not(Not(q) >> p)
not_p_or_q = Not(p) + q

for a, b in itertools.product([0, 1], [0, 1]):
   F = model({p: a, q: b})
   assert F(p_implies_q) == (not (not F(q) and F(p)))\
       == F(not_p_or_q) == (not F(p) or F(q))
\end{minted}
\end{example}



\section{Diagrams in Python} \label{section:monoidal}

In the previous section, we introduced the idea of arrows in free categories as abstract data pipelines and functor application as their evaluation in concrete categories such as $\mathbf{Pyth}$, $\mathbf{Mat}$ or $\mathbf{Circ}$ where the computation happens.
For now, our pipelines are rather basic because they are linear: we cannot express functions of multiple arguments, nor tensors of order higher than 2, nor circuits with multiple qubits in any explicit way.

In this section, we move from the one-dimensional syntax of arrows in free categories to the two-dimensional syntax of \emph{string diagrams}, the arrows of free \emph{monoidal categories}.
The data for a (strict\footnote
{We will assume that our monoidal categories are strict, i.e. the axioms for monoids are equalities rather than natural isomorphisms subject to coherence conditions.}) monoidal category $C$ is that of a category together with:
an object $1 \in C_0$ called the \emph{unit} and a pair of overloaded binary operations called the \emph{tensor} on objects $\otimes : C_0 \times C_0 \to C_0$ and on arrows $\otimes : C_1 \times C_1 \to C_1$, translated to \py{@} in Python.
The axioms for monoidal categories are the following:
\begin{itemize}
\item $(C_0, \otimes, 1)$ and $(C_1, \otimes, \id(1))$ are monoids,
\item the tensor defines a functor $\otimes : C \times C \to C$, i.e. the following \emph{interchange law} $(f \fcmp f') \otimes (g \fcmp g') = (f \otimes g) \fcmp (f' \otimes g')$ holds for all arrows $f, f', g, g' \in C_1$.
\end{itemize}
We will use the following terminology: an object $x$ is called a \emph{system}, an arrow $f : 1 \to x$ from the unit is called a \emph{state} of the system $x$, an arrow $f : x \to 1$ into the unit is called an \emph{effect} of $x$ and an arrow $a : 1 \to 1$ from the unit to itself is called a \emph{scalar}.

A functor $F : C \to D$ between monoidal categories $C$ and $D$ is (strict\footnote
{We will assume that our monoidal functors are strict, i.e. $F(x \otimes y) = F(x) \otimes F(y)$ and $F(1) = 1$ are equalities rather than natural transformations.}) monoidal whenever it is also a monoid homomorphism on objects and arrows.
Thus, monoidal categories themselves form a category $\mathbf{MonCat}$ with monoidal functors as arrows.
A transformation $\alpha : F \to G$ between two monoidal functors $F, G : C \to D$ is monoidal itself when $\alpha(x \otimes y) = \alpha(x) \otimes \alpha(y)$ for all objects $x, y \in C$.

\begin{example}
Every monoid $M$ can also be seen as a discrete monoidal category, i.e. with only identity arrows.
\end{example}

\begin{example}
A monoidal category with one object is a commutative monoid.
Indeed in any monoidal category, the interchange law implies that scalars form a commutative monoid, by the following Eckmann-Hilton argument:
\begin{align*}
a \fcmp b
&\quad = \quad 1 \otimes a \s \fcmp \s b \otimes 1
\quad = \quad (1 \fcmp b) \s \otimes \s (a \fcmp 1)
\quad = \quad b \otimes a
\\
&\quad = \quad (b \fcmp 1) \s \otimes \s (1 \fcmp a)
\quad = \quad b \otimes 1 \s \fcmp \s 1 \otimes a
\quad = \quad b \fcmp a
\end{align*}
\end{example}

\begin{example}
A monoidal category with at most one arrow between any two objects is called a \emph{preordered monoid}.
The functoriality axiom implies that the preorder is in fact a pre-congruence, i.e. $a \leq b$ and $c \leq d$ implies $a \times c \leq b \times d$.
Given the presentation of a monoid $(X, R)$ with $R \sub X^\star \times X^\star$, we can construct a preordered monoid with $(\leq_R) = \bigcup_{n \in \N} R^n \sub X^\star \times X^\star$ the (non-symmetric) reflexive transitive closure of $R$.
Thus, the inequality problem (i.e. given two lists $x, y \in X^\star$ and a presentation $(X, R)$, decide whether $x \leq_R y$) is a generalisation of the word problem for monoids.
\end{example}

\begin{example}
The category $\mathbf{FinSet}$ is monoidal with the singleton $1$ as unit and Cartesian product as tensor.
Again, this is not a strict monoidal category but it is equivalent to one: take the category with natural numbers $m, n \in \N$ as objects and functions $[m] \to [n]$ as arrows for $[n] = \{ 0, 1, \dots, n - 1 \}$.
The states can be identified with elements and the only effect is discarding, i.e. the constant function into the singleton.
$\mathbf{FinSet}$ is also monoidal with the empty set $0$ as unit and disjoint union as tensor.
\end{example}

\begin{example}\label{example:endofunctors are monoidal}
For any category $C$, there is a monoidal category $C^C$ where the objects are enfodunctors with composition as tensor and the arrows are natural transformations $\alpha : F \to F'$, $\beta : G \to G'$ with vertical composition $(\alpha \otimes \beta)(x) : G(F(x)) \to G'(F'(x))$ as tensor.
\end{example}

\begin{example}
The category $\mathbf{Pyth}$ is monoidal with unit \py{()} and \py{tuple[t1, t2]} as the tensor of types \py{t1} and \py{t2}.
Given two functions \py{f} and \py{g}, we can define their tensor \py{f @ g = lambda x, y: f(x), g(y)}.
\end{example}

\begin{python}
{\normalfont Implementation of $\mathbf{Pyth}$ as a (non-strict pre)monoidal category with tuple as tensor.}

\begin{minted}{python}
class Function:
    ...
    def tensor(self, other: Function) -> Function:
        dom, cod = tuple[self.dom, other.dom], tuple[self.cod, other.cod]
        return Function(lambda x, y: (self(x), other(y)), dom, cod)
\end{minted}
\end{python}

\begin{remark}
As discussed in remark~\ref{remark:category-Pyth}, it's not clear whether $\mathbf{Pyth}$ is a category.
It's even less clear whether it can be called a monoidal category, for two reasons.
First, $\mathbf{Pyth}$ is not strict monoidal: \py{(x, (y, z)) != ((x, y), z)} and \py{((), x) != x != (x, ())} are not strictly equal but only naturally isomorphic.
These natural isomorphisms are subject to coherence conditions which make sure that all the ways to rebracket \py{(((x, y), z), w)} into \py{(x, (y, (z, w)))} are the same.
In practice, this bureaucracy of parenthesis does not pose any problem: MacLane's coherence theorem~\cite[VII]{MacLane71} makes sure that every monoidal category is monoidally equivalent\footnote
{An \emph{equivalence} of categories is an adjunction where the unit and counit are in fact natural isomorphisms.
It is a \emph{monoidal equivalence} when they are also monoidal transformations.} to a strict one.

Second, the interchange law only holds for the subcategory of $\mathbf{Pyth}$ with \emph{pure functions} as arrows.
Indeed, if the functions \py{f} and \py{g} are impure (e.g. they call \py{random} or \py{print}) then their tensor \py{f @ g} will depend on the order in which they are evaluated, i.e. \py{f @ id >> id @ g != id @ g >> f @ id}.
As we will discuss in section~\ref{section:premonoidal}, $\mathbf{Pyth}$ is in fact a \emph{premonoidal category}.
The states, i.e. the functions \py{f : () -> t}, can be identified with their value \py{f(): t}.
There is only one pure effect, i.e. a unique pure function \py{f : t -> ()} called \emph{discarding}, and thus a unique pure scalar.
If we take all impure functions into account, the scalars form a non-commutative monoid of side-effects.
\end{remark}

\begin{example}
We can also make $\mathbf{Pyth}$ monoidal with the \emph{tagged union} as tensor on objects and \py{typing.NoReturn} as unit.
Given two types \py{t0, t1}, their tagged union \py{t0 + t1} is the union of the types \py{tuple[0, t0]} and \py{tuple[1, t1]}\footnote
{What we really mean is \py{tuple[Literal[0], t0] | tuple[Literal[1], t1]}.}, i.e. a term \py{(b, x): t0 + t1} is a pair of a Boolean \py{b: bool} and a term \py{x: t0} if \py{b} else a term \py{x: t1}.
Given two functions \py{f, g} we can define their tensor as $\py{lambda b, x:} \py{(b, f(x) if b else g(x))}$.
\end{example}

\begin{example}
The category $\mathbf{Mat}_\S$ is monoidal with addition of natural numbers as tensor on objects and the \emph{direct sum} $f \oplus g = \big(\begin{smallmatrix}f & 0\\ 0 & g\end{smallmatrix}\big)$ as tensor on arrows.
When the rig $\S$ is commutative, $\mathbf{Mat}_\S$ is also monoidal with multiplication of natural numbers as tensor on objects and the \emph{Kronecker product} as tensor on arrows.
The inclusion functor $\mathbf{FinSet} \to \mathbf{Mat}_\B$ is monoidal in two ways: it sends disjoint unions to direct sums and Cartesian products to Kronecker products.
\end{example}

\begin{python}
{\normalfont Implementation of $\mathbf{Mat}_\S$ as a monoidal category with \py{direct_sum} and \py{Kronecker} as tensor.}

\begin{minted}{python}
class Matrix:
    ...
    def direct_sum(self, other: Matrix) -> Matrix:
        dom, cod = self.dom + other.dom, self.cod + other.cod
        left, right = (len(m.inside[0]) if m.inside else 0 for m in (self, other))
        inside = [row + right * [0] if i < len(self.inside) else left * [0] + row
                  for i, row in enumerate(self.inside + other.inside)]
        return type(self)(inside, dom, cod)

    def Kronecker(self, other: Matrix) -> Matrix:
        dom, cod = self.dom * other.dom, self.cod * other.cod
        inside = [[self.inside[i_dom][i_cod] * other.inside[j_dom][j_cod]
            for i_cod in range(self.cod) for j_cod in range(other.cod)]
            for i_dom in range(self.dom) for j_dom in range(other.dom)]
        return type(self)(inside, dom, cod)
\end{minted}
\end{python}

\begin{example}
The category $\mathbf{Circ}$ is monoidal with addition of natural numbers as tensor on objects and \emph{parallel composition} of circuits as tensor on arrows.
The evaluation functor $\mathtt{eval} : \mathbf{Circ} \to \mathbf{Mat}_\C$ is monoidal: it sends the parallel composition of circuits to the Kronecker product of their unitary matrices.
\end{example}

\subsection{Foo monoidal categories}

Again, implementing a monoidal category in Python means nothing but defining a pair of classes for objects and arrows with a \py{tensor} method that satisfies the axioms.
Less trivially, we want to implement the arrows of \emph{free monoidal categories} which can then be interpreted in arbitrary monoidal categories via the application of monoidal functors: this is the content of the \py{discopy.monoidal} module.
As in the case of free categories, free monoidal categories will be the image of a functor $F : \mathbf{MonSig} \to \mathbf{MonCat}$, the left adjoint to the forgetful functor $U : \mathbf{MonCat} \to \mathbf{MonSig}$ from monoidal categories to \emph{monoidal signatures}.
A monoidal signature $\Sigma$ is a monoidal category without identity, composition or tensor: a pair of sets $\Sigma_0, \Sigma_1$ and a pair of functions $\dom, \cod : \Sigma_1 \to \Sigma_0^\star$ from boxes to lists of objects.
A morphism of monoidal signatures $f : \Sigma \to \Sigma'$ is a pair of functions $f : \Sigma_0 \to \Sigma'_0$ and $f : \Sigma_1 \to \Sigma'_1$ with $f \fcmp \dom = \dom \fcmp f^\star$ and $f \fcmp \cod = \cod \fcmp f^\star$.
Thus, we have defined the category $\mathbf{MonSig}$ of monoidal signatures and their morphisms.

In order to define the forgetful functor $U : \mathbf{MonCat} \to \mathbf{MonSig}$, we will make our lives easier and add an extra assumption: that monoidal categories are \emph{free on objects} (foo), i.e. that the monoid of objects $(C_0, \otimes, 1)$ is a free monoid $C_0 = X^\star$ generated by some set of objects $X$.
This means we can take the data for a monoidal category $C$ to be the following:
\begin{itemize}
\item a class $C_0$ of \emph{generating objects} and a class $C_1$ of arrows,
\item domain and codomain functions $\dom, \cod : C_1 \to C_0^\star$,
\item a function $\id : C_0^\star \to C_1$ and a (partial) operation $\then : C_1 \times C_1 \to C_1$,
\item an operation on arrows $\tensor : C_1 \times C_1 \to C_1$ such that $\dom (f \otimes g) = \dom(f) \dom(g)$ and $\cod (f \otimes g) = \cod(f) \cod(g)$.
\end{itemize}
The axioms for the objects to be a monoid now come for free, we only need to require that tensor on arrows is a monoid with the interchange law.
With this definition of (free-on-objects) monoidal category, we can define the forgetful functor $U : \mathbf{MonCat} \to \mathbf{MonSig}$: it forgets the identity, composition and tensor on arrows, but not the tensor on objects which is free.

\begin{remark}\label{remark:coloured-PRO}
In the cases of monoidal categories where the objects are the natural numbers with addition as tensor, such as $\mathbf{FinSet}$ with disjoint union, $\mathbf{Mat}_\S$ with direct sum or $\mathbf{Circ}$, the monoid of objects is already free: $(\N, +, 0)$ is the free monoid generated by the singleton set.
These monoidal categories are also called \emph{PROs} (for PROduct categories).
When the objects are generated by a more-than-one-element set they are also called \emph{coloured PROs}, which is the standard name for foo-monoidal categories.
For example, we can take all the Python types as colours and define $\mathbf{Pyth}$ as a foo-monoidal category with \py{tuple[type, ...]} as objects.
\end{remark}

\begin{remark}
Given any non-foo monoidal category $C$, we can construct an equivalent foo-monoidal category $C'$ with objects $C_0^\star$ the free monoid over the objects of $C$ and $C'(x, y) = C(\epsilon_{C_0^\star}(x), \epsilon_{C_0^\star}(y))$ for $\epsilon_{C_0} : C_0^\star \to C_0$ the counit of the list adjunction.
That is, an arrow $f : x \to y$ between two lists $x, y \in C_0^\star$ in $C'$ is an arrow $f : \epsilon_{C_0}(x) \to \epsilon_{C_0}(y)$ between their multiplication in $C$.
From left to right, the equivalence $C \simeq C'$ sends every object $x \in C_0$ to its singleton list $x \in C_0^\star$ and every arrow to itself, from right to left it sends every list to its multiplication and every arrow to itself.
Note that the functor $F : C \to C'$ witnessing the equivalence is not strict monoidal, indeed $F(x \otimes y)$ is a singleton list whereas the list $F(x) \otimes F(y)$ has two elements.
\end{remark}

\begin{example}
Take a monoid $M$ seen as a discrete monoidal category, we get an equivalent monoidal category $M'$ with objects the free monoid $M^\star$ and an isomorphism $x_1 \dots x_n \to y_1 \dots y_m$ whenever $x_1 \times \dots \times x_n = y_1 \times \dots \times y_m$ in $M$.
\end{example}

\begin{python}
{\normalfont Syntactic sugar for whiskering and tensor.}

\begin{minted}{python}
class Tensorable:
    @classmethod
    def whisker(cls, other):
        return other if isinstance(other, Tensorable) else cls.id(other)

    __matmul__ = lambda self, other: self.tensor(self.whisker(other))
    __rmatmul__ = lambda self, other: self.whisker(other).tensor(self)
\end{minted}
\end{python}

\begin{python}\label{listing:monoidal.Function}

{\normalfont $\mathbf{Pyth}$ as a foo-monoidal category with \py{tuple[type, ...]} as objects, \py{Function} as arrows and tuple as tensor.}

\begin{minted}{python}
tuplify = lambda stuff: stuff if isinstance(stuff, tuple) else (stuff, )
untuplify = lambda stuff: stuff[0] if len(stuff) == 1 else stuff

class Function(cat.Function, Tensorable):
    inside: Callable
    dom: tuple[type, ...]
    cod: tuple[type, ...]

    @inductive
    def tensor(self, other: Function) -> Function:
        def inside(*xs):
            left, right = xs[:len(self.dom)], xs[len(self.dom):]
            return untuplify(tuplify(self(*left)) + tuplify(other(*right)))
        return Function(inside, self.dom + other.dom, self.cod + other.cod)
\end{minted}
\end{python}

In the case of $\mathbf{Mat}_\S$ with Kronecker product as tensor, we can define an equivalent category $\mathbf{Tensor}_\S$ where the objects are lists of natural numbers and the arrows $f : x_1 \dots x_n \to y_1 \dots y_m$ are $(x_1 \times \dots \times x_n) \times (y_1 \times \dots \times y_m)$ matrices, i.e. tensors of order $m + n$.
Note that we could define yet another equivalent category where the objects are lists of prime numbers instead.

\begin{python}\label{listing:tensor}
{\normalfont Implementation of the foo-monoidal category $\mathbf{Tensor}_\S \simeq \mathbf{Mat}_\S$ with \py{tuple[int, ...]} as objects and \py{Tensor[dtype]} as arrows.}

\begin{minted}{python}
def product(x, unit=1): return unit if not x else product(x[1:], x[0] * unit)

class Tensor(Tensorable, Matrix):
    inside: list[list[Number]]
    dom: tuple[int, ...]
    cod: tuple[int, ...]

    def downgrade(self) -> Matrix:
        return Matrix[self.dtype](
            self.inside, product(self.dom), product(self.cod))

    @classmethod
    def id(cls, x: tuple[int, ...]) -> Tensor:
        return cls(Matrix.id(product(x)).inside, x, x)

    @inductive
    def then(self, other: Tensor) -> Tensor:
        inside = Matrix.then(*map(Tensor.downgrade, (self, other))).inside
        return type(self)(inside, self.dom, other.cod)

    @inductive
    def tensor(self, other: Tensor) -> Tensor:
        inside = Matrix.Kronecker(*map(Tensor.downgrade, (self, other))).inside
        return type(self)(inside, self.dom + other.dom, self.cod + other.cod)

    def __getitem__(self, key : int | tuple) -> Tensor:
        if isinstance(key, tuple):
            key = sum(
                key[i] * product(self.dom[i + 1:]) for i in range(len(key)))
        inside = Matrix.__getitem__(self.downgrade(), key).inside
        dom, cod = ((), self.cod) if product(self.dom) == 1 else ((), ())
        return type(self)(inside, dom, cod)

for attr in ("__bool__", "__int__", "__float__", "__complex__"):
    setattr(Tensor, attr, lambda self: getattr(self.downgrade(), attr)())
\end{minted}
\end{python}

\subsection{Free monoidal categories}\label{subsection:free-monoidal}

Now how do we go on constructing the left adjoint $F : \mathbf{MonSig} \to \mathbf{MonCat}$?
In the same way that lists in the free monoid $X^\star$ can be defined as equivalence classes of expressions built from generators in $X$, product and unit, we can construct the arrows of the free monoidal category $F(\Sigma)$ as equivalence classes of expressions built from boxes in $\Sigma_1$, identity, composition and tensor.
In order to find good representatives for these equivalence classes, we will need the following technical lemma.

\begin{definition}
Given a monoidal signature $\Sigma$, we define a signature of \emph{layers} $L(\Sigma)$ with $\Sigma_0^\star$ as objects and triples $(x, f, y) \in \Sigma_0^\star \times \Sigma_1 \times \Sigma_0^\star$ as boxes with $\dom(x, f, y) = x \dom(f) y$ and $\cod(x, f, y) = x \cod(f) y$.
Given a morphism of monoidal signatures $f : \Sigma \to \Sigma'$, we get a morphism between their signatures of layers $L(f) : L(\Sigma) \to L(\Sigma')$.
Thus, we have defined a functor $L : \mathbf{MonSig} \to \mathbf{Sig}$.
\end{definition}

\begin{lemma}
Fix a monoidal signature $\Sigma$.
Every well-typed expression built from boxes in $\Sigma_1$, identity of objects in $\Sigma_0^\star$, composition and tensor is equal to:
$$
\id(x) \s \text{for} \s x \in \Sigma_0^\star \quad \text{or} \quad
\id(x_1) \otimes f_1 \otimes \id(y_1)
\s \fcmp \s \dots \s \fcmp \s
\id(x_n) \otimes f_n \otimes \id(y_n)$$
for some list of layers $(x_1, f_1, y_1), \dots, (x_n, f_n, y_n) \in L(\Sigma)$.
\end{lemma}

\begin{proof}
By induction on the structure of well-typed expressions.
The only non-trivial case is for the tensor $f \otimes g$ of two expressions $f : x \to y$ and $g : z \to w$, where we need to apply the interchange law to push the tensor through the composition $f \otimes g \s = \s (f \fcmp \id(y)) \otimes (\id(z) \fcmp g) \s = \s f \otimes \id(z) \ \fcmp \ \id(y) \otimes g$.
\end{proof}

We have all the ingredients to define the free monoidal category $F(\Sigma)$: it is a quotient $F(L(\Sigma)) / R$ of the free category generated by the signature of layers $L(\Sigma)$.
Its objects, which we call \emph{types}, are lists in the free monoid $\Sigma_0^\star$.
Its arrows, which we call \emph{diagrams}, are paths with lists in $\Sigma_0^\star$ as nodes and layers $(x, f : s \to t, y) \in L(\Sigma)$ as edges $x s y \to x t y$.
The equality of diagrams is the smallest congruence generated by the \emph{right interchanger}:
$$
(a x b, \ g, \ c) \s \fcmp \s (a, \ f, \ b w c)
\quad \to_R \quad
(a, \ f, \ b z c) \s \fcmp \s (a y b, \ g, \ c)
$$
for all types $a, b, c \in \Sigma_0^\star$ and boxes $f : x \to y$ and $g : z \to w$.
That is, we can interchange two consecutive layers whenever the output of the first box is not connected to the input of the second, i.e. there is an identity arrow $\id(b)$ separating them.
Note that for an effect $f : x \to 1$ followed by a state $g : 1 \to y$, we have two options: we can apply the right interchanger $(1, f, 1) \fcmp (1, g, 1) \to_R (1, g, x) \fcmp (y, f, 1)$ or its opposite $(1, f, 1) \fcmp (1, g, 1) \leftarrow_R (x, g, 1) \fcmp (1, f, y)$.
For the composition of two scalars $a : 1 \to 1$ and $b : 1 \to 1$, we can apply interchangers indefinitely $a \fcmp b \to_R b \fcmp a \to_R a \fcmp b$: this is the Eckmann-Hilton argument.
Delpeuch and Vicary~\cite{VicaryDelpeuch22} give a quadratic solution to the word problem for free monoidal categories, i.e. deciding when two diagrams are equal.
It is linear time in the connected case, and quadratic in the general case.
The right interchanger is confluent and for connected diagrams, i.e. when the Eckmann-Hilton argument does not apply.
It reaches a normal form in a cubic number of steps, the worst-case is given in example~\ref{example:spiral}.

We have defined the equality of diagrams, there remains to define the tensor operation.
First, we define the \emph{whiskering} $f \otimes z$ of a diagram $f$ by an object $z \in \Sigma_0^\star$ on the right: we tensor $z$ to the right-hand side of each layer $(x_i, f_i, y_i)$, i.e. $f \otimes z = (x_1, f_1, y_1 z) \fcmp \dots \fcmp (x_n, f_n, y_n z)$ and symmetrically for the whiskering $z \otimes f$ on the left.
Then, we can define the tensor $f \otimes g$ of two diagrams $f : x \to y$ and $g : z \to w$ in terms of whiskering $f \otimes g = f \otimes z \s \fcmp \s y \otimes g$.
Note that we could have chosen to define $f \otimes g = x \otimes g \fcmp f \otimes w$, the two definitions are equated by the interchanger.

Given a morphism of monoidal signatures $f : \Sigma \to \Sigma'$, we get a monoidal functor $F(f) : F(\Sigma) \to F(\Sigma')$ by relabeling: we have defined a functor $F : \mathbf{MonSig} \to \mathbf{MonCat}$.
We now have to show that it is indeed the left adjoint of $U : \mathbf{MonCat} \to \mathbf{MonSig}$.
This is very similar to the monoid case.
The unit $\eta_\Sigma : \Sigma \to U(F(\Sigma))$ sends objects to themselves and boxes $f : x \to y \in \Sigma$ to diagrams $(1, f, 1) \in L(\Sigma)$, i.e. the layer with empty lists on both sides of $f$.
The counit $\epsilon_C : F(U(C)) \to C$ is the functor which sends diagrams with boxes in $C$ to their evaluation, i.e. the formal composition and tensor of diagrams in $F(U(C))$ is sent to the concrete composition and tensor of arrows in $C$.
In the next section, we will show that this construction is in fact equivalent to the topological definition of diagrams as labeled graphs embedded in the plane.

\begin{python}\label{listing:Ty}
{\normalfont Outline of the class \py{monoidal.Ty}.}
\begin{minted}{python}
class Ty(Ob):
    def __init__(self, inside: Optional[tuple[Ob | str, ...]] = ()):
        self.inside = tuple(x if isinstance(x, Ob) else Ob(x) for x in inside)
        name = ' @ '.join(map(str, inside)) if inside\
            else "{}()".format(type(self).__name__)
        super().__init__(name)

    def tensor(self, *others: Ty) -> Ty:
        if all(isinstance(other, Ty) for other in others):
            inside = self.inside + sum([other.inside for other in others], ())
            return self.cast(inside)
        return NotImplemented  # This will allow whiskering on the left.

    def __getitem__(self, key):
        if isinstance(key, slice):
            return self.cast(self.inside[key])
        return self.cast((self.inside[key], ))

    __matmul__ = __add__ = tensor
    __pow__ = lambda self, n: self.cast(n * self.inside)
    __len__ = lambda self: len(self.inside)

    cast = classmethod(lambda cls, inside: cls(inside))
\end{minted}
\end{python}

The implementation of the class \py{Ty} for types (i.e. lists of objects) is straightforward, it is sketched in listing~\ref{listing:Ty}.
The only subtlety is in the use of the class method \py{cast} which allows the tensor of objects in a subclass to stay within the subclass, without having to redefine the \py{tensor} method.
We also use it to define indexing (which returns a type of length one), slicing and exponentiation by a natural number.

\begin{example}
We can define a \py{Qubits} subclass and be sure that the tensor of qubits is still an instance of \py{Qubits}, not merely \py{Ty}.

\begin{minted}[fontsize=\footnotesize]{python}
class Qubits(Ty):
    __str__ = lambda self: "qubit ** {}".format(len(self))

qubit = Qubits('1')

nstance(qubit ** 0, Qubits) and isinstance(qubit ** 42, Qubits)
\end{minted}
\end{example}

The implementation of \py{Layer} as a subclass of \py{cat.Box} is sketched in listing~\ref{listing:Layer}.
It has methods \py{__matmul__} and  \py{__rmatmul__} for whiskering on the right and left respectively, and \py{cast} for turning boxes into layers with units on both sides.
We use empty slices of the box's domain as units, so that \py{Layer} can be used with any subclass of \py{Ty} as attributes.
Instead of getting the units by calling \py{Ty()} directly, we use the domain of the box to slice empty types of the appropriate \py{Ty} subclass.
This will prove useful in sections~\ref{subsection:rigid}, \ref{subsection:closed} and~\ref{subsection:towards-higher} where we will subclass \py{Ty} to define the types for free rigid, closed and 2-categories respectively.

\begin{python}\label{listing:Layer}
{\normalfont Outline of the class \py{monoidal.Layer}.}

\begin{minted}[fontsize=\footnotesize]{python}
class Layer(cat.Box):
    def __init__(self, left: Ty, box: Box, right: Ty):
        self.left, self.box, self.right = left, box, right
        name = ("{} @ ".format(left) if left else "") + box.name\
            + (" @ {}".format(right) if right else "")
        dom, cod = left @ box.dom @ right, left @ box.cod @ right
        super().__init__(name, dom, cod)

    def __matmul__(self, other: Ty) -> Layer:
        return Layer(self.left, self.box, self.right @ other)

    def __rmatmul__(self, other: Ty) -> Layer:
        return Layer(other @ self.left, self.box, self.right)

    def __iter__(self): yield self.left; yield self.box; yield self.right

    @classmethod
    def cast(cls, old: Box) -> Layer:
        return cls(old.dom[:0], old, old.dom[len(old.dom):])
\end{minted}
\end{python}

Now we have all the ingredients to define \py{Diagram} as a subclass of \py{Arrow} with instances of \py{Layer} as boxes.
The \py{tensor} method is defined in terms of left and right whiskering of layers.
The \py{interchange} method takes an integer \py{i < len(self)} and returns the diagram with boxes \py{i} and \py{i + 1} interchanged, or raises an \py{AssertionError} if they are connected.
It also takes an optional argument \py{left: bool} which allows to choose between left and right in case we're interchanging an effect then a state.
The \py{normal_form} method implements applies \py{interchange} until it reaches a normal form, or raises \py{NotImplementedError} if the diagram is disconnected.
The \py{draw} method renders the diagram as an image, it implements the drawing algorithm discussed in the next section.

\begin{python}\label{listing:monoidal.Diagram}
{\normalfont Outline of the class \py{monoidal.Diagram}.}

\begin{minted}{python}
class Diagram(cat.Arrow, Tensorable):
    inside: tuple[Layer, ...]
    dom: Ty
    cod: Ty

    @inductive
    def tensor(self, other: Diagram) -> Diagram:
        layers = tuple(layer @ other.dom for layer in self.inside)\
            + tuple(self.cod @ layer for layer in other.inside)
        dom, cod = self.dom @ other.dom, self.cod @ other.cod
        return self.cast(Diagram(layers, dom, cod))

    def interchange(self, i: int, left=False) -> Diagram: ...
    def normal_form(self, left=False) -> Diagram: ...
    def draw(self, **params): ...
\end{minted}
\end{python}

Again, we have a class method \py{cast} which takes an old \py{cat.Arrow} and turns it into a new object of type \py{cls}, a given subclass of \py{Diagram}.
This means we do not need to repeat the code for identity or composition which is already implemented by \py{cat.Arrow}.
In turn, when the user defines a subclass of \py{Diagram}, they do not need to repeat the code for identity, composition or tensor.
The implementation of \py{monoidal.Box} as a subclass of \py{cat.Box} and \py{Diagram} is relatively straightforward, we only need to make sure that a box is equal to the diagram of just itself.
We also want the \py{cast} method of \py{Box} to be that of \py{Diagram}.

\begin{python}
{\normalfont Outline of the class \py{monoidal.Box}.}
\begin{minted}[fontsize=\footnotesize]{python}
class Box(cat.Box, Diagram):
    def __init__(self, name: str, dom: Ty, cod: Ty, **params):
        cat.Box.__init__(self, name, dom, cod, **params)
        Diagram.__init__(self, (Layer.cast(self), ), dom, cod)

    def __eq__(self, other):
        if isinstance(other, Box):
            return cat.Box.__eq__(self, other)
        if isinstance(other, Diagram):
            return other.inside == (Layer.cast(self), )
        return False

    __hash__ = cat.Box.__hash__
    cast = Diagram.cast
\end{minted}
\end{python}

\begin{example}\label{example:circuit-diagrams}
We can define \py{Circuit} as a subclass of \py{Diagram}. \py{Gate}, \py{Bra} and \py{Ket} are subclasses of \py{Box} and \py{Circuit}.
Now we can compose and tensor gates together and the result will be an instance of \py{Circuit}.

\begin{minted}{python}
class Circuit(Diagram): pass

class Gate(Box, Circuit): pass

class Bra(Box, Circuit):
    def __init__(self, *bits: bool):
        name = "Bra({})".format(', '.join(map(str, bits)))
        self.bits, dom, cod = bits, qubit ** len(bits), qubit ** 0
        Box.__init__(self, name, dom, cod)

    def dagger(self) -> Circuit: return Ket(*self.bits)

class Ket(Box, Circuit):
    def __init__(self, *bits: bool):
        name = "Ket({})".format(', '.join(map(str, bits)))
        self.bits, dom, cod = bits, qubit ** 0, qubit ** len(bits)
        Box.__init__(self, name, dom, cod)

    def dagger(self) -> Circuit: return Bra(*self.bits)

Gate.cast = Ket.cast = Circuit.cast

X, Y, Z, H = [Gate(name, qubit, qubit) for name in "XYZH"]
CX = Gate("CX", qubit ** 2, qubit ** 2)
sqrt2 = Gate("$\\sqrt{2}$", qubit ** 0, qubit ** 0)
assert isinstance(sqrt2 @ Ket(0, 0) >> H @ qubit >> CX, Circuit)
\end{minted}
\end{example}

The \py{monoidal.Functor} class is a subclass of \py{cat.Functor}.
It overrides the \py{__call__} method to define the image of types and layers, and it delegates to its superclass for the image of boxes and composition.
To make the syntax look nicer, we define the domain of the object mapping as types of length one rather than their generating object, e.g. we can define a functor with \py{ob={x: y}} for \py{x = Ty('x')} and \py{y = Ty('y')} rather than \py{ob={x.inside[0]: y}}.
We also implement some syntactic sugar for the codomain so that we can define e.g. a \py{Tensor}-valued functor with \py{ob={x: n}} rather than \py{ob={x: [n]}} for a \py{x = Ty('x')} and \py{n: int}.
We make use of Python's duck typing so that the codomain can be \py{Ty} or \py{tuple} indifferently, in both cases computed using the built-in \py{sum} with \py{Ty()} or \py{()} as unit.

\begin{python}
{\normalfont Implementation of monoidal functors.}
\begin{minted}{python}
class Functor(cat.Functor):
    dom = cod = Category(Ty, Diagram)

    def __call__(self, other : Ty | Diagram) -> Ty | Diagram:
        if isinstance(other, Ty):
            return sum([self(obj) for obj in other.inside], self.cod.ob())
        if isinstance(other, Ob):
            result = self.ob[self.dom.ob((other, ))]
            return result if isinstance(result, self.cod.ob)\
                else self.cod.ob((result, ))
        if isinstance(other, Layer):
            return self(other.left) @ self(other.box) @ self(other.right)
        return super().__call__(other)
\end{minted}

Note that the keys of the dictionary \py{ob} are \py{Ty} of length $1$.
\end{python}

\begin{example}
We can simulate quantum circuits by applying a functor from \py{Circuit} to \py{Tensor}.
We override the \py{__call__} method to define the image of \py{Bra} and \py{Ket} on the fly.

\begin{minted}{python}
class Eval(Functor):
    def __init__(self, ob, ar):
        super().__init__(ob, ar,
                         dom=Category(Qubits, Circuit),
                         cod=Category(tuple[int, ...], Tensor[complex]))

    def __call__(self, other):
        if isinstance(other, Ket):
            if not other.bits: return Tensor.id(())
            head, *tail = other.bits
            return Tensor[complex]([[not head, head]], (), (2, ))\
                @ self(Ket(*tail))
        if isinstance(other, Bra):
            return self(other.dagger()).dagger()
        return super().__call__(other)

Circuit.eval = lambda self: Eval(
    ob={qubit: 2},
    ar={X: [[0, 1], [1, 0]], Y: [[0, -1j], [1j, 0]], Z: [[1, 0], [0, -1]],
        H: [[1 / sqrt(2), 1 / sqrt(2)], [1 / sqrt(2), -1 / sqrt(2)]],
        CX: [[1, 0, 0, 0], [0, 1, 0, 0], [0, 0, 0, 1], [0, 0, 1, 0]],
        sqrt2: [[sqrt(2)]]})(self)

circuit = sqrt2 @ Ket(0, 0) >> H @ qubit >> CX
superposition = Ket(0, 0) + Ket(1, 1)
assert circuit.eval() == Circuit.eval(superposition)
\end{minted}
\end{example}

\begin{remark}
DisCoPy uses a more compact encoding of diagrams than their list of layers.
Indeed, a diagram is uniquely specified by a domain, a list of boxes and a list of \emph{offsets}, i.e. the length of the type to the left of each box.

\begin{minted}{python}
@dataclass
class Encoding:
    dom: Ty
    boxes_and_offsets: tuple[tuple[Box, int], ...]

Diagram.boxes = property(lambda self: tuple(box for _, box, _ in self.inside))
Diagram.offsets = property(
    lambda self: tuple(len(left) for left, _, _ in self.inside))

def encode(diagram: Diagram) -> Encoding:
    return Encoding(diagram.dom, tuple(zip(diagram.boxes, diagram.offsets)))

def decode(encoding: Encoding) -> Diagram:
    diagram = Diagram.id(encoding.dom)
    for box, offset in encoding.boxes_and_offsets:
        left, right = diagram.cod[:offset], diagram.cod[offset + len(box.dom):]
        diagram >>= left @ box @ right
    return diagram

x, y, z = map(Ty, "xyz")
f, g, h = Box('f', x, y), Box('g', y, z), Box('h', y @ z, x)
encoding = Encoding(dom=x @ y, boxes_and_offsets=((f, 0), (g, 1), (h, 0)))
assert decode(encoding) == f @ g >> h and encode(f @ g >> h) == encoding
\end{minted}
\end{remark}

\subsection{Quotient monoidal categories}\label{subsection:quotient-monoidal}

Once we have defined freeness, we need to define quotients.
The quotient $C / R$ of a monoidal category $C$ by a binary relation $R \sub \coprod_{x, y \in C_0^\star} C(x, y) \times C(x, y)$ has the same objects $C_0$ and arrows equivalence classes of arrows in $C_1$ under the smallest \emph{monoidal congruence} containing $R$.
A congruence $(\sim_R)$ is monoidal when $f \sim_R f'$ and $g \sim_R g'$ implies $f \otimes g \sim f' \otimes g'$.
Explicitly, we can construct $C / R$ as the quotient category for the rewriting relation $\to_R$ where:
$$u \s \fcmp \s b \otimes f \otimes c \s \fcmp \s v \quad
\to_R \quad u \s \fcmp \s b \otimes g \otimes c \s \fcmp \s v$$
for all $(f, g) \in R$, $u : a \to b \otimes \dom(f) \otimes c$ and $v : b \otimes \cod(f) \otimes c \to d$.
Intuitively, if we can equate $f$ and $g$ then we can equate them in any context, i.e. with any objects $b$ and $c$ tensored on the left and right and any arrows $u$ and $v$ composed above and below.
A proof that two diagrams are equal in the quotient can itself be thought of as a diagram in three dimensions, i.e. the movie of a diagram being rewritten into another.
These higher-dimensional diagrams will be mentioned in section~\ref{section:summary-and-future}.
Again, every monoidal category $C$ is isomorphic to the quotient of a free monoidal category $C = F(\Sigma) / R$: take $\Sigma = U(C)$ and the relation $R \sub F(U(C)) \times F(U(C))$ given by every binary composition and tensor.

The word problem for categories reduces to that of monoidal categories, indeed the signature of a category can be seen as a monoidal signature where boxes all have domain and codomain of length one.
Thus, deciding equality of diagrams in arbitrary quotient monoidal categories is just as undecidable.
The implementation of a quotient is nothing but a subclass of \py{Diagram} with an equality method that respects the axioms of a monoidal congruence.
The easy way is to define equality of diagrams to be equality of their evaluation by a monoidal functor into a category where equality is decidable.
The hard way is to define a \py{normal_form} method which sends every diagram to a chosen representative of its equivalence class.
DisCoPy provides some basic tools to define such a normal form: \emph{pattern matching} and \emph{substitution}.

\begin{python}
{\normalfont Implementation of diagram pattern matching and substitution.}

\begin{minted}{python}
@dataclass
class Match:
    top: Diagram
    bottom: Diagram
    left: Ty
    right: Ty

    def subs(self, target):
        return self.top >> self.left @ target @ self.right >> self.bottom

def match(self, pattern: Diagram) -> Iterator[Match]:
    for i in range(len(self) - len(pattern) + 1):
        for j in range(len(self[i].dom) - len(pattern.dom) + 1):
            match = Match(
                self[:i], self[i + len(pattern):],
                self[i].dom[:j], self[i].dom[j + len(pattern.dom):])
            well_typed = match.top.cod == match.left @ pattern.dom @ match.right\
                and match.left @ pattern.cod @ match.right == match.bottom.dom
            if well_typed and self == match.subs(pattern): yield match
\end{minted}
\end{python}

Now implementing a quotient reduces to implementing a \emph{rewriting strategy}, i.e. a function which inputs diagrams and returns either a choice of match or \py{StopIteration}, then proving that it is \emph{confluent} (i.e. the order in which we pick matches does not matter) and \emph{terminating} (i.e. there are no infinite sequences of rewrites).
Our simple pattern matching routine can be extended in several ways.
First, it can only find matches on the nose: we could apply interchangers to the diagram until we find a match (with the cubic-time complexity that this implies).
Second, it can only find and substitute one match at a time: we could iterate through lists of compatible matches and implement their simultaneous substitution.
Third, instead of looking for \py{pattern} as a subdiagram of \py{self} directly, we could iterate through the functors \py{F} such that \py{F(pattern)} is a subdiagram.
This would allow to implement infinite families of equations such as those between quantum gates parameterised by continuous phases.

Why should computer scientists care about such diagram rewriting?
One reason is that diagrams are free data structures in the same sense that lists are free: they are a two-dimensional generalisation of lists.
Another reason is that they allow an elegant definition of a Turing-complete problem: given a finite monoidal signature $\Sigma$ and a pair of lists $x, y \in \Sigma_0^\star$, decide whether there is a diagram $f : x \to y$ in $F(\Sigma)$.
Indeed, the word problem for monoids (which is equivalent to the halting problem for Turing machines) reduces to the existence problem for diagrams: given the presentation of a monoid $X^\star / R$,
take objects $\Sigma_0 = X$ and boxes $\Sigma_1 = R$ with $\dom, \cod : \Sigma_1 \injects X^\star \times X^\star \to X^\star$ the left and right hand-side of each related pair.
For any pair $x, y \in X^\star$, we have that $x \leq_R y$ if and only if there is a diagram $f : x \to y$ in $F(\Sigma)$: the preordered monoid generated by the relation $R$ is the preorder collapse of the free monoidal category $F(\Sigma)$.
While monoid presentations define \emph{decision problems} (i.e. with a Boolean output), free monoidal categories naturally define \emph{function problems}: given a pair of types, output a diagram.

If we compose the two reductions together, we get a free monoidal category where diagrams are the possible runs of a given Turing machine.
Moreover, a monoidal functor from the category of one machine to another corresponds to a reduction between the problems they solve, the domain machine being simulated by the codomain.
Thus, we could very well take finite monoidal signatures as our definition of machine and diagrams as our definition of computation: algorithmic complexity is given by the size of signatures, time and space complexity are given by the length and width\footnote
{The width of a diagram is the maximum width of its layers, which is not preserved by interchangers.
In the diagrams generated by Turing machines, we cannot apply interchangers anyway: every box is connected to the next by the head of the machine.}
of diagrams.
Now if two-dimensional diagrams encode computations on one-dimensional lists, we can think of three-dimensional diagrams either as computations on two-dimensional data, or as higher-order computations.
For example, the optimisation steps of a (classical or quantum) compiler can be thought of as a three-dimensional diagram, with (classical or quantum) circuits as domain and codomain.

\begin{example}\label{example:simplify-circuits}
We can simplify quantum circuits using pattern matching.

\begin{minted}{python}
def simplify(circuit, rules):
    for source, target in rules:
        for match in circuit.match(source):
            return simplify(match.subs(target), rules)
    return circuit

rules = [(Ket(b) >> X, Ket(int(not b)))
         for b in [0, 1]] + [
         (Ket(b0) @ Ket(b1) >> CX, Ket(b0) @ Ket(int(not b1 if b0 else b1)))
         for b0 in [0, 1] for b1 in [0, 1]]
circuit = Ket(1) @ Ket(0) >> CX >> qubit @ X

assert simplify(circuit, rules) == Ket(1) >> qubit @ Ket(0)
\end{minted}
\end{example}

\subsection{Daggers, sums and bubbles}\label{subsection:monoidal-daggers-sums-bubbles}

As in the previous section, we introduce three extra pieces of implementation: daggers, sums and bubbles.
A $\dagger$-monoidal category is a monoidal category with a dagger (i.e. an identity-on-objects involutive contravariant endofunctor) that is also a monoidal functor, a $\dagger$-monoidal functor is both a $\dagger$-functor and a monoidal functor.
They are implemented by adding a \py{dagger} method to the \py{Layer} class.
For example, $\mathbf{Tensor}_\S$ is $\dagger$-monoidal with any conjugate transpose as dagger.
The category $\mathbf{Mat}_\S$ with direct sum as tensor is also $\dagger$-monoidal.

\begin{python}
{\normalfont Implementation of free $\dagger$-monoidal categories.}

\begin{minted}{python}
class Layer:
    ...
    def dagger(self) -> Layer:
        return Layer(self.left, self.box.dagger(), self.right)
\end{minted}
\end{python}

A monoidal category is commutative-monoid-enriched when sums distribute over the tensor, i.e.
$$(f + f') \s \otimes \s (g + g')
\quad = \quad
f \otimes g \s + \s f \otimes g' \s + \s f' \otimes g \s + \s f' \otimes g'$$
$$\text{and} \quad f \otimes 0 \s = \s 0 \s = \s 0 \otimes f$$
They are implemented by a adding method a \py{tensor} method to \py{Sum}, as well as overriding \py{Diagram.tensor} so that \py{f @ (g + h) == Sum.cast(f) @ (g + h)} for all diagrams \py{f}.

\begin{python}
{\normalfont Implementation of free CM-enriched monoidal categories.}

\begin{minted}{python}
class Diagram(monoidal.Diagram):
    @inductive
    def tensor(self, other):
        return self.sum.cast(self).tensor(other)\
            if isinstance(other, Sum) else super().tensor(other)

class Sum(cat.Sum, Box):
    @inductive
    def tensor(self, other: Sum) -> Sum:
        terms = tuple(f @ g for f in self.terms for g in self.cast(other).terms)
        return Sum(terms, self.dom @ other.dom, self.cod @ other.cod)

    id = lambda x: Sum.cast(Diagram.id(x))

Diagram.sum = Sum
\end{minted}
\end{python}

Bubbles for monoidal categories are the same as bubbles for categories, their implementation requires no extra work.
As we mentioned in the previous section, bubbles do give us a strictly more expressive syntax however: they can encode operations on arrows that cannot be expressed in terms of composition or tensor.

\begin{python}
{\normalfont Implementation of free monoidal categories with bubbles.}

\begin{minted}{python}
class Bubble(cat.Bubble, Box): pass

Diagram.bubble = lambda self, **kwargs: Bubble(self, **kwargs)
\end{minted}
\end{python}

\begin{example}
As in example~\ref{example:endofunctor-bubbles}, any monoidal endofunctor $\beta : C \to C$ also defines a bubble on the monoidal category $C$, we can define a bubble-preserving functor $F(U(C)^\beta) \to C$ which interprets bubbled diagrams as functor application.
However, the disjoint union $C + D$ of two monoidal categories does not yield a well-defined monoidal category: we cannot tensor arrows of $C$ with those of $D$.
Thus, the case of monoidal functors $C \to D$ requires diagrams with different colours (for the inside and the outside of the bubble) which we will mention in section~\ref{section:extra structure}.
\end{example}

\begin{example}\label{example:postprocessed-circuit}
We can implement the Born rule as a bubble on $\py{Circ}$ interpreted as element-wise squared amplitude.
We can also implement any classical post-processing as a bubble.

\begin{minted}{python}
Born_rule = lambda x: abs(x) ** 2
Circuit.measure = lambda self: self.bubble(method="squared_amplitude")
Tensor.squared_amplitude = lambda self: self.map(Born_rule)

assert Circuit.eval((Ket(0) >> H >> Bra(0)).measure())[0][0] == .5

biased_ReLU = lambda x: max(0, 2 * x.real - 1)
Circuit.post_process = lambda self: self.bubble(method="non_linearity")
Tensor.non_linearity = lambda self: self.map(biased_ReLU)

circuit = Ket(0, 0) >> H @ qubit >> CX >> Bra(0, 0)
post_processed_circuit = circuit.measure().post_process()
assert Circuit.eval(post_processed_circuit).inside\
    == biased_ReLU(Born_rule(complex(circuit.eval())))
\end{minted}
\end{example}

\begin{example}\label{example:monoidal-formula}
We can implement the formulae of first-order logic using Peirce's \emph{existential graphs}.
They are the first historical examples of string diagrams as well as the first definition of first-order logic~\cite{BradyTrimble98,BradyTrimble00,MelliesZeilberger16,HaydonSobocinski20}.
Predicates of arity $n$ are boxes with a codomain of length $n$, if there are more than one generating objects we get a many-sorted logic.
The wires of the diagram correspond to variables, open wires in the domain and codomain are free variables, the others are existentially quantified.
Thus, the composition of the diagrams $f : x \to y$ and $g : y \to z$ encodes the formula $\exists y \ f(x, y) \land g(y, z)$ with two free variables $x, z$ and $y$ bound.
The diagram obtained by composing a predicate $p$ with the dagger of a predicate $q$ encodes the formula $\exists x \ p(x) \land q(x)$.
Bubbles, which Peirce calls \emph{cuts}, encode negation.
The evaluation of a formula in a finite model corresponds to the application of a bubble-preserving functor into $\mathbf{Mat}_\B$.

\begin{minted}{python}
class Formula(Diagram):
    cut = lambda self: Cut(self)

class Cut(Bubble, Formula):
    method = "_not"
    cast = Formula.cast

class Predicate(Box, Formula):
    cast = Formula.cast

def model(size: dict[Ty, int], data: dict[Predicate, list[bool]]):
    return Functor(
        ob=size, ar={p: [data[p]] for p in data},
        dom=Category(Ty, Formula), cod=Category(tuple[int, ...], Tensor[bool]))

x = Ty('x')
dog, god, mortal = [Predicate(name, Ty(), x) for name in ("dog", "god", "mortal")]
all_dogs_are_mortal = (dog.cut() >> mortal.dagger()).cut()
gods_are_not_mortal = (god >> mortal.dagger()).cut()
there_is_no_god_but_god = god >> (Formula.id(x).cut() >> god.dagger()).cut()

size = {x: 2}

for dogs, gods, mortals in itertools.product(*3 * [
        itertools.product(*size[x] * [[0, 1]])]):
    F = model(size, {dog: dogs, god: gods, mortal: mortals})
    assert F(all_dogs_are_mortal) == all(
        not F(dog)[i] or F(mortal)[i] for i in range(size[x]))
    assert F(gods_are_not_mortal) == all(
        not F(god)[i] or not F(mortal)[i] for i in range(size[x]))
    assert F(there_is_no_god_but_god) == any(F(god)[i] and not any(
        F(god)[j] and j != i for j in range(size[x])) for i in range(size[x]))
\end{minted}

Note that for now our syntax is somehow limited: we can only write formulae where each variable appears at most twice, once for the source and target of its wire.
In section~\ref{subsection:hypergraph} we will introduce the diagrammatic syntax for arbitrary formulae, essentially by adding explicit boxes for equality.
\end{example}

\subsection{From tacit to explicit programming}\label{subsection:tacit-to-explicit}

We get to the end of this section and the reader may have noticed that we have not drawn a single diagram yet: drawing will be the topic of the next section.
This absence of drawing intends to demonstrate that diagrams are not only a great tool for visual reasoning, they can also be thought of as a \emph{data structure for abstract pipelines}.
Monoidal functors then allow to evaluate these abstract pipelines in terms of concrete computation, be it Python functions, tensor operations or quantum circuits.
This abstract programming style, defining programs in terms of composition rather than arguments-and-return-value, is called \emph{point-free} or \emph{tacit programming}.
Because of the difficulty of writing any kind of complex program in that way, it has also been called the \emph{pointless style}.
DisCoPy provides a \py{@diagramize} decorator which allows the user to define diagrams using the standard \emph{explicit} syntax for Python functions instead of the point-free syntax.
Given \py{dom: Ty}, \py{cod: Ty} and \py{signature: tuple[Box, ...]} as parameters, it adds to each box a \py{__call__} method which takes the objects of its domain as input and returns the objects of its codomain.

\begin{example}
We can define quantum circuits as Python functions on qubits.

\begin{minted}{python}
kets0 = Ket(0, 0)

@diagramize(dom=Qubits(2), cod=Qubits(2), signature=(sqrt2, kets0, H, CX))
def circuit():
    sqrt2(); qubit0, qubit1 = kets0
    return CX(H(qubit0), qubit1)

assert circuit == sqrt2 @ kets0 >> H @ qubit >> CX
\end{minted}
\end{example}

The underlying algorithm constructs a graph with nodes for each object of the domain and the codomain of each box, as well as of the whole diagram.
There is an edge from a codomain node of a box (or a domain node of the whole diagram) to the domain node of another (or a codomain node of the whole diagram) whenever they are connected.
There is also a node for each box and an edge from that box node to its domain and codomain nodes.
First, we initialise the graph of the identity diagram and feed the objects of its codomain as input to the decorated function.
When a box is applied to a list of nodes, it adds edges going into each object of its domain and returns nodes for each object of its codomain.
Finally, the return value of the decorated function is taken as the codomain of the whole diagram.

The \py{graph2diagram} algorithm which translates the resulting graph into a diagram will be covered in the next section.
It will allow to automatically read \emph{pictures of diagrams} (i.e. matrices of pixels) and translate them into \py{Diagram} objects.
Listing~\ref{listing:diagram2graph} shows the implementation of the inverse translation \py{diagram2graph} which outputs only planar graphs as we will show in the next section by constructing their embedding in the plane, i.e. their drawing.

\begin{python}\label{listing:diagram2graph}
{\normalfont Translation from \py{Diagram} to \py{Graph}.}

We use the graph data structure from NetworkX~\cite{HagbergEtAl08}.

\begin{minted}{python}
from networkx import Graph

@dataclass
class Node:
    kind: str
    label: Ty | Box
    i: int
    j: int

def diagram2graph(diagram: Diagram) -> Graph:
    graph = Graph()
    scan = [Node('dom', x, i, -1) for i, x in enumerate(diagram.dom)]
    graph.add_edges_from(zip(scan, scan))
    for j, (left, box, _) in enumerate(diagram.inside):
        box_node = Node('box', box, -1, j)
        dom_nodes = [Node('dom', x, i, j) for i, x in enumerate(box.dom)]
        cod_nodes = [Node('cod', x, i, j) for i, x in enumerate(box.cod)]
        graph.add_edges_from(zip(scan[len(left): len(left @ box.dom)], dom_nodes))
        graph.add_edges_from(zip(dom_nodes, len(box.dom) * [box_node]))
        graph.add_edges_from(zip(len(box.cod) * [box_node], cod_nodes))
        scan = scan[len(left):] + cod_nodes + scan[len(left @ box.dom):]
    graph.add_edges_from(zip(scan, [
        Node('cod', x, i, len(diagram)) for i, x in enumerate(diagram.cod)]))
    return graph
\end{minted}
\end{python}

Note that in order to construct a \py{monoidal.Diagram} we need to assume \emph{plane graphs} as input, i.e. graphs with an embedding in the plane.
This means the \py{diagramize} method cannot accept functions which swap the order of variables such as \py{lambda x, y: y, x}.
We also need to assume that every codomain node is connected to exactly one domain node.
In terms of Python functions, this means we have to use every variable exactly once.
In section~\ref{section:extra structure} we will discuss the case of diagrams induced by non-planar graphs, with potentially multiple edges between domain and codomain nodes.


\section{Drawing \& reading} \label{section:drawing}

The previous section defined diagrams as a data structure based on lists of layers, in this section we define \emph{pictures of diagrams}.
Concretely, such a picture will be encoded in a computer memory as a bitmap, i.e. a matrix of colour values.
Abstractly, we will define these pictures in terms of topological subsets of the Cartesian plane.
We first recall the topological definition from Joyal's and Street's unpublished manuscript \emph{Planar diagrams and tensor algebra}~\cite{JoyalStreet88} and then discuss the isomorphism between the two definitions.
In one direction, the isomorphism sends a \py{Diagram} object to its drawing.
In the other direction, it reads the picture of a diagram and translates it into a \py{Diagram} object, i.e. its domain, codomain and list of layers.

\subsection{Labeled generic progressive plane graphs}

A \emph{topological graph}, also called 1-dimensional cell complex, is a tuple $(G, G_0, G_1)$ of a Hausdorff space $G$ and a pair of a closed subset $G_0 \sub G$ and a set of open subsets $G_1 \sub P(G)$ called \emph{nodes} and \emph{wires} respectively, such that:
\begin{itemize}
\item $G_0$ is discrete and $G - G_0 = \bigcup G_1$,
\item each wire $e \in G_1$ is homeomorphic to an open interval and its boundary is contained in the nodes $\partial e \sub G_0$.
\end{itemize}
From a topological graph $G$, one can construct an undirected graph in the usual sense by forgetting the space $G$, taking $G_0$ as nodes and edges $(x, y) \in G_0 \times G_0$ for each $e \in G_1$ with $\partial e = \{ x, y \}$.
A topological graph is finite (planar) if its undirected graph is finite (planar, i.e. there is some embedding in the plane).

A \emph{plane graph} between two real numbers $a < b$ is a finite, planar topological graph $G$ with an embedding in $\R \times [a, b]$.
We define the domain $\dom(G) = G_0 \ \cap \ \R \times \{ a \}$, the codomain $\cod(G) = G_0 \ \cap \ \R \times \{ b \}$ as lists of nodes ordered by horizontal coordinates and the set $\boxes(G) = G_0 \ \cap \ \R \times (a, b)$.
We require that:
\begin{itemize}
    \item $G \ \cap \ \R \times \{ a \} = \dom(G)$ and $G \ \cap \ \R \times \{ b \} = \cod(G)$, i.e. the graph touches the horizontal boundaries only at domain and codomain nodes,
    \item every domain and codomain node $x \in G \ \cap \ \R \times \{ a, b \}$ is in the boundary of exactly one wire $e \in G_1$, i.e. wires can only meet at box nodes.
\end{itemize}
A plane graph is \emph{generic} when the projection on the vertical axis $p_1 : \R \times \R \to \R$ is injective on $G_0 \ - \ \R \times \{ a, b \}$, i.e. no two box nodes are at the same height.
From a generic plane graph, we can get a list $\boxes(G) \in G_0^\star$ ordered by height.
A plane graph is \emph{progressive} (also called \emph{recumbent} by Joyal and Street) when $p_1$ is injective on each wire $e \in G_1$, i.e. wires go from top to bottom and do not bend backwards.

From a progressive plane graph $G$, one can construct a directed graph by forgetting the space $G$, taking $G_0$ as nodes and edges $(x, y) \in G_0 \times G_0$ for each $e \in G_1$ with $\partial e = \{ x, y \}$ and $p_1(x) < p_1(y)$.
We can also define the domain and the codomain of each box node $\dom, \cod : \boxes(G) \to G_1^\star$ with
$\dom(x) = \{ e \in G_1 \ \vert \partial e = \{ x, y \}, p_1(x) < p_1(y) \}$ the wires coming in from the top and
$\cod(x) = \{ e \in G_1 \ \vert \partial e = \{ x, y \}, p_1(x) > p_1(y) \}$ the wires going out to the bottom, these sets are linearly ordered as follows.
Take some $\epsilon > 0$ such that the horizontal line at height $p_1(x) - \epsilon$ crosses each of the wires in the domain.
Then list $\dom(x) \in G_1^\star$ in order of horizontal coordinates of their intersection points, i.e. $e < e'$ if $p_0(y) < p_0(y')$ for the projection $p_0 : \R \times \R \to \R$ and $y^{(')} = e^{(')} \cap \{ p_1(x) - \epsilon \} \times \R$. Symmetrically we define the list of codomain nodes $\cod(x) \in G_1^\star$ with a horizontal line at $p_1 + \epsilon$.

A \emph{labeling} of progressive plane graph $G$ by a monoidal signature $\Sigma$ is a pair of functions from wires to objects $\lambda : G_1 \to \Sigma_0$ and from boxes to boxes $\lambda : \boxes(G) \to \Sigma_1$ which commutes with the domain and codomain.
From an lgpp (\emph{labeled generic progressive plane}) graph, one can construct a \py{Diagram}.

\begin{python}\label{listing:lgpp2diagram}
{\normalfont Reading a labeled generic progressive plane graphs as a \py{Diagram}.}
\vspace{5pt}
\hrule
\vspace{-15pt}
\begin{flalign*}
\py{def} \s & \py{read(} \s G, \s \lambda : G_1 \to \py{Ty}, \s \lambda : \boxes(G) \to \py{Box} \s \py{) -> Diagram:}&&\\
& \py{dom = [} \s \lambda(e) \s \py{for} \s x \in \dom(G) \s \py{for} \s e \in G_1 \s \py{if} \s x \in \partial e \s \py{]}&&\\
& \py{boxes = [} \s \lambda(x) \s \py{for} \s x \in \boxes(G) \s \py{]}&&\\
& \py{offsets = [len(} \s G_1 \ \cap \ \{ p_0(x) \} \times \R \s \py{) for} \s x \in \boxes(G) \s \py{]}&&\\
& \py{return decode(dom, zip(boxes, offsets))} &&
\vspace{-10pt}
\end{flalign*}
\hrule
\end{python}

\subsection{From diagrams to graphs and back}

In the other direction, there are many possible ways to draw a given \py{Diagram} as a lgpp graph, i.e. to embed its graph into the plane.
Vicary and Delpeuch \cite{VicaryDelpeuch22} give a linear-time algorithm to compute such an embedding with the following disadvantage: the drawing of a tensor $f \otimes g$ does not necessarily look like the horizontal juxtaposition of the drawings for $f$ and $g$.
For example, if we tensor an identity with a scalar, the wire representing the identity will wiggle around the node representing the scalar.
DisCoPy uses a quadratic-time drawing algorithm with the following design decision: we make every wire a straight line and as vertical as possible.
We first initialise the lgpp graph of the identity with a constant spacing between each wire, then for each layer we update the embedding so that there is enough space for the output wires of the box before we add it to the graph.
The resulting plane graph is then either plotted on the screen using Matplotlib~\cite{Hunter07} or translated to TikZ~\cite{Tantau13} code that can be integrated to a \LaTeX \ document.
All the diagrams in this thesis were drawn using DisCoPy together with TikZiT\footnote{\url{https://tikzit.github.io}} for manual editing.

\begin{python}
{\normalfont Outline of \py{Diagram.draw} from \py{Diagram} to \py{PlaneGraph}.}

\begin{minted}{python}
Embedding = dict[Node, tuple[float, float]]
PlaneGraph = tuple[Graph, Embedding]

def make_space(position: Embedding, scan: list[Node], box: Box, offset: int
        ) -> tuple[Embedding, float]:
    """ Update the graph to make space and return the left of the box. """

def draw(self: Diagram) -> PlaneGraph:
    graph = diagram2graph(self)
    box_nodes = [Node('box', box, -1, j) for j, box in enumerate(self.boxes)]
    dom_nodes = scan = [Node('dom', x, i, -1) for i, x in enumerate(self.dom)]
    position = {node: (i, -1) for i, node in enumerate(dom_nodes)}
    for j, (left, box, _) in enumerate(self.inside):
        box_node = Node('box', box, -1, j)
        position, left_of_box = make_space(position, scan, box, len(left))
        position[box_node] = (
            left_of_box + max(len(box.dom), len(box.cod)) / 2, j)
        for i, x in enumerate(box.dom):
            cod_node, = filter(lambda node: node.kind != "box",
                               graph.neighbors(Node('dom', x, i, j)))
            position[Node('dom', x, i, j)] = (position[cod_node][0], j - .1)
        for i, x in enumerate(box.cod):
            position[Node('cod', x, i, j)] = (left_of_box + i, j + .1)
        box_cod_nodes = [Node('cod', x, i, j) for i, x in enumerate(box.cod)]
        scan = scan[:len(left)] + box_cod_nodes + scan[len(left @ box.dom):]
    for i, x in enumerate(self.cod):
        cod_node = Node('cod', x, i, len(self))
        position[cod_node] = (position[scan[i]][0], len(self))
    return graph, position

Diagram.draw = draw
\end{minted}
\end{python}

Note that when we draw the plane graph for a diagram, we do not usually draw the box nodes as points.
Instead, we draw them as boxes, i.e. a box node $x \in \boxes(G)$ is depicted as the rectangle with corners $(l, p_1(x) \pm \epsilon)$ and $(r, p_1(x) \pm \epsilon)$ for $l, r \in \R$ the left- and right-most coordinate of its domain and codomain nodes.
In this way, we do not need to draw the in- and out-going wires of the box node: they are hidden by the rectangle.
Exceptions include \emph{spider boxes} where we draw the box node (the head) and its outgoing wires (the legs of the spider) as well as \emph{swap, cup and cap boxes} where we do not draw the box node at all, only its outgoing wires which are drawn as Bézier curves to look like swaps, cups and caps respectively.
These special boxes will be discussed, and drawn, in section~\ref{section:extra structure}.

\begin{example}
{\normalfont Drawing of a box, an identity, a layer, a composition and a tensor.}

\begin{minted}{python}
a, b, c, x, y, z, w = map(Ty, "abcxyzw")
Box('box', a @ b, x @ y @ z).draw()
\end{minted}

\ctikzfig{img/basics/box}

\begin{minted}{python}
Diagram.id(x @ y @ z).draw()
\end{minted}

\ctikzfig{img/basics/id}

\begin{minted}{python}
layer = a @ Box('f', x, y) @ b
layer.draw()
\end{minted}

\ctikzfig{img/basics/layer}

\begin{minted}{python}
top, bottom = Box('top', a @ b, x @ y @ z), Box('bottom', x @ y @ z, c)
(top >> bottom).draw()
\end{minted}

\ctikzfig{img/basics/composition}

\begin{minted}{python}
left, right = Box('left', a @ b, x @ y @ z), Box('right', x @ y @ z, c)
(left @ right).draw()
\end{minted}

\ctikzfig{img/basics/tensor}
\end{example}

\begin{example}
{\normalfont Drawing of the interchanger in the general case.}

\begin{minted}{python}
f, g = Box('f', x, y), Box('g', z, w)
(a @ f @ b @ g @ c).interchange(0).draw(); (a @ f @ b @ g @ c).draw()
\end{minted}

\begin{center}
\input{img/basics/interchanger-left.tikzstyles}
$\quad \to_R \quad$ \input{img/basics/interchanger-right.tikzstyles}
\end{center}
\end{example}

\begin{example}
{\normalfont Drawing of the interchangers for an effect then a state.}

\begin{minted}{python}
f, g = Box('f', x, Ty()), Box('g', Ty(), w)
(f >> g).interchange(0).draw()
(f >> g).draw()
(f >> g).interchange(0, left=True).draw()
\end{minted}

\begin{center}
\input{img/basics/state-effect-left.tikzstyles}
$\quad \to_R \quad$ \input{img/basics/state-effect.tikzstyles}
$\quad \to_R \quad$ \input{img/basics/state-effect-right.tikzstyles}
\end{center}
\end{example}

\begin{example}
{\normalfont Drawing of the Eckmann-Hilton argument.}

\begin{minted}{python}
f, g = Box('f', Ty(), Ty()), Box('g', Ty(), Ty())
(f @ g).draw()
(f @ g).interchange(0).draw()
(f @ g).interchange(0).interchange(0).draw()
\end{minted}

\begin{center}
\input{img/basics/Eckmann-Hilton-left.tikzstyles}
$\quad \to_R \quad$ \input{img/basics/Eckmann-Hilton-right.tikzstyles}
$\quad \to_R \quad$ \input{img/basics/Eckmann-Hilton-left.tikzstyles}
$\quad \to_R \quad \dots$
\end{center}
\end{example}

\begin{example}\label{example:spiral}
The following spiral diagram is the cubic worst-case for interchanger normal form.
It is also the quadratic worst-case for drawing, at each layer of the first half we need to update the position of every preceding layer in order to make space for the output wires.

\begin{minted}{python}
x = Ty('x')
f, g = Box('f', Ty(), x @ x), Box('g', x @ x, Ty())
u, v = Box('u', Ty(), x), Box('v', x, Ty())

def spiral(length: int) -> Diagram:
    diagram, n = u, length // 2 - 1
    for i in range(n):
        diagram >>= x ** i @ f @ x ** (i + 1)
    diagram >>= x ** n @ v @ x ** n
    for i in range(n):
        diagram >>= x ** (n - i - 1) @ g @ x ** (n - i - 1)
    return diagram

diagram = spiral(8)
for i in [1, 2, 3]: diagram[:i + 1].draw()
diagram.draw(); diagram.normal_form().draw()
Diagram.to_gif(*diagram.normalize())
\end{minted}
\begin{center}
\input{img/spiral/1.tikzstyles},
\input{img/spiral/2.tikzstyles},
\input{img/spiral/3.tikzstyles}, ... \\
\vspace{5pt}
\input{img/spiral/8.tikzstyles} $\quad \sim \quad$
\input{img/spiral/nf.tikzstyles}
\end{center}
The interchangers between these two diagrams can be downloaded as a \py{.gif}\footnote
{\href{https://github.com/oxford-quantum-group/discopy/blob/f364ce218890d87fda4aa5c1f4f770f07c7b4f25/docs/_static/imgs/spiral.gif}{\nolinkurl{https://github.com/oxford-quantum-group/discopy/.../imgs/spiral.gif}}} video.
\end{example}

Next, we define the inverse translation \py{graph2diagram}.
\pagebreak

\begin{python}\label{listing:graph2diagram}
{\normalfont Translation from \py{PlaneGraph} to \py{Diagram}.}

\begin{minted}{python}
def graph2diagram(graph: Graph, position: Embedding) -> Diagram:
    dom = Ty(*[node.label for node in graph.nodes
               if node.kind == 'dom' and node.j == -1])
    boxes = [node.label for node in graph.nodes if node.kind == 'box']
    scan, offsets = [Node('dom', x, i, -1) for i, x in enumerate(dom)], []
    for j, box in enumerate(boxes):
        left_of_box = position[Node('dom', box.dom[0], 0, j)][0]\
            if box.dom else position[Node('box', box, -1, j)][0]
        offset = len([node for node in scan if position[node][0] < left_of_box])
        box_cod_nodes = [Node('cod', x, i, j) for i, x in enumerate(box.cod)]
        scan = scan[:offset] + box_cod_nodes + scan[offset + len(box.dom):]
        offsets.append(offset)
    return decode(Encoding(dom, list(zip(boxes, offsets))))
\end{minted}
\end{python}

\begin{proposition}\label{proposition:graph2diagram(self.draw())}
The equality \py{graph2diagram(self.draw()) == self} holds for all \py{self: Diagram}.
\end{proposition}

\begin{proof}
By induction on \py{n = len(self)}.
If \py{n == 0} we get that \py{dom == self.dom} and \py{boxes == offsets == []}.
If the proposition holds for \py{self}, then it holds for \py{self} \py{>> Layer(left, box, right)}.
Indeed, we have:
\begin{itemize}
\item \py{dom == self.dom and boxes == self.boxes + [box]}
\item \py{(x, Node('cod', self.cod[i], i, n)) in graph}\\
\noindent \py{for i, x in enumerate(scan)}
\end{itemize}
Moreover, the horizontal coordinates of the nodes in \py{scan} are strictly increasing,
thus we get the desired \py{offsets == self.offsets + [len(left)]}.
\end{proof}

From a labeled generic progressive plane graph, we get a unique diagram \emph{up to deformation}.
A deformation $h : G \to G'$ between two labeled plane graphs $G, G'$ is a continuous map $h : G \times [0, 1] \to \R \times \R$ such that:
\begin{itemize}
\item $h(G, t)$ is a plane graph for all $t \in [0, 1]$, $h(G, 0) = G$ and $h(G, 1) = G'$,
\item $x \in \boxes(G)$ implies $h(x, t) \in \boxes(h(G, t))$ for all $t \in [0, 1]$,
\item $h(G, t) \fcmp \lambda = \lambda$ for all $t \in [0, 1]$, i.e. the labels are preserved throughout.
\end{itemize}
A deformation is progressive (generic) when $h(G, t)$ is progressive (generic) for all $t \in [0, 1]$.
We write $G \sim G'$ when there exists some deformation $h : G \to G'$, this defines an equivalence relation.

\begin{proposition}\label{proposition:g2d then d2g}
\py{Diagram.draw(graph2diagram(} $G$ \py{))} $\sim G$ for all lgpp graphs $G$, up to generic progressive deformation.
\end{proposition}

\begin{proof}
By induction on the length of $\boxes(G)$.
If there are no boxes, $G$ is the graph of the identity and we can deform it so that each wire is vertical with constant spacing.
If there is one box, $G$ is the graph of a layer and we can cut it in three vertical slices with the box node and its outgoing wires in the middle.
We can apply the case of the identity to the left and right slices, for the middle slice we make the wires straight with a constant spacing between the domain and codomain.
Because $G$ is generic, we can cut a graph with $n > 2$ boxes in two horizontal slices between the last and the one-before-last box, then apply the case for layers and the induction hypothesis.
To glue the two slices back together while keeping the wires straight, we need to make space for the wires going out of the box.

This deformation is indeed progressive, i.e. we never bend wires, we only make them straight.
It is also generic, i.e. we never move a box node past another.
\end{proof}

\begin{proposition}\label{proposition:d2g then g2d}
There is a progressive deformation $h : G \to G'$ between two lgpp graphs iff \py{graph2diagram(} $G$ \py{) == graph2diagram(} $G'$ \py{)} up to interchanger.
\end{proposition}

\begin{proof}
By induction on the number $n$ of \emph{coincidences}, the times at which the deformation $h$ fails to be generic, i.e. two or more boxes are at the same height.
WLOG (i.e. up to continuous deformation of deformations) this happens at a discrete number of time steps $t_1, \dots, t_n \in [0, 1]$.
Again WLOG at each time step there is at most two boxes at the same height, e.g. if there are two boxes moving below a third at the same time, we deform the deformation so that they move one after the other.
The list of boxes and offsets is preserved under generic deformation, thus if $n = 0$ then \py{graph2diagram(} $G$ \py{) == graph2diagram(} $G'$ \py{)} on the nose.
If $n = 1$, take \py{i: int} the index of the box for which the coincidence happens and \py{left: bool} whether it is a left or right interchanger, then \py{graph2diagram(} $G$ \py{).interchange(i, left) == graph2diagram(} $G'$ \py{)}.
Given a deformation with $n + 1$ coincidences, we can cut it in two time slices with $1$ and $n$ coincidences respectively then apply the cases for $n = 1$ and the induction hypothesis.

For the converse, a proof of \py{graph2diagram(} $G$ \py{) == graph2diagram(} $G'$ \py{)}, i.e. a sequence of $n$ interchangers, translates into a deformation with $n$ coincidences.
DisCoPy can output these proofs as videos using \py{Diagram.normalize} to iterate through the rewriting steps and \py{Diagram.to_gif} to produce a \py{.gif} file.
\end{proof}

\subsection{A natural isomorphism}

We have established an isomorphism between the class of lgpp graphs (up to progressive deformation) and the class of \py{Diagram} objects (up to interchanger).
It remains to define lgpp graphs as the arrows of a monoidal category, i.e. to define identity, composition and tensor.
For every monoidal signature $\Sigma$, there is a monoidal category $G(\Sigma)$ with objects $\Sigma_0^\star$ and arrows the equivalence classes of lgpp graphs with labels in $\Sigma$.
The domain and codomain of an arrow is given by the labels of the domain and codomain of the graph.
The identity $\id(x_1 \dots x_n)$ is the graph with wires $(i, a) \to (i, b)$ for $i \leq n$ and $a, b \in \R$ the horizontal boundaries.
The tensor of two graphs $G$ and $G'$ is given by horizontal juxtaposition, i.e. take $w = \max(p_0(G)) + 1$ the right-most point of $G$ plus a margin and set $G \otimes G' = G \cup \{ (p_0(x) + w, p_1(x)) \ \vert \ x \in G' \}$.
The composition $G \fcmp G'$ is given by vertical juxtaposition and connecting the codomain nodes of $G$ to the domain nodes of $G'$.
That is, $G \fcmp G' = s^+(G) \cup s^-(G') \cup E$ for $s^\pm(x) = \big( p_0(x), \frac{p_1(x) \pm (b - a)}{2} \big)$ and wires $s^+(\cod(G)_i) \to s^-(\dom(G')_i) \in E$ for each $i \leq \len(\cod(G)) = \len(\dom(G'))$.

The deformations for the unitality axioms are straightforward: there is a deformation $G \fcmp \id(\cod(G)) \sim G \sim \id(\dom(G)) \fcmp G$ which contracts the wires of the identity graph, the unit of the tensor is the empty diagram so we have an equality $G \otimes \id(1) = G = \id(1) \otimes G$.
The deformations for the associativity axioms are better described by the hand-drawn diagrams of Joyal and Street in figure~\ref{fig:assoc}.

\begin{figure}[H]
\centering
\includegraphics[scale=0.2]{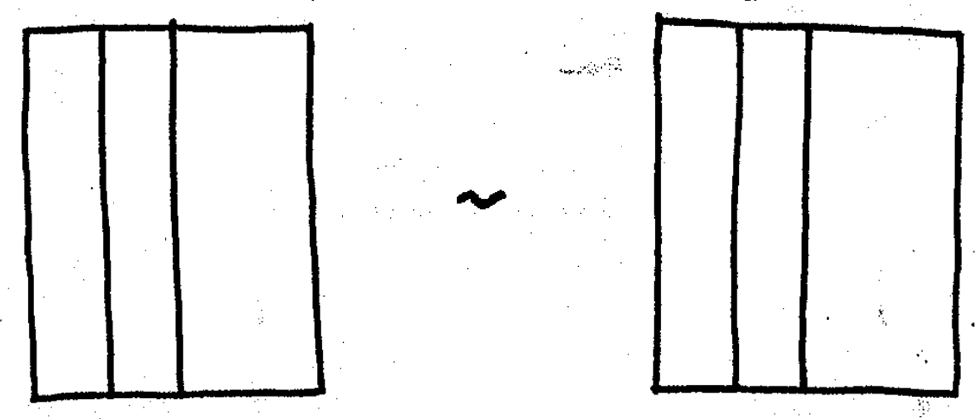}
\qquad \quad \includegraphics[scale=0.2]{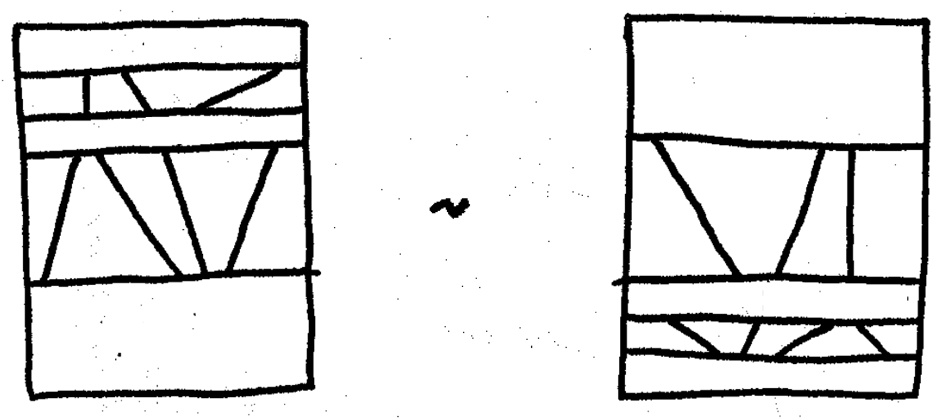}
\caption{Deformations for the associativity of tensor and composition.}
\label{fig:assoc}
\end{figure}

The interchange law holds on the nose, i.e. $(G \otimes G') \fcmp (H \otimes H') = (G \fcmp H) \otimes (G' \fcmp H')$, as witnessed by figure~\ref{fig:interchange}, the hand-drawn diagram which is the result of both sides.

\begin{figure}[H]
\centering
\includegraphics[scale=0.1]{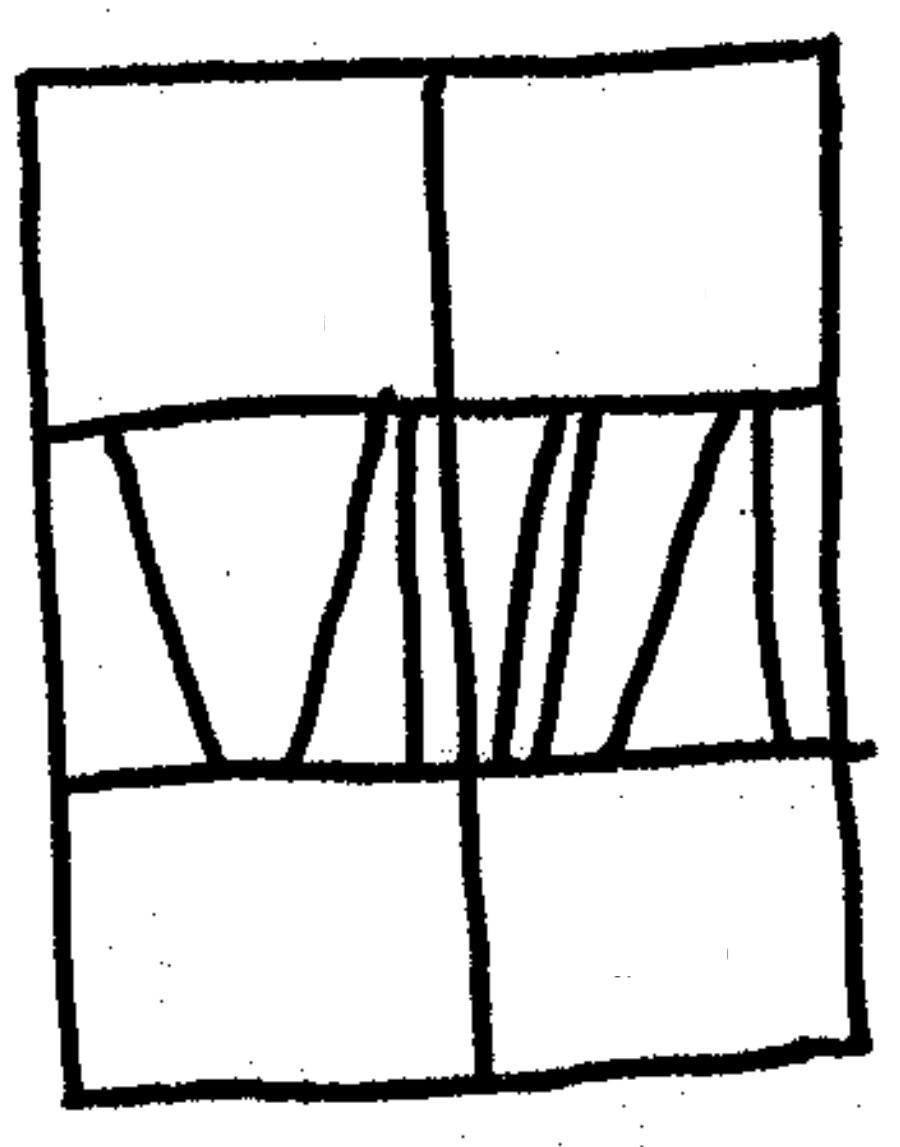}
\caption{The graph of the interchange law.}
\label{fig:interchange}
\end{figure}

Thus, we have defined a monoidal category $G(\Sigma)$.
Now given a morphism of monoidal signatures $f : \Sigma \to \Sigma'$, we can define a functor $G(f) : G(\Sigma) \to G(\Sigma')$ which sends a graph to itself relabeled with $f \fcmp \lambda$, its image on arrows is given in listing~\ref{listing:G_of_f}.
Hence, we have defined a functor\footnote
{For lack of a better notation, we use the same letter $G$ to refer to an arbitrary graph as well as for the functor $G  : \mathbf{MonSig} \to \mathbf{MonCat}$.} $G : \mathbf{MonSig} \to \mathbf{MonCat}$ which we claim is naturally isomorphic to the free functor $F : \mathbf{MonSig} \to \mathbf{MonCat}$ defined in the previous section.
An abstract way to prove this is to appeal to the universal property of free monoidal categories: if the topological and the combinatorial definitions are both free monoidal, they are necessarily isomorphic.
More concretely, we can implement this natural isomorphism as a commutative diagram in $\mathbf{Pyth}$ and think of it as a software test for our drawing and reading algorithms.

\begin{python}\label{listing:G_of_f}
{\normalfont Implementation of the functor $G : \mathbf{MonSig} \to \mathbf{MonCat}$.}

\begin{minted}{python}
SigMorph = tuple[dict[Ob, Ob], dict[Box, Box]]

def G(f: SigMorph) -> Callable[[Graph], Graph]:
    def G_of_f(graph: Graph) -> Graph:
        relabel = lambda node: Node('box', f[1][node.label], node.i, node.j)\
            if node.kind == 'box'\
            else Node(node.kind, f[0][node.label], node.i, node.j)
        return Graph(map(relabel, graph.edges))
    return G_of_f
\end{minted}
\end{python}

\begin{proposition}\label{proposition:nat-iso}
There is a natural isomorphism $F \simeq G : \mathbf{MonSig} \to \mathbf{MonCat}$ for $F$ the combinatorial definition of diagrams in section~\ref{subsection:free-monoidal} and $G$ the topological definition in terms of labeled generic progressive plane graphs.
\end{proposition}

\begin{proof}
From propositions~\ref{proposition:g2d then d2g} and \ref{proposition:d2g then g2d}, we have an isomorphism between \py{Diagram} and \py{PlaneGraph} (up to deformation and interchanger respectively) given by \py{d2g = Diagram.draw} and \py{g2d = graph2diagram}.
Now define the image of $F$ on arrows \py{F = lambda f: Functor(ob=f[0], ar=f[1])}.
Given a morphism of monoidal signatures \py{f: SigMorph} we have the following two naturality squares.
\begin{center}
    \begin{tikzcd}
    \py{Diagram} \ar{r}{\py{d2g}} \ar{d}[']{\py{F(f)}} &
    \py{PlaneGraph} \ar{d}{\py{G(f)}}\\
    \py{Diagram} \ar{r}{\py{d2g}} &
    \py{PlaneGraph}
    \end{tikzcd}
    \qquad \text{and} \qquad
    \begin{tikzcd}
    \py{PlaneGraph} \ar{r}{\py{g2d}} \ar{d}[']{\py{G(f)}} &
    \py{Diagram} \ar{d}{\py{F(f)}}\\
    \py{PlaneGraph} \ar{r}{\py{g2d}} &
    \py{Diagram}
    \end{tikzcd}
\end{center}
\end{proof}

\subsection{Daggers, sums and bubbles}

As in the previous sections, we now discuss the drawing of daggers, sums and bubbles.
When we draw a diagram in the free $\dagger$-monoidal category, we add some asymmetry to the drawing of each box so that it looks like the vertical reflection of its dagger.

\begin{example}
{\normalfont Drawing of the axiom for unitaries.}

\begin{minted}{python}
f, g = Box('f', x, y), Box('g', z, w)
(f >> f[::-1]).draw(); Diagram.id(x).draw()
(f[::-1] >> f).draw(); Diagram.id(y).draw()
\end{minted}

\begin{center}
\input{img/basics/unitaries-left.tikzstyles} \hspace{100pt}
\input{img/basics/unitaries-right.tikzstyles}
\end{center}
\end{example}

\begin{example}
{\normalfont Drawing of the axiom for $\dagger$-monoidal categories.}

\begin{minted}{python}
f, g = Box('f', x, y), Box('g', z, w)
(f @ g)[::-1].draw(); (f[::-1] @ g[::-1]).draw()
assert (f @ g)[::-1].normal_form() == f[::-1] @ g[::-1]
\end{minted}

\begin{center}
\input{img/basics/dagger-monoidal-left.tikzstyles}
$\quad \to_R \quad$ \input{img/basics/dagger-monoidal-right.tikzstyles}
\end{center}
\end{example}

When we draw a sum, we just draw each term with an addition symbol in between.
More generally, \py{drawing.equation} allows to draw any list of diagrams and \py{drawing.Equation} allows to draw equations within equations.

\begin{example}
{\normalfont Drawing of a commutativity equation.}

\begin{minted}{python}
from discopy import drawing

f, g, h = Box('f', x, y), Box('g', y, z), Box('h', x, z)
drawing.equation(drawing.Equation(f >> g, h, symbol='$+$'),
                 drawing.Equation(h, f >> g, symbol='$+$'))
\end{minted}
\ctikzfig{img/basics/equation}
\end{example}

\begin{example}
{\normalfont Drawing a composition and tensor of sums.}

\begin{minted}{python}
f0, g0 = Box("f0", x, y), Box("g0", y, z)
f1, g1 = Box("f1", x, y), Box("g1", y, z)
((f0 + f1) >> (g0 + g1)).draw()
\end{minted}
\ctikzfig{img/basics/sums-composition}

\begin{minted}{python}
((f0 + f1) @ (g0 + g1)).draw()
\end{minted}
\ctikzfig{img/basics/sums-tensor}
\end{example}

The case of drawing bubbles is more interesting.
One solution would be to draw the bubble as a rectangle like any other box, then draw the content of the bubble inside the rectangle.
However, this would require some clever scaling so that the boxes of the diagram inside the bubble have the same size as the boxes outside, i.e. we would need to add more complexity to our drawing algorithm.
The solution implemented in DisCoPy is to apply a faithful functor $\py{downgrade} : F(\Sigma^\beta) \to F(\Sigma \cup \mathtt{open}^\beta \cup \mathtt{close}^\beta)$ from the free monoidal category with bubbles $F(\Sigma^\beta)$ to the free monoidal category generated by the following signature.
Take the objects $\mathtt{open}^\beta_0 = \mathtt{close}^\beta_0 = \Sigma_0 + \{ \bullet \}$ and boxes
for opening $\mathtt{open}^\beta(x) : \beta_\dom(x) \to \bullet \otimes x \otimes \bullet$ and closing $\mathtt{close}^\beta(x) : \bullet \otimes x \otimes \bullet \to \beta_\cod(x)$ bubbles for each type $x \in \Sigma_0^\star$.
Now define $\py{downgrade(f.bubble())} = \mathtt{open}^\beta(\py{f.dom}) \fcmp (\bullet \otimes \py{f} \otimes \bullet) \fcmp \mathtt{close}^\beta(\py{f.cod})$ for any diagram $\py{f}$ inside a bubble.
That is, we draw a bubble as its opening, its inside with identity wires on both sides then its closing.
The $\bullet$-labeled wires are drawn with Bézier curves so that the bubble looks a bit closer to a circle than a rectangle.
In the case of bubbles that are length-preserving on objects, we also want to override the drawing of its opening and closing boxes so that the wires go straight through the bubble rather than meeting at the box node.

\begin{example}
{\normalfont Drawing of a bubbled diagram and a first-order logic formula.}

\begin{minted}{python}
f, g, h = Box('f', x, y), Box('g', y, z), Box('h', y @ z, x)
(f @ g >> h).bubble(dom=a @ b, cod=c, name="$\\beta$").draw()
\end{minted}
\ctikzfig{img/basics/bubble}
\begin{minted}{python}
god = Predicate("G", x)
formula = god >> (Formula.id(x).cut() >> god.dagger()).cut()
formula.draw()
\end{minted}
\ctikzfig{img/basics/there-is-but-one-god}
\end{example}

\subsection{Automatic diagram recognition}

We conclude this section with an application of proposition~\ref{proposition:nat-iso} to \emph{automatic diagram recognition}: turning pictures of diagrams into diagrams.
In listing~\ref{listing:lgpp2diagram}, we described an abstract reading algorithm which took lgpp graphs as input and returned diagrams.
We make it a concrete algorithm by taking \emph{bitmaps} as input: grids of Boolean pixels describing a black-and-white picture.
The algorithm \py{read} listed below takes as input a pair of bitmaps for the box nodes and the wires of the plane graph, it returns a \py{Diagram}.
It is more general than the \py{graph2diagram} algorithm of listing~\ref{listing:lgpp2diagram} where we assumed that the embedding of the graph looked like the output of \py{Diagram.draw}. i.e. that edges are straight vertical lines.
Indeed, our reading algorithm will accept \emph{any bitmaps} as input and always return a valid diagram, however bended the edges are.
If the bitmaps indeed represent a progressive generic plane graph $G$, then we get \py{read(} $G$ \py{).draw()} $\sim G$ up to progressive generic deformation.
If not, the output will still be a diagram but its drawing may not look anything like the input.

\begin{python}
{\normalfont Implementation of the abstract reading algorithm of listing~\ref{listing:lgpp2diagram}.}

\begin{minted}{python}
from numpy import array, argmin
from skimage.measure import regionprops, label

def read(box_pixels: array, wire_pixels: array) -> Diagram:
    connected_components = lambda img: regionprops(label(img))
    box_nodes, wires = map(connected_components, (box_pixels, wire_pixels))
    source, target, length, width = [], [], len(box_pixels), len(box_pixels[0])
    critical_heights = [0] + [
        int(node.centroid[0]) for node in box_nodes] + [length]
    for wire, region in enumerate(wires):
        top, bottom = (
            minmax(i for i, _ in region.coords) for minmax in (min, max))
        source.append(argmin(abs(array(critical_heights) - top)))
        target.append(argmin(abs(array(critical_heights) - bottom)))
    scan = [wire for wire, node in enumerate(source) if node == 0]
    dom, boxes_and_offsets = Ty('x') ** len(scan), []
    for depth, box_node in enumerate(box_nodes):
        input_wires = [wire for wire in scan if target[wire] == depth + 1]
        output_wires = [
            wire for wire, node in enumerate(source) if node == depth + 1]
        dom, cod = Ty('x') ** len(input_wires), Ty('x') ** len(output_wires)
        box = Box('box_{}_{}'.format(len(dom), len(cod)), dom, cod)
        height, left = map(int, box_node.centroid)
        left_of_box = [wire for wire in scan if wire not in input_wires
                       and dict(wires[wire].coords).get(height, width) < left]
        offset = max(len(left_of_box), 0)
        boxes_and_offsets.append((box, offset))
        scan = scan[:offset] + output_wires + scan[offset + len(input_wires):]
    return decode(dom, tuple(boxes_and_offsets))
\end{minted}
\end{python}

We use the \py{array} data structure of NumPy~\cite{VanDerWaltEtAl11} for bitmaps.
We compute the connected components of box and wire pixels with Scikit-Image~\cite{WaltEtAl14}, using their default ordering by lexicographic order of top-left pixel.
We then define a list of critical heights: the top of the picture, the height of the centroid of each box component, then the bottom of the picture.
For each wire component, we define its source and target as the closest critical height to its top-most and bottom-most pixel.
We define the domain of the diagram as the list of wires with the domain as source.
We then scan through the picture top to bottom, keeping a list \py{scan} of the open wires at each height.
For each box, we find its input wires in this list and define the offset as the number of wires left of the box node that are not inputs, then we update \py{scan} with the output wires.
We get an encoding \py{dom, boxes_and_offsets} which yields a valid diagram by construction.

\begin{example}
Suppose we take the following picture of a diagram as input, where the red pixels are boxes and the black pixels are wires:
\begin{center}
\includegraphics[scale=.2]{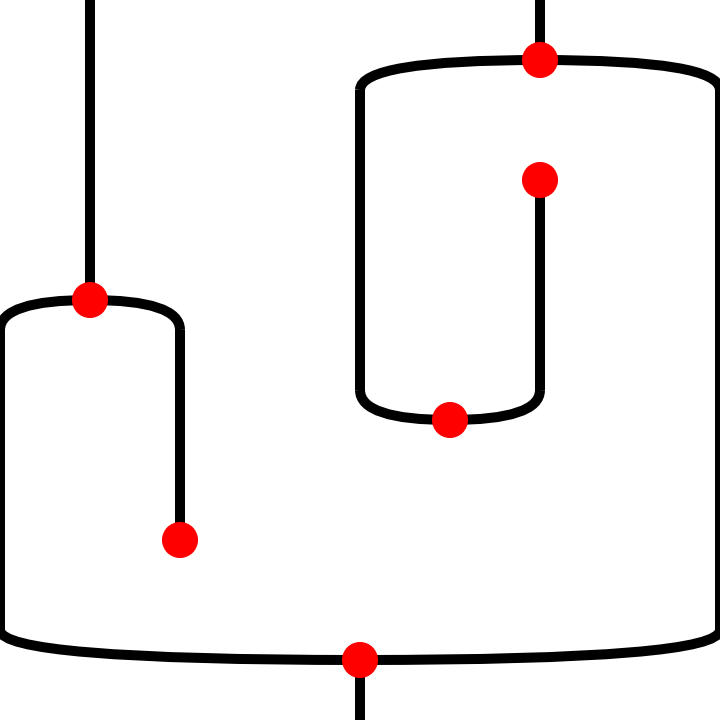}
\end{center}
We count $\{ 1, \dots, 6 \}$ boxes and 8 wires
$$
\{ 0 \to 3, \s
0 \to 1, \s
1 \to 4, \s
1 \to 6, \s
2 \to 4, \s
3 \to 6, \s
3 \to 5, \s
6 \to 7 \}
$$
for $0$ and $7$ the domain and codomain of the whole diagram respectively.
\begin{center}
\includegraphics[scale=.2]{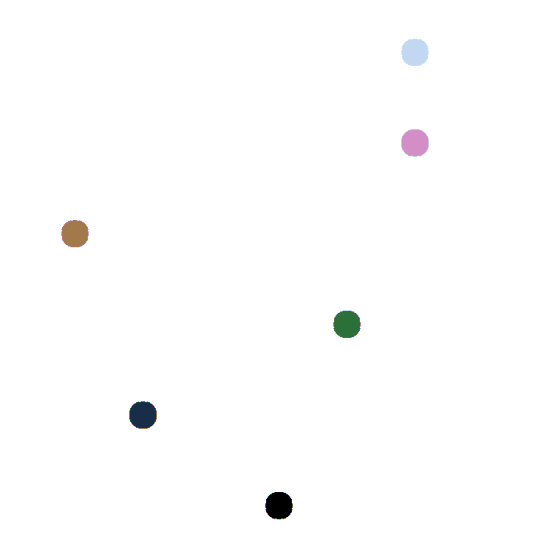}
\hspace{10pt}
\includegraphics[scale=.2]{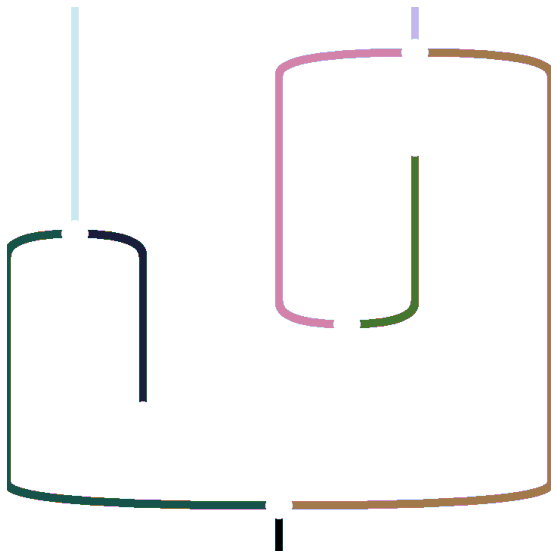}
\end{center}

From this, we reconstruct the following diagram by scanning top to bottom:

\ctikzfig{img/bitmap2diagram/finish}

which is indeed equal to the input picture, up to generic progressive deformation.
\end{example}

\begin{example}
Suppose we start from the following pastiche of Kandinsky's \emph{Punkt und linie zu fläche} (point and line on the plane):

\begin{center}
\includegraphics[scale=.4]{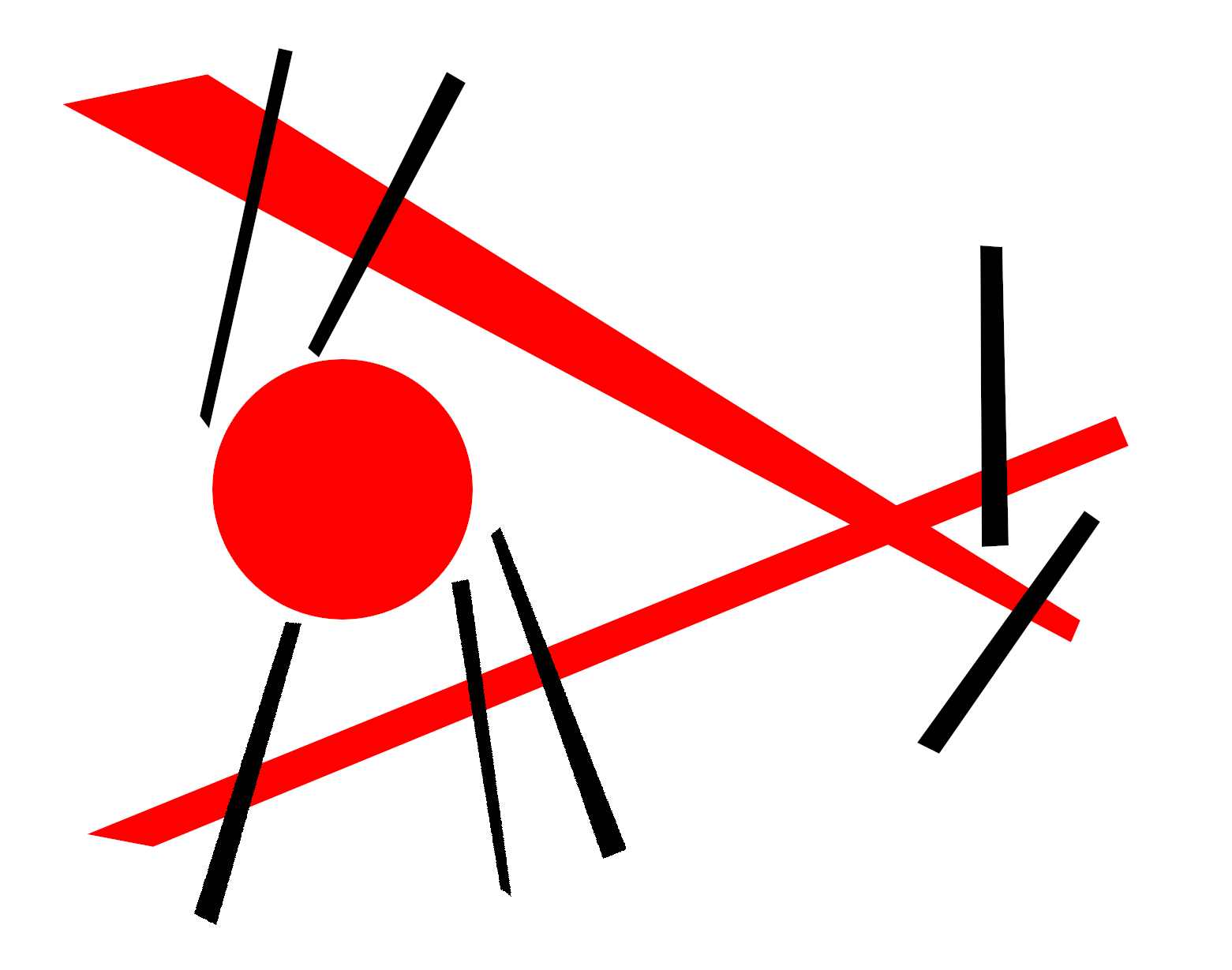}
\end{center}
We count $\{ 1, \dots, 9 \}$ boxes and 7 wires
$$\{
0 \to 3, \s
1 \to 3, \s
2 \to 4, \s
4 \to 8, \s
4 \to 9, \s
6 \to 10, \s
6 \to 10 \}
$$
for $0$ and $10$ the domain and codomain of the whole diagram respectively.
\begin{center}
\includegraphics[scale=.2]{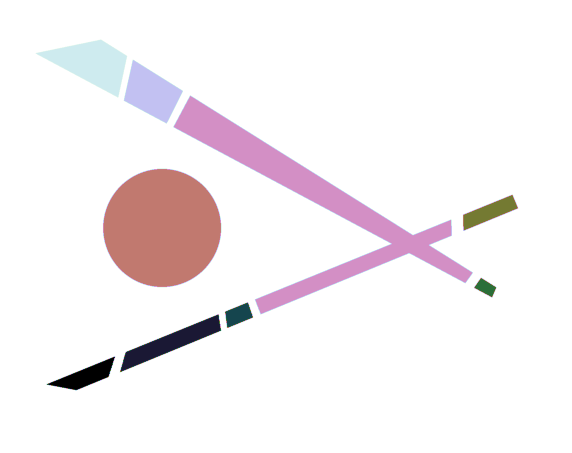}
\hspace{10pt}
\includegraphics[scale=.2]{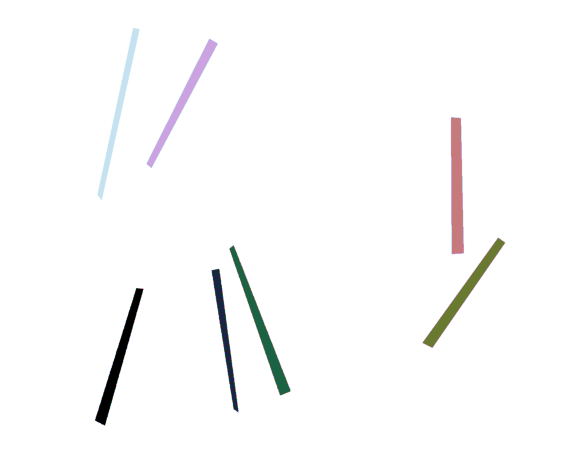}
\end{center}

From this, we reconstruct the following diagram by scanning top to bottom:

\ctikzfig{img/punkt-und-linie/finish}

which looks nothing like the input because Kandinsky's abstract paintings are not generic progressive plane graphs.
\end{example}

What could be the applications of such a reading algorithm?
Of course, real-world pictures do not come in clean red and black bitmaps.
We need some more computer vision to label each pixel as belonging either to a box, a wire or the background.
This can be achieved by training a convolutional neural network on a large number of example pictures, each annotated with their red and black bitmaps.
Once we have trained the machine to classify whether pixels belong to a box or wire, we can also train it to output the label of each pixel, i.e. the label of the box or wire it belongs to.
With enough training data, we could automatically turn pictures of Bob's blackboard into circuits that we can execute on a quantum computer.
Less trivially, this could be applied to \emph{document layout analysis}:
from the picture of say, a hand-drawn calendar from the middle ages, we could train our recognition algorithm to output a diagram that encodes the structure of the calendar.
In previous work~\cite{BorosEtAl19}, we trained convolutional neural networks to extract lines of text in medieval manuscripts, but we had to exclude calendars from our analysis due to the complexity of their layouts.
Our diagram recognition machines would enhance the automated analysis of such hand-drawn documents with structured layouts.

Our simple algorithm is robust to any deformation of lgpp graph, but there is much room for improvement.
The easiest assumption to remove is genericity, i.e. boxes need not be at distinct heights.
Non-generic progressive plane graphs, i.e. with potentially horizontal wires between boxes at the same height, have been characterised as the arrows of free \emph{double categories}.
Delpeuch~\cite{Delpeuch20} shows how they can be encoded as lgpp graphs with an extra label on each wire for whether it is horizontal or vertical.
We can also remove the progressivity assumption, i.e. wires can bend backwards.
We have two options: either a) we write a geometric algorithm than can find the endpoints of any bended wire, or b) we train another neural network to detect each point of non-progressivity, i.e. the cups and caps where the vertical derivative of a wire changes sign.
Such non-progressive plane graphs can be encoded as lgpp graphs with \emph{cups and caps boxes}, they are the arrows of free \emph{pivotal categories} which we discuss in section~\ref{subsection:rigid}.

Next, we can get rid of the planarity assumption: the projection of the topological graph onto the plane need not be an embedding, i.e. wires can cross.
We can improve our reading algorithm in a similar way: either a) some geometric algorithm or b) some black-box neural network that can detect the points of non-planarity, i.e. the intersection of wires.
Such non-planar progressive graphs can be encoded as lgpp graphs with \emph{swap boxes}, they are the arrows of free \emph{symmetric monoidal categories} which we discuss in section~\ref{subsection:symmetric}.
Non-planar non-progressive graphs, i.e. where wires can bend and swap, are the arrows of free \emph{compact closed categories}, which play the starring role in categorical quantum mechanics.

Finally, we can even remove the graph assumption: wires need not be homeomorphic to an open interval.
We merely require that they are one-dimensional open subsets of the plane, i.e. wires can split and merge.
We do not even need to assume that their boundary is in the nodes of the graph, i.e. wires can start or end anywhere in the plane.
Again, we can improve our reading algorithm by encoding such hypergraphs (i.e. where wires can have any number of sources and targets) as lgpp graphs with \emph{spider boxes} for the splits, merges, starts and ends of each wire.
Labeled hypergraphs are the arrows of free \emph{hypergraph categories}, which we discuss in section~\ref{subsection:hypergraph}.
In section~\ref{section:premonoidal}, we discuss the relationship between this definition of hypergraph diagrams as (equivalence classes of) planar diagrams with swap and spider boxes and the more traditional graph-based definition.


\section{Adding extra structure} \label{section:extra structure}


\subsection{Rigid categories \& wire bending} \label{subsection:rigid}

In sections~\ref{section:cat} and \ref{section:monoidal} we discussed the fundamental notion of \emph{adjunction} with the example of free-forgetful functors.
The definition of left and right adjoints in terms of unit and counit natural transformations makes sense in $\mathbf{Cat}$, but it can be translated in the context of any monoidal category $C$.
An object $x^l \in C_0$ is the left adjoint of $x \in C_0$ whenever there are two arrows $\ttcup(x) : x^l \otimes x \to 1$ and $\ttcap(x) : 1 \to x \otimes x^l$ (also called counit and unit) such that:
\begin{itemize}
\item $\ttcap(x) \otimes x \ \fcmp \ x \otimes \ttcup(x) \s = \s \id(x)$,
\ctikzfig{img/rigid/snake-left}
\item $x^l \otimes \ttcap(x) \ \fcmp \ \ttcup(x) \otimes x^l \s = \s \id(x^l)$.
\ctikzfig{img/rigid/snake-right}
\end{itemize}
This is equivalent to the condition that the functor $x^l \otimes - : C \to C$ is the left adjoint of $x \otimes - : C \to C$.
Symmetrically, $x^r \in C_0$ is the right adjoint of $x \in C_0$ if $x$ is its left adjoint.
We say that $C$ is \emph{rigid} (also called \emph{autonomous}) if every object has a left and right adjoint.
From this definition we can deduce a number of properties:
\begin{itemize}
    \item adjoints are unique up to isomorphism,
    \item adjoints are monoid anti-homomorphisms, i.e. $(x \otimes y)^l \simeq y^l \otimes x^l$ and $1^l \simeq 1$,
    \item left and right adjoints cancel, i.e. $(x^l)^r \simeq x \simeq (x^r)^l$,
\end{itemize}
We say that $C$ is strictly rigid whenever these isomorphisms are in fact equalities, again one can show that any rigid category is monoidally equivalent to a strict one.
One can also show that cups and caps compose by nesting:
\begin{itemize}
\item $\ttcup(x \otimes y) \s = \s y^l \otimes \ttcup(x) \otimes y \ \fcmp \ \ttcup(y)$,
\ctikzfig{img/rigid/nesting-cups}
\item $\ttcap(x \otimes y) \s = \s \ttcap(x) \ \fcmp \ x \otimes \ttcap(y) \otimes x^l$,
\ctikzfig{img/rigid/nesting-caps}
\item $\ttcup(1) \s = \s \ttcap(1) \s = \s \id(1)$, drawn as the equality of three empty diagrams.
\end{itemize}
The first two equations are drawn as diagrams in a non-foo monoidal category, i.e. with wires for composite types and explicit boxes for tensor.
This can be taken as an inductive definition, once we have defined the cups and caps for generating objects, we have defined them for all types.
Thus, we can take the data for a (strictly) rigid category $C$ to be that of a free-on-objects monoidal category together with:
\begin{itemize}
    \item a pair of unary operators $(-)^l, (-)^r : C_0 \to C_0$ on generating objects,
    \item and a pair of functions $\ttcup, \ttcap : C_0 \to C_1$ witnessing that $x^l$ and $x^r$ are the left and right adjoints of each generating object $x \in C_0$.
\end{itemize}
Diagrams in rigid categories are more flexible than monoidal categories: we can bend wires.
They owe their name to the fact that they are less flexible than \emph{pivotal categories}.
For any rigid category $C$, there are two contravariant endofunctors, called the left and right \emph{transpose} respectively.
They send objects to their left and right adjoints, and each arrow $f : x \to y$ to
\begin{center}
\input{img/rigid/transpose-left.tikzstyles}
\hspace{50pt} and \hspace{50pt}
\input{img/rigid/transpose-right.tikzstyles}
\end{center}
A rigid category $C$ is called \emph{pivotal} when it has a monoidal natural isomorphism $x^l \sim x^r$ for each object $x$, which implies that the left and right transpose coincide: we can rotate diagrams by 360 degrees~\cite[§4.2]{Selinger10}.
We say $C$ is strictly piotal when this isomorphism is an equality.
This is the case for any rigid category $C$ with a dagger structure: the dagger of the cup (cap) for an object $x$ is the cap (cup) of its left adjoint $x^l$.
When this is the case, $C$ is called $\dagger$-pivotal.
We say $C$ is strictly piotal when left and right transpose are equal.

\begin{example}
Recall from example~\ref{example:endofunctors are monoidal} that for any category $C$, the category $C^C$ of endofunctors and natural transformations is monoidal.
Its subcategory with endofunctors that have both left and right adjoints is rigid.
Its subcategory with endofunctors that have equal left and right adjoints is pivotal.
\end{example}

\begin{example}
$\mathbf{Tensor}_\S$ is $\dagger$-pivotal with left and right adjoints given by list reversal, cups and caps by the Kronecker delta $\ttcup(n)(i, j) = \ttcap(n)(i, j) = 1$ if $i = j$ else $0$.
Note that for tensors of order greater than 2, the \emph{diagrammatic transpose} defined in this way differs from the usual \emph{algebraic transpose}: the former reverses list order while the latter is the identity on objects.
\end{example}

\begin{example}
$\mathbf{Circ}$ is $\dagger$-pivotal with the preparation of the Bell state as cap and the post-selected Bell measurement as cup (both are scaled by $\sqrt{2}$).
The snake equations yield a proof of correctness for the (post-selected) quantum teleportation protocol.
\end{example}

\begin{example}\label{example:pregroups}
A discrete rigid category is a group: if the cups and caps are identities then they define an inverse for the tensor.
A rigid preordered monoid (i.e. a rigid category with at most one arrow between any two objects) is called a (quasi\footnote
{In his original definition~\cite{Lambek99}, Lambek also requires that pregroups are \emph{partial orders}, i.e. preorders with antisymmetry $x \leq y$ and $y \leq x$ implies $x = y$.
This implies that pregroups are strictly rigid, but also that they cannot be free on objects: $\ttcup(x) \otimes \id(x) : x \otimes x^l \otimes x \to x$ and $\id(x) \otimes \ttcap(x) : x \to x \otimes x^l \otimes x$ together would imply $x = x \otimes x^l \otimes x$.
}) \emph{pregroup}, their application to NLP will be discussed in section~\ref{section:NLP}.
A commutative pregroup is a (preordered) abelian group: left and right adjoints coincide with the multiplicative inverse.

Natural examples of non-free non-commutative pregroups are hard to come by.
One exception is the monoid of \emph{monotone unbounded functions} $\Z \to \Z$ with composition as multiplication and pointwise order.
The left adjoint of $f : \Z \to \Z$ is defined such that $f^l(m)$ is the minimum $n \in \Z$ with $m \leq f(n)$ and symmetrically $f^r(m)$ is the maximum $n \in \N$ with $f(n) \leq m$.
Extending Cayley's theorem from groups to pregroups, Buszkowski~\cite[Proposition~2]{Buszkowski01} proved that every pregroup $G$ is in fact isomorphic to a subpregroup (i.e. a monoidal subcategory) of monotone functions $G \to G$.
\end{example}

Any monoidal functor $F : C \to D$ between two rigid categories $C$ and $D$ preserves left and right adjoints up to isomorphism, we say it is strict when it preserves them up to equality.
Thus, we have defined a subcategory $\mathbf{RigidCat} \injects \mathbf{MonCat}$.
We define a \emph{rigid signature} $\Sigma$ as a monoidal signature where the generating objects have the form $\Sigma_0 \times \Z$.
We identify $x \in \Sigma_0$ with $(x, 0) \in \Sigma_0 \times \Z$ and define the left and right adjoints $(x, z)^l = (x, z - 1)$ and $(x, z)^r = (x, z + 1)$.
The objects $\Sigma_0$ are called \emph{basic types}, their iterated adjoints $\Sigma_0 \times \Z$ are called \emph{simple types}.
The integer $z \in \Z$ is called the \emph{adjunction number} of the simple type $(x, z) \in \Sigma_0 \times \Z$ by Lambek and Preller~\cite{PrellerLambek07} and its \emph{winding number} by Joyal and Street~\cite{JoyalStreet88}.
Again, a morphism of rigid signatures $f : \Sigma \to \Sigma'$ is a pair of functions $f : \Sigma_0 \to \Sigma'_0$ and $f : \Sigma_1 \to \Sigma'_1$ which commute with domain and codomain.

There is a forgetful functor $U : \mathbf{RigidCat} \to \mathbf{RigidSig}$ which sends any strictly-rigid foo-monoidal category to its underlying rigid signature.
We now describe its left adjoint $F^r : \mathbf{RigidSig} \to \mathbf{RigidCat}$.
Given a rigid signature $\Sigma$, we define a monoidal signature $\Sigma^r = \Sigma \cup \{ \ttcup (x) \}_{x \in \Sigma_0} \cup \{ \ttcap (x) \}_{x \in \Sigma_0}$.
The free rigid category is the quotient $F^r(\Sigma) = F(\Sigma^r) / R$ of the free monoidal category by the snake equations $R$.
That is, the objects are lists of simple types $(\Sigma_0 \times \Z)^\star$, the arrows are equivalence classes of diagrams with cup and cap boxes.
This is implemented in the \py{rigid} module of DisCoPy as outlined below.

\begin{python}
{\normalfont Implementation of objects and types of free rigid categories.}

\begin{minted}{python}
@dataclass
class Ob(cat.Ob):
    z: int = 0

    l = property(lambda self: Ob(self.name, self.z - 1))
    r = property(lambda self: Ob(self.name, self.z + 1))

    @classmethod
    def cast(cls, old: cat.Ob) -> Ob:
        return old if isinstance(old, cls) else cls(str(old), z=0)

class Ty(monoidal.Ty, Ob):
    def __init__(self, inside=()):
        monoidal.Ty.__init__(self, inside=tuple(map(Ob.cast, inside)))

    l = property(lambda self: type(self)([x.l for x in self.inside[::-1]]))
    r = property(lambda self: type(self)([x.r for x in self.inside[::-1]]))
\end{minted}
\end{python}

\begin{example}
We can check the axioms for objects in rigid categories hold on the nose.

\begin{minted}{python}
x, y = Ty('x'), Ty('y')
assert Ty().l == Ty() == Ty().r
assert (x @ y).l == y.l @ x.l and (x @ y).r == y.r @ x.r
assert x.r.l == x == x.l.r
\end{minted}
\end{example}

\py{rigid.Ob} and \py{rigid.Ty} are subclasses of \py{cat.Ob} and \py{monoidal.Ty} respectively, with \py{property} methods (i.e. attributes that are computed on the fly) \py{l} and \py{r} for the left and right adjoints.
Thanks to the \py{cast} method, we do not need to override the \py{tensor} method inherited from \py{monoidal.Ty}.
In turn, subclasses of \py{rigid.Ty} will not need to override \py{l} and \py{r}.
Similarly, the \py{rigid.Diagram} class is a subclass of \py{monoidal.Diagram}, thanks to the \py{cast} we do not need to reimplement the identity, composition or tensor.
\py{rigid.Box} is a subclass of \py{monoidal.Box} and \py{rigid.Diagram}, with \py{Box.cast = Diagram.cast}.
We need to be careful with the order of inheritance however: diagram equality is defined in terms of box equality, so if we had \py{Box.__eq__ = Diagram.__eq__} then checking equality would enter an infinite loop.
\py{Cup} (\py{Cap}) is a subclass of \py{Box} initialised by a pair of types \py{x, y} such that \py{len(x) == len(y) == 1} \py{x == y.l} (\py{x.l == y}, respectively).
The class methods \py{cups} and \py{caps} construct diagrams of nested cups and caps by induction, with \py{Cup} and \py{Cap} as a base case.

\begin{python}
{\normalfont Implementation of the arrows of free rigid categories.}

\begin{minted}{python}
class Diagram(monoidal.Diagram):
    def transpose(self, left=True) -> Diagram:
        if left: ... # Symmetric to the right case.
        return self.caps(self.dom.r, self.dom) @ self.id(self.cod.r)\
            >> self.id(self.dom.r) @ self @ self.id(self.cod.r)\
            >> self.id(self.dom.r) @ self.cups(self.cod, self.cod.r)

class Box(monoidal.Box, Diagram):
    cast = Diagram.cast

class Cup(Box):
    def __init__(self, x: Ty, y: Ty):
        assert len(x) == 1 and x == y.l
        super().__init__("Cup({}, {})".format(repr(x), repr(y)), x @ y, x[:0])

class Cap(Box):
    def __init__(self, x: Ty, y: Ty):
        assert len(x) == 1 and x.l == y
        super().__init__("Cap({}, {})".format(repr(x), repr(y)), x[:0], x @ y)

def nesting(factory):
    def method(cls, x: Ty, y: Ty) -> Diagram:
        if len(x) == 0: return cls.id(x[:0])
        if len(x) == 1: return factory(x, y)
        head = factory(x[0], y[-1])
        if head.dom:  # We are nesting cups.
            return x[0] @ method(cls, x[1:], y[:-1]) @ y[-1] >> head
        return head >> x[0] @ method(cls, x[1:], y[:-1]) @ y[-1]
    return classmethod(method)

Diagram.cups, Diagram.caps = nesting(Cup), nesting(Cap)
\end{minted}
\end{python}

The \emph{snake removal} algorithm listed below computes the normal form of diagrams in rigid categories.
It is a concrete implementation of the abstract algorithm described in pictures by Dunn and Vicary~\cite[2.12]{DunnVicary19}.
First, we implement a subroutine \py{follow_wire}.
It takes a codomain node (given by the index \py{i} of its box and the index \py{j} of itself in the box's codomain)
and follows the wire till it finds either the domain of another box or the codomain of the diagram.
When we follow a wire, we compute two lists of \emph{obstructions}, the index of each box on its left and right.
The \py{find_snake} function calls \py{follow_wire} for each \py{Cap} in the diagram until it finds one that is connected to a \py{Cup}, or returns \py{None} otherwise.
A \py{Yankable} snake is given by the index of its cup and cap, the two lists of obstructions on each side and whether it is a left or right snake.
\py{unsnake} applies \py{interchange} repeatedly to remove the obstructions, i.e. to make the cup and cap consecutive boxes in the diagram, then returns the diagram with the snake removed.
Each snake removed reduces the length $n$ of the diagram by 2, hence the \py{snake_removal} algorithm makes at most $n / 2$ calls to \py{find_snake}.
Finally, we call \py{monoidal.Diagram.normal_form} which takes at most cubic time.
Finding a snake takes quadratic time (for each cap we need to follow the wire at each layer) as well as removing it (for each obstruction we make a linear number of calls to \py{interchange}).
Thus, we can compute normal forms for diagrams in free rigid categories in cubic time.

\begin{python}\label{listing:snake-removal}
{\normalfont Outline of the snake removal algorithm.}

\begin{minted}{python}
Obstruction = tuple[tuple[int, ...], tuple[int, ...]]
Yankable = tuple[int, int, Obstruction, bool]

def follow_wire(
    self: Diagram, i: int, j: int) -> tuple[int, int, Obstruction]: ...
def find_snake(self: Diagram) -> Optional[Yankable]: ...
def unsnake(self: Diagram, yankable: Yankable) -> Diagram: ...

def snake_removal(self: Diagram) -> Diagram:
    yankable = find_snake(diagram)
    return snake_removal(unsnake(diagram, yankable)) if yankable else diagram

Diagram.normal_form = lambda self:\
    monoidal.Diagram.normal_form(snake_removal(self))
\end{minted}
\end{python}

\begin{example}
We can check that the snake equations hold up to normal form.

\begin{minted}{python}
t = x @ y

left_snake = Diagram.id(t.l).transpose(left=False)
right_snake = Diagram.id(t).transpose(left=True)

assert left_snake.normal_form() == Diagram.id(t)\
    and right_snake.normal_form() == Diagram.id(t.l)

drawing.equation(
    drawing.Equation(left_snake, Diagram.id(t)),
    drawing.Equation(right_snake, Diagram.id(t.l)),
    symbol='and', space=2, draw_type_labels=False)
\end{minted}

\ctikzfig{img/rigid/double-snake-equation}
\end{example}

\begin{example}
We can check that left and right transpose cancel up to normal form.

\begin{minted}{python}
f = Box('f', x, y)

lr_transpose = f.transpose(left=True).transpose(left=False)
rl_transpose = f.transpose(left=False).transpose(left=True)

assert lr_transpose.normal_form() == f == rl_transpose.normal_form()
drawing.equation(lr_transpose, f, rl_transpose)
\end{minted}

\ctikzfig{img/rigid/transpose-inverse}
\end{example}

\begin{python}\label{example:pivotal-circuit}
{\normalfont Implementation of $\mathbf{Circ}$ as a pivotal category.}

\begin{minted}{python}
class Qubits(monoidal.Qubits, Ty):
    l = r = property(lambda self: self)

class Circuit(monoidal.Circuit, Diagram):
    cups = nesting(lambda *_: sqrt2 @ Ket(0, 0) >> H @ qubit)
    caps = lambda x, y: Circuit.cups(x, y).dagger()
\end{minted}
\end{python}

\begin{example}
We can verify the teleportation protocol for two qubits.

\begin{minted}{python}
Bell_state = Circuit.caps(qubit ** 2, qubit ** 2)
Bell_effect = Circuit.cups(qubit ** 2, qubit ** 2)

assert (Bell_state @ qubit ** 2 >> qubit ** 2 @ Bell_effect).eval()\
    == (qubit ** 2).eval()\
    == (qubit ** 2 @ Bell_state >> Bell_effect @ qubit ** 2).eval()
\end{minted}
\end{example}

\py{rigid.Functor} is implemented as a subclass of \py{monoidal.Functor} with the magic method \py{__call__} overriden.
The image on types and on objects \py{x} with \py{x.z == 0} remains unchanged.
The image on objects \py{x} with \py{x.z < 0} is defined by \py{F(x) = F(x.r).l} and symmetrically for \py{x.z > 0}.
Indeed, when defining a strict rigid functor we only need to define the image of basic types, the image of their iterated adjoints is completely determined.
The only problem arises when the objects in the codomain do not have \py{l} and \py{r} attributes, such as the implementation of $\mathbf{Tensor}_\S$ with \py{list[int]} as objects.
In this case, we assume that the left and right adjoints are given by list reversal.

\begin{python}
{\normalfont Implementation of strict rigid functors.}
\begin{minted}{python}
class Functor(monoidal.Functor):
    dom = cod = Category(Ty, Diagram)

    def __call__(self, other):
        if isinstance(other, Ty) or isinstance(other, Ob) and other.z == 0:
            return super().__call__(other)
        if isinstance(other, Ob):
            if not hasattr(self.cod.ob, 'l' if other.z < 0 else 'r'):
                return self(Ob(other.name, z=0))[::-1]
            return self(other.r).l if other.z < 0 else self(other.l).r
        if isinstance(other, Cup):
            return self.cod.ar.cups(self(other.dom[:1]), self(other.dom[1:]))
        if isinstance(other, Cap):
            return self.cod.ar.caps(self(other.cod[:1]), self(other.cod[1:]))
        return super().__call__(other)
\end{minted}
\end{python}

\begin{python}
{\normalfont Implementation of $\mathbf{Tensor}_\S$ as a pivotal category.}
\begin{minted}{python}
Tensor.cups = classmethod(lambda cls, x, y: cls(cls.id(x).inside, x + y, ()))
Tensor.caps = classmethod(lambda cls, x, y: cls(cls.id(x).inside, (), x + y))
\end{minted}
\end{python}

\begin{example}
We can check that $\mathbf{Tensor}_\S$ is indeed pivotal.

\begin{minted}{python}
F = Functor(
    ob={x: 2, y: 3}, ar={f: [[1, 2], [3, 4], [5, 6]]},
    cod=Category(tuple[int, ...], Tensor[int]))

assert F(left_snake) == F(Diagram.id(t.l)) == F(right_snake)
assert F(f.transpose()) == F(f).transpose() == F(f.transpose(left=False))

# Diagrammatic and algebraic transpose differ for tensors of order >= 2.
assert F(f @ x).transpose() != F((f @ x).transpose())
\end{minted}
\end{example}

Free pivotal categories are defined in a similar way to free rigid categories, with the two-element field $\Z / 2 \Z$ instead of the integers $\Z$, i.e. simple types with adjunction numbers of the same parity are equal.
In this case, we usually write $x^l = x^r = x^\star$ with $(x^\star)^\star = x$.
Given a pivotal signature $\Sigma$ with objects of the form $\Sigma_0 \times (\Z / 2 \Z)$, the free pivotal category is the quotient $F^p(\Sigma) = F^r(\Sigma) / R$ of the free rigid category by the relation $R$ equating the left and right transpose of the identity for each generating object.
While the diagrams of free rigid categories can have snakes, those of free pivotal categories can have circles: we can compose $\ttcap(x) : 1 \to x^l \otimes x$ then $\ttcup(x^\star) : x^l \otimes x \to 1$ to form a scalar diagram called the \emph{dimension} of the system $x$.
We also draw the wires with an orientation: the wire for $x$ is labeled with an arrow going down, the one for $x^\star$ with an arrow going up.

To the best of our knowledge, the word problem for pivotal categories is still open.
When defining the normal form of pivotal diagrams, we would need to make a choice between the diagrams for left or right transpose of a box.
Another solution is to add a new box $f^T : y^\star \to x^\star$ for the transpose of every box $f : x \to y$ in the signature, and set it as the normal form of both diagrams.
We can add some asymmetry to the drawing of the box $f$, and draw $f^T$ as its 180° degree rotation.
If the category is also $\dagger$-pivotal, we get a four-fold symmetry: the box, its dagger, its transpose and its dagger-transpose (also called its conjugate).
This is still being developed by the DisCoPy community.

\begin{python}
{\normalfont Implementation of free $\dagger$-pivotal categories.}

\begin{minted}{python}
class Ob(rigid.Ob):
    l = r = property(lambda self: self.cast(Ob(self.name, (self.z + 1) % 2)))

class Ty(rigid.Ty, Ob):
    def __init__(self, inside=()):
        rigid.Ty.__init__(self, inside=tuple(map(Ob.cast, inside)))

class Diagram(rigid.Diagram): pass

class Box(rigid.Box, Diagram):
    cast = Diagram.cast

class Cup(rigid.Cup, Box):
    def dagger(self):
        return Cap(self.dom[0], self.dom[1])

class Cap(rigid.Cap, Box):
    def dagger(self):
        return Cup(self.cod[0], self.cod[1])

Diagram.cups, Diagram.caps = nesting(Cup), nesting(Cap)

class Functor(rigid.Functor):
    dom = cod = Category(Ty, Diagram)
\end{minted}
\end{python}


\subsection{Braided categories \& wire crossing} \label{subsection:symmetric}

With rigid and pivotal categories, we have removed the assumption that diagrams are progressive: we can bend wires.
With braided and symmetric monoidal categories, we now remove the planarity assumption: wires can cross.

The data for a \emph{braided category} is that of a monoidal category $C$ together with a \emph{brading} natural isomorphism $B(x, y) : x \otimes y \to y \otimes x$ (and its inverse $B^{-1}$) drawn as a wire for $x$ crossing under (over) a wire $y$.
Braidings are subject to the following \emph{hexagon equations}:
\begin{itemize}
\item $B(x, y \otimes z) \s = \s B(x,y) \otimes z \ \fcmp \ y \otimes B(x,z)$,
\ctikzfig{img/symmetric/hexagon-left}
\item $B(x \otimes y, z) \s = \s x \otimes B(y,z) \ \fcmp \ B(x,z) \otimes y$
\ctikzfig{img/symmetric/hexagon-right}
\end{itemize}
which owe their name to the shape of the corresponding commutative diagrams when $C$ is non-strict monoidal.
We also require that $B(x, 1) = \id(x) = B(1, x)$\footnote
{Note that in a non-strict monoidal category this axiom is unnecessary, it follows from the coherence conditions.},
i.e. braiding a wire $x$ with the unit $1$ does nothing, we do not need to draw it.
The hexagon equations may be taken as an inductive definition: we can decompose the braiding $B(x, y \otimes z)$ of an object with a tensor in terms of two simpler braids $B(x,y)$ and $B(x,z)$.
Thus, we can take the data for a braided category to be that of a foo-monoidal category together with a pair of functions $B, B^{-1} : C_0 \times C_0 \to C_1$ which send a pair of generating objects to their braiding and its inverse.
Once we have specified the braids of generating objects, the braids of any type (i.e. list of objects) is uniquely determined.
A monoidal functor $F : C \to D$ between two braided categories $C$ and $D$ is braided when $F(B(x, y)) = B(F(x), F(y))$.
Thus, we get a category $\mathbf{BraidCat}$ with a forgetful functor $U : \mathbf{BraidCat} \to \mathbf{MonSig}$, we now describe its left adjoint.

Given a monoidal signature $\Sigma$, the free braided category is a quotient $F^B(\Sigma) = F(\Sigma^B) / R$ of the free monoidal category generated by $\Sigma^B =  \Sigma \cup B \cup B^{-1}$ for the braiding $B(x, y) : x \otimes y \to y \otimes x$ and its inverse $B^{-1}(x, y) : y \otimes x \to x \otimes y$ for each pair of generating objects $x, y \in \Sigma_0$.
The relation $R$ is given by the following axioms for a natural isomorphism:
\begin{itemize}
\item $B(x, y) \fcmp B^{-1}(y, x)
\s = \s \id(x \otimes y) \s = \s B^{-1}(y, x) \fcmp B(x, y)$,
\ctikzfig{img/symmetric/isomorphism}
\item $f \otimes x \ \fcmp \ B(b, x) \s = \s B(a, x) \ \fcmp \ x \otimes f \s$ and $\s x \otimes f \ \fcmp \ B(x, b) \s = \s B(x, a) \ \fcmp \ f \otimes x$.
\begin{center}
\input{img/symmetric/naturality-left.tikzstyles}
\hfill
\input{img/symmetric/naturality-right.tikzstyles}
\end{center}
\end{itemize}
for all generating objects $x, y \in \Sigma_0$ and boxes (including braidings) $f : a \to b$ in $\Sigma^B$.
From $B$ being an isomorphism on generating objects, we can prove it is self-inverse on any type by induction.
Similarly, from $B$ being natural on the left and right for each box, we can prove by induction that it is in fact natural for any diagram.
Note that the naturality axiom holds for boxes with domains and codomains of arbitrary length.
In particular, it holds for $f = B(y, z)$ in which case we get the following Yang-Baxter equation:
\ctikzfig{img/symmetric/yang-baxter}
It also holds for any scalar $f : 1 \to 1$, which allows to pass them through a wire:
\ctikzfig{img/symmetric/naturality-scalars}

A braided category $C$ is \emph{symmetric} if the braiding $B$ is its own inverse $B = B^{-1} = S$, in this case it is called a \emph{swap} and drawn as the intersection of two wires.
A symmetric functor is a braided functor between symmetric categories.
A $\dagger$-braided category is a braided category with a dagger structure, such that the braidings are unitaries, i.e. their inverse is also their dagger.
A $\dagger$-symmetric category is a $\dagger$-braided category that is also symmetric.

\begin{remark}
A symmetric (braided) category with one generating object is called a PROP (PROB) for PROduct and Permutation (Braid).
Indeed, the arrows of the free PROP with no generating boxes (i.e. only swaps) are \emph{permutations}, the arrows of the free braided PRO with no boxes are called \emph{braids}.

Both are \emph{groupoids}, i.e. all their arrows are isomorphisms, which also implies that they are $\dagger$-braided with the dagger given by the inverse.
For every $n \in \N$, the arrows $f : x^n \to x^n$ in the free PROP (PROB) are the elements of the $n$-th symmetric group $S_n$ (braid group $B_n$).
\end{remark}

DisCoPy implements free $\dagger$-symmetric ($\dagger$-braided) categories with a class \py{Swap} (\py{Braid}) initialised by types of length one and a class method \py{swap} (\py{braid}) for types of arbitrary length.
The method \py{simplify} cancels every braid followed by its inverse.
The \py{naturality} method applies the naturality axiom to the box at a given index \py{i: int}.
The optional argument \py{left: bool} allows to choose between left and right naturality axioms, \py{down: bool} allows to move the box either up or down the braid and \py{braid: Callable} allows to apply naturality to any subclass of \py{Braid}.

\begin{python}
{\normalfont Implementation of free $\dagger$-braided categories.}

\begin{minted}{python}
class Diagram(monoidal.Diagram):
    def simplify(self):
        for i, ((x, f, _), (y, g, _)) in enumerate(
                zip(self.inside, self.inside[1:])):
            if x == y and isinstance(f, Braid) and f == g[::-1]:
                inside = self.inside[:i] + self.inside[i + 2:]
                return self.cast(Diagram(inside, self.dom, self.cod)).simplify()
        return self

class Box(monoidal.Box, Diagram):
    cast = Diagram.cast

class Braid(Box):
    def __init__(self, x: Ty, y: Ty, is_dagger=False):
        assert len(x) == len(y) == 1
        name = "{}({}, {})[::-1]".format(type(self), y, x) if is_dagger\
            else "{}({}, {})".format(type(self), x, y)
        super().__init__(name, x @ y, y @ x, is_dagger)

    def dagger(self): return Braid(*self.cod, is_dagger=not self.is_dagger)

def hexagon(factory) -> Callable:
    def method(cls, x: Ty, y: Ty) -> Diagram:
        if len(x) == 0: return cls.id(y)
        if len(x) == 1:
            if len(y) == 1: return factory(x[0], y[0])
            return method(cls, x, y[:1]) @ cls.id(y[1:])\
                >> cls.id(y[:1]) @ method(cls, x, y[1:])  # left hexagon equation.
        return cls.id(x[:1]) @ method(cls, x[1:], y)\
            >> method(cls, x[:1], y) @ cls.id(x[1:])  # right hexagon equation.
    return classmethod(method)

Diagram.braid, Diagram.swap = hexagon(Braid), hexagon(Swap)

def naturality(self: Diagram, i: int, left=True, down=True, braid=None):
    braid = braid or self.braid
    layer, box = self.inside[i], self.inside[i].box
    if left and down:
        source = layer.left[-1] @ box >> braid(layer.left[-1], box.cod)
        target = braid(layer.left[-1], box.dom) >> box @ layer.left[-1]
    elif left: ...
    elif down: ...
    else:
        source = braid(layer.right[0], box.dom) >> box @ layer.right[0]
        target = layer.right[0] @ box >> braid(layer.right[0], box.cod)
    match = Match(top=self[:i] if down else self[:i - len(source) + 1],
                  bottom=self[i + len(source):] if down else self[i + 1:],
                  left=layer.left[:-1] if left else layer.left,
                  right=layer.right if left else layer.right[1:])
    assert self == match.subs(source)
    return match.subs(target)

Diagram.naturality = naturality

class Functor(monoidal.Functor):
    dom = cod = Category(Ty, Diagram)

    def __call__(self, other):
        if isinstance(other, Braid) and not other.is_dagger:
            return self.cod.ar.braid(self(other.dom[0]), self(other.dom[1]))
        return super().__call__(other)
\end{minted}
\end{python}

\begin{example}
We can check the hexagon equations hold on the nose.

\begin{minted}{python}
x, y, z = map(Ty, "xyz")

assert Diagram.braid(x, y @ z) == Braid(x, y) @ z >> y @ Braid(x, z)
assert Diagram.braid(x @ y, z) == x @ Braid(y, z) >> Braid(x, z) @ y
\end{minted}

We can check that \py{Braid} is an isomorphism up to a \py{simplify} call.

\begin{minted}{python}
assert (Diagram.braid(x, y @ z) >> Diagram.braid(x, y @ z)[::-1]).simplify()\
    == Diagram.id(x @ y @ z)\
    == (Diagram.braid(y @ z, x)[::-1] >> Diagram.braid(y @ z, x)).simplify()
\end{minted}

We can check that \py{Braid} and its dagger are natural.

\begin{minted}{python}
a, b = Ty('a'), Ty('b')
f = Box('f', a, b)

for braid in [Diagram.braid, (lambda x, y: Diagram.braid(y, x)[::-1])]:
    source, target = x @ f >> braid(x, b), braid(x, a) >> f @ x
    assert source.naturality(0, braid=braid) == target
    assert target.naturality(1, left=False, down=False, braid=braid) == source
\end{minted}
\end{example}

\begin{python}
{\normalfont Implementation of free $\dagger$-symmetric categories.}

\begin{minted}{python}
class Diagram(braided.Diagram): pass

class Box(braided.Box, Diagram):
    cast = Diagram.cast

class Swap(Braid, Box):
    def dagger(self): return Swap(*self.cod)

Diagram.braid = Diagram.swap = hexagon(Swap)

class Functor(braided.Functor):
    dom = cod = Category(Ty, Diagram)

    def __call__(self, other):
        if isinstance(other, Swap):
            return self.cod.ar.swap(self(other.dom[0]), self(other.dom[1]))
        return super().__call__(other)
\end{minted}
\end{python}

\begin{python}
{\normalfont Implementation of $\mathbf{Pyth}$ and $\mathbf{Tensor}_\S$ as symmetric categories.}

\begin{minted}{python}
@staticmethod
def function_swap(x: tuple[type, ...], y: tuple[type, ...]) -> Function:
    def inside(*xs):
        return untuplify(tuplify(xs)[len(x):] + tuplify(xs)[:len(x)])
    return Function(inside, dom=x + y, cod=y + x)

Function.swap = Function.braid = function_swap

@classmethod
def tensor_swap(cls, x: tuple[int, ...], y: tuple[int, ...]) -> Tensor:
    inside = [[(i0, j0) == (i1, j1)
        for j0 in range(product(y)) for i0 in range(product(x))]
        for i1 in range(product(x)) for j1 in range(product(y))]
    return cls(inside, dom=x + y, cod=y + x)

Tensor.swap = Tensor.braid = tensor_swap
\end{minted}
\end{python}

\begin{example}
We can check the axioms for symmetric categories hold in $\mathbf{Tensor}_\S$ and $\mathbf{Pyth}$.

\begin{minted}{python}
swap_twice = Diagram.swap(x, y @ z) >> Diagram.swap(y @ z, x)

F = Functor(
    ob={a: 1, b: 2, x: 3, y: 4, z: 5},
    ar={f: [[1-2j, 3+4j]]},
    cod=Category(tuple[int, ...], Tensor[complex]))

assert F(f @ x >> Swap(b, x)) == F(Swap(a, x) >> x @ f)
assert F(x @ f >> Swap(x, b)) == F(Swap(x, a) >> f @ x)
assert F(swap_twice) == Tensor.id(F(x @ y @ z))

G = Functor(
    ob={a: complex, b: float, x: int, y: bool, z: str},
    ar={f: lambda z: abs(z) ** 2},
    cod=Category(tuple[type, ...], Function))

assert G(f @ x >> Swap(b, x))(1j, 2) == G(Swap(a, x) >> x @ f)(1j, 2)
assert G(x @ f >> Swap(x, b))(2, 1j) == G(Swap(x, a) >> f @ x)(2, 1j)
assert G(swap_twice)(42, True, "foo") == (42, True, "foo")
\end{minted}
\end{example}

\begin{remark}
Note that the naturality axioms in $\mathbf{Pyth}$ hold only for its subcategory of pure functions, as we will see in section~\ref{section:premonoidal} $\mathbf{Pyth}$ is in fact a symmetric \emph{premonoidal} category.
This is also the case for $\mathbf{Tensor}_\S$ when the rig $\S$ is non-commutative.
\end{remark}

A \emph{compact closed category} is one that is both rigid and symmetric, which implies that it is also pivotal; a $\dagger$-compact closed category is both $\dagger$-pivotal and $\dagger$-symmetric.
The arrows of free $\dagger$-compact closed categories (i.e. equivalence classes of diagrams with cups, caps and swaps) are also called \emph{tensor networks}, a graphical equivalent to \emph{Einstein notation} and \emph{abstract index notation}, first introduced by Penrose~\cite{Penrose71}.
Unlike the computer scientists however, physicists tend to identify the diagram (syntax) with its image under some interpretation functor to the category of tensors (semantics).

A \emph{tortile category}, also called a \emph{ribbon category}\footnote
{Here again we take a \emph{strict} definition, where the twist is an identity rather than an isomorphism.
In a non-strict tortile category, the wires would be drawn as ribbons, i.e. two wires side by side.
The twist isomorphism would be drawn as the two wires being braided twice.
In a strict tortile category, the ribbon has no width thus the twist is invisible.},
is a braided, pivotal category which furthermore satisfies the following \emph{untwisting} equation:
\ctikzfig{img/symmetric/untwisting}
The scalars of the free tortile category with no boxes (i.e. equivalence classes of diagrams with only cups, caps and braids) are called \emph{links} in general and \emph{knots} when they are connected.
Untwisting, the self-inverse equation and the Yang-Baxter equation (i.e. naturality with respect to braids) are called the three \emph{Reidemeister moves}, they completely characterise the continuous deformations of circles embedded in three-dimensional space~\cite{Reidemeister13}.

The \emph{unknotting problem} (given a knot, can it be untied, i.e. continuously deformed to a circle?) is a candidate NP-intermediate problem: it is decidable~\cite{Haken61} and in NP~\cite{Lackenby15}, but there is neither a proof of it being NP-complete nor a polynomial-time algorithm.
Delpeuch and Vicary~\cite{DelpeuchVicary21} proved that the word problem for free braided categories is unknotting-hard.
Hence, there is little hope of finding a simple polynomial-time algorithm for computing normal forms of braided diagrams.
It is not known whether it is even decidable.

The word problem for free symmetric categories reduces to the \emph{graph isomorphism problem}~\cite{PattersonEtAl21}, another potential NP-intermediate problem.
The word problem for free compact closed categories also reduces to graph isomorphism~\cite{Selinger07}.
To the best of our knowledge, it is not known whether they are graph-isomorphism-hard, i.e. whether there is a reduction the other way around that sends any graph to a diagram with swaps (and cups and caps) so that graphs are isomorphic if and only their diagrams are equal.
Thus, there could be a simple polynomial-time algorithm for computing normal forms of diagrams in symmetric and compact closed categories.
In any case, DisCoPy does not implement any normal forms for diagrams with braids yet.

Of course, we can also enrich rigid and braided categories in commutative monoids, i.e. we can take formal sums of diagrams with cups, caps and braids in the same way as any other box.
We can also define bubbles and draw them in the same way as for monoidal diagrams.

\begin{python}
{\normalfont Implementation of free tortile categories and functors.}

\begin{minted}{python}
class Ty(pivotal.Ty, braided.Ty): pass
class Diagram(pivotal.Diagram, braided.Diagram): pass
class Box(pivotal.Box, braided.Box, Diagram):
    cast = Diagram.cast

class Cup(pivotal.Cup, Box): pass
class Cap(pivotal.Cap, Box): pass
class Braid(braided.Braid, Box): pass

Diagram.braid = hexagon(Braid)
Diagram.cups, Diagram.caps = nesting(Cup), nesting(Cap)

class Functor(pivotal.Functor, braided.Functor):
    dom = cod = Category(Ty, Diagram)

    def __call__(self, other):
        if isinstance(other, Braid):
            return braided.Functor.__call__(self, other)
        return pivotal.Functor.__call__(self, other)
\end{minted}
\end{python}

\begin{example}
We can define knot polynomials such as the Kauffman bracket using tortile functors into a self-dual category where the braiding is defined as a weighted sum of diagrams.

\begin{minted}{python}
Ty.l = Ty.r = property(lambda self: self)
x, A = Ty('x'), Box('A', Ty(), Ty())

class Polynomial(Diagram):
    def braid(x, y):
        return (A @ x @ y) + (Cup(x, y) >> A.dagger() >> Cap(x, y))

Kauffman = Functor(
    ob={x: x}, ar={}, cod=Category(Ty, Polynomial))

drawing.equation(Braid(x, x).bubble(), Kauffman(Braid(x, x)))
\end{minted}

\ctikzfig{img/symmetric/kauffman-bracket}
\end{example}


\subsection{Hypergraph categories \& wire splitting} \label{subsection:hypergraph}

With compact closed and tortile categories, we have removed both the progressivity and the planarity assumptions: wires can bend and cross.
With \emph{hypergraph categories} we remove the assumption that diagrams are graphs: wires can split and merge, they need not be homeomorphic to an open interval.
A hypergraph category is a symmetric category with \emph{coherent special commutative spiders}, let's spell out what this means.

An object $x$ in a monoidal category $C$ has \emph{spiders} with \emph{phases} in a monoid $(\Phi, +, 0)$ if it comes equipped with a family of arrows $\spider_{\phi, a, b}(x) : x^a \to x^b$ for every phase $\phi \in \Phi$ and pair of natural numbers $a, b \in \N$, such that the following \emph{spider fusion} equation holds for all $a, b, c, d, n \in \N$.
$$\spider_{\phi, a, c + n + 1}(x) \otimes x^b
\ \fcmp \ x^c \otimes \spider_{\phi', c + n + 1, d}(x)
\s = \s \spider_{\phi + \phi', a + b, c + d}(x)$$
We also require that our spiders satisfy the \emph{special} condition $\spider_{0, 1, 1}(x) = \id(x)$.
Spiders owe their name to their arachnomorphic drawing, for example $\spider_{\phi, 2, 6}$ is drawn as a node (the head, labeled by its phase when it's non-zero) and its wires (the eight legs of the spider, two of them menacing us):
\ctikzfig{img/hypergraph/spider}
Once drawn, the spider fusion equation has the intuitive graphical meaning that if one or more legs of two spiders touch, they fuse and add up their phase.
\ctikzfig{img/hypergraph/spider-fusion}
From spider fusion, we can deduce the following properties:
\begin{itemize}
\item $\merge(x) = \spider_{0, 2, 1}(x)$ and $\unit(x) = \spider_{0, 0, 1}(x)$ form a monoid,
\begin{center}
\input{img/hypergraph/assoc.tikzstyles}
\hfill
\input{img/hypergraph/unit.tikzstyles}
\end{center}
\item $\ttsplit(x) = \spider_{0, 1, 2}(x)$ and $\counit(x) = \spider_{0, 1, 0}(x)$ form a comonoid,
\begin{center}
\input{img/hypergraph/coassoc.tikzstyles}
\hfill
\input{img/hypergraph/counit.tikzstyles}
\end{center}
\item $\ttsplit(x) \fcmp \merge(x) = \id(x)$, called the \emph{special} condition,
\ctikzfig{img/hypergraph/special}
\item $\merge(x) \otimes x \fcmp x \otimes \ttsplit(x) \s = \s x \otimes \merge(x) \fcmp \ttsplit(x) \otimes x$, called the \emph{Frobenius law}.
\ctikzfig{img/hypergraph/frobenius}
\end{itemize}
In fact, when the phases are trivial $\Phi = \{ 0 \}$ these four axioms are sufficient to deduce spider fusion, spiders are also called \emph{special Frobenius algebras}.
Indeed, given a monoid $\merge(x) : x \otimes x \to x, \s \unit(x) : 1 \to x$ and a comonoid $\ttsplit(x) : x \to x \otimes x, \s \counit(x) : x \to 1$ subject to the Frobenius law, we can construct $\spider_{a, b}(x) : x^a \to x^b$ by induction on the number of legs.
The base case is given by the special condition $\spider_{1, 1}(x) = \id(x)$.
Then we define spiders with $a \in \N$ input legs for $a \neq 1$:
\begin{itemize}
\item $\spider_{0, b}(x) \s = \s \unit(x) \fcmp \spider_{1, b}(x)$,
\item $\spider_{a + 2, b}(x) \s = \s \merge(x) \otimes x^a \ \fcmp \ \spider_{a + 1, b}(x)$,
\end{itemize}
\begin{center}
\input{img/hypergraph/induction-base.tikzstyles}
\hfill
\input{img/hypergraph/induction-step.tikzstyles}
\end{center}
Finally we define spiders with one input leg by induction on the output legs $b \in \N$:
\begin{itemize}
\item $\spider_{1, 0}(x) \s = \s \counit(x)$,
\item $\spider_{1, b + 2}(x) \s = \s \spider_{1, b + 1}(x) \ \fcmp \ \ttsplit(x) \otimes x^b$.
\ctikzfig{img/hypergraph/induction-step-one-legged}
\end{itemize}
One can show that this satisfies the spider fusion law, again by induction on the legs~\cite[Lemma 5.20]{HeunenVicary19a}.
In this way, we can construct an infinite family of spiders from just the four boxes $\merge(x), \unit(x), \ttsplit(x), \counit(x)$ and a finite set of equations: a spider is nothing but a big multiplication followed by a big co-multiplication.
As for the phases, we can recover them from a family of \emph{phase shifts} $\{ \shift_\phi(x) : x \to x \}_{\phi \in \Phi}$ such that:
\begin{itemize}
\item $\shift_-(x)$ is a monoid homomorphism $\Phi \to C(x, x)$, i.e. $\shift_0(x) = \id(x)$ and $\shift_{\phi}(x) \fcmp \shift_{\phi'} = \shift_{\phi + \phi'}(x)$,
\ctikzfig{img/hypergraph/phase-hom}
\item phase shifts commute with the product, $\shift_\phi(x) \otimes x \ \fcmp \ \merge(x)
\ = \ \merge(x) \fcmp \shift_\phi(x)
\ = \ x \otimes \shift_\phi(x) \ \fcmp \ \merge(x)$,
\ctikzfig{img/hypergraph/phase-commute-product}
\item phase shifts commute with the coproduct, $\ttsplit(x) \ \fcmp \ \shift_\phi(x) \otimes x
\ = \ \shift_\phi(x) \fcmp \ttsplit(x)
\ = \ x \otimes \ttsplit(x) \ \fcmp \ \shift_\phi(x)$.
\ctikzfig{img/hypergraph/phase-commute-coproduct}
\end{itemize}
We can then define $\spider_{\phi, a, b}(x) = \spider_{a, 1}(x) \fcmp \shift_\phi(x) \fcmp \spider_{1, b}(x)$ and check that indeed, spiders fuse up to addition of their phase.
Thus when the monoid is finite, we get a finite number of boxes and equations, i.e. a finite presentation of the spiders.
In fact instead of taking it as data, we could have equivalently defined the monoid of phases $\Phi$ as the set of endomorphisms $x \to x$ that satisfy the last two conditions.

\begin{remark}
Given any Frobenius algebra on an object $x$, we can show that $x$ is its own left and right adjoint.
Indeed, take $\ttcup(x) = \unit(x) \fcmp \ttsplit(x)$ and $\ttcap(x) = \merge(x) \fcmp \counit(x)$, then the Frobenius law and the (co)unit law of the (co)monoid implies the snake equations.
Thus, a category with (not-necessarily special) spiders on every object is automatically a pivotal category.

\ctikzfig{img/hypergraph/spider-implies-snake}
\end{remark}

\begin{example}
In any pivotal category, there is a Frobenius algebra for every object of the form $x^\star \otimes x$ given by:
\begin{itemize}
\item $\merge(x^\star \otimes x) = x^\star \otimes \ttcup(x^\star) \otimes x$ and $\unit(x) = \ttcap(x)$,
\item $\ttsplit(x^\star \otimes x) = x^\star \otimes \ttcap(x) \otimes x$ and $\counit(x) = \ttcup(x^\star)$.
\end{itemize}
\ctikzfig{img/hypergraph/pair-of-pants}
Due to the drawing of its comonoid, this is called the \emph{pair of pants} algebra.
The special condition requires the dimension of the system $x$ to be the unit, i.e. the circle is equal to the empty diagram.
Non-special Frobenius algebras can still be drawn as spiders, they satisfy a modifed version of spider fusion where we keep track of the number of circles, i.e. the number of splits followed by a merge.
We can extend our inductive definition so that all the circles are in between the product and coproduct, see~\cite[Theorem 5.21]{HeunenVicary19a}.
\end{example}

\begin{example}\label{example:tensor-spider}
The category $\mathbf{Tensor}_\S$ has spiders for every dimension $n \in \N$ with phases in any submonoid of $\phi \in (\S, \times, 1)^n$.
They are given by $\spider_{\phi, a, b}(n) = \sum_{i \leq n} \phi_i \ket{i}^{\otimes a} \bra{i}^{\otimes b}$ where $\ket{i}$ ($\bra{i}$) is the $i$-th basis row (column) vector.

\begin{minted}{python}
class Tensor:
    ...
    @classmethod
    def spider(cls, a: int, b: int, n: int, phase=None) -> Tensor:
        phase = phase or n * [1]
        inside = [[sum(phase)]] if not a and not b\
            else [[phase[xs[0]] for xs in itertools.product(*b * [range(n)])
                   if all(x == xs[0] for x in xs)]]\
            if not a else cls.spider([], a + b, n).inside
        return cls(inside, dom=a * [n], cod=b * [n])
\end{minted}

When $\S$ is a field, we can divide every $\phi_i$ by $\phi_0$, or equivalently require that $\phi_0 = 1$.
Indeed, we can represent any spider with $\phi_0 \neq 1$ as a spider with $\phi_0 = 1$ multiplied by the scalar $\phi_0$, which is called a \emph{global phase}.
When $\S = \C$ and $n = 2$, we usually take the monoid of phases to be the unit circle and write it in terms of addition of angles.
\end{example}

\begin{example}\label{example:circuit-spider}
In the category $\mathbf{Circ}$ of quantum circuits, if we allow post-selected measurements then we can construct spiders with the unit circle as phases.
The spiders with no inputs legs are called the (generalised) GHZ states:
$$
\spider_{\alpha, 0, b} = \ket{0}^{\otimes b} + e^{i \alpha} \ket{1}^{\otimes b}
$$
Note that we need to scale by $\frac{1}{\sqrt{2}}$ to make this a normalised quantum state.
The spiders with $a > 0$ input legs can be thought of as measuring $a$ qubits, post-selecting on all of them giving the same result and then preparing $b$ copies of this result.
The evaluation functor $\mathbf{Circ} \to \mathbf{Tensor}_\C$ sends spiders to spiders.
\end{example}

Spiders allow us to draw diagrams where wires can split and merge, connecting an arbitrary number of boxes.
The PRO of Frobenius algebras (without the special condition), i.e. diagrams with only spider boxes, defines a notion of ``well-behaved'' 1d subspaces of the plane, up to continuous deformation.
Indeed, it is equivalent to the category of \emph{planar thick tangles}~\cite{Lauda05}.
Intuitively, planar thick tangles can be thought of as planar wires with a width, i.e. that we can draw with pens or pixels.
The inductive definition of spiders in terms of monoids and comonoids has the topological interpretation that any wire can be deformed so that all its singular points (i.e. where the wire crosses itself) are binary splits and merges.
The special condition has the non-topological consequence that we can contract the holes in the wires, splitting a wire then merging it back does nothing.

If the monoidal category $C$ is braided, we can remove the planarity assumption and define \emph{commutative spiders} as those where the monoid and comonoid are commutative, i.e.
\begin{align*}
\spider_{\phi, a + b, c + d}(x) \ \fcmp \ B(x^c, x^d)
\s &= \s \spider_{\phi, a + b, c + d}(x) \\
\s &= \s
B(x^a, x^b) \ \fcmp \ \spider_{\phi, a + b, c + d}(x)
\end{align*}
\ctikzfig{img/hypergraph/co-commutativity}
Together with spider fusion, this implies that the monoid of phases is also commutative.
The PROB of commutative Frobenius algebras (without the special condition), i.e. diagrams with only spiders and braids, defines a notion of ``well-behaved'' 1d subspaces of 3d space, up to continuous deformation.
When the category is furthermore symmetric, the PROP of commutative spiders defines a notion of ``well-behaved'' 1d spaces up to diffeomorphism, or equivalently 1d subspaces of 4d space, i.e. one where wires can pass through each other and all knots untie.
It is equivalent to the category of two-dimensional \emph{cobordisms}~\cite{Abrams96}, i.e. oriented 2d manifolds with a disjoint union of circles as boundary.
Intuitively, a 2d cobordism can be thought of as a (non-planar) wire with a width, i.e. one that we can draw.

When $C$ is braided, we can also give an inductive definition of spiders for tensors.
Indeed, given the spiders for $x$ and $y$ we can construct the following comonoid:
\begin{itemize}
\item $\spider_{1, 0}(x \otimes y) \s = \s \spider_{1, 0}(x) \otimes \spider_{1, 0}(y)$,
\ctikzfig{img/hypergraph/coherence-unit}
\item $\spider_{1, 2}(x \otimes y) \s = \s \spider_{1, 2}(x) \otimes \spider_{1, 2}(y) \ \fcmp \ x \otimes S(x, y) \otimes y$
\ctikzfig{img/hypergraph/coherence-product}
\end{itemize}
and construct a monoid in a symmetric way, then show that they satisfy the spider fusion equations for $x \otimes y$.
We can also show that the identity of the unit defines a family of spiders, i.e. $\spider_{a, b}(1) = \id(1)$.
If we take them as axioms rather than definitions, these are called the \emph{coherence conditions} for spiders.

Thus we get to our definition: a \emph{hypergraph category} is a symmetric category with coherent special commutative spiders on each object.
We can take the data to be that of a foo-monoidal category $C$ together with a function $\spider : \N \times \N \times C_0 \to C_1$ or equivalenty, with four functions $\merge, \unit, \ttsplit, \counit : C_0 \to C_1$.
Once we fix the spiders for generating objects, we get spiders for any type (i.e. list of objects).
A hypergraph functor is a symmetric functor $F : C \to D$ between hypergraph categories such that $F \fcmp \spider_{a, b} = \spider_{a, b} \fcmp F$.
Thus we get a category $\mathbf{HypCat}$ with a forgetful functor $U : \mathbf{HypCat} \to \mathbf{MonSig}$.
Its left adjoint $F^H : \mathbf{MonSig} \to \mathbf{HypCat}$ is defined as a quotient $F^S(\Sigma^H) / R$ of the free symmetric category generated by $\Sigma^H = \Sigma \spider$ and the relation $R$ given by the equations for commutative spiders.
Equivalently, we can take $\Sigma^H = \bigcup \{ \Sigma, \merge, \unit, \ttsplit, \counit \}$ and $R$ given by the equations for special commutative Frobenius algebras.
A \emph{$\dagger$-hypergraph category} is a $\dagger$-symmetric category (i.e. the swaps are unitaries) where the dagger is a hypergraph functor.
We also require that the monoid of phases is in fact a group with the dagger as inverse or equivalently, that phase shifts are unitaries.

\begin{example}
For every commutative rig $\S$, $\mathbf{Tensor}_\S$ is a $\dagger$-hypergraph category with the transpose as dagger.
Arguably, special commutative Frobenius algebras were first defined by Peirce~\cite{Peirce06} with their interpretation in the category of relations, or equivalently $\mathbf{Tensor}_\B$.
Indeed, they correspond to what Peirce calls \emph{lines of identity}: they express in two dimensions what one-dimensional first-order logic would express with equality symbols.
For example, take a binary predicate encoded as a box $p : 1 \to x^2$ (interpreted as the formula $\exists \ a \cdot \exists \ b \cdot p(a, b)$) then the diagram $p \fcmp \merge(x)$ is interpreted as the formula $\exists \ a \cdot \exists \ b \cdot p(a, b) \land a = b$ or equivalently $\exists \ a \cdot p(a, a)$.
Thus, every first-order logic formula can be written as a diagram with boxes for predicates, spiders for identity and bubbles for negation.
The equivalence of formulae can be defined as a quotient of a free hypergraph category with bubbles, i.e. all the rules of first-order logic can be given in terms of diagrams.
\end{example}

\begin{example}
The category of complex tensors $\mathbf{Tensor}_\C$ is $\dagger$-hypergraph with the spiders given in example~\ref{example:tensor-spider}.
Any unitary matrix $U : n \to n$ defines another family of spiders $U^{\otimes a} \fcmp \spider_{\phi, a, b}(n) \fcmp (U^\dagger)^{\otimes b}$.
In fact, every unitary arises in this way, see Heunen and Vicary~\cite[Corollary 5.32]{HeunenVicary19a}.
Thus, the axioms for spiders allow us to define any orthonormal basis without ever mentioning basis vectors: they are merely the states $v : 1 \to n$ for which the comonoid is natural, i.e. $v \fcmp \ttsplit(x) = v \otimes v$ and $v \fcmp \counit(x) = \id(1)$.
\end{example}

\begin{example}
The category $\mathbf{Circ}$ is $\dagger$-hypergraph with the spiders defined in example~\ref{example:circuit-spider}, the evaluation functor $\mathbf{Circ} \to \mathbf{Tensor}_\C$ is a $\dagger$-hypergraph functor.
\end{example}

DisCoPy implements spiders for types of length one (i.e. generating objects) as a subclass of \py{Box} and spiders for arbitrary types as a method \py{Diagram.spiders}.

\begin{python}
{\normalfont Implementation of $\dagger$-hypergraph categories and functors.}

\begin{minted}{python}
class Spider(Box):
    def __init__(self, a: int, b: int, x: Ty, phase=None):
        assert len(x) == 1
        self.object, self.phase = x, phase or 0
        name = "Spider({})".format(', '.join(map(str, (a, b, x, phase))))
        super().__init__(name, dom=x ** a, cod=x ** b)

    def dagger(self):
        a, b, x = len(self.cod), len(self.dom), self.object
        phase = None if self.phase is None else -self.phase
        return Spider(a, b, x, phase)

def coherence(factory):
    def method(cls, a: int, b: int, x: Ty, phase=None) -> Diagram:
        if len(x) == 0 and phase is None: return cls.id(x)
        if len(x) == 1: return factory(a, b, x, phase)
        if phase is not None:  # Coherence for phase shifters.
            shift = cls.tensor(*[factory(1, 1, obj, phase) for obj in x])
            return method(cls, a, 1, x) >> shift >> method(cls, 1, b, x)
        if (a, b) in [(1, 0), (0, 1)]: # Coherence for (co)units.
            return cls.tensor(*[factory(a, b, obj) for obj in x])
        # Coherence for binary (co)products.
        if (a, b) in [(1, 2), (2, 1)]:
            spiders, braids = (
                factory(a, b, x[0], phase) @ method(cls, a, b, x[1:], phase),
                x[0] @ cls.braid(x[0], x[1:]) @ x[1:])
            return spiders >> braids if (a, b) == (1, 2) else braids >> spiders
        if a == 1:  # We can now assume b > 2.
            return method(cls, 1, b - 1, x)\
                >> method(cls, 1, 2, x) @ (x ** (b - 2))
        if b == 1:  # We can now assume a > 2.
            return method(cls, 2, 1, x) @ (x ** (a - 2))\
                >> method(cls, a - 1, 1, x)
        return method(cls, a, 1, x) >> method(cls, 1, b, x)
    return classmethod(method)

Diagram.spiders = coherence(Spider)
Diagram.cups = nesting(lambda x, _: Spider(0, 2, x))
Diagram.caps = nesting(lambda x, _: Spider(2, 0, x))

class Functor(symmetric.Functor):
    def __call__(self, other):
        if isinstance(other, Spider):
            a, b = len(other.dom), len(other.cod)
            x, phase = other.object, other.phase
            return self.cod.ar.spiders(a, b, self(x), phase)
        return super().__call__(other)
\end{minted}
\end{python}

\begin{example}
We can now extend example~\ref{example:monoidal-formula} to arbitrary formulae of first-order logic.
Every variable that appears exactly twice is encoded as a wire (possibly with cups and caps), every variable that appears $n \neq 2$ is encoded as an $n$-legged spider.
For example, the formula $\forall c \ \forall o \ O(c, o) \land R(c) \land C(c) \implies U(c, o)$ (interpreted as ``every object of a rigid cartesian category is also its unit'') can be encoded as a diagram with a wire for $o$ and a four-legged spider for $c$.

\begin{minted}{python}
class Formula(Diagram):
    cut = lambda self: Cut(self)

class Cut(Bubble, Formula):
    method = "_not"
    cast = Formula.cast

class Predicate(Box, Formula):
    cast = Formula.cast

def model(size: dict[Ty, int], data: dict[Predicate, list[bool]]):
    return Functor(ob=size, ar={p: [data[p]] for p in data},
                   dom=Category(Ty, Formula),
                   cod=Category(list[int], Tensor[bool]))

objects, categories = Ty('o'), Ty('c')
has_object, has_unit = [Predicate(p, Ty(), categories @ objects) for p in "OU"]
is_rigid, is_cartesian = [Predicate(p, Ty(), categories) for p in "RC"]

rigid_cartesian_implies_trivial = (
    has_object >> Formula.spiders(1, 3, categories) @ objects
    >> (is_rigid @ is_cartesian @ has_unit.cut()).dagger()).cut()

size = {objects: 2, categories: 2}
predicate_values = itertools.product(*size[categories] * [[0, 1]])
relation_values = itertools.product(*size[categories] * size[objects] * [[0, 1]])

for O, U, R, C in itertools.product(
        *(2 * [predicate_values] + 2 * [relation_values])):
    F = model(size, {has_object: O, has_unit: U, is_rigid: R, is_cartesian: C})
    is_rigid_cartesian_and_has_object = lambda i, j:\
        F(has_object)[i, j] and F(is_rigid)[i] and F(is_cartesian)[i]
    assert F(rigid_cartesian_implies_trivial) == all(
        not is_rigid_cartesian_and_has_object(i, j) or F(has_unit)[i, j]
        for i in range(size[categories]) for j in range(size[objects]))

rigid_cartesian_implies_trivial.draw()
\end{minted}

\ctikzfig{img/hypergraph/rigid-cartesian-implies-trivial}
\end{example}

The equality of hypergraph diagrams reduces to hypergraph isomorphism, it will be discussed in section~\ref{subsection:hypergraph-vs-premonoidal}.
The equality of non-commutative spiders is not implemented yet, spider fusion would be a natural extension of the snake removal algorithm for rigid diagrams: we find pairs fusable spiders then apply interchangers to make them adjacent.
The possible obstructions are more serious for spiders than for cups and caps however, for example consider the diagram $\spider_{0, 3}(x) \ \fcmp \ x \otimes f \otimes g \ \fcmp \ \spider_{3, 0}(x)$.
The two three-legged spiders want to fuse but the boxes $f$ and $g$ stand on the way, the best we can do is to bend their output wires with two cups and get a four-legged spider $\spider_{0, 4}(x) \ \fcmp \ x \otimes f \otimes g \otimes x \ \fcmp \ \ttcup(x) \otimes \ttcup(x)$.

\ctikzfig{img/hypergraph/obstruction}


\subsection{Products \& coproducts} \label{subsection:cartesian}

With hypergraph diagrams, we have enough syntax to discuss quantum protocols and first-order logic.
However, the spiders of hypergraph categories are of no use if we want to interpret our diagrams as (pure) Python functions with \py{tuple} as tensor.
Indeed, $\mathbf{Pyth}$ has the property that every function $f : x \to y \otimes z$ into a composite system $y \otimes z$ is in fact a tensor product $f = f_0 \otimes f_1$ of two separate functions $f_0 : x \to y$ and $f_1 : x \to z$.
If a Python type $x$ had caps (let alone spiders) then we could break them in two with the consequence that the identity function on $x$ is constant, i.e. $x$ is trivial~\cite[Proposition~4.76]{CoeckeKissinger17}.
Moreover, there is only one (pure) effect of every type, discarding it.
Thus if a Python type $x$ had cups then we could break them apart as well with the same consequence: only the trivial Python type can have spiders.
A similar argument destroys our hopes for time reversal in Python: if we had a monoidal dagger on $\mathbf{Pyth}$, every state would be equal to every other.

Now if we go back to the intuition of diagrams as pipelines and their wires as carrying data, not all might be lost about spiders.
Indeed, it makes sense to split a data-carrying wire: it means we are copying information.
Closing a data-carrying wire is the counit of the copying comonoid, it means we are deleting information.
In this context, the special condition would translate as follows: if we copy some data then merge the two copies back together, then we haven't done anything.
In order for the spider fusion equations to hold, we would need the monoid to take any two inputs and assert that they are equal or abort the computation otherwise, i.e. we would need side effects.
Even more weirdly, we would need the unit of the monoid to be equal to anything else.

Rather than complaining that classical computing is weird because we cannot coherently merge data back together, we should embrace this as a feature, not a bug: in Python we can copy and discard data (at least assuming that we have enough RAM and that the garbage collector is doing its job).
This means we can still keep the comonoid half of our spiders, forget that they are spiders and come to realise that they are in fact \emph{natural comonoids}, i.e. every function is a comonoid homomorphism.
Indeed, the functions \py{copy = lambda *xs: xs + xs} and \py{delete = lambda *xs: ()} define a pair of natural transformations $x \to x \otimes x$ and $x \to 1$ in $\mathbf{Pyth}$:
\begin{itemize}
    \item \py{copy(f(xs)) == f(copy(xs)[:n]), f(copy(xs)[n:])}
    \item \py{delete(f(xs)) == delete(xs)}
\end{itemize}
for all pure functions \py{f} and inputs \py{xs} with \py{n = len(xs)}.
Once drawn as a diagram, the naturality equations for comonoids allows us to either copy or delete boxes by passing them through either the coproduct or the counit.

\ctikzfig{img/cartesian/naturality}

A \emph{cartesian category} is a symmetric category with coherent, natural commutative comonoids.
The category $\mathbf{Pyth}$ is an example of cartesian category, as well as the categories $\mathbf{Set}$, $\mathbf{Mon}$, $\mathbf{Cat}$, $\mathbf{MonCat}$, etc.
The category $\mathbf{Mat}_\S$ is also a cartesian category with the direct sum as tensor.
Our definition of cartesian is convenient if we want to draw string diagrams and interpret them as functions but it is rather cumbersome: checking that a given category fits the definition involves a lot of structure (tensor, swaps and comonoids) and many axioms relating them.
In practice, we usually take an equivalent definition: a category $C$ is cartesian if it has \emph{categorical products} and a \emph{terminal object}.
An object $1 \in C_0$ is terminal if there is a unique arrow $\counit(x) : x \to 1$ from each object $C_0$.
An object $x_0 \times x_1 \in C_0$ is the product of two objects $x_0, x_1 \in C_0$ if it comes equipped with a pair of arrows $\pi_0 : x_0 \times x_1 \to x_0$ and $\pi_1 : x_0 \times x_1 \to x_1$
such that for all pairs of arrows $f_0 : y \to x_0$ and $f_1 : y \to x_1$ there is a unique $f = \langle f_0, f_1 \rangle : y \to x_0 \times x_1$ such that $f \fcmp \pi_0 = f_0$ and $f \fcmp \pi_1 = f_1$.
These definitions are usually drawn as commutative diagrams where the full lines are universally quantified and the dotted line is uniquely existentially quantified.

From these two \emph{universal properties} we can deduce that terminal objects and categorical products are unique up to a unique isomorphism.
Given two arrows $f : a \to b$ and $g : c \to d$ we have two arrows $\pi_0 \fcmp f : a \times c \to b$ and $\pi_1 \fcmp g : a \times c \to d$, thus there is a unique $f \times  g = \langle \pi_0 \fcmp f, \pi_1 \fcmp g \rangle : a \times b \to c \times d$.
One can show that this makes the category $C$ a (non-strict) monoidal category.
Furthermore, we can show $C$ is symmetric with the swaps given by $S(x, y) = \langle \pi_1, \pi_0 \rangle : x \times y \to y \times x$.
Finally, we can show $C$ has coherent natural commutative comonoids given by $\ttsplit(x) = \langle \id(x), \id(x) \rangle : x \to x \times x$ and $\counit(x) : x \to 1$.

In the other direction, if $C$ has coherent natural commutative comonoids we can deduce that $1$ is a terminal object from the naturality of the counit.
For any arrows $f_0 : y \to x_0$ and $f_1 : y \to x_1$ we can define
\begin{itemize}
\item $\langle f_0, f_1 \rangle = \ttsplit(y) \fcmp f_0 \otimes f_1$,
\item $\pi_0 = \id(x_0) \otimes \counit(x_1)$ and $\pi_1 = \counit(x_0) \otimes \id(x_1)$,
\end{itemize}
and show that $\otimes = \times$ is in fact a categorical product, see Selinger's survey~\cite[Section 6.1]{Selinger10}.
A functor is cartesian when it preserves the categorical product, or equivalently if it is a symmetric functor that preserves the comonoid.
This defines a category $\mathbf{CCat}$ of cartesian categories and functors.
We can assume that cartesian categories are free-on-objects, i.e. the monoid axioms for objects are equalities rather than natural transformations.
Thus, we get a forgetful functor $U : \mathbf{CCat} \to \mathbf{MonSig}$ with its left adjoint given by a quotient of the free symmetric category $F^C(\Sigma) = F^S(\Sigma \cup \ttsplit \cup \counit) / R$ with the relations $R$ given by the naturality equations for each box.

Taking the opposite definition, a \emph{cocartesian category} is one with a categorical coproduct, or equivalently with a coherent natural commutative monoid.
For example, the category $\mathbf{Set}$ is cocartesian with the disjoint union as tensor.
The category $\mathbf{Pyth}$ is cocartesian with tagged union: the merging function takes a tagged element of an n-fold union and forgets the tag.
While cartesian structures can be thought of in terms of data copying, cocartesian structures formalise conditional branching.
Indeed, when we interpret cocartesian diagrams in $\mathbf{Pyth}$ parallel wires encode the different branches of a program, merging two wires of the same type means forgetting the difference between two branches.

DisCoPy implements free (co)cartesian categories with subclasses of \py{Box} for making and merging \py{n} copies of a type \py{x} of length one.
The class methods \py{copy} and \py{merge} extend this to types of arbitrary length by calling the \py{coherence} subroutine of the previous section.
Cartesian functors take \py{Copy} (\py{Merge}) boxes of its domain to the \py{copy} (\py{merge}) method of its codomain.

\begin{python}
{\normalfont Implementation of free (co)cartesian categories and functors.}

\begin{minted}{python}
class Diagram(symmetric.Diagram):
    @classmethod
    def copy(cls, x: Ty, n=2) -> Diagram:
        def factory(a, b, x, _):
            assert a == 1
            return Copy(x, b)
        return coherence(factory).__func__(cls, 1, n, x)

    @classmethod
    def merge(cls, x: Ty, n=2) -> Diagram:
        return cls.copy(x, n).dagger()

class Box(symmetric.Box, Diagram):
    cast = Diagram.cast

class Swap(symmetric.Swap, Box): pass

Diagram.swap = Diagram.braid = hexagon(Swap)

class Copy(Box):
    def __init__(self, x: Ty, n: int = 2):
        super().__init__(name="Copy({}, {})".format(x, n), dom=x, cod=x ** n)

    dagger = lambda self: Merge(self.dom, len(self.cod))

class Merge(Box):
    def __init__(self, x: Ty, n: int = 2):
        super().__init__(name="Merge({}, {})".format(x, n), dom=x ** n, cod=x)

    dagger = lambda self: Copy(self.cod, len(self.dom))

class Functor(symmetric.Functor):
    dom = cod = Category(Ty, Diagram)

    def __call__(self, other):
        if isinstance(other, Copy):
            return self.cod.ar.copy(self(other.dom), len(other.cod))
        if isinstance(other, Merge):
            return self.cod.ar.merge(self(other.cod), len(other.dom))
        return super().__call__(other)
\end{minted}
\end{python}

\begin{python}\label{listing:python-co-cartesian}
{\normalfont Implementation of $\mathbf{Pyth}$ as a cartesian category.}

\begin{minted}{python}
class Function:
    ...
    @staticmethod
    def copy(x: tuple[type, ...], n: int):
        return Function(lambda *xs: n * xs, dom=x, cod=n * x)
\end{minted}
\end{python}

\begin{example}
We can implement the architecture of a neural network as a cartesian diagram and its evaluation as a functor to \py{Function}.

\begin{minted}{python}
x = Ty('x')
add = lambda n: Box('$+$', x ** n, x)
ReLU = Box('$\\sigma$', x, x)
weights = [Box('w{}'.format(i), x, x) for i in range(4)]
bias = Box('b', Ty(), x)

network = Diagram.copy(x @ x, 2)\
>> Diagram.tensor(*weights) @ bias >> add(5) >> ReLU

F = Functor(ob={x: int}, ar={
        add(5): lambda *xs: sum(xs),
        ReLU: lambda x: max(0, x),
        bias: lambda: -1, **{
            weight: lambda x, w=w: x * w
            for weight, w in zip(weights, range(4))}},
    cod=Category(tuple[type, ...], Function))

assert F(network)(42, 43) == max(0, sum([42 * 0, 43 * 1, 42 * 2, 43 * 3, -1]))
\end{minted}
\end{example}

\begin{example}
In a cartesian category, every monoid is automatically a \emph{bialgebra}, i.e. the monoid is a comonoid homomorphism or equivalently, the comonoid is a homomorphism for the monoid.
When furthermore the monoid has an inverse, then it is automatically a \emph{Hopf algebra}, the generalisation of groups to arbitrary monoidal categories.

\begin{minted}{python}
x = Ty('x')
copy, discard = Copy(x), Copy(x, n=0)
add, minus, zero = Box('+', x @ x, x), Box('-', x, x), Box('0', Ty(), x)

drawing.equation(add >> copy, copy @ copy >> x @ Swap(x, x) @ x >> add @ add)
drawing.equation(zero >> copy, zero @ zero)
\end{minted}
\begin{center}
\input{img/cartesian/add-hom-coproduct.tikzstyles}
\hfill
\input{img/cartesian/zero-hom-coproduct.tikzstyles}
\end{center}
\begin{minted}{python}
drawing.equation(add >> discard, discard @ discard)
drawing.equation(zero >> discard, Diagram.id(Ty()))
\end{minted}
\begin{center}
\input{img/cartesian/add-hom-counit.tikzstyles}
\hfill
\input{img/cartesian/zero-hom-counit.tikzstyles}
\end{center}
\begin{minted}{python}
drawing.equation(copy >> minus @ x >> add,
                 discard >> zero,
                 copy >> x @ minus >> add)
\end{minted}
\ctikzfig{img/cartesian/hopf}

A bialgebra which is also a Frobenius algebra is necessarily trivial, i.e. isomorphic to the unit.

\begin{minted}{python}
drawing.equation(
    Diagram.id(x),
    x @ zero >> x @ copy >> add @ x >> discard @ x,
    x @ zero @ zero >> discard @ discard @ x,
    discard >> zero)
\end{minted}

\ctikzfig{img/cartesian/bialgebra-frobenius-implies-trivial}
\end{example}

A cartesian category that is also a PROP is called a \emph{Lawvere theory}~\cite{Lawvere63}, they were first introduced as a high-level language for \emph{universal algebra}. take the boxes $x^n \to x$ to be the primitive $n$-ary operations of your language (e.g. for rigs we have unary boxes for $0$ and $1$, binary boxes for $+$ and $\times$) then the diagrams in the free Lawvere theory are all the terms of your language.
We can write down any universally quantified axiom as a relation between diagrams and the cartesian functors from the resulting quotient to $\mathbf{Set}$ are \emph{algebras}, i.e. sets equipped with operations that satisfy the axioms.
The natural transformations between models are precisely the homomorphisms between the algebras, i.e. the functions that commute with the operations.
If we add colours back in and allow many generating objects, we can define many-sorted theories such as that of modules, with a ring acting on a group.
If we take every pair of objects $(x, y) \in C_0 \times C_0$ as colour and a box $(x, y) \otimes (y, z) \to (x, z)$ for every possible composition, then we can even define the Lawvere theory of categories with $C_0$ as objects.
Thus, categories (with some fixed objects) can also be seen as the functors from this Lawvere theory to $\mathbf{Set}$, functors can also be seen as the natural transformations between such functors.

The free Lawvere theory with no boxes (only swaps, coproducts and counits) is equivalent to $\mathbf{FinSet}^{op}$, the opposite category to finite sets and functions, with the disjoint union as tensor.
Indeed, a cartesian diagram $f : x^n \to x^m$ can be seen as the graph of a function from the $m$ to $n$ elements.
This free Lawvere theory is also called the \emph{theory of equality}, a functor from it to $\mathbf{Set}$ (i.e. an algebra for the theory) is just a set with its equality relation, a natural transformation between those functors is just a function.
Thus, equality of cartesian diagrams with no boxes reduces to equality of functions between finite sets, which can be implemented as equality of finite dictionaries in Python.
It is easy to show that the rules for naturality are confluent and terminating when applied from left to right, i.e. we copy (delete) every box by passing it down through all possible coproducts (counits).
At each rewrite step there is one fewer box node above a comonoid node, thus we can reduce the word problem for free cartesian categories to that of free symmetric categories and then to graph isomorphism.

\begin{python}
Implementation of the free cartesian category $\mathbf{FinSet}^{op}$ with \py{int} as objects and \py{Dict} as arrows.

\begin{minted}{python}
@dataclass
class Dict(Composable, Tensorable):
    inside: dict[int, int]
    dom: int
    cod: int

    __getitem__ = lambda self, key: self.inside[key]

    @staticmethod
    def id(x: int = 0): return Dict({i: i for i in range(x)}, x, x)

    @inductive
    def then(self, other: Dict) -> Dict:
        inside = {i: self[other[i]] for i in range(other.cod)}
        return Dict(inside, self.dom, other.cod)

    @inductive
    def tensor(self, other: Dict) -> Dict:
        inside = {i: self[i] for i in range(self.cod)}
        inside.update({
            self.cod + i: self.dom + other[i] for i in range(other.cod)})
        return Dict(inside, self.dom + other.dom, self.cod + other.cod)

    @staticmethod
    def swap(x: int, y: int) -> Dict:
        inside = {i: i + x if i < x else i - x for i in range(x + y)}
        return Dict(inside, x + y, x + y)

    @staticmethod
    def copy(x: int, n: int) -> Dict:
        return Dict({i: i % x for i in range(n * x)}, x, n * x)
\end{minted}
\end{python}

\begin{example}
We can check equality of cartesian diagrams with one generating object and no boxes.

\begin{minted}{python}
x = Ty('x')
copy, discard, swap = Copy(x, 2), Copy(x, 0), Swap(x, x)
F = Functor({x: 1}, {}, cod=Category(int, Dict))

assert F(copy >> discard @ x) == F(Diagram.id(x)) == F(copy >> x @ discard)
assert F(copy >> copy @ x) == F(Copy(x, 3)) == F(copy >> x @ copy)
assert F(copy >> swap) == F(copy)
\end{minted}
\end{example}

A \emph{rig category} has two monoidal structures $\oplus$ and $\otimes$ that satisfy the equations of a rig up to natural isomorphism.
This is the case for the category $\mathbf{Set}$ as well as for $\mathbf{Pyth}$.
The Kronecker product is not a cartesian product for $\mathbf{Mat}_\S$ (since this role is taken by direct sums) but it does form a rig category with the direct sum.
The arrows of free rig categories can be described as (equivalence classes of) three-dimensional \emph{sheet diagrams} where composition, additive and multiplicative tensor are encoded in three orthogonal axes~\cite{ComfortEtAl20}.
These 3d diagrams are a complete language for \emph{dataflow programming}~\cite{Delpeuch20a}, they are not implemented in DisCoPy yet.

\subsection{Biproducts}\label{subsection:biproducts}

A category has \emph{biproducts} if the cartesian and cocartesian structures coincide,
a $\dagger$-category has \emph{$\dagger$-biproducts} when furthermore the monoid is the dagger of the comonoid.
This is the case in the $\dagger$-category $\mathbf{Mat}_\S$ with direct sum $\oplus$ as tensor.

\begin{python}
{\normalfont Implementation of $\dagger$-biproducts for $\mathbf{Mat}_\S$.}

\begin{minted}{python}
class Matrix:
    ...
    @classmethod
    def copy(cls, x: int, n: int) -> Matrix:
        inside = [[
            i + int(j % n * x) == j for j in range(n * x)] for i in range(x)]
        return cls(inside, x, n * x)

    @classmethod
    def merge(cls, x: int, n: int) -> Matrix:
        return cls.copy(x, n).dagger()

    @classmethod
    def basis(cls, x: int, i: int) -> Matrix:
        return cls([[i == j for j in range(x)]], x ** 0, x)
\end{minted}
\end{python}

Biproducts and matrices happen to be intimately related.
Indeed, given any commutative-monoid-enriched category $C$ we can construct its \emph{free biproduct completion} $\mathbf{Mat}_C$ as the monoidal category with objects given by $C_0^\star$ and arrows $f : x \to y$ given by matrices $f_{ij} : x_i \to y_j$ of arrows in $C_1$.
Composition in $\mathbf{Mat}_C$ is an extension of the usual matrix multiplication with composition as product.
In particular, if $C = \S$ is a rig, i.e. a one-object CM-enriched category, then this definition coincides with the usual one.
Any category with biproducts is automatically enriched in commutative monoids with $f + g : x \to y$ given by $\ttsplit(x) \fcmp f \oplus g \fcmp \merge(x)$ and the zero morphism $0 = \counit(x) \fcmp \unit(y)$.
One can verify that the free completion is indeed the left adjoint to the forgetful functor from biproducts to CM-enrichment, see~\cite[Exercise VIII.2.6]{MacLane71}.
Similarly, if $C$ is a $\dagger$-category then $\mathbf{Mat}_C$ is its free $\dagger$-biproduct completion with the element-wise dagger of the transpose.
If $C$ is also a monoidal category, then $\mathbf{Mat}_C$ is a rig category with the tensor given by an extension of the usual Kronecker product with tensor as product.

We implement free $\dagger$-biproduct completion as a subclass of \py{Matrix[Sum]} with \py{tuple[Ty, ...]} as objects and addition given by the formal sum of diagrams.
We use Python's duck typing to lift the code for composition, tensor and direct sum from \py{int} to \py{tuple[Ty, ...]}.
We also use a \py{contextmanager} to temporarily replace the multiplication of two \py{Sum} entries by composition or tensor.
Again, we override equality so that diagrams are equal to the matrix of just themselves.

\begin{python}
{\normalfont Implementation of free $\dagger$-biproduct completion.}

\begin{minted}{python}
@dataclass
class FakeInt:
    inside: tuple[Ty, ...] = (Ty(), )

    __index__ = lambda self: len(self.inside)
    __iter__ = property(lambda self: self.inside.__iter__)
    __add__ = lambda self, other: FakeInt(self.inside + other.inside)
    __mul__ = lambda self, other: FakeInt(
        tuple(x0 @ x1 for x0 in self.inside for x1 in other))
    __rmul__ = lambda self, n: FakeInt(n * self.inside)
    __pow__ = lambda self, n: product(n * (self, ), unit=FakeInt())

class Diagram(monoidal.Diagram):
    def __eq__(self, other):
        if isinstance(other, Biproduct):
            return other.inside == [[self]]
        return monoidal.Diagram.__eq__(self, other)

    def direct_sum(self, *others):
        return Biproduct.cast(self).direct_sum(*others)

    __or__ = direct_sum

class Box(monoidal.Box, Diagram):
    cast = Diagram.cast

class Sum(monoidal.Sum, Box):
    id = lambda x: Sum.cast(Diagram.id(x))

Diagram.sum = Sum

class Biproduct(Matrix):
    dtype = Sum

    def __init__(self, inside: list[list[Sum]], dom: FakeInt, cod: FakeInt):
        self.dom, self.cod, self.inside = dom, cod, [[
            self.dtype.id(x) if val == 1
            else self.dtype.zero(x, y) if val == 0
            else self.dtype.cast(val)
            for y, val in zip(cod, row)] for x, row in zip(dom, inside)]

    @contextmanager
    def fake_multiplication(self, method):
        self.dtype.__mul__ = getattr(self.dtype, method)
        yield
        delattr(self.dtype, "__mul__")

    @classmethod
    def cast(cls, old: Diagram):
        if isinstance(old, cls): return old
        return cls([[old]], FakeInt((old.dom, )), FakeInt((old.cod, )))

    @inductive
    def then(self, other: Biproduct | Diagram) -> Biproduct:
        with self.fake_multiplication("then"):
            return Matrix.then(self, self.cast(other))

    @inductive
    def tensor(self, other: Biproduct | Diagram) -> Biproduct:
        with self.fake_multiplication("tensor"):
            return Matrix.Kronecker(self, self.cast(other))

    @inductive
    def direct_sum(self, other: Biproduct | Diagram) -> Biproduct:
        with self.fake_multiplication("then"):
            return Matrix.direct_sum(self, self.cast(other))

    dagger = lambda self: self.transpose().map(lambda f: f.dagger())
    __eq__ = lambda self, other: Matrix.__eq__(self, self.cast(other))
\end{minted}
\end{python}

\begin{example}
We can define the object \py{bit} as the list of two empty types, with \py{true} and \py{false} the two basis states
then we can implement conditional expressions as biproducts.

\begin{minted}{python}
unit = FakeInt()
true = Biproduct.copy(unit, 2)\
    >> (Biproduct.id(unit) | Biproduct.zero(unit, unit))
false = Biproduct.copy(unit, 2)\
    >> (Biproduct.zero(unit, unit) | Biproduct.id(unit))

x, y = Ty('x'), Ty('y')
f, g = Box('f', x, y), Box('g', x, y)
conditional = (f | g) >> Biproduct.merge(FakeInt((y, )), 2)

assert true @ FakeInt((x, )) >> conditional == f\
    and false @ FakeInt((x, )) >> conditional == g
\end{minted}
\end{example}

\begin{example}\label{example:biproduct-measurement}
We can implement quantum measurements as biproducts of two quantum effects and classical control as a biproduct of two quantum states.
When we compose classical control with measurement, we get a matrix where the entries are scalar diagrams.
The squared amplitude of the evaluation of these scalars give us the measurement probabilities for each classical choice of state.
We leave the implementation of such biproduct-valued functors to future work: it would require to augment the syntax of types from monoids to semirings.
\end{example}

As we mentioned at the end of section~\ref{section:monoidal}, DisCoPy uses a \emph{point-free} syntax and it can be rather tedious to define any complex diagram in this way.
It is straightforward to extend the \py{diagramize} method to cartesian diagrams, so that they can be defined using the standard syntax for Python functions, where we can use arguments any number of times in any order.
Extending it to cocartesian diagrams so that they can be defined using the standard Python syntax for conditionals will likely be more challenging.
Given enough engineering, it would be possible to turn any pure Python function into a diagram, however this will require more structure than just (co)cartesian categories.
Functions with side effects can be seen as arrows in \emph{premonoidal categories} which are the topic of section~\ref{section:premonoidal}, while recursive functions are arrows in \emph{traced categories}, both will be discussed in section~\ref{section:premonoidal}.
Higher-order functions are modeled as arrows in \emph{closed categories}, the topic of the next section.


\subsection{Closed categories} \label{subsection:closed}

As we have seen in sections~\ref{subsection:hypergraph} and \ref{subsection:cartesian}, cartesian categories like $\mathbf{Pyth}$ and hypergraph categories like $\mathbf{Tensor}$ are two orthogonal extensions of monoidal categories.
The former have natural comonoids on each object, the latter have spiders on each object, an object that has both is necessarily trivial.
Nevertheless, the category $\mathbf{Pyth}$ does share a common structure with rigid categories beyond being monoidal: both are \emph{closed} monoidal categories.
A monoidal category $C$ is left-closed if for every object $x \in C_0$, the functor $x \otimes - : C \to C$ has a right adjoint $x \backslash - : C \to C$ called \emph{$x$ under $-$}.
Symmetrically, $C$ is right-closed if the functor $- \otimes x : C \to C$ has a right adjoint $- / x : C \to C$ called \emph{$-$ over $x$}.\footnote
{There is a simple mnemonic to remember what comes over or under: the input is under the slash in the same way that the denominator is under the fraction bar, it gets canceled when multiplied on the appropriate side $x / y \otimes y \to x$ and $y \otimes y \backslash x \to x$.}
A \emph{closed (monoidal) category} is one that is closed on the left and the right.
For example, a rigid category is a closed category where the over and under types have the form $x \backslash y = x^r \otimes y$ and $y / x = y \otimes x^l$.
When the category is symmetric, over and under types coincide, they are called \emph{exponentials} $x \backslash y = y / x = y^x$.

\begin{example}\label{example:residuated-monoids}
A discrete monoidal category (i.e. a monoid) is closed if and only if it is a group.
A closed preordered monoid (i.e. a closed category with at most one arrow between any two objects) is also called a \emph{residuated monoid}~\cite{Coecke13}, their application to NLP will be discussed in section~\ref{section:NLP}.
The powerset of any monoid $M$ can be given the structure of a residuated monoid where:
\begin{itemize}
    \item $X \otimes Y = \{ x y \in M \ \vert \ x \in X \land y \in Y \}$,
    \item $(X / Y) = \{ z \in M \ \vert \ \forall y \in Y \cdot z y \in X \}$,
    \item $(X \backslash Y) = \{ z \in M \ \vert \ \forall x \in X \cdot x z \in Y \}$.
\end{itemize}
for all subsets $X, Y \sub M$.
\end{example}

As the name suggests, a \emph{cartesian closed category} is a cartesian category that is also closed.
Examples of cartesian closed categories include $\mathbf{Set}$ with the exponential $Y^X$ given by the set of functions from $X$ to $Y$ and $\mathbf{Cat}$ with $D^C$ the category of functors from $C$ to $D$ with natural transformations as arrows.
The category $\mathbf{Pyth}$ with \py{tuple[type, ...]} as objects and pure functions between tuples as arrows is also cartesian closed, the exponential of two lists of types \py{x, y} is given by \py{Callable[x, tuple[y]]}.
The natural isomorphism $\Lambda : \mathbf{Pyth}(x \times y, z) \to \mathbf{Pyth}(y, z^x)$ is called \emph{currying}, after the founding father of functional programming Haskell Curry.
A function of two arguments $x \times y \to z$ is the same as a one-argument higher-order function $y \to z^x$.
Taking the equivalent definition of adjunctions, the unit $\eta_y : y \to (y \times x)^x$ is given by concatenation, i.e. \py{lambda *ys: lambda *xs: ys + xs}, while the counit $\epsilon_y : y^x \times x \to x$ is given by evaluation, i.e. \py{lambda f, *xs: f(*xs)}.
In fact, we can take the data for a cartesian closed category to be that of a cartesian category $C$ together with:
\begin{itemize}
\item an operation $\exp : C_0^\star \times C_0^\star \to C_0$ sending every pair of types to a generating object,
\item an operation $\mathtt{ev} : C_0^\star \times C_0^\star \to C_1$ sending every pair of types $x, y$ to an arrow $\mathtt{ev}(x, y) : \exp(y, x) \times x \to y$,
\item an operation $\Lambda_n : C_1 \to C_1$ for each $n \in \N$, sending every arrow $f : x \times y \to z$ with $x$ of length $n$ to an arrow $\Lambda_n(f) : y \to \exp(z, x)$.
\end{itemize}
We can take the axioms to be those of cartesian categories together with $\Lambda_n(f \times x \ \fcmp \ \mathtt{ev}(z, x)) = f$ for all $f : y \to \exp(z, x)$.
Intuitively, if we take a higher-order function, evaluate it then abstract away the result, we get back to where you started, i.e. $\Lambda_n^{-1}(f) = f \times x \ \fcmp \ \mathtt{ev}(z, x)$.

\begin{python}\label{example:closed-function}
{\normalfont Implementation of $\mathbf{Pyth}$ as a cartesian closed category.}

\begin{minted}{python}
Ty = tuple[type, ...]

def exp(base: Ty, exponent: Ty) -> Ty:
    return (Callable[exponent, tuple[base]], )

class Function:
    ...
    def curry(self, n=1, left=True) -> Function:
        inside = lambda *xs: lambda *ys: self(*(xs + ys) if left else (ys + xs))
        if left:
            dom = self.dom[:len(self.dom) - n]
            cod = exp(self.cod, self.dom[len(self.dom) - n:])
        else: dom, cod = self.dom[n:], exp(self.cod, self.dom[:n])
        return Function(inside, dom, cod)

    @staticmethod
    def ev(base: Ty, exponent: Ty, left=True) -> Function:
        if left:
            inside = lambda f, *xs: f(*xs)
            return Function(inside, exp(base, exponent) + exponent, base)
        inside = lambda *xs: xs[-1](*xs[:-1])
        return Function(inside, exponent + exp(base, exponent), base)

    def uncurry(self, left=True) -> Function:
        base, exponent = self.cod[0].__args__[-1], self.cod[0].__args__[:-1]
        base = tuple(base.__args__) if is_tuple(base) else (base, )
        return self @ exponent >> Function.ev(base, exponent) if left\
            else exponent @ self >> Function.ev(base, exponent, left=False)

    exp = under = over = staticmethod(exp)
\end{minted}
\end{python}

\begin{example}
We can check the axioms for cartesian closed categories hold in $\mathbf{Pyth}$.

\begin{minted}{python}
x, y, z = (complex, ), (bool, ), (float, )
f = Function(dom=y, cod=exp(z, x),
             inside=lambda y: lambda x: abs(x) ** 2 if y else 0)
g = Function(dom=x + y, cod=z, inside=lambda x, y: f(y)(x))

assert f.uncurry().curry()(True)(1j) == f(True)(1j)
assert g.curry().uncurry()(True, 1j) == g(True, 1j)
\end{minted}
\end{example}

A (strict) cartesian closed functor is a cartesian functor $F$ which respects the exponential, i.e. $F(y^x) = F(y)^{F(x)}$.
Thus, we get a category $\mathbf{CCCat}$ of cartesian closed categories and functors, with a forgetful functor $U : \mathbf{CCCat} \to \mathbf{MonSig}$.
Its left adjoint $F^{CC} : \mathbf{MonSig} \to \mathbf{CCCat}$ can be constructed in two steps.
First, we define a \emph{closed signature} as a signature $\Sigma$ with a pair of binary operators $(- / -), \ (- \backslash -) : \Sigma_0 \times \Sigma_0 \to \Sigma_0$, i.e. for every pair $x, y \in \Sigma_0$ there are two generating objects $x / y, \ y \backslash x \in \Sigma_0$.
A closed monoidal signature is both a monoidal signature (i.e. its generating objects are a free monoid) and a closed signature.
Morphisms of closed monoidal signatures are defined in the obvious way, thus we get a category $\mathbf{CMonSig}$ with two forgetful functors $U : \mathbf{CCCat} \to \mathbf{CMonSig}$ and $U : \mathbf{CMonSig} \to \mathbf{MonSig}$.
The left adjoint $F^C : \mathbf{MonSig} \to \mathbf{CMonSig}$ takes a monoidal signature and freely adds extra objects for the over and under slashes, it can be defined by induction on the number of nested slashes.
Now we can define the left adjoint $F : \mathbf{CMonSig} \to \mathbf{CCCat}$ as a quotient of the free cartesian category with boxes for evaluations, bubbles for currying (i.e. by induction on the number of nested curryings) and relations given by the axioms for natural isomorphisms.
Composing the two adjunctions we get $F^{CC} : \mathbf{MonSig} \to \mathbf{CCCat}$.

\begin{remark}
Diagrams in free cartesian closed categories can also be seen as terms of the \emph{simply typed lambda calculus} (up to $\beta\eta$-equivalence) or as proofs in \emph{minimal logic} (the fragment of propositional logic with only conjunction and implication).
See Abramsky and Tzevelekos~\cite{AbramskyTzevelekos11} for an introduction to this \emph{Curry-Howard-Lambek correspondence}.
The problem of deciding the ($\beta\eta$) equivalence of a given pair of simply typed lambda terms (or equivalently the equivalence of two minimal logic formulae) is decidable~\cite{Tait67} but not elementary recursive~\cite{Statman79}, i.e. its time complexity is not bounded by any tower of exponentials.
As such, the word problem for free cartesian closed categories is as intractable as it gets.
If we remove the copying and discarding, the word problem for free symmetric closed categories is decidable in linear time~\cite{Voreadou77}.
The algorithm is based on a variant of Gentzen's cut-elimination theorem~\cite{Gentzen35} for \emph{multiplicative intuitionistic linear logic} (MILL), a \emph{substructural logic} where the \emph{weakening} and \emph{contraction} rules are omitted, i.e. we cannot discard or copy assumptions.
To the best of our knowledge, the case of (non-symmetric) closed monoidal categories is still open.
We conjecture it can be solved with a variant of cut-elimination for a \emph{non-commutative logic} omitting the \emph{exchange} rule, i.e. we cannot swap assumptions.
\end{remark}

DisCoPy implements the types of the closed diagrams with a subclass \py{Exp} of \py{Ty} for exponentials.
\py{Over} and \py{Under} are two subclasses of \py{Exp}, shortened to \py{x >> y} and \py{y << x}, and attached to the \py{Diagram} class as two static methods \py{over} and \py{under}.
We need to initalise the type \py{self: Exp} with some list of objects, our only choice is \py{self.inside=[self]}.
Thus, we need to override the equality and printing methods so that we don't fall into infinite recursion.
We also need to override the \py{cast} method so that the unit law is satisfied, i.e. \py{self @ Ty() == self == Ty() @ self}.

\begin{python}
{\normalfont Implementation of the types in free closed categories.}

\begin{minted}{python}
class Ty(monoidal.Ty):
    @classmethod
    def cast(cls, old: monoidal.Ty) -> Ty:
        return old[0] if len(old) == 1 and isinstance(old[0], Exp) else cls(old)

    def __pow__(self, other):
        return Exp(self, other) if isinstance(other, Ty)\
            else super().__pow__(other)

class Exp(Ty):
    cast = Ty.cast

    def __init__(self, base, exponent):
        self.base, self.exponent = base, exponent
        super().__init__(inside=(self, ))

    def __eq__(self, other):
        return isinstance(other, type(self))\
            and (self.base, self.exponent) == (other.base, other.exponent)

    __str__ = lambda self: "({} ** {})".format(self.base, self.exponent)

class Over(Exp):
    __str__ = lambda self: "({} << {})".format(self.base, self.exponent)

class Under(Exp):
    __str__ = lambda self: "({} >> {})".format(self.exponent, self.base)

Ty.__lshift__ = lambda self, other: Over(self, other)
Ty.__rshift__ = lambda self, other: Under(other, self)
\end{minted}
\end{python}

Closed diagrams are implemented with two subclasses of \py{Box} for currying and evaluation, which are attached to the \py{Diagram} class as two static methods \py{curry} and \py{ev}.
We shorten \py{Diagram.ev(base, exponent, left)} to \py{Ev(exponent >> base)} if \py{left} else \py{Ev(base << exponent)}.
Closed functors map \py{Exp} types to the \py{over} and \py{under} methods of their codomain, similarly for \py{Curry} and \py{Uncurry} boxes.

\begin{python}
{\normalfont Implementation of free closed categories and functors.}

\begin{minted}{python}
class Diagram(monoidal.Diagram):
    curry = lambda self, n=1, left=True: Curry(self, n, left)

    @staticmethod
    def ev(base: Ty, exponent: Ty, left=True) -> Ev:
        return Ev(base << exponent if left else exponent >> base)

    def uncurry(self: Diagram, left=True) -> Diagram:
        base, exponent = self.cod.base, self.cod.exponent
        return self @ exponent >> Ev(base << exponent) if left\
            else exponent @ self >> Ev(exponent >> base)

class Box(monoidal.Box, Diagram):
    cast = Diagram.cast

class Ev(Box):
    def __init__(self, x: Exp):
        self.base, self.exponent = x.base, x.exponent
        self.left = isinstance(x, Over)
        dom, cod = (x @ self.exponent, self.base) if self.left\
            else (self.exponent @ x, self.base)
        super().__init__("Ev" + str(x), dom, cod)

class Curry(Box):
    def __init__(self, diagram: Diagram, n=1, left=True):
        self.diagram, self.n, self.left = diagram, n, left
        name = "Curry({}, {}, {})".format(diagram, n, left)
        if left:
            dom = diagram.dom[:len(diagram.dom) - n]
            cod = diagram.cod << diagram.dom[len(diagram.dom) - n:]
        else: dom, cod = diagram.dom[n:], diagram.dom[:n] >> diagram.cod
        super().__init__(name, dom, cod)

Diagram.over, Diagram.under, Diagram.exp = map(staticmethod, (Over, Under, Exp))

class Functor(monoidal.Functor):
    dom = cod = Category(Ty, Diagram)

    def __call__(self, other):
        for cls, attr in [(Over, "over"), (Under, "under"), (Exp, "exp")]:
            if isinstance(other, cls):
                method = getattr(self.cod.ar, attr)
                return method(self(other.base), self(other.exponent))
        if isinstance(other, Curry):
            return self.cod.ar.curry(
                self(other.diagram), len(self(other.cod.exponent)), other.left)
        if isinstance(other, Ev):
            return self.cod.ar.ev(
                self(other.base), self(other.exponent), other.left)
        return super().__call__(other)
\end{minted}
\end{python}

\begin{example}
We can check the axioms by applying functors into $\mathbf{Pyth}$ and evaluating the result on some test input.

\begin{minted}{python}
x, y, z = map(Ty, "xyz")
f, g = Box('f', y, z << x), Box('g', y, z >> x)

F = Functor(
    ob={x: complex, y: bool, z: float},
    ar={f: lambda y: lambda x: abs(x) ** 2 if y else 0,
        g: lambda y: lambda z: z + 1j if y else -1j}
    cod=Category(Ty, Function))

assert F(f.uncurry().curry())(True)(1j) == F(f)(True)(1j)
assert F(g.uncurry(left=False).curry(left=False))(True)(1.2) == F(g)(True)(1.2)
\end{minted}
\end{example}

Understanding the relationship between closed and rigid categories will also explain how we draw closed diagrams, using the bubble notation introduced by Baez and Stay~\cite[Section~2.6]{BaezStay10}.
Indeed, in a free rigid category the exponentials are given as a tensor of a type and an adjoint, thus we can draw them as two wires side by side.
On the other hand, the exponentials of a free closed category are defined as generating objects, they ought to be drawn as one wire but we can decide to draw them as two, inseparable wires.
This constraint can be materialised by a \emph{clasp} that binds the two wires together.
Similarly, in free rigid categories we can draw evaluation and currying as diagrams with cups and caps while in a free closed category they are defined as generating boxes, which ought to be drawn as black boxes.
We can decide to draw them the same way as in a rigid category, with a bubble surrounding them to prohibit illicit rewrites.
Once drawn in this way, the equations for currying become a special case of the snake equations, although in general closed categories do not have boxes for cups and caps.

\begin{python}
{\normalfont Implementation of free rigid categories as closed categories.}

\begin{minted}{python}
rigid.Ty.__lshift__ = lambda self, other: self @ other.l
rigid.Ty.__rshift__ = lambda self, other: self.r @ other
rigid.Diagram.over = staticmethod(lambda base, exponent: base << exponent)
rigid.Diagram.under = staticmethod(lambda base, exponent: exponent >> base)

@classmethod
def ev(cls, base: rigid.Ty, exponent: rigid.Ty, left=True) -> rigid.Diagram:
    return base @ cls.cups(exponent.l, exponent) if left\
        else cls.cups(exponent, exponent.r) @ base

def curry(self: rigid.Diagram, n=1, left=True) -> rigid.Diagram:
    if left:
        base, exponent = self.dom[:n], self.dom[n:]
        return base @ self.caps(exponent, exponent.l) >> self @ exponent.l
    offset = len(self.dom) - n
    base, exponent = self.dom[offset:], self.dom[:offset]
    return self.caps(exponent.r, exponent) @ base >> exponent.r @ self

Diagram.ev, Diagram.curry = ev, curry
\end{minted}
\end{python}

\begin{example}
We can draw closed diagrams by applying a functor to a rigid category with bubbled evaluation and currying.

\begin{minted}{python}
class ClosedDrawing(rigid.Diagram):
    ev = staticmethod(lambda base, exponent, left=True:
        rigid.Diagram.ev(base, exponent, left).bubble())
    curry = lambda self, n=1, left=True:\
        rigid.Diagram.curry(self, n, left).bubble()

Draw = Functor(lambda x: x, lambda f: f, cod=Category(rigid.Ty, ClosedDrawing))
Diagram.draw = lambda self, **params: Draw(self).draw(**params)

f, g, h = Box('f', x, z << y), Box('g', x @ y, z), Box('h', y, x >> z)

drawing.equation(f.uncurry().curry(), f)
drawing.equation(h.uncurry(left=False).curry(left=False), h)
\end{minted}

\begin{center}
\input{img/closed/right-curry.tikzstyles}
\hfill
\input{img/closed/left-curry.tikzstyles}
\end{center}

\begin{minted}{python}
drawing.equation(g.curry().uncurry(), g, g.curry(left=False).uncurry(left=False))
\end{minted}

\ctikzfig{img/closed/uncurry}
\end{example}


\subsection{Traced categories} \label{subsection:traced}

The category $\mathbf{Pyth}$ has another categorical structure in common with hypergraph categories like $\mathbf{Tensor}$: they are both \emph{traced symmetric categories}.
A symmetric category is traced when it comes equipped with a family of functions $\mathtt{trace}_x(a, b) : C(a \otimes x, b \otimes x) \to C(a, b)$ subject to axioms that formalise the intuition that we can trace a morphism $f : a \otimes x \to b \otimes x$ by connecting its input and output $x$-wires in a loop, see Joyal et al.~\cite{JoyalEtAl96}.

Every compact closed category has a trace given by cups and caps, traced symmetric categories allow to express recursion and fixed points more generally in non-rigid categories.
Indeed in the case of a cartesian category such as $\mathbf{Pyth}$, the trace can equivalently be given in terms of \emph{fixed point operators} $\mathtt{fix}_x : C(a \times x, x) \to C(a, x)$~\cite[Proposition~6.8]{Selinger10}.
Dually, in a cocartesian category the trace can be defined in terms \emph{iteration operators} $\mathtt{iter}_x : C(x, a + x) \to C(x, a)$.
When the category has biproducts, it is sufficient to define a \emph{repetition operator} $\mathtt{repeat}_x : C(x, x) \to C(x, x)$~\cite[Proposition~6.11]{Selinger10}.
In the category of finite sets and relations $\mathbf{Mat}_\B$ with the direct sum as tensor, this coincides with the usual notion of reflexive transitive closure~\cite[Proposition~6.3]{JoyalEtAl96}.

\begin{python}
{\normalfont Implementation of the syntax for free traced categories.}

\begin{minted}{python}
class Diagram(symmetric.Diagram):
    def trace(self, n=1):
        return Trace(self, n)

class Box(symmetric.Box, Diagram):
    cast = Diagram.cast

class Trace(Box):
    def __init__(self, diagram: Diagram, n=1):
        assert diagram.dom[-n:] == diagram.cod[-n:]
        self.diagram, name = diagram, "Trace({}, {})".format(diagram, n)
        super().__init__(name, diagram.dom[:-n], diagram.cod[:-n])

class Functor(symmetric.Functor):
    dom = cod = Category(Ty, Diagram)

    def __call__(self, other):
        if isinstance(other, Trace):
            n = len(self(other.diagram.dom)) - len(self(other.dom))
            return self.cod.ar.trace(self(other.diagram), n)
        return super().__call__(other)
\end{minted}
\end{python}

\begin{example}
We can draw traced diagrams by applying a traced functor into compact-closed categories with bubbles.

\begin{minted}{python}
def compact_trace(self, n=1):
    return self.dom[:-n] @ self.caps(self.dom[-n:], self.dom[-n:].r)\
        >> self @ self.dom[-n:].r\
        >> self.cod[:-n] @ self.cups(self.cod[-n:], self.cod[-n:].r)

compact.Diagram.trace = compact_trace

class TracedDrawing(compact.Diagram):
    trace = lambda self, n: compact_trace(self, n).bubble()

Draw = Functor(lambda x: x, lambda f: f, cod=Category(Ty, TracedDrawing))
Diagram.draw = lambda self, **params: Draw(self).draw(**params)

a, b, x = map(Ty, "abx")
Box('f', a @ x, b @ x).trace().draw()
\end{minted}
\ctikzfig{img/closed/trace}
\end{example}

\begin{python}\label{listing:traced-python}
{\normalfont Implementation of $\mathbf{Pyth}$ as a traced cartesian category.}

\begin{minted}{python}
class Function:
    ...
    def fix(self, n=1):
        if n > 1: return self.fix().fix(n - 1)
        dom, cod = self.dom[:-1], self.cod
        def inside(*xs, y=None):
            result = self.inside(*xs + (() if y is None else (y, )))
            return y if result == y else inside(*xs, y=result)
        return Function(inside, dom, cod)

    def trace(self, n=1):
        dom, cod, traced = self.dom[:-n], self.cod[:-n], self.dom[-n:]
        fixed = (self >> self.discard(cod) @ traced).fix()
        return self.copy(dom) >> dom @ fixed\
            >> self >> cod @ self.discard(traced)
\end{minted}
\end{python}

\begin{example}
We can compute the golden ratio as a fixed point.
Note that in order to find a fixed point we need a default value to start from.

\begin{minted}{python}
phi = Function(lambda x=1: 1 + 1 / x, [int], [int]).fix()
assert phi() == (1 + sqrt(5)) / 2
\end{minted}
\end{example}

\begin{python}\label{listing:traced-matrix}
{\normalfont Implementation of $\mathbf{Mat}_\S$ as a traced biproduct category.}

\begin{minted}{python}
class Matrix:
    ...
    def repeat(self):
        assert self.dtype is bool and self.dom == self.cod
        return sum(
            Matrix.id(self.dom).then(*n * [self]) for n in range(self.dom + 1))

    def trace(self, n=1):
        assert self.dtype is bool
        A, B, C, D = (row >> self >> column
                      for row in [self.id(self.dom - n) @ self.unit(n),
                                  self.unit(self.dom - n) @ self.id(n)]
                      for column in [self.id(self.cod - n) @ self.discard(n),
                                     self.discard(self.cod - n) @ self.id(n)])
        return A + (B >> D.repeat() >> C)
\end{minted}
\end{python}


\section{A premonoidal approach} \label{section:premonoidal}

In the previous section, we have seen that cartesian closed categories give us enough syntax to interpret (simply typed) lambda terms.
Thus, we can execute the diagrams in a free cartesian closed category as functions by applying a functor into $\mathbf{Set}$ or $\mathbf{Pyth}$, we can also interpret them as functors by applying a functor into $\mathbf{Cat}$ with the cartesian product as tensor.
Now if we remove the cartesian assumption, the diagrams of free closed categories give us a programming language with higher-order functions where we cannot copy, discard or even swap data: the (non-commutative) \emph{linear lambda calculus}.
With a more restricted language, we get a broader range of possible interpretations.
For example, there can be only one cartesian closed structure on $\mathbf{Cat}$ (because any other would be naturally isomorphic) but are there any other monoidal closed structures?
Foltz, Lair and Kelly~\cite{FoltzEtAl80} answer the question with the positive: $\mathbf{Cat}$ has exactly two closed structures: the usual cartesian closed structure with the exponential $D^C$ given by the category of functors $C \to D$ and natural transformations, and a second one where the exponential $C \Rightarrow D$ is given by the category of functors $C \to D$ and \emph{transformations}, with no naturality requirement.

The corresponding tensor product on $\mathbf{Cat}$, i.e. the left adjoint $C \Box - \dashv C \Rightarrow -$, is called the \emph{funny tensor product}, maybe because mathematicians thought it was funny not to require naturality.
More explicitly, the funny tensor $C \Box D$ can be described as the push-out of $C \times D_0 \leftarrow C_0 \times D_0 \to C_0 \times D$ where $C_0, D_0$ are the discrete categories of objects, or equivalently as a quotient of the coproduct $(C \times D_0 + C_0 \times D) / R$ where the relations are given by $(0, \id(x), y) = (1, x, \id(y))$.
Even more explicitly, the objects of $C \Box D$ are given by the cartesian product $C_0 \times D_0$, the arrows are alternating compositions $(0, f_1, y_1) \fcmp (1, x_2, g_2) \fcmp \dots \fcmp (0, f_{n - 1}, y_{n -1}) \fcmp (1, x_n, g_n)$ of arrows in one category paired with an object of the other.
When the two categories are in fact monoids, the funny tensor is called the \emph{free product} because it sends free monoids to free monoids, i.e. $X^\star \Box Y^\star = (X + Y)^\star$.
This is also true for free categories, i.e. $F(\Sigma) \Box F(\Sigma') = F(\Sigma \times \Sigma'_0 \cup \Sigma_0 \times \Sigma)$, so maybe $\Box$ should be called free rather than funny.
While a functor on a cartesian product $F : C \times D \to E$ can be seen as a functor of two arguments that is functorial in both simultaneously, i.e. $F(f \fcmp f', g \fcmp f') = F(f, g) \fcmp F(f', g')$, a functor on a funny product $F : C \Box D \to E$ is functorial separately in its $C$ and $D$ arguments, i.e. $F(0, f \fcmp f', y) = F(0, f, y) \fcmp F(0, f', y)$ and $F(1, x, g \fcmp g') = F(1, x, g) \fcmp F(1, x, g')$.

\subsection{Premonoidal categories \& state constructions}\label{subsection:state-construction}

Now recall that a (strict) monoidal category $C$ is a monoid in $(\mathbf{Cat}, \times)$, i.e. the tensor is a functor on a cartesian product $\otimes : C \times C \to C$.
In a similar way, we define a (strict) \emph{premonoidal category} as a monoid in $(\mathbf{Cat}, \Box)$, i.e. a category $C$ with an associative, unital functor on the funny product $\boxtimes : C \Box C \to C$.
A (strict) \emph{premonoidal functor}\footnote
{As we mention in remark~\ref{remark:non-strict-premonoidal}, the original definition from Power and Robinson~\cite{PowerRobinson97} requires premonoidal functors to be center-preserving.
This is not necessary for strict premonoidal functors.}
is a functor that commutes with $\boxtimes$, i.e. $F(f \boxtimes g) = F(f) \boxtimes F(g)$, thus we get a category $\mathbf{PreMonCat}$.
As for monoidal categories, we can show that every premonoidal category is equivalent to a foo one, i.e. where the monoid of objects is free, thus we get a forgetful functor $U : \mathbf{PreMonCat} \to \mathbf{MonSig}$.
The image of $\boxtimes$ on objects may be given by concatenation, its image on arrows is called \emph{whiskering}, it is denoted by $\boxtimes(0, f, x) = f \boxtimes x$ and $\boxtimes(1, x, f) = x \boxtimes f$.
As we have seen in section~\ref{section:monoidal}, from whiskering we can define a (biased) tensor product on arrows $f \boxtimes g = f \boxtimes \dom(g) \fcmp \cod(f) \boxtimes g$ and conversely, we can define whiskering as tensoring with identity arrows.

Thus, we can take the data for a premonoidal category $C$ to be the same as that of a foo-monoidal category and the only axioms to be those for $(C_1, \boxtimes, \id(1))$ being a monoid.
That is, a premonoidal category is almost a monoidal category, only the interchange law does not necessarily hold.
Every monoidal category (functor) is also a premonoidal category (functor), hence we have an inclusion functor $\mathbf{MonCat} \injects \mathbf{PreMonCat}$.
An arrow of a premonoidal category $C$ is called \emph{central} if it interchanges with every other arrow, a transformation is called central if every component is central.
Every identity is central and composition preserves centrality, thus we can define the \emph{center} $Z(C)$ as the subcategory of central arrows and show that $Z : \mathbf{PreMonCat} \to \mathbf{MonCat}$.
A \emph{symmetric premonoidal category} is a premonoidal category with a central natural isomorphism $S : x \boxtimes y \to y \boxtimes x$ such that the hexagon equations hold and $S(x, 1) = \id(x) = S(1, x)$.

\begin{remark}\label{remark:non-strict-premonoidal}
The definition of non-strict premonoidal category and functor requires some caution.
Indeed, in order to get a coherence theorem (i.e. in order to prove that every premonoidal category is equivalent to a strict one) we need to assume that the associator and unitor morphisms are central.
Then for the composition of premonoidal functors to be well-defined, we need them to be center-preserving which rules out important examples~\cite{StatonLevy13}.
This motivates the definition of (non-cartesian) \emph{Freyd categories}, also called \emph{effectful categories}~\cite{Roman22}: a triple $(C, V, J)$ of a premonoidal category $C$, a monoidal category $V$ of \emph{values} and an identity-on-objects strict premonoidal functor $J : V \to C$ whose image is central.
An effectful functor is a functor between effectful categories which restricts to a monoidal functor on their values.
Every premonoidal category can be taken as an effectful category with its center as values (in which case effectful functors are center-preserving) or with its discrete category of objects as values (in which case effectful functors don't necessarily preserve centers).
\end{remark}

\begin{example}
A premonoidal category with one object is just a set with two monoid structures.
They do not satisfy the interchange law so the Eckmann-Hilton argument does not apply, the two monoids need not coincide nor be commutative.
In this case, the notion of center coincides with the usual notion of center of a monoid, i.e. the submonoid of elements that commute with everything else.
Indeed, the monoidal center of a one-object premonoidal is the intersection of the centers of its two monoid structures.
\end{example}

\begin{example}
For any small category $C$, the category $C \Rightarrow C$ of endofunctors $C \to C$ with (not-necessarily-natural) transformations as arrows is premonoidal.
\end{example}

\begin{example}
The category of matrices $\mathbf{Mat}_\S$ with entries in a rig $\S$ with the Kronecker product as tensor is a premonoidal category, it is monoidal precisely when $\S$ is commutative.
\end{example}

\begin{example}
The category $\mathbf{Pyth}$ with \py{tuple} as tensor is premonoidal.
Every pure function is in the center $Z(\mathbf{Pyth})$, but the converse is not necessarily true: take the side effect $f : x \to 1$ which increments a private, internal counter every time it is called.
It is impure, but not enough that we can observe it by parallel composition, i.e. although it does not commute with copy and discard, it can still be interchanged with any other function.
In other words, it is in the monoidal center, but not the cartesian center (i.e. the subcategory of comonoid homomorphisms).
\end{example}

Premonoidal categories were introduced by Power and Robinson~\cite{PowerRobinson97} as a way to model programming languages with side effects, reformulating an earlier framework of Moggi~\cite{Moggi91} which captured notions of computation as \emph{monads}.
Our last two examples can be seen as special cases of a more general pattern: they are \emph{Kleisli categories} for a \emph{strong monad}.
Infamously, a monad is just a monoid $T : C \to C, \mu : T \fcmp T \to T$, $\eta : 1 \to T$ in the category $C^C$ of endofunctors with natural transformations as arrows.
Its Kleisli category $K(T)$ has the same objects as $C$ and arrows given by $K(T)(x, y) = C(x, T(y))$, with the identity given by the unit $\id_K(x) = \eta(x)$ and composition given by post-composition with the multiplication, i.e. $f \fcmp_K g = f \fcmp T(g) \fcmp \mu(z)$ for $f : x \to T(y)$ and $g : y \to T(z)$.
Now if $C$ happens to be a monoidal category, we can ask for $T$ to be a monoidal functor, but we also want the multiplication and unit of the monad to play well with the monoidal structure.
We could ask for a \emph{monoidal monad} where $\mu$ and $\eta$ are monoidal transformations, i.e. for the monad to be a monoid in the category of monoidal endofunctors and monoidal natural transformations and show that the Kleisli category $K(T)$ inherits a monoidal structure.
More generally, we can ask only for a (bi)\emph{strong monad}, equipped with two natural transformations $\sigma(a, b) : a \otimes T(b) \to T(a \otimes b)$ and $\tau(a, b) : T(a) \otimes b \to T(a \otimes b)$ subject to sufficient conditions for the Kleisli category $K(T)$ to inherit a premonoidal structure.
It is monoidal precisely when the monad is commutative, i.e. the two arrows from $T(x) \otimes T(y)$ to $T(x \otimes y)$ are equal.

\begin{example}
Take the category $C = \mathbf{Set}$ and the \emph{distribution monad} $T(X) = \{ p : X \to \S \ \vert \ p \text{has finite support} \}$ for a rig $\S$ with the image on arrows given by \emph{pushforward} $T(f : X \to Y)(p : X \to \S) : y \mapsto \sum_{x \in f^{-1}(y)} p(x)$, the multiplication and unit induced by the rig multiplication and unit.
If we construct its Kleisli category and take the subcategory spanned by finite sets, we get the category $\mathbf{Mat}_S$ of matrices seen as functions $m : Y \to \S^X \simeq X \times Y \to \S$.
One can show this is a strong monad, and it is commutative precisely when the rig is commutative.
\end{example}

\begin{example}
Take any closed symmetric category $C$ and the \emph{state monad} $T(x) = s^{s \otimes x}$ for some object $s$, an arrow $f : x \to y$ in the Kleisli category $K(T)$ is given by an arrow $f : s \otimes x \to s \otimes y$ in $C$ (up to uncurrying).
One can show that $T$ is strong and thus $K(T)$ is premonoidal.
When $C = \mathbf{Set}$ the state monad is a non-commutative as it gets: $T$ is commutative if and only if $s$ is trivial.
Whiskering an arrow $f : s \otimes x \to s \otimes y$ by an object $z$ on the left is given by pre- and post-composition with swaps $z \boxtimes f = S(s, z) \otimes x \ \fcmp \ z \otimes f \ \fcmp \ S(z, s) \otimes y$, whiskering on the right is easier $f \boxtimes z = f \otimes z$.
\end{example}

Jeffrey~\cite{Jeffrey97} then gave the first definition of free premonoidal categories, his construction formalises the intuition that non-central arrows are to be thought as arrows with side effects.
The \emph{state construction} takes as input a symmetric monoidal category $C$ and an object $s$, and builds a symmetric premonoidal category $\mathbf{St}(C, s)$ with the same objects as $C$, arrows given by $\mathbf{St}(C, s)(x, y) = C(s \otimes x, s \otimes y)$ and whiskering defined as in the state monad.
Intuitively, an arrow in $\mathbf{St}(C, s)$ is an arrow in $C$ which also updates a global state encoded in the object $s$, which we can draw as an extra wire passing through every box of the diagram, preventing them from being interchanged.
More formally, given a monoidal signature $\Sigma$ we can construct the free symmetric premonoidal category by taking the state construction $\mathbf{St}(F^S(\Sigma + \{ s \}), s)$ over the free symmetric category with an extra object $s$, then taking the subcategory spanned by objects of the form $s \otimes t$ for $t \in \Sigma_0^\star$.
We can generalise this to (non-symmetric) free premonoidal categories but we still need symmetry at least for the extra object, i.e. natural isomorphisms $S(s, x) : s \otimes x \to x \otimes s$ for each object $x$, subject to hexagon and unit equations.
We refer the reader to Roman~\cite{Roman22} for a formalisation of this result in the framework of effectful categories.

We call this definition of the free premonoidal category as a state construction over a free monoidal category with an extra swappable object the \emph{monoidal approach} to premonoidal categories.
In what we call the \emph{premonoidal approach} to monoidal categories, definitions go the other way around with free premonoidal categories as the fundamental notion and free monoidal categories as an interesting quotient.
Indeed, we have been using the arrows of free premonoidal categories all along: they are string diagrams, defined as lists of layers without quotienting by interchanger.
Equivalently, they are labeled generic progressive plane graphs up to generic deformation, i.e. with at most one box node at each height.
While in the monoidal approach, string diagrams are defined as non-planar graphs and the ordering of boxes is materalised by extra wires connecting the boxes in sequence, in the premonoidal approach we take this ordering as data: boxes are in a list.
This comes with an immediate advantage: equality of premonoidal diagrams can be defined in terms of equality of lists, hence it is decidable in linear time whereas equality of monoidal diagrams has quadratic complexity and equality of symmetric diagrams could be as hard as graph isomorphism.
Another advantage of representing string diagrams with lists rather than graphs is that the code for functor application, i.e. the interpretation of diagrams, is a simple \py{for} loop rather than an elaborate graph algorithm.
Similarly, the algorithm for drawing premonoidal diagrams requires almost no choices, the order of wires and boxes is fixed, we can only choose their shape and the spacing between them.
On the other hand, drawing graphs requires complex heuristics and graphical interfaces in order to get satisfying results.

\subsection{Hypergraph versus premonoidal diagrams}\label{subsection:hypergraph-vs-premonoidal}

In order to compare the graph-based and list-based approaches, we need to say a few words about how string diagrams for symmetric categories are implemented.
Recall from sections~\ref{subsection:symmetric} and \ref{subsection:hypergraph} that equality of diagrams in symmetric and hypergraph categories reduce to graph and hypergraph isomorphisms respectively.
This can be made explicit by implementing these diagrams as graphs and hypergraphs rather than lists of layers with explicit boxes for swaps and spiders.
Given a monoidal signature $\Sigma$, a \emph{hypergraph diagram} (also called hypergraphs with \emph{ports}) $f$ is given by:
\begin{itemize}
\item its domain and codomain $\dom(f), \cod(f) \in \Sigma_0^\star$,
\item a list of boxes $\boxes(f) \in \Sigma_1^\star$ from which we define:
\begin{itemize}
    \item $\mathtt{input\_ports}(f) = \cod(f) + \coprod_i \dom(f_i)$,
    \item $\mathtt{output\_ports}(f) = \dom(f) + \coprod_i \cod(f_i)$,
    \item and $\mathtt{ports}(f) = \mathtt{input\_ports}(f) + \mathtt{output\_ports}(f)$
\end{itemize}
\item a number of \emph{spiders} $\mathtt{spiders}(f) = n \in \N$ together with their list of types $\mathtt{spider\_types}(f) \in \Sigma_0^{n}$,
\item a set of \emph{wires} $\mathtt{wires}(f) : \mathtt{ports}(f) \to \mathtt{spiders}(f)$.
\end{itemize}
The tensor of two hypergraph diagrams is given by concatenating their domain, codomain, boxes and spiders.
The composition is defined in terms of \emph{pushouts}.
Given $f : x \to y$ and $g : y \to x$ we have a \emph{span} of functions $\mathtt{spiders}(f) \leftarrow y \rightarrow \mathtt{spiders}(g)$ induced by the wires from the codomain of $f$ and the domain of $g$, we define $\mathtt{spiders}(f \fcmp g)$ as the size of the quotient set $(\mathtt{spiders}(f) + \mathtt{spiders}(g)) / R$ under the relation given by the codomain wires of $f$ and the domain wires of $g$.
Concretely, this is computed as the reflexive transitive closure of the binary relation on $\mathtt{spiders}(f) + \mathtt{spiders}(g)$.
The identity diagram $\id(x)$ has $\mathtt{spiders}(f) = \vert x \vert$ and wires given by the two injections $\vert x \vert + \vert x \vert \to \mathtt{spiders}(f)$.
Now we can define a notion of interchanger which takes a hypergraph diagram $f$ and some index $i < \vert \boxes(f) \vert$ and returns the diagram with boxes $i$ and $i + 1$ interchanged, i.e. with the wires relabeled appropriately.
While the interchanger of monoidal diagrams is ill-defined when the boxes are connected, that of hypergraph diagrams is always defined.
The category $\mathbf{Hyp}(\Sigma)$ with equivalence classes of hypergraph diagrams is in fact isomorphic to the free hypergraph category $F^H(\Sigma)$ which we defined in section~\ref{subsection:hypergraph} in terms of special commutative Frobenius algebras~\cite[Theorem 3.3]{BonchiEtAl16}.
The data structure for hypergraph diagrams has swaps and spiders built-in: they are hypergraph diagrams with no boxes.

We say a hypergraph diagram is \emph{bijective} when each spider to be connected to either zero or two ports, so that they define a bijection $\mathtt{ports}(f) \to \mathtt{ports}(f)$.
We conjecture the subcategory of bijective hypergraph diagrams is isomorphic to the free compact-closed category defined in section~\ref{subsection:rigid} in terms of cups and caps.
Spiders connected to zero ports correspond to dimension scalars, i.e. circles composed of a cap then a cup.
A hypergraph diagram is \emph{monogamous} when each spider is connected to exactly one input port and one output port, so that they define a bijection $\mathtt{output\_ports}(f) \to \mathtt{input\_ports}(f)$.
We conjecture the subcategory of monogamous hypergraph diagrams is the free traced symmetric category as defined in section~\ref{subsection:traced}.
We have not been able to find a proof of this statement nor of the compact-closed case in the literature, although they are a straightforward generalisation of \cite[Theorem 3.3]{BonchiEtAl16}.

A hypergraph diagram is \emph{progressive} when it is monogamous and furthermore when the output port of box $i$ is connected to the input port of box $j$ we have $i < j$.
An equivalent condition is that the underlying hypergraph obtained by forgetting the ports is \emph{acyclic} and the order of boxes witnesses that acyclicity.
An interchanger between boxes $i$ and $i + 1$ is progressive when the two boxes are not connected, i.e. progressive interchangers preserve progressivity.
The resulting category of progressive hypergraph diagrams up to progressive interchangers is isomorphic to the free symmetric category defined in section~\ref{subsection:symmetric} in terms of braidings~\cite[Theorem~3.12]{BonchiEtAl16}.
If we remove the interchanger quotient, we get premonoidal versions of free hypergraph, compact closed, traced and symmetric categories.
Traced preonoidal categories were introduced by Benton and Hyland~\cite{BentonHyland03} in order to model recursion in the presence of side-effects.
To the best of our knowledge, no one has ever considered premonoidal compact closed categories: the snake equations still hold but we cannot yank the snakes away if there are obstructions, i.e. the snake removal algorithm of section~\ref{section:monoidal} does not apply.

At every level of this symmetric-traced-compact-hypergraph hierarchy, premonoidal diagrams have the same linear-time algorithm for deciding equality, the data structure for hypergraph diagrams faithfully encode the premonoidal axioms without needing to compute any quotient: normal form becomes identity.
Now suppose we want to compute the interpretation of such a hypergraph diagram in a category given only access to its methods for identity, composition, tensor, swaps and spiders.
This requires to compute the isomorphism $\mathbf{Hyp}(\Sigma) \to F^H(\Sigma)$, i.e. we want to describe the given hypergraph diagram as a chosen representative in its equivalence class of premonoidal diagrams with explicit boxes for swaps and spiders.
The inverse isomorphism $F^H(\Sigma) \to \mathbf{Hyp}(\Sigma)$ is computed by applying a premonoidal functor (i.e. a \py{for} loop on a list of layers) sending swap and spider boxes to hypergraph diagrams with no boxes.

This isomorphism is implemented in the \py{hypergraph} module of DisCoPy which is outlined below.
The composition method calls a \py{pushout} subroutine which takes as input the numbers of spiders on the left and right, the wires from some common boundary ports to the left and right spiders, and returns the two injections into their pushout.
The three properties for bijective, monogamous and progressive diagrams implement the subcategory of compact closed, traced and symmetric diagrams respectively.
The three corresponding methods take a diagram and add explicit boxes for spiders, cups and caps so that \py{f.make_bijective().is_bijective} for all \py{f: Diagram} and similarly for monogamous and progressive.
The \py{downgrade} method calls \py{make_progressive} to construct a \py{compact.Diagram} with explicit boxes for swaps and spiders.

The method \py{cast} applies a \py{compact.Functor} from premonoidal to hypergraph diagrams so that we have \py{cast(f.downgrade())} \py{== f} on the nose for any \py{f: Diagram} and \py{cast(f).downgrade()} is equal to any \py{f:} \py{compact.Diagram} up to the special commutative Frobenius axioms.
The \py{draw} method uses a randomised force-based layout algorithm for graphs to compute an embedding from the hypergraph diagram to the plane where wires do not cross too much.
This is still an experimental feature and the results of \py{f.draw()} usually look much worse than the drawing of \py{f.downgrade().draw()} using the deterministic algorithm of section~\ref{section:drawing}.

\begin{python}\label{listing:discopy-hypergraph}
{\normalfont Outline of the \py{discopy.hypergraph} module.}

\begin{minted}{python}
def pushout(left: int, right: int,
            left_wires: tuple[int, ...], right_wires: tuple[int, ...]
            ) -> tuple[dict[int, int], dict[int, int]]: ...

@dataclass
class Diagram(Composable, Tensorable):
    dom: Ty
    cod: Ty
    boxes: tuple[Diagram, ...]
    wires: tuple[int, ...]
    spider_types: tuple[Ty, ...]

    @staticmethod
    def id(x: Ty) -> Diagram: ...
    def then(self, *others: Diagram) -> Diagram: ...
    def tensor(self, *others: Diagram) -> Diagram: ...
    def interchange(self, i: int) -> Diagram: ...

    swap: Callable[[Ty, Ty], Diagram] = staticmethod(...)
    spiders: Callable[[int, int, Ty], Diagram] = staticmethod(...)

    is_bijective: bool = property(...)
    is_monogamous: bool = property(...)
    is_progressive: bool = property(...)

    def make_bijective(self) -> Diagram: ...
    def make_monogamous(self) -> Diagram: ...
    def make_progressive(self) -> Diagram: ...

    def downgrade(self) -> compact.Diagram: ...

    cast = staticmethod(compact.Functor(
        ob=lambda x: Ty(x.inside[0]),
        ar=lambda box: Box(box.name, box.dom, box.cod),
        cod=Category(Ty, Diagram)))

    def draw(self, **params): ...

class Box(Diagram):
    def __init__(self, name: str, dom: Ty, cod: Ty):
        boxes, spider_types, wires = (self, ), tuple(map(Ty, dom @ cod)), ...
        self.name = name; super().__init__(dom, cod, boxes, wires, spider_types)

    __eq__ = lambda self, other: cat.Box.__eq__(self, other)\
            if isinstance(other, Box) else super().__eq__(other)
\end{minted}
\end{python}

In some cases however, we can compute the interpretation of hypergraph diagrams without having to downgrade them back to premonoidal diagrams.
This is the case for \emph{tensor networks}, i.e. hypergraph diagrams interpreted in the category $\mathbf{Tensor}_\S$.
Indeed, rather than applying premonoidal functors into our naive \py{Matrix} class, DisCoPy can translate hypergraph diagrams as input to \emph{tensor contraction} algorithms such as the Einstein summation of NumPy~\cite{VanDerWaltEtAl11} combined with the just-in-time compilation of JAX~\cite{BradburyEtAl20}, or the specialised TensorNetwork library~\cite{RobertsEtAl19}.
Another example is that of quantum circuits: they are inherently symmetric diagrams.
Indeed, the data structure for circuits in a quantum compiler such as t$\vert$ket$\rangle$~\cite{SivarajahEtAl20} is secretly some premonoidal symmetric category: objects are lists of qubit (and bit) identifiers, arrows (i.e. circuits) are lists of operations.
When circuits are encoded as premonoidal diagrams as we have done in example~\ref{example:circuit-diagrams}, qubits are forced into a line because their wires are ordered from left to right, thus applying gates to non-adjacent qubits is encoded in terms of swap boxes.
When we apply a premonoidal functor to the category of t$\vert$ket$\rangle$ circuits, those swap boxes are not interpreted as the physical operation of applying three CNOT gates, but as the logical operation of relabeling our qubit identifiers.
It is then the job of the compiler to map these symmetric diagrams (where every qubit can talk to every other) onto the architecture of the machine and potentially introduce physical swaps when a logical gate applies to physical qubits that are not adjacent.

The main advantage of representing diagrams as hypergraphs rather than lists is that we can use graph rewriting algorithms to implement quotient categories.
Indeed, the \emph{double push-out} (DPO) rewriting of Ehrig et al.~\cite{EhrigEtAl73} can be extended from graphs to hypergraph diagrams so that we can match the left-hand side of an axiom in a hypergraph diagram and then compute the substitution with the right-hand side~\cite{BonchiEtAl20}.
Abstractly, DPO rewriting takes two hypegraph diagrams \py{self} and \py{pattern} and iterates through all possible \py{match} (i.e. pairs of diagrams for top and bottom and pairs of types for left and right as defined in section~\ref{subsection:quotient-monoidal}) such that \py{match.subs(pattern)} is equal to \py{self} up to interchanger.
This can be extended to the case of symmetric diagrams by implementing a Boolean property \py{match.is_convex} that makes sure that pattern matching does not introduce spiders~\cite{BonchiEtAl16}.
DPO rewriting has been the basis of tools such as Quantomatic and its successor PyZX~\cite{KissingerVanDeWetering19}~\cite{KissingerZamdzhiev15} for automated diagrammatic reasoning, it is also at the core of circuit optimisation in the t$\vert$ket$\rangle$ compiler.
DisCoPy implements back and forth translations from diagrams to both PyZX and t$\vert$ket$\rangle$, thus we can use their rewriting engines to implement quotient categories, i.e. to define normal forms.

However, the hypergraph approach breaks down in the case of non-symmetric monoidal categories.
Indeed, the monoidal functor from the free monoidal category to the free symmetric category is not faithful, for example it sends two nested circles and two circles side by side to the same hypergraph diagram.
We conjecture that the category of planar\footnote
{A hypergraph diagram is planar when we can embed it in the plane, or equivalently it is the image of a swap-free premonoidal diagram.} progressive hypergraph diagrams is the free \emph{spacial category}, i.e. one with $s \otimes x = x \otimes s$ for all scalars $s : 1 \to 1$ and objects $x$.
Translated in terms of the topological definition of string diagrams, this would correspond to taking labeled progressive plane graphs up to deformation of three-dimensional space rather than up to deformation of the plane~\cite[Conjecture~3.4]{Selinger10}.
When presented as quotients of free monoidal categories, spacial categories require an infinite family of axioms indexed by all possible scalar diagrams: in the absence of symmetry we cannot decompose the equation $(f \fcmp g) \otimes x = x \otimes (f \fcmp g)$ for a state $f : 1 \to y$ followed by an effect $g : y \to 1$ in terms of two smaller equations about $f$ and $g$ passing through the wire $x$.
That every braided monoidal category is spacial follows from the naturality of the braiding, we do not know of any natural example of non-free non-braided spacial category.
Moreover, it is an open question whether we can extend DPO rewriting to the case of spacial monoidal categories, i.e. whether there is an efficiently checkable condition that ensures that pattern matching does not introduce swaps.

Thus, DisCoPy's planar premonoidal approach to string diagrams allows to define diagrams in non-symmetric categories that cannot be defined as hypergraphs.
Although the concrete examples of categories we have discussed so far (functions, matrices, circuits) are all symmetric, planarity is essential if we are to model grammatical structure in terms of string diagrams as we will in section~\ref{section:NLP}.
Indeed, the left to right order of wires in a planar diagram encode the chronological order of words in a sentence, allowing arbitrary swaps would make grammaticality permutation-invariant: if a sentence is grammatical, then so would be any random shuffling of it.
More impotantly, planarity in grammar has been given a cognitive explanation.
In order to minimise the computational resources needed by the brain, human languages tend to minimise the distance between words that are syntactically connected~\cite{FutrellEtAl15} and the minimisation of swaps comes as a side-effect~\cite{Cancho06}.
It can also be given a complexity-theoretic explanation: planar grammatical structures such as Chomsky's syntax trees, Lambek's pregroup diagrams or Gaifman's dependency trees (which we introduce in section~\ref{section:NLP}) are all \emph{context-free}, they have the same expressive power as \emph{push-down automaton}.
As we will mention in section~\ref{subsection:chomsky}, \emph{cross-serial dependencies} are counter examples where grammatical wires are allowed to cross, albeit in a restricted way that makes them \emph{mildly context-sensitive}~\cite{Stabler04}.
Yeung and Kartsaklis~\cite{YeungKartsaklis21} showed that up to word reordering, the diagrams for these cross-serial dependencies can always be rewritten in a planar way using naturality.
They then used DisCoPy to encode every sentence of Alice in Wonderland as a diagram, ready to be translated into a circuit and sent to a quantum computer.

\subsection{Towards higher-dimensional diagrams}\label{subsection:towards-higher}

The premonoidal approach is also well-suited to be generalised from two- to arbitrary-dimensional diagrams.
The first step would be to add some colours to our diagrams, generalising them from monoidal categories to (strict) \emph{2-categories}, or equivalently from premonoidal categories to \emph{sesquicategories}.
The data for a sesquicategory $C$ is given by:
\begin{itemize}
    \item the data for a category $(C_0, C_1, \dom_0, \cod_0, \fcmp_0, \id_0)$ where the objects $C_0$ and arrows $C_1$ are called the class of $\emph{0- and 1-cells}$,
    \item a class $C_2$ of \emph{2-cells},
    \item domain and codomain $\dom_1, \cod_1 : C_2 \to C_1$,
    \item an identity $\id_1 : C_1 \to C_2$ and a partial composition $(\fcmp_1) : C_2 \times C_2 \to C_1$.
\end{itemize}
such that the following holds
\begin{itemize}
    \item $C(x, y) = \{ f \in C_2 \ \vert f : x \to y \}$ is a category for every pair of 1-cells $x, y \in C_1$,
    \item composition of 1-cells is a functor $(\fcmp)_0 : C(x, y) \Box C(y, z) \to C(x, z)$ for $\Box$ the funny tensor product on $\mathbf{Cat}$.
\end{itemize}
The axioms for 2-categories are the same but now composition is a bifunctor $\fcmp_1 : C(x, y) \times C(y, z) \to C(x, z)$ on a cartesian product.
Every bifunctor is also functorial in its two arguments separately thus every 2-category is also a sesquicategory.
The canonical example of a (2-category) sesquicategory is $\mathbf{Cat}$ with categories as 0-cells, functors as 1-cells and (natural) transformations as 2-cells.
Every monoidal (premonoidal) category is a 2-category (sesquicategory) with one 0-cell.
A 2-functor $F : C \to D$ between two 2-categories is given by three functions $\{ F_i : C_i \to D_i \}_{0 \leq i \leq 2}$ such that $(F_0, F_1) : (C_0, C_1) \to (D_0, D_1)$ and $(F_1, F_2) : C(x, y) \to D(F_1(x), F_1(y))$ are functors for all $x, y \in C_1$.

Free sesquicategories are defined in the same way as free premonoidal categories (i.e. as lists of layers) except that now every type comes itself with a domain and codomain, represented as the background colours on the left and right of the wire.
Thus we need a \emph{2-signature} $\Sigma = (\Sigma_0, \Sigma_1, \Sigma_2, \dom, \cod)$ where
\begin{itemize}
    \item $\Sigma_0$ is a set of colours,
    \item $\Sigma_1$ is a set of objects with colours as domain and codomain,
    \item $\Sigma_2$ is a set of boxes with domain and codomain in the free category $F(\Sigma_0, \Sigma_1)$, i.e. lists of generating objects with composable colours.
\end{itemize}
We also need to require the \emph{globular conditions} $\dom(\dom(f)) = \dom(\cod(f))$ and $\cod(\dom(f)) = \cod(\cod(f))$ that ensure that the top-left (top-right) colour is the same as the bottom-left (bottom-right, respectively).
Intuitively, the only changes in background colour happen at the wires, labeled by a generating object with the appropriate domain and codomain.
The free 2-category $F^{2C}(\Sigma)$ can then be described by its set of \emph{0-cells} $\Sigma_0$ (the colours), its category of \emph{1-cells} $F(\Sigma_0, \Sigma_1)$ (the types) and a category for every pair of 1-cells: the category of coloured diagrams up to interchanger.
The implementation is straightforward: we just need to make \py{Ty} a subclass of both \py{monoidal.Ty} (so that it can be used as domain and codomain for diagrams) and \py{Arrow} (so that it can have a domain and codomain itself).

\begin{python}\label{listing:free-sesquicategory}
{\normalfont Implementation of the free sesquicategory with \py{Colour} as 0-cells, \py{Ty} as 1-cells and \py{Diagram} as 2-cells.}

\begin{minted}{python}
class Colour(cat.Ob):
    pass

class TyArrow(cat.Arrow, monoidal.Ty):
    @inductive
    def tensor(self, other):
        if isinstance(other, TyArrow):
            return cat.Arrow.then(self, other)
        return NotImplemented  # Allows whiskering on the left.

    __matmul__ = tensor

class Ty(cat.Box, TyArrow):
    cast = TyArrow.cast

class Layer(monoidal.Layer):
    def __init__(self, left: Ty, box: monoidal.Box, right: Ty):
        assert left.cod == box.dom.dom and box.dom.cod == right.dom
        super().__init__(left, box, right)

class Diagram(monoidal.Diagram):
    pass

class Box(monoidal.Box, Diagram):
    def __init__(self, name: str, dom: Ty, cod: Ty):
        assert (dom.dom, dom.cod) == (cod.dom, cod.cod)
        monoidal.Box.__init__(self, name, dom, cod)
        Diagram.__init__(self, (Layer.cast(self), ), dom, cod)

    cast = Diagram.cast

@dataclass
class TwoCategory:
    colours: type = Colour
    ob: type = Ty
    ar: type = Diagram

@dataclass
class TwoFunctor(monoidal.Functor):
    colours: DictOrCallable[Colour, Colour]
    ob: DictOrCallable[Ty, Ty]
    ar: DictOrCallable[Box, Diagram]

    dom: TwoCategory = TwoCategory()
    cod: TwoCategory = TwoCategory()

    def __call__(self, other):
        if isinstance(other, Colour):
            return self.colours[other]
        if isinstance(other, Ty):
            return self.ob[other]
        if isinstance(other, TyArrow):
            base_case = self.cod.ob.id(self(other.dom))
            return base_case.then(*[self(box) for box in other.inside])
        return super().__call__(other)
\end{minted}
\end{python}

\begin{python}\label{listing:Transformation}
{\normalfont Implementation of $\mathbf{Cat}$ as a sesquicategory with transformations as 2-cells.}

\begin{minted}{python}
class Transformation(Composable, Tensorable):
    def __init__(self, inside: Callable, dom: Functor, cod: Functor):
        assert (dom.dom, dom.cod) == (cod.dom, cod.cod)
        self.inside, self.dom, self.cod = inside, dom, cod

    @staticmethod
    def id(F: Functor):
        return Transformation(F.cod.ar.id, dom=F, cod=F)

    @inductive
    def then(self, other: Transformation) -> Transformation:
        return Transformation(lambda x: self(x) >> other(x), self.dom, other.cod)

    @inductive
    def tensor(self, other: Transformation) -> Transformation:
        return self @ other.dom >> self.cod @ other

    def __matmul__(self, other: Transformation | Functor) -> Transformation:
        if isinstance(other, Functor):
            return Transformation(
                lambda x: other(self(x)), self.dom >> other, self.cod >> other)
        return self.tensor(other)

    def __rmatmul__(self, other: Transformation | Functor) -> Transformation:
        if isinstance(other, Functor):
            return Transformation(
                lambda x: self(other(x)), other >> self.dom, other >> self.cod)
        raise TypeError

    def __call__(self, other: Ob) -> Arrow:
        inside, dom, cod = self.inside(other), self.dom(other), self.cod(other)
        return self.cod.cod.ar(inside, dom, cod)

Cat = TwoCategory(Category, Functor, Transformation)
\end{minted}
\end{python}

\begin{example}
We can interpret colours as categories, types as functors and diagrams as transformations.

\begin{minted}{python}
a = Colour('a')
x = Ty('x', dom=a, cod=a)
f, g = Box('f', Ty.id(a), x), Box('g', x @ x, x)

Pyth = Category(tuple[type, ...], Function)
List = Functor(
    ob=lambda xs: list[xs],
    ar=lambda f: lambda xs: list(map(f, xs)),
    dom=Pyth, cod=Pyth)
Unit = Transformation(
    lambda _: lambda x: [x], dom=Functor.id(Pyth), cod=List)
Mult = Transformation(
    lambda _: lambda xs: sum(xs, []), dom=List >> List, cod=List)

F = TwoFunctor(
    colours={a: Pyth}, ob={x: List}, ar={f: Unit, g: Mult}, cod=Cat)

assert F(f @ x >> g)(int)([1, 2, 3])\
    == F(x @ f >> g)(int)([1, 2, 3])\
    == F(Diagram.id(x))(int)([1, 2, 3]) == [1, 2, 3]

assert F(g @ x >> g)(int)([[[42]]]) == [42] == F(x @ g >> g)(int)([[[42]]])
\end{minted}
\end{example}

We have already discussed another way to construct a 2-categories: taking types as 0-cells, diagrams as 1-cells and rewrites as 2-cells, in fact this gives a \emph{premonoidal sesquicategory}.
A monoidal 2-signature $\Sigma$ is a 2-signature where the objects in $\Sigma_1$ have lists of colours $\Sigma_0^\star$ as domain and codomain and the boxes in $\Sigma_2$ have domain and codomain in the free premonoidal category $F^P(\Sigma_0, \Sigma_1)$.
Thus, a box $r : f \to g$ in a monoidal 2-signature may be seen as a rewrite rule with parallel diagrams $f : x \to y$ and $g : x \to y$ as domain and codomain.
It generates a free premonoidal sesquicategory with types $\Sigma_0^\star$ as 0-cells, diagrams $F^P(\Sigma_0, \Sigma_1)$ as 1-cells and rewrites as 2-cells.
We can construct it explicitly by generalising layers to \emph{slices} with not only types on the left and right but also diagrams on the top and bottom, i.e. a rewrite rule together with a match.
\py{Rule} is a subclass of \py{Box} with diagrams as domain and codomain, \py{Slice} is a box made of a rule inside a match with methods for left and right whiskering as well as pre- and post-composition.
\py{Rewrite} is a subclass of \py{Diagram} with \py{Slice} as layers, it inherits its vertical composition (i.e. two rewrites applied in sequence) from the diagram class as well as its tensor product (i.e. two rewrites applied in parallel on the tensor of two diagrams).
The horizontal composition (i.e. two rewrites applied in parallel on the composition of two diagrams) can be implemented by temporarily replacing left and right whiskering by pre- and post-composition before calling \py{Diagram.tensor}.

\begin{python}\label{listing:free-monoidal-2-category}
{\normalfont Outline of the implementation of free premonoidal sesquicategories.}

\begin{minted}{python}
class Slice(monoidal.Box):
    def __init__(self, rule: Rule, match: Match):
        dom, cod = match.subs(rule.dom), match.subs(rule.cod)
        super().__init__("Slice({}, {})".format(rule, match), dom, cod)

    @classmethod
    def cast(cls, old: Rule) -> Slice:
        x, y = old.dom.dom, old.cod.cod
        top, bottom, left, right = old.id(x), old.id(y), x[:0], y[len(y):]
        return cls(old, Match(top, bottom, left, right))

class Rewrite(Diagram):
    inside: tuple[Slice, ...]
    dom: Diagram
    cod: Diagram

class Rule(monoidal.Box, Rewrite):
    def __init__(self, name: str, dom: Diagram, cod: Diagram):
        monoidal.Box.__init__(self, name, dom, cod)
        Rewrite.__init__(self, (Slice.cast(self), ), dom, cod)
\end{minted}
\end{python}

Note that when the monoidal 2-signature $\Sigma$ is in fact a simple 2-signature, i.e. every box $f \in \Sigma_1$ has domain and codomain of length one, the definition of a rewrite coincides with the definition of coloured diagram.
Indeed we can relabel everything one level down: the types are colours, the boxes are types and the rules are boxes: a coloured diagram can be seen as a rewrite of one dimensional diagrams, i.e. lists of types with composable colours.
Symmetrically, rewriting an 2-dimensional diagram can itself be seen as a 3-dimensional diagram.
If we compose the two constructions (colours and rewrites) together, we get the free 3-sesquicategory with rewrites of coloured diagrams as 3-cells.
We can keep on going with \emph{modifications} of rewrites, i.e. 4-dimensional diagrams, by generalising layers one step further with not only a pair of types (left and right) and a pair of diagrams (top and bottom) but also a pair of rewrites (before and after).
What could be the use of a such four-dimensional diagram?
For example, a free 4-category with a single 0-, 1- and 2-cell is the same as a free symmetric category (once we relabel everything three levels down).
Indeed, the swaps are given by the interchange law and the 4d space in which the diagrams live allows wires to cross and every knot to be untied: every diagram interpreted in $\mathbf{Pyth}$ or $\mathbf{Mat}_\S$ is secretly four-dimensional.
One dimension lower, a free 3-category with a single 0- and 1-cell is the same as a free braided category, this is only the tip of the \emph{periodic table} of k-tuply monoidal n-categories~\cite[Section~2.5]{BaezStay10}.

The proof assistant Globular~\cite{BarEtAl18} allows to construct 4-dimensional diagrams using a graphical interface for drawing slices and projections in two dimensions.
In fact, the drawing algorithm presented in section~\ref{section:drawing} was reverse engineered from that of Globular.
Its successor homotopy.io~\cite{ReutterVicary19} went from four to arbitrary dimensions based on a data structure for diagrams in free n-sesquicategories~\cite{BarVicary17}.
Interfacing DisCoPy with homotopy.io is in the backlog of features yet to be implemented, so that the user can define diagrams by drag-and-dropping boxes then interpret them in arbitrary n-categories.
One of the new feature of homotopy.io compared to its predecessor is the possibility of drawing non-generic diagrams, i.e. with more than one box on the same layer.
This amounts to taking the free category over $L^+(\Sigma) = (\Sigma_0 + \Sigma_1)^\star \simeq \Sigma_0^\star \Box \Sigma_1^\star$ rather than $L(\Sigma) = \Sigma_0^\star \times \Sigma_1 \times \Sigma_0^\star$.
This is also in DisCoPy's backlog, implementing the syntax is straightforward but then it requires to extend the algorithms for functors, drawing, normal forms, etc.

Layers with arbitrarily many boxes also allow to define the \emph{depth} of an arrow in any quotient of a free premonoidal category as the minimum number of layers in its equivalence class of diagrams.
As we mentioned in section~\ref{subsection:quotient-monoidal}, premonoidal diagrams also have a well-defined notion of \emph{width} (the maximum number of parallel wires) which we can extend in the same way to define the width of any quotient.
This makes diagrams a foundational data structure for computational complexity theory: a signature can be seen as both a machine and a language, a diagram as both code and data.
In the other direction, this also allows to borrow results from complexity theory to characterise the computational resources required in solving problems about diagrams.
This will be needed in the next chapter when we will look at NLP problems through the lens of diagrams.


\section{Summary \& future work} \label{section:summary-and-future}

This chapter gave a comprehensive overview of DisCoPy and the mathematics behind its design principles: we take the definitions of category theory (as strictly and freely as possible) and translate them into a Pythonic syntax.
Figure~\ref{fig:summary} summarises the different modules and their inheritance hierarchy, implementing a subset of the hierarchy of graphical languages surveyed by Selinger~\cite{Selinger10}.
We hope it may be useful both as an introduction to monoidal categories for the Python programmer, and an introduction to Python programming for the applied category theorist.

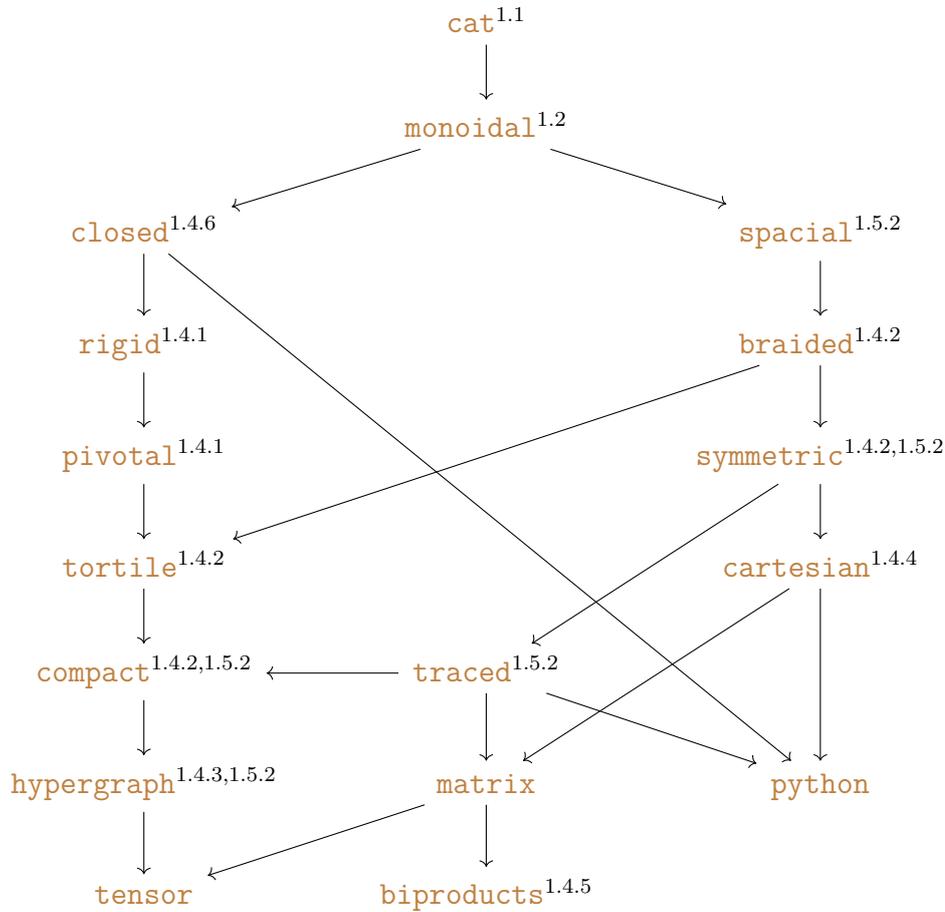
\begin{figure}
\centering
\begin{tikzcd}
	& {\py{cat}^{\ref{section:cat}}} \\
	& {\py{monoidal}^{\ref{section:monoidal}}} \\
	{\py{closed}^{\ref{subsection:closed}}} && {\py{spacial}^{\ref{subsection:hypergraph-vs-premonoidal}}} \\
	{\py{rigid}^{\ref{subsection:rigid}}} && {\py{braided}^{\ref{subsection:symmetric}}} \\
	{\py{pivotal}^{\ref{subsection:rigid}}} && {\py{symmetric}^{\ref{subsection:symmetric}, \ref{subsection:hypergraph-vs-premonoidal}}} \\
	{\py{tortile}^{\ref{subsection:symmetric}}} && {\py{cartesian}^{\ref{subsection:cartesian}}} \\
	{\py{compact}^{\ref{subsection:symmetric}, \ref{subsection:hypergraph-vs-premonoidal}}} & {\py{traced}^{\ref{subsection:hypergraph-vs-premonoidal}}} \\
	{\py{hypergraph}^{\ref{subsection:hypergraph}, \ref{subsection:hypergraph-vs-premonoidal}}} & {\py{matrix}} & {\py{python}} \\
	{\py{tensor}} & {\py{biproducts}^{\ref{subsection:biproducts}}}
	\arrow[from=1-2, to=2-2]
	\arrow[from=2-2, to=3-1]
	\arrow[from=2-2, to=3-3]
	\arrow[from=3-1, to=4-1]
	\arrow[from=3-3, to=4-3]
	\arrow[from=4-1, to=5-1]
	\arrow[from=8-2, to=9-1]
	\arrow[from=8-2, to=9-2]
	\arrow[from=5-1, to=6-1]
	\arrow[from=6-1, to=7-1]
	\arrow[from=7-1, to=8-1]
	\arrow[from=8-1, to=9-1]
	\arrow[from=4-3, to=5-3]
	\arrow[from=5-3, to=6-3]
	\arrow[from=5-3, to=7-2]
	\arrow[from=7-2, to=7-1]
	\arrow[from=6-3, to=8-3]
	\arrow[from=6-3, to=8-2]
	\arrow[from=7-2, to=8-2]
    \arrow[from=7-2, to=8-3]
	\arrow[from=3-1, to=8-3]
    \arrow[from=4-3, to=6-1]
\end{tikzcd}
\caption{DisCoPy's modules and the sections where they are discussed, arrows indicate software dependency.}
\label{fig:summary}
\end{figure}

Note that the code presented in this thesis represents a significant refactoring of the original implementation of DisCoPy \py{v0.4.2} as available online at the time this thesis is submitted\footnote
{\url{https://github.com/oxford-quantum-group/discopy/releases/tag/0.4.2}}.
It is available as a standalone version\footnote
{\url{https://github.com/toumix/thesis}} which will later be merged with the original repository and released as DisCoPy \py{v1.0}.
We list some of the significant changes between the two versions.
\begin{itemize}
\item We add type annotations throughout the codebase, using the postponed evaluation of annotations introduced in Python \py{3.7}~\cite{Langa17}.
\item We simplify the inheritance mechanism using Python's \py{classmethod} decorator.
This improves the code reuse for composition of diagrams, application of functors, etc.
\item We reorganise the codebase so that it follows more closely the hierarchy of categorical structures.
For example, we move the code for \py{Braid} outside of the \py{monoidal} module into its own \py{braided} module, we also introduce e.g. the \py{tortile} module which imports from both \py{rigid} and \py{braided}.
\item We make the syntax more uniform for arrows in different categories, which are all initialised with the same attributes \py{inside}, \py{dom} and \py{cod}.
\item We implement whiskering, i.e. tensoring with the identity of a given type on the left or right.
This avoids to clutter diagram definitions with \py{Id}.
\item Arrows in concrete categories like \py{Matrix}, \py{Tensor} and \py{Function} are no longer subclasses of \py{Box}. Instead, we implement the syntactic sugar for composition, whiskering, etc. with abstract classes \py{Composable} and \py{Tensorable}.
\item We make the \py{Matrix} and \py{Tensor} classes parameterised by the datatype of their entries.
This makes use of the magic method \py{__class_getitem__} which appeared in Python \py{3.10}~\cite{Levkivskyi17}.
\end{itemize}

We list but a few of the many potential directions for further developments.

\begin{itemize}
\item DisCoPy was implemented mainly with correctness in mind, thus there is much room for improving performance.
For now, this has not been quite necessary since the diagrams we manipulate are exponentially smaller than the computation they represent.
However if we want to implement any serious rewriting efficiently, we will need to port the core algorithms to a lower-level language such as Rust~\cite{KlabnikNichols19} and wrap them with Python bindings.
This strategy has improved the time performance of PyZX by over 4000 on a small benchmark consisting of the fusion of 1 million spiders\footnote{\url{https://github.com/quantomatic/quizx}}.

\item As we mentioned in section~\ref{subsection:tacit-to-explicit}, DisCoPy uses a \emph{point-free}, \emph{tacit programming} style which can get very verbose as soon as diagrams have more than a few boxes.
One of the features in our backlog is implementing an \emph{explicit} syntax where diagrams are defined as decorated Python functions taking the wires in their domain as argument, applying boxes to them and returning their codomain.
We already have a working version of this for planar diagrams, it would be straightforward to extend it to any cartesian diagram where we can swap, copy and discard arguments.
What would be less straightforward is to extend it to the syntax of structures beyond cartesian: cocartesian (control flow), closed (higher-order functions) and traced (iteration and recursion).
One starting point for this, rather than reinventing the wheel, would be to use JAX~\cite{BradburyEtAl20} expressions as an intermediate language between pure Python and diagrams.

\item There are many more ways we can interpret diagrams as code, i.e. many more functors into concrete categories we can implement.
One example is probabilistic functions which can be modeled as arrows of \emph{Markov categories}~\cite{FritzEtAl20a} where the objects have comonoids but only the counit is natural.
DisCoPy has already been interfaced with the probabilistic programming language Pyro~\cite{BinghamEtAl19} in order to learn both the structure and the parameters of a machine learning model end-to-end~\cite{Sennesh20}.

\item Some of these concrete categories will not be strictly associative: $(x \otimes y) \otimes z$ and $x \otimes (y \otimes z)$ can represent two different ways of storing the same data, and using one versus the other may have an impact on performance.
Diagrams for non-strict monoidal categories have been used to give an elementary proof of MacLane's coherence theorem for monoidal categories~\cite{WilsonEtAl22}.
We have also drawn them throughout this thesis when discussing coherence for rigid, braided and hypergraph categories.
For now we had to cheat and manually define a new type \py{xy} with boxes from \py{x @ y} to \py{xy} an back, better support for such monoidal coherence is also in the backlog.

\item Categories with a tensor product that is not necessarily associative or unital, sometimes called \emph{magmoidal categories}, also play a role in linguistics.
Indeed, the Lambek calculus in its 1961 version~\cite{Lambek61} is non-associative and non-unital, which gives a finer control over the grammaticality of trees rather than lists.
With \emph{skew monoidal} categories~\cite{UustaluEtAl18}, one re-introduces the natural transformation for associativity but in only one direction.
In another generalisation, Grishin~\cite{Grishin83} introduced a coproduct and its left and right adjoints as dual to the tensor product.
This new binary operation comes with interaction rules for distributing over the tensor, see Moortgat~\cite{Moortgat09} for a modern presentation.
Wijnholds~\cite{Wijnholds15,Wijnholds17} gave a distributional compositional semantics to this Lambek-Grishin calculus in terms of \emph{weakly distributive categories}~\cite{CockettSeely97}.
We leave the implementation of categories with multiple non-associative monoidal structures and their potential application to QNLP as a direction for future work.

\item There are many more constructions from category theory that could be implemented in DisCoPy.
One example is the \emph{$\mathbf{Int}$ construction} which defines the free compact-closed category generated by a traced symmetric category $C$~\cite[Section~4]{JoyalEtAl96}.
Generalising the way the integers $\Z$ are constructed as a quotient of pairs of natural numbers, the objects of $\mathbf{Int}(C)$ are given by pairs of objects in $C$, the arrows by pairs of arrows going in opposite direction and their composition by the trace.
The $\mathbf{Int}$ construction allows to reason about \emph{bidirectional processes} such as \emph{optics} in functional programming~\cite{LavoreRoman19}.
It is also related to the notion of \emph{combs} or \emph{open diagrams}~\cite{Roman20a} which have been used to reason about processes with feedback~\cite{Roman20} as well as causal quantum processes~\cite{KissingerUijlen19}.
Other examples include \emph{open learners}~\cite{FongJohnson19} and \emph{open games}~\cite{Hedges17,Hedges19a} which formalise machine learning and game theory in terms of monoidal categories with some notion of bidirectionality.
\end{itemize}


\chapter{Quantum natural language processing} \label{chapter-2:qnlp}

This chapter introduces quantum natural language processing (QNLP) models as monoidal functors from grammar to quantum circuits.
Building on the previous chapter, we show how to implement QNLP models in DisCoPy and how to train them to solve NLP tasks such as classification and question answering.

\section{Formal grammars and quantum complexity}\label{section:NLP}

The previous chapter has put much emphasis on string diagrams and its role at the intersection of mathematics and computer science.
From the programming perspective, diagrams are a two-dimensional generalisation of lists which may describe the run of a Turing machine, the syntax of a first-order logic formula or the architecture of a neural network.
In fact, we will see that string diagrams also play a key role in linguistics, where they allow to encode the grammatical structure of sentences.
First, section~\ref{subsection:chomsky} reviews formal grammars, the notion of ambiguity and the computational complexity of parsing.
Then we discuss categorial grammars, from the Lambek calculus and Montague semantics to pregroup grammars and DisCoCat models.
Finally, we summarise previous work on the Frobenius anatomy of anaphora and investigate the quantum complexity of DisCoCat models.


\subsection{Formal grammars, parsing and ambiguity}\label{subsection:chomsky}

The word ``grammar'' comes from the ancient Greek ``$\rm{\gamma \rho\acute{\alpha}\mu\mu\alpha}$'' (line of writing), it is cognate to the words ``glamour'' and ``grimoire'' \cite{RuppliThorel05, Davies10, Lambek14}.
The practice of grammar itself goes back to India somewhere between the 6th and 4th century BCE~\cite{BhateKak93}, where the Sanskrit philologist P\={a}\d{n}ini introduced what was later recognised as \emph{context-sensitive grammars}.
More than two thousand years later, Chomsky~\cite{Chomsky56,Chomsky57} gave grammars their modern definition.
A \emph{formal grammar}, also called \emph{unrestricted} or \emph{type-0} grammar, is a tuple $G = (V, X, R, s)$ where:
\begin{itemize}
    \item $V$ and $X$ are finite sets called \emph{terminal} and \emph{non-terminal} symbols respectively, we will also call them the \emph{vocabulary} and the \emph{basic types},
    \item $R$ is a finite set of \emph{production rules} $x \to y$ where $x, y \in (V + X)^\star$,
    \item $s \in X$ is called the \emph{start symbol} or the \emph{sentence type}.
\end{itemize}
A string of words $w = w_1 \dots w_n \in V^\star$ is a \emph{grammatical sentence} whenever\footnote
{Formal grammars are usually defined in the other direction, i.e. $w \in L(G)$ iff $s \leq_R w$.
We choose the opposite convention so that we won't have to switch in the next section.} $w \leq_R s$
for $(\leq_R)$ the reflexive transitive closure of the rewriting relation as defined in section~\ref{subsection:quotient-categories}.
Thus, the grammar $G$ generates a \emph{language} $L(G) \sub V^\star$, the set of all grammatical sentences.
Although formal grammars are called unrestricted, the right-hand side of the rules in $R$ is usually restricted to be non-empty.
This makes no difference as to the classes of languages that can be generated, i.e. for every grammar with empty right-hand sides there is a grammar without that generates the same language.

Equivalently, a formal grammar is a finite monoidal signature $G$ with an injection from the words in the vocabulary and the sentence type into the generating objects $V + \{ s \} \injects G_0$.
Indeed, we can define the non-terminal symbols as $X = G_0 - V$ then a rewrite rule is nothing but a box with lists of symbols as domain and codomain.
The language of $G$ may then be defined as $L(G) = \{ w \in V^\star \ \vert \ \exists f : w \to s \ \in \G \}$ for $\G$ the free monoidal category generated by $G$, a diagram $f : w \to s$ to the sentence type $s$ is proof that the string of words $w$ is grammatical.
We call the diagram $f : w \to s$ a \emph{grammatical structure} for the sentence $w$, we say a sentence is \emph{ambiguous} whenever it has more than one grammatical structure.
The \emph{parsing problem} is to decide, given a grammar $G$ and a string $w \in V^\star$, whether $w \in L(G)$.
It is easily shown to be equivalent to the word problem for monoids and the halting problem for Turing machines, thus it is undecidable.
Moreover, there exists a \emph{universal grammar} $G$ such that the parsing problem with $G$ fixed and only the string $w \in V^\star$ as input is undecidable.
That is, for any other grammar $G'$ and string $w \in V(G')^\star$ we can compute some other string $w' \in V(G)^\star$ such that $w \in L(G')$ if and only if $w' \in L(G)$.

If we are to build a \emph{parser}, i.e. a machine that computes the grammatical structure of a given string, type-0 grammars are too general: their parsing problem is undecidable.
Going one level up in Chomsky's hierarchy, a \emph{context-sensitive grammar} (CSG, also called a \emph{type-1} grammar) is a formal grammar $G$ where the rules have the form $a b c \to a x c$ for a non-terminal symbol $x \in X$ and lists of symbols $a, b, c \in G_0^\star$ where $\len(b) \geq 1$\footnote
{If we care about whether a language contains the empty string $1$ or not, we also have to allow for the rule $s \to 1$.}.
The parsing problem for CSG was the first to be shown complete for the class $\mathtt{NPSPACE}$ of problems solvable in non-deterministic polynomial space~\cite{Kuroda64}.
Savitch~\cite{Savitch70} then proved $\mathtt{NPSPACE} = \mathtt{PSPACE}$, hence that parsing CSG is in fact complete for deterministic polynomial space.
Another $\mathtt{PSPACE}$-complete problem is the parsing problem for \emph{non-contracting grammars}, where we have that $\len(y) \leq \len(x)$ for every rule $x \to y$.
Indeed, every CSG is also non-contracting and for every non-contracting grammar $G$, there is a CSG $G'$ with $L(G) = L(G')$~\cite[Theorem~11]{Chomsky63}.

Two grammars $G$ and $G'$ over the same vocabulary $V$ are \emph{weakly equivalent} whenever they generate the same language, i.e. $L(G) = L(G') \sub V^\star$.
For example, every non-contracting grammar is weakly-equivalent to a CSG.
A \emph{strong equivalence} preserves not only the generated languages but also the grammatical structure, i.e. it defines a bijection\footnote
{Chomsky~\cite{Chomsky63} defines two grammars to be strongly equivalent when they generate ``the same set of structural descriptions'' but he doesn't define sameness of structural descriptions.
Our definition only asserts that the two grammars assign the same number of grammatical structures to any string, not that these structures are isomorphic themselves.
Asking for an equivalence of monoidal categories $\G \simeq \G'$ would be too strong: when two free categories are equivalent, they are automatically isomorphic.} $\G(w, s) \simeq \G'(f(w), s)$ for all strings $w \in V^\star$.
For example, every CSG is strongly equivalent to a non-contracting grammar (itself) in a trivial way.

For a less trivial example, every formal grammar $G$ is strongly equivalent to a \emph{lexicalised} one, where the rules are a union $R = D \cup R'$ of \emph{dictionary entries} $D \sub V \times X$ assigning possible types to each word and production rules $R' \sub X^\star \times X^\star$ not involving the vocabulary.
Indeed, given a grammar $G$ we can add a new basic type $w'$ and a dictionary entry $w \to w'$ for each word $w \in V$ to get a lexicalised grammar $G'$.
Every grammatical structure $f : w_1 \dots w_n \to s$ in $\G'$ factorises as $f = d \fcmp f'$ for a tensor of dictionary entries $d : w_1 \dots w_n \to w_1' \dots w_n'$ and a diagram $f' : w_1' \dots w_n' \to s$ with no dictionary entries, which is isomorphic to a grammatical structure in $\G$.
Once the grammar is lexicalised, we usually draw dictionary entries as boxes labeled by the corresponding word and we omit the wires for terminal symbols.
We also assume that any \emph{semantic functor} $F : \G \to C$ from a lexicalised grammar $\G$ to some concrete category $C$ maps dictionary entries to states, i.e. $F(w) = 1$ for all words $w \in V$, which implies that the interpretation of any grammatical structure is a also state $F(f) : 1 \to F(s)$.

Unless $\mathtt{P} = \mathtt{NP} = \mathtt{PSPACE}$, there can be no efficient parser for context-sensitive grammars in general.
This motivates the introduction of \emph{context-free grammars} (CFGs, also called \emph{type-2} grammars) where the right-hand side of each rule has length one.
We can assume that the grammar is lexicalised, so that the rules have the form either $w \to x$ or $y_1 \dots y_n \to x$ for a basic type $x \in X$, a word $w \in V$ and a list of basic types $y_1 \dots y_n \in X^\star$.
In this case, grammatical structures $f : w_1 \dots w_n \to s$ have the shape of a \emph{syntax tree} with the words $w_1 \dots w_n$ as leaves and the sentence type $s$ as root.
The interchanger normal form of a syntax tree is called its \emph{left-most derivation}, when two rules apply in parallel the left-most is always applied first.
A CFG is in \emph{Chomsky normal form} (CNF\footnote
{CNF is \emph{not} a normal form in the sense that it computes representatives of equivalence classes, a given grammar may have many non-isomorphic CNFs.
In fact, deciding whether two CFGs are weakly equivalent is undecidable~\cite[Theorem~26]{Chomsky63}.}) when it is lexicalised and the rules are of the form either $s \to 1$ or $x y \to z$ for $x, y \in X - \{ s \}$ qnd $z \in X$, i.e. where all the syntax trees are binary and the sentence type appears only at the root.
Every context-free grammar $\G$ can be converted to some weakly equivalent $\G'$ in CNF, with at most a quadratic blow-up in size.
There is a monoidal functor $\G \to \G'$ mapping every $n$-ary rule to a tree of $n - 1$ binary rules when $n \geq 2$ and to the identity when $n < 2$.
This means that \emph{nullable} types, i.e. from which we can derive the empty string, are all sent to the monoidal unit.
Thus in the presence of unary and nullary rules, the functor cannot be faithful and the equivalence cannot be strong.

The CYK (Cocke–Younger–Kasami) algorithm solves the parsing problem for CNF in cubic time using dynamic programming.
Valiant~\cite{Valiant75} then reduced the problem to Boolean matrix multiplication, yielding a solution in time $O(n^{\log_2 7})$ via Strassen's algorithm.
Today, the fastest algorithm known for matrix multiplication, hence for parsing context-free grammars, is the galactic algorithm by Alman and Williams~\cite{AlmanWilliams21}.
Parsing context-free grammars is in fact complete for $\mathtt{P}$, the class of problems solvable in deterministic polynomial time~\cite{JonesLaaser74}.
Hence, whatever grammatical framework we may come up with, if its parsing problem is solvable in polynomial time then there exists a logarithmic-space reduction to CFG parsing: it takes a grammar and a string, returns a CFG and a new string such that the input is grammatical if and only if the output is.
Crucially, the output CFG depends not only on the input grammar but also on the input string: $\mathtt{P}$-completeness does not imply that there exists one fixed CFG that generates the same language.
This opens the door to grammars that are more expressive than context-free but still efficiently-parsable.

Indeed, there is evidence for some degree context-sensitivity in natural language~\cite{Huybregts84,Shieber85}.
The most studied examples are the \emph{cross-serial dependencies} of Dutch and Swiss German, which have been abstracted as the formal language $\{ w^k \ \vert \ w \in V^\star, k \leq n \}$ for some (low) constant threshold $k \leq n$.
Thus, several \emph{mildly context-sensitive grammar} (MCSG) formalisms have been introduced, which generate all of the context-free languages as well as cross-serial dependencies, yet are still parsable in polynomial time.
All MCSGs proposed so far have fallen into one of three classes of weak equivalence~\cite{Weir88}.
Thus, there is reasonable consensus over the kind of computational power required to parse human language, at least up to weak equivalence, see Kallmeyer~\cite{Kallmeyer10} for a standard survey.
However, there is no consensus yet on the syntactic way this computational power should be expressed: apart from some isolated results~\cite{SchifferMaletti21}, there is no classification of MCSGs up to strong equivalence.

Whether two grammars are strongly equivalent matters when we want to define their semantics, i.e. we want to compute the interpretation of a sentence given its grammatical structure.
Indeed, weakly equivalent grammars may assign different sets of possible parsing to the same ambiguous sentence, which will correspond to different interpretations.
For example, we can apply a monoidal functor from a CFG to a category of neural networks, which yields a \emph{recursive neural network} that computes the meaning of a sentence given its parse tree~\cite{SocherEtAl11,SocherEtAl13}.
Different trees will result in different network architectures, so how do we know we have picked the right one?
We can use a \emph{probabilistic grammar}~\cite{Salomaa69} to compute the most likely grammatical structure given some training data, in some cases with theoretical guarantees that this is indeed learnable efficiently~\cite{ClarkEtAl06,ShibataYoshinaka16}.

DisCoPy implements formal grammars with \py{Parsing}, a subclass of \py{Diagram} with \py{Word} and \py{Production} as boxes.
It does not implement any parsing algorithm, however it is straightforward to encode the output of an existing parser e.g. that of NLTK~\cite{LoperBird02} into a \py{Parsing} diagram so that we can compute the semantics of sentences by applying a \py{Functor}.

\begin{python}
{\normalfont Implementation of the \py{grammar} module and its interface with NLTK.}

\begin{minted}{python}
class Parsing(monoidal.Diagram):
    @staticmethod
    def fromtree(tree: nltk.Tree) -> Parsing:
        if len(tree) == 1 and isinstance(tree[0], str):
            return Word(tree[0], Ty(tree.label()))
        subtrees = Parsing.tensor(*[Parsing.fromtree(t) for t in tree])
        return subtrees >> Production(dom=subtrees.cod, cod=Ty(tree.label()))

class Word(monoidal.Box, Parsing):
    def __init__(self, name: str, cod: Ty, dom=Ty()):
        monoidal.Box.__init__(self, name, dom, cod)

class Production(monoidal.Box, Parsing):
    def __init__(self, dom: Ty, cod: Ty):
        name = "Production({}, {})".format(dom, cod)
        monoidal.Box.__init__(self, name, dom, cod)

Word.cast = Production.cast = Parsing.cast
\end{minted}
\end{python}

\begin{example}
We use the recursive descent parser from NLTK to parse an ambiguous expression and draw its possible parsings.

\begin{minted}{python}
from nltk import CFG, BottomUpChartParser as Parser

grammar = """
n -> a n
n -> n n
a -> 'black'
a -> 'metal'
n -> 'metal'
n -> 'fan'
"""
parser = Parser(CFG.fromstring(grammar)).parse

for tree in parser("black metal fan".split()): Parsing.fromtree(tree).draw()
\end{minted}
\begin{center}
\input{img/nlp/black-metal-fan/0.tikzstyles}
\hfill
\input{img/nlp/black-metal-fan/1.tikzstyles}
\hfill
\input{img/nlp/black-metal-fan/2.tikzstyles}
\end{center}
If we fed these syntax trees as input to the recursive neural network of Socher et al.~\cite{SocherEtAl11} (which was trained to generate images from text descriptions) we would expect to get the following images as output:
\begin{center}
\includegraphics[width=0.25\textwidth]{img/nlp/black-metal-fan/0.png}
\hfill
\includegraphics[width=0.25\textwidth]{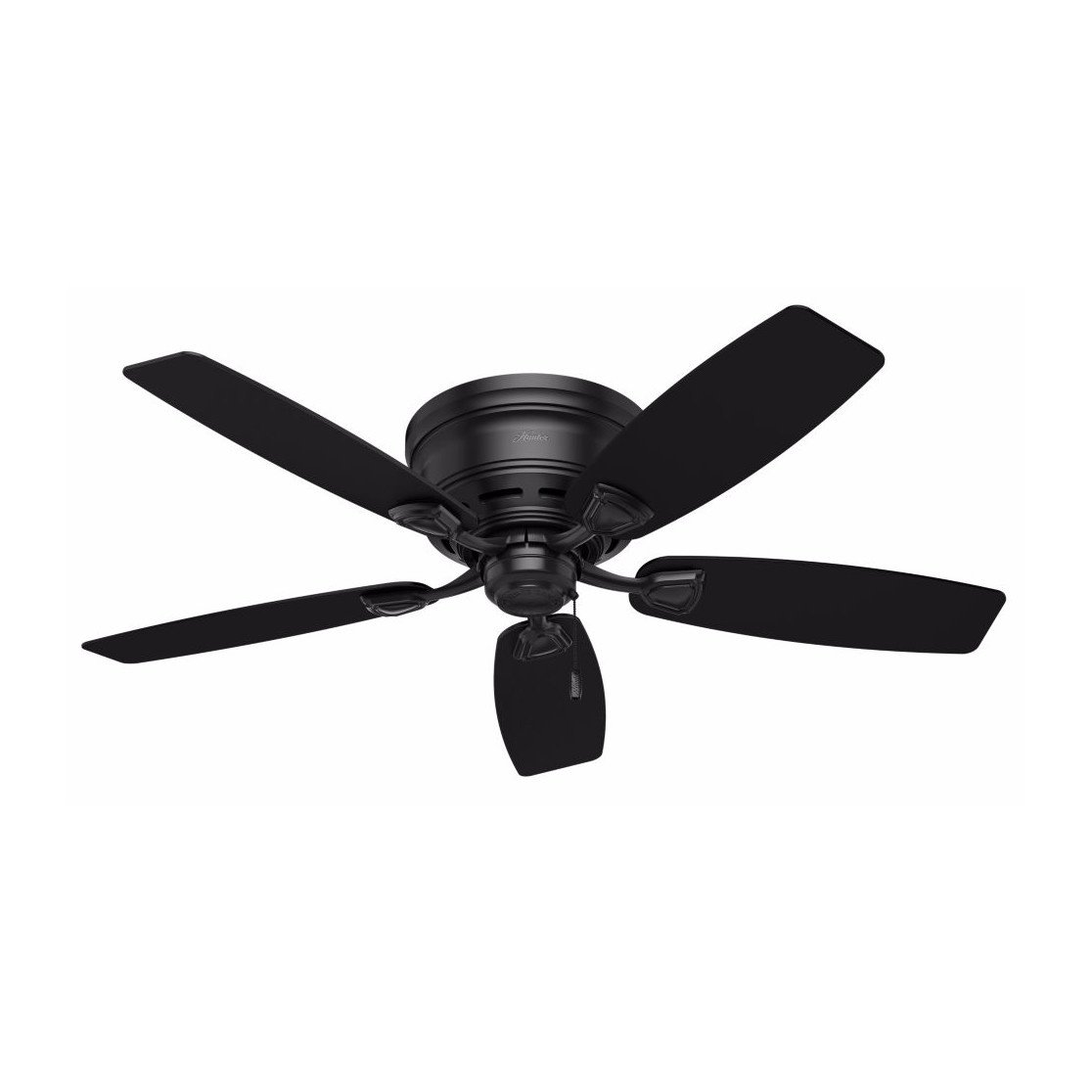}
\hfill
\includegraphics[width=0.25\textwidth]{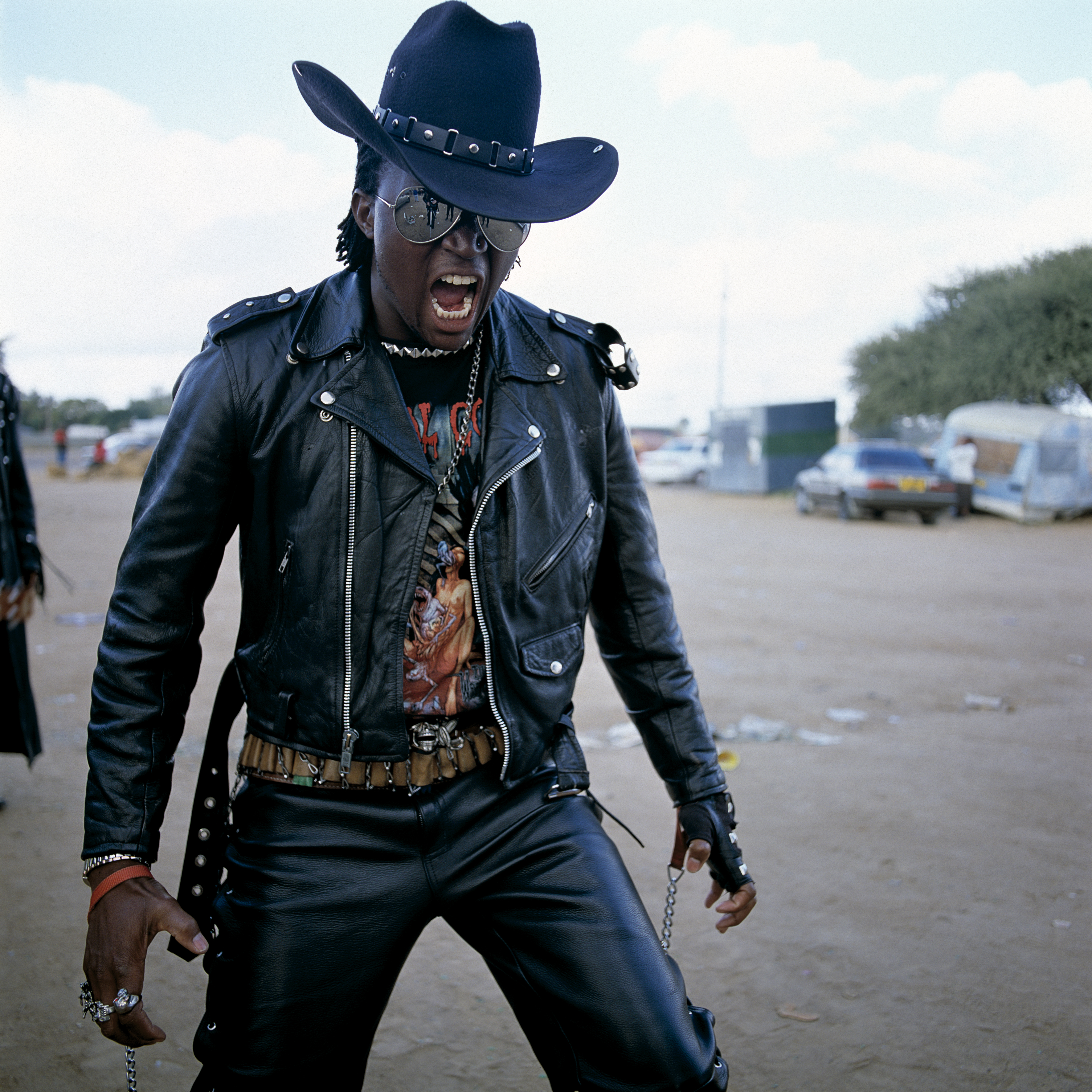}
\end{center}
\end{example}

In the framework of context-free grammars, ambiguity arises in at least two ways: 1) we have to choose from the many weakly equivalent grammars that generate the same language, 2) once the grammar is fixed we have to choose from the many syntax trees that generate the same sentence.
The second type of ambiguity cannot be alleviated: Parikh~\cite{Parikh61,Parikh66} defined a context-free language that is \emph{inherently ambiguous} in the sense that no unambiguous grammar can generate it~\cite[Theorem~29]{Chomsky63}.
In the same paper, Parikh unravels a deep connection between the theory of context-free grammars and that of free rigs: CFGs $G$ with symbols $G_0$ are in one to one corresondance with endomorphisms of the free rig $f_G : \N[G_0] \to \N[G_0]$ which fix the vocabulary, i.e. $f_G(w) = w$ for all words $w \in V$~\cite[Section~3]{Parikh66}.
Indeed, the free rig $\N[X]$ generated by a set $X$ has underlying set $\N^{X^\star}$, it can be thought of as the set of \emph{languages with multiplicities}.
Iterating the endomorphism $n$ times then projecting on the vocabulary with $\pi_V : \N[G_0] \to \N[V]$, the formal sum $\pi_V(f_G^n(s)) \in \N[G_0]$ has a term for each grammatical sentence that can be generated by a syntax tree of depth $n$ and the coefficients given by the ambiguity of the sentence, i.e. the number of different syntax trees~\cite{MichaelisKracht97}.
Thus we can define the language of a CFG as $L(G) = \cup_{n \in \N} \mathtt{sign}(\pi_V(f_G^n(s)))$ for $\mathtt{sign} : \N[V] \to \B[V] \simeq \B^{V^\star}$ the quotient map induced by $1 + 1 = 1$, i.e. forgetting multiplicities.

Parikh~\cite[Theorem~2]{Parikh66} states that if we define the map $p : V^\star \to \N^V$ which sends lists to bags by forgetting word order, then the direct image $p(L(G)) \sub \N^V$ is indistinguishable from that of a \emph{regular grammar}.
Regular grammars (also called \emph{type-3}) have rules of the form either $x \to 1$ or $x \to w y$ for non-terminals $x, y \in X$ and word $w \in V$, they are the least expressive level in Chomsky's hierarchy.
The subsets $p(L(G)) \sub \N^V$ generated by CFGs (or equivalently by regular grammars) are called \emph{semilinear}, they are finite unions of affine subspaces\footnote
{A subspace of $\N^V$ is affine if it has the form $\{ u_0 + t_1 u_1 + \dots + t_n u_n \ \vert \ t_1, \dots, t_n \in \N \}$ for some $u_0, \dots u_n \in \N^V$.
Confusingly, affine subspaces are called ``linear'' in the literature, hence ``semilinear''.
}.
Semilinearity is sometimes required as an extra condition for a grammar to be considered mildly context-sensitive, although there is evidence that some natural languages like Old Georgian are not semilinear~\cite{MichaelisKracht97}.
Regular grammars can be be equivalently defined as \emph{regular expressions}: elements of the free \emph{Kleene algebra} $K(V)$, the free idempotent rig with a closure\footnote
{A closure is an idempotent monad on a preorder, here given by $a \leq b$ iff $\exists c \cdot a + c = b$.} $(-)^\star : K(V) \to K(V)$.
Equality of regular expressions is decidable, hence so is the weak equivalence of regular grammars.
Moreover, every regular language can be generated unambiguously.
Thus, Parikh's theorem tells us intuitively that the hardness of natural language comes from its non-commutativity: if we forget about word order then everything is decidable and ambiguity disappears.
From our applied category theory perspective, this also means that language cannot be fully\footnote
{That is, faithful functors from non-regular CFGs to symmetric categories cannot be full.} investigated in symmetric categories, we need a planar monoidal data structure such as DisCoPy's \py{Diagram}.


\subsection{From the Lambek calculus to DisCoCat models}\label{subsection:lambek-discocat}

Even if we cannot get rid of the inherent ambiguity of natural language, we can still try to reduce the artificial ambiguity of our grammar formalism, i.e. the number of weakly equivalent grammars that generate the same language.

The \emph{categorial grammar} tradition may be summed up in a slogan: \emph{all the grammar is in the dictionary}~\cite{Preller07}.
Indeed, there is no need for language-specific production rules if the types of our grammar have enough structure, if we go from monoidal to closed categories.
In the \emph{Lambek calculus}~\cite{Lambek58}~\footnote
{The original calculus did not include a unit for the tensor product, here we follow the presentation given by Lambek~\cite{Lambek88} thirty years later.
We only consider \emph{string languages}, as opposed to the \emph{tree languages} generated by the non-associative Lambek calculus of 1961~\cite{Lambek61}.}, a categorial grammar is defined as a tuple $G = (V, X, D, s)$ where:
\begin{itemize}
\item $V$ and $X$ are finite sets called the \emph{vocabulary} and the \emph{basic types} with $s \in X$ the \emph{sentence type},
\item $D \sub V \times T(X)$ is a finite set of \emph{dictionary entries} with $T(X) \supseteq X$ the set of formal expressions with $1, (x \otimes y), (x / y), (x \backslash y) \in T(X)$ for all $x, y \in T(X)$.
\end{itemize}
Equivalently, the Lambek grammar $G$ may be seen as a closed monoidal signature (as defined in section~\ref{subsection:closed}) with dictionary entries as boxes where the domain is a single word.
In a \emph{basic categorial grammar}, also called an AB grammar after Ajdukiewicz~\cite{Ajdukiewicz35} and Bar-Hillel~\cite{Bar-Hillel54}, the dictionary is restricted to a closed signature, i.e. types are generated without the tensor product and unit.
The language of a categorial grammar $G$ is given by $L(G) = \{ w \in V^\star \ \vert \ \exists \ f : w \to s \in \G \}$ for $\G$ the free closed category generated by the dictionary.

More explicitly, a grammatical structure $f : w_1 \dots w_n \to s$ is given by a tensor of dictionary entries $(w_i, t_i) \in D$ followed by a closed diagram $t_1 \dots t_n \to s$ composed only of evaluation and currying.
Traditionally, these closed diagrams have been defined in terms of a \emph{sequent calculus} à la Gentzen, see Lambek~\cite{Lambek88} for a translation between the two definitions.
If we uncurry the identity on exponential types $x / y$ and $y \backslash x$ then curry them back the other way, we get the \emph{type raising} rules $x \to y / (x \backslash y)$ and $x \to (y / x) \backslash y$ which are analogous to the \emph{continuation-passing style} in functional programming~\cite{DeGroote01}.
Although it does not affect the expressive power of the Lambek calculus, type raising allows \emph{incremental parsing} where sequences of words are processed strictly from left to right~\cite{Dowty88,Steedman91}, a feature which is well-motivated from a cognitive perspective.
In previous work, Shiebler, Sadrzadeh and the present author~\cite{ShieblerEtAl20} investigate incrementality in terms of a functor from grammars to automata.

\begin{python}
{\normalfont Implementation of categorial grammars as closed categories.}

\begin{minted}{python}
class Parsing(closed.Diagram, grammar.Parsing):
    def type_raise(x: closed.Ty, y: closed.Ty, left=True) -> Parsing:
        return Parsing.id(x >> y).uncurry().curry(left=False) if left\
            else Parsing.id(y << x).uncurry(left=False).curry()

class Ev(closed.Ev, Parsing): pass
class Word(grammar.Word, Parsing): pass
Ev.cast = Word.cast = Parsing.cast
\end{minted}
\end{python}

\begin{example}
We can take $X = \{ s, n, np \}$ and assign common noun the type $n$, determiners $(np / n)$ and transitive verbs $((np \backslash s) / np)$.

\begin{minted}{python}
n, np, s = map(closed.Ty, ('n', 'np', 's'))
man, island = (Word(noun, n) for noun in ("man", "island"))
no, an = (Word(determinant, np << n) for determinant in ("no", "an"))
_is = Word("is", (np >> s) << np)

no_man_is_an_island = no @ man @ _is @ an @ island\
    >> Ev(np << n) @ ((np >> s) << np) @ Ev(np << n)\
    >> Parsing.type_raise(np, s) @ Ev((np >> s) << np)\
    >> Ev(s << (np >> s))

no_man_is_an_island.draw()
\end{minted}
\ctikzfig{img/nlp/no-man-is-an-island}
\end{example}

Bar-Hillel et al.~\cite{Bar-HillelEtAl60} showed that basic categorial grammars are strongly equivalent to context-free grammars in \emph{Greibach normal form}~\cite{Greibach65}, where every production has the form $x \to w y$ for a non-terminal $x \in X$, a word $w \in V$ and a string of non-terminals $y \in X^\star$.
Thus, their parsing problem can be solved in polynomial time.
Pentus~\cite{Pentus93} then showed that Lambek grammars are weakly equivalent to CFGs as well, although their parsing problem is $\mathtt{NP}$-complete~\cite{Pentus06}.
This means that unless $\mathtt{P} = \mathtt{NP}$ there are Lambek grammars for which the smallest weakly equivalent CFG will have exponential size.
Many extensions of the Lambek calculus have been introduced to go beyond its context-free limitation and give a more fine-grained description of syntactic phenomena, see Moortgat~\cite{Moortgat14} for a survey.
Additional unary operators called \emph{modalities} allow to break away from the planarity and linearity of closed diagrams, introducing rules for swaps and comonoids in a controlled way to model phenomena such as \emph{parasitic gaps}, \emph{ellipsis} and \emph{anaphora}.
See Moortgat~\cite[4.2]{Moortgat97} for a survey of modalities in linguistics and McPheat et al.~\cite{McPheatEtAl21} where the author and collaborators introduce a diagrammatic syntax and functorial semantics for such modalities.
The \emph{combinatory categorial grammars} (CCGs) of Steedman~\cite{Steedman87,Steedman00} take a different approach inspired by the combinatory logic of Schönfinkel~\cite{Schonfinkel24} and Curry~\cite{Curry30}, a variable-free predecessor to the lambda-calculus.
In particular, CCGs include \emph{crossed composition} rules which make them mildly context-sensitive, see Kartsaklis and Yeung~\cite{YeungKartsaklis21} for their implementation in DisCoPy.
Another extension is the \emph{abstract categorial grammar} (ACG) framework of de Groote~\cite{Groote01} defined in terms of a homomorphism from abstract to concrete syntax.
This allows to obtain a refinement of the Chomsky hierarchy which characterises mild context-sensitivity in terms of two parameters: the order of the abstract syntax and the complexity of the homomorphism~\cite{DeGrootePogodalla04}.

A key feature of categorial grammars as free closed categories, is that we can define their semantics as functors into any closed category: once the image of dictionary entries is defined, the image of any grammatical structure is fixed.
Montague~\cite{Montague70a,Montague70,Montague73} introduced the idea of semantics as a homomorphism from syntax to logic, i.e. as a closed functor $F : \G \to F^{CC}(\Sigma)$ from a categorial grammar into a free cartesian closed category with logical connectives and predicates as boxes.
Although Montague himself did not care much about syntax~\footnote
{``I fail to see any great interest in syntax except as a preliminary to semantics.''~\cite{Montague70a}},
his method provides a general recipe to interpret any categorial grammar in terms of lambda expressions.
From a computational perspective however, Montague semantics is too expressive: as we mentioned in section~\ref{subsection:closed}, the word problem for free cartesian closed categories, or equivalently the normalisation of simply-typed lambda terms, is not elementary recursive.
The \emph{Entscheidungsproblem} of Hilbert and Ackermann~\cite{HilbertAckerman28} is to decide, given a first-order logic formula, whether it is valid (i.e. true in every interpretation).
Church~\cite{Church36} proved this is undecidable, thus even if we manage to translate sentences as first-order logic formulae (which could take non-elementary time) checking if a given sentence is valid (or if two sentences are equivalent) is also undecidable.
Intuitively, we can reduce Turing's halting problem to the validity of the sentence ``the machine halts''.
The \emph{model checking problem} is to decide whether a formula is valid in a given finite model, it is $\mathtt{PSPACE}$-complete in the size of the formula~\cite[Theorem~4.3]{Gradel02} and in $\mathtt{L}$ (logarithmic space) if the formula is fixed~\cite[Corollary~4.5]{Gradel02}.

\begin{example}
We can interpret natural language as arbitrary Python code by applying a closed functor $\G \to \mathbf{Pyth}$.

\begin{minted}{python}
x, y = map(Ty, "xy")
program, runs = Word("program", x), Word("runs", x >> y)
program_runs = program @ runs >> Ev(x >> y)

F = closed.Functor(
    dom=Category(Ty, Parsing), cod=Category(tuple[type, ...], Function),
    ob={x: int, y: int},
    ar={program: lambda: 42, runs: lambda: lambda n: n * 10})
assert F(program_runs)() == 420
\end{minted}
\end{example}

\begin{example}
We can implement Montague semantics as a closed functor $\G \to \mathbf{Pyth}$ which sends the sentence type to \py{Formula}, the implementation of diagrammatic first-order logic à la Peirce given in example~\ref{example:monoidal-formula}.
Montague~\cite{Montague73} defines common nouns and intransitive verb phrases as functions from terms to formulae, in diagrammatic logic the same role is played by states and effects, i.e. open formulae with one open wire \py{x} in the codomain and domain respectively.
Montague's noun phrases are functions from intransitive verb phrase to sentence, in our setting they are given by functions from open to closed formulae.
The transitive verb ``is'' must take two such functions \py{P} and \py{Q} as input and return a closed formula, the only thing we can do is apply \py{P} to the identity diagram on \py{x} to get a state, before feeding the dagger of the result to \py{Q}.

\begin{minted}{python}
x = Ty('x')

Montague = closed.Functor(
    dom=Category(Ty, Parsing), cod=Category(tuple[type, ...], Function),
    ob={s: Formula, n: Formula, np: exp(Formula, Formula)},
    ar={no: lambda: lambda state: lambda effect: (state >> effect).bubble(),
        man: lambda: Predicate("man", x),
        _is: lambda: lambda P: lambda Q: Q(P(Formula.id(x)).dagger()),
        an: lambda: lambda state: lambda effect: state >> effect,
        island: lambda: Predicate("island", x)})
Montague(no_man_is_an_island)().draw()
\end{minted}

\ctikzfig{img/nlp/peirce-montague}

\begin{minted}{python}
size = {x: 2}

for mans, islands in itertools.product(*2 * [
        itertools.product(*size[x] * [[0, 1]])]):
    F = model(size, {Predicate("man", x): mans, Predicate("island", x): islands})
    assert F(Montague(no_man_is_an_island)()) == not any(
        F(Predicate("man", x))[i] and F(Predicate("island", x))[i]
        for i in range(size[x]))
\end{minted}
\end{example}

Returning to linguistics half a century after his seminal \emph{Mathematics of sentence structure}, Lambek~\cite{Lambek99,Lambek01,Lambek08} introduced \emph{pregroup grammars} as a simplification of his original calculus replacing closed categories by rigid categories.
That is, the dictionary of a pregroup grammar $G = (V, X, D, s)$ has the shape $D \sub V \times (X \times \Z)^\star$, it assigns words to lists of iterated adjoints of basic types.
Again, the language of $G$ is defined as $L(G) = \{ w \in V^\star \ \vert \ \exists \ f : w \to s \in \G \}$ where now $\G$ is the free rigid category generated by the dictionary.
Equivalently, a sentence $w = w_1 \dots w_n \in V^\star$ is grammatical if there are dictionary entries $(w_i, t_i) \in D$ such that $t_1 \dots t_n \leq s$ holds in the free pregroup, i.e. the preorder collapse of $\G$.
In fact, Lambek first defined his pregroup grammars in terms of partial orders (i.e. preorders with antisymmetry) then Preller and he~\cite{PrellerLambek07} reformulated them in terms of free compact 2-categories (i.e. rigid categories with colours) so that they could account for ambiguity.
Moving from preorders to free categories also allows to define pregroup semantics as functors, indeed a functor with a preorder as domain is required to map all the parsings of an ambiguous sentence to the same meaning.
Even worse, a monoidal functor with a pregroup as domain has to obey the equation $F(x) = F(x \otimes x^l \otimes x)$ for all types $x \in X$, which makes the functor trivial in categories of interest such as $\mathbf{Set}$ or $\mathbf{Mat}_\S$.

Every rigid category is also closed, thus for any Lambek grammar $G$ we can construct a pregroup grammar $G'$ and a closed functor $\G \to \G'$ which sends categorial types $x \backslash y$ and $x / y$ to pregroup types $x^r y$ and $x y^l$.
In general this functor need not be faithful: it maps both $(x \otimes y) / z$ and $x \otimes (y / z)$ to the same pregroup type $x y z^l$.
Although they cannot be strongly equivalent to Lambek grammars, Buszkowski~\cite{Buszkowski01} proved that pregroup grammars are context-free, hence they are still weakly-equivalent.
As for the complexity of their parsing problem, Lambek~\cite{Lambek99} first showed it was decidable with the following \emph{switching lemma}: any pregroup inequality $x \leq z$ can be factored into $x \leq y \leq z$ where $x \leq y$ using only cups then $y \leq z$ using only caps.
The proof is essentially given by the snake removal algorithm of listing~\ref{listing:snake-removal} applied to rigid diagrams with only cups and caps: the resulting normal form can be shown to have all cups preceding the caps.
As a corollary, if there is a grammatical structure $f : w_1 \dots w_n \to s$ then there is one using only dictionary entries and cups which we can find by brute force search, Oehrle~\cite{Oehrle04} showed that pregroup parsing can in fact be solved in cubic time.
Preller~\cite{Preller07a} gave sufficient conditions on the dictionary for unambiguous pregroup grammars to be parsed in linear time, the algorithm was later improved and implemented in DisCoPy by Rizzo~\cite{Rizzo21}.
Multiple extensions of pregroup grammars have been proposed to go beyond context-free languages including taking products of free pregroups~\cite[Section~28]{Lambek08}, grammars with infinite dictionaries~\cite{Preller10} or a notion of buffer~\cite{GenkinEtAl10}.

One can define the semantics of a pregroup grammar $G$ as a monoidal functor $F : \G \to C$ into any rigid category $C$, in the \emph{categorical compositional distributional} (DisCoCat) models of Clark, Coecke and Sadrzadeh~\cite{ClarkEtAl08,ClarkEtAl10} one takes $C = \mathbf{Mat}_\S$ the category of matrices.
More explicitly, a DisCoCat model $F : \G \to \mathbf{Mat}_\S$ is defined by a dimension $F(x) \in \N$ for every basic type $x \in X$ and a vector $F(w, t) : 1 \to F(t)$ for every dictionary entry $(w, t) \in D$.
It then defines the semantics of a grammatical sentence $f : w_1 \dots w_n \to s$ as the contraction of a tensor network where the nodes are dictionary entries and the edges are given by the cups.
Note that the two seminal articles~\cite{ClarkEtAl08,ClarkEtAl10} do not mention rigid categories and stick with the definition of pregroup grammars in terms of partial orders.
Unable to define non-trivial functors $P \to \mathbf{Mat}_\S$ for $P$ the free pregroup generated by $G$, i.e. the preorder collapse of $\G$, they resort to defining DisCoCat in terms of a product category $\mathbf{Mat}_\S \times P$.
Although they do not say it explicitly, they in fact worked with a subcategory of $\G \times \mathbf{Mat}_\S$ called the \emph{Grothendieck construction} on a monoidal functor $F : \G \to \mathbf{Mat}_\S$, see Bradley et al.~\cite{BradleyEtAl18} for an application of this observation to model translation and language evolution in DisCoCat.

DisCoCat models came out of a quest to accommodate the \emph{compositional} approach to NLP which focused on grammatical structure and the \emph{distributional} approach that represented words as vectors extracted from text data.
A key feature of this approach compared to its predecessor is that DisCoCat models define a similarity measure between any two expressions of the same type, even though they do not have the same grammatical structure.
Indeed, any two pregroup diagrams $f, g : w_1 \dots w_n \to x$ will be mapped to vectors of dimension $n = F(x)$ such that the inner product $\langle F(f) \vert F(g) \rangle$ yields a measure of their similarity.
We can then use standard machine learning techniques to solve problems such as classification, an approach that has received some empirical support on small-scale datasets~\cite{GrefenstetteSadrzadeh11}.
The name of DisCoPy stands for \emph{Distributional Compositional Python}, indeed it was first meant as an implementation of DisCoCat models before it turned into an implementation of monoidal functors in general.

\begin{python}
{\normalfont Implementation of pregroup grammars as rigid categories.}

\begin{minted}{python}
class Parsing(rigid.Diagram, categorial.Parsing): pass
class Cup(rigid.Cup, Parsing): pass
class Word(categorial.Word, Parsing): pass
Cup.cast = Word.cast = Parsing.cast
\end{minted}
\end{python}

\begin{example}\label{example:alice-loves-bob}
Computing the meaning of ``Alice loves Bob'' was the first example in DisCoPy's documentation.

\begin{minted}{python}
s, n = rigid.Ty('s'), rigid.Ty('n')
Alice, loves, Bob =\
    Word('Alice', n), Word('loves', n.r @ s @ n.l), Word('Bob', n)
sentence = Alice @ loves @ Bob >> Cup(n, n.r) @ s @ Cup(n.l, n)

sentence.draw()
\end{minted}

\ctikzfig{img/nlp/alice-loves-bob}

\begin{minted}{python}
F = rigid.Functor(
    dom=Category(rigid.Ty, Parsing), cod=Category(tuple[int, ...], Tensor),
    ob={s: 1, n: 2},
    ar={Alice: [[1, 0]], loves: [[0, 1], [1, 0]], Bob: [[0, 1]]})

assert F(sentence)
\end{minted}
\end{example}


\subsection{Anaphora and the quantum complexity of language}\label{section:anaphora}

While the meaning of \emph{lexical words} (also called \emph{content} words) such as nouns and verbs are extracted from text data, DisCoCat models allow to encode \emph{functional words} such as pronouns and conjunctions in terms of the Frobenius algebras, a.k.a. spiders, we discussed in section~\ref{subsection:hypergraph}.
This \emph{Frobenius anatomy of word meanings} was first applied to relative pronouns~\cite{SadrzadehEtAl13,SadrzadehEtAl14}, then to coordination~\cite{Kartsaklis16} as well as intonation~\cite{KartsaklisSadrzadeh15}.
In previous work with Coecke, de Felice and Marsden~\cite{CoeckeEtAl18a} as well as in a subsequent dissertation~\cite{Toumi18a}, we proposed the use of spiders to model \emph{anaphora}, expressions such as personal pronouns whose meaning depends on another expression in context, connecting the diagrams for sentences together into a diagram for \emph{discourse}.
This proposal came with an algorithm for constructing a \emph{relational database} from any such discourse diagram and for translating the diagrams of questions into database queries.

This discourse-to-database translation was refined in later work with de Felice and Meichanetzidis~\cite{FeliceEtAl19} where we defined the \emph{functorial question answering} problem as the application of a given DisCoCat model to a question diagram.
In the case of Boolean-valued models $F : \G \to \mathbf{Mat}_\S$, we proved that this question-answering problem is in fact equivalent to \emph{conjunctive query evaluation}, which is $\mathtt{NP}$-complete by a celebrated theorem of Chandra and Merlin~\cite{ChandraMerlin77}.
Conjunctive queries can be defined as Peircean diagrams with no bubbles, i.e. only spiders and predicate boxes, or equivalently as the first-order logic formulae with existentials and conjunction but no negation, see Bonchi et. al for a diagrammatic treatment~\cite{BonchiEtAl18}.
The Chandra-Merlin theorem is based on the construction of a \emph{canonical model} for the given query, i.e. a canonical functor given a diagram, then reducing evaluation to the problem of graph homomorphism between models.
Once translated in terms of Boolean DisCoCat models, the same result implies that question-answering (i.e. the application of a functor to a diagram) is equivalent to the \emph{entailment problem}: given two sentences, does the truth of one imply that of the other?

DisCoCat models with anaphoric spiders are unsatisfying in two opposite ways: they are not expressive enough to encode negation\footnote
{Note that the results of \cite{FeliceEtAl19} assume that the sentence type is mapped to the unit, i.e. the meaning of a sentence is a scalar in $\B$.
Preller~\cite{Preller14a,Preller14} takes an alternative, four-valued approach to first-order logic with pregroups where the sentence type is given dimension $2$: a sentence is either true, false, neither or both.}
but if we allow arbitrary anaphora, they are already too expressive to be computed efficiently.
The first point cannot be avoided if we want our models to be tractable: adding negation to conjunctive queries will generate all of first-order logic, for which question-answering (i.e. model checking) is $\mathtt{PSPACE}$-complete and entailment (i.e. the Entscheidungsproblem) undecidable.
We may get around the second dissatisfaction by adding restrictions on anaphoric expressions.
For example, requiring that the corresponding query has \emph{bounded tree-width}~\cite{ChekuriRajaraman00} ensures that question answering is solvable in polynomial time.
This restriction may be motivated in terms of \emph{bounded short-term memory}: a query with tree-width $k$ corresponds to a first-order logic with $k$ variables, each of which may be bound and reused multiple times.
We refer to Abramsky and Shah~\cite{AbramskyShah21} for a comonadic approach to such resource bounds.

If we go from Booleans to natural numbers, we get a \emph{counting problem}: we want to know not only whether but \emph{how many} answers a question has in a given model.
This answer-counting problem is complete for $\mathtt{\#P}$, the generalisation of $\mathtt{NP}$ from decision to counting problems.
By extension, evaluating DisCoCat models over the real or complex numbers (with finite precision) is also $\mathtt{\#P}$-complete.
The closest decision complexity class is $\mathtt{PP}$, also called \emph{Majority-P}, the class of problems solvable in probabilistic polynomial time with no error bounds, which amounts to computing the most significant digit of a $\mathtt{\#P}$ problem.
If we write $\mathtt{P}^\mathtt{X}$ for the class of problems solvable in polynomial time with access to an oracle solving any problem of $\mathtt{X}$ in one step, we can use a binary search to prove $\mathtt{P}^\mathtt{PP} = \mathtt{P}^\mathtt{\#P}$.
One indication for how much harder counting is compared to decision problems is \emph{Toda's theorem}~\cite{Toda91}.
It states that $\mathtt{PH} \sub \mathtt{P}^\mathtt{\#P}$ where $\mathtt{PH}$ is the \emph{polynomial hierarchy}, the union of all towers of $\mathtt{NP}$-oracles, i.e. $\mathtt{PH} = \cup_{n \in \N} \Sigma_n$ where $\Sigma_0 = \mathtt{P}$ and $\Sigma_{n + 1} = \mathtt{NP}^{\Sigma_n}$.
Intuitively, a $\mathtt{PP}$-oracle for counting problems is at least as powerful as any tower of $\mathtt{NP}$-oracles for decision problems.

In a beautifully simple theorem, Aaronson~\cite{Aaronson05} shows that $\mathtt{PP} = \mathtt{PostBQP}$, the class of problems solvable in polynomial time by a quantum computer given \emph{post-selection}, the ability to make the possible necessary and to choose what outcome we get from a quantum measurement\footnote
{Assuming the many-world hypothesis, Aaronson~\cite{Aaronson05} gives a simple method to achieve post-selection: committing suicide if we do not get the desired outcome.
A less brutal method is to keep on trying until we do, in exponential time on average.}.
In one direction, this equality says that the evaluation of a post-selected quantum circuit can be reduced to tensor contraction: quantum gates, bras and kets are nodes, the qubits connecting them are edges.
In the other, it means that we can use post-selected quantum computation to contract any tensor network, or equivalently to evaluate any monoidal functor from a free compact closed category into $\mathbf{Mat}_\C$.
The related counting class $\mathtt{\#P}$ was originally introduced by Valiant~\cite{Valiant79} to show the completeness the \emph{matrix permanent}.
Aaronson and Arkhipov~\cite{AaronsonArkhipov11} then related it to the complexity of \emph{boson sampling}, a restricted model of quantum computation which they prove cannot be simulated classically unless the \emph{polynomial hierarchy collapses to the third level}.
Although this is less unlikely than $\mathtt{P} = \mathtt{NP}$, i.e. a collapse at level zero, this is still believed to be a strong indication that quantum computers cannot be efficiently classically simulated.

Removing post-selection from $\mathtt{PostBQP}$ we get $\mathtt{BQP}$, bounded-error quantum polynomial time, arguably the largest class of decision problems that a physical machine\footnote
{By a physical machine we mean a machine that obeys the laws of quantum mechanics.
This excludes machines exploiting features of general relativity such as closed time-like curves (CTCs).
Using the CTCs of Deutsch~\cite{Deutsch91} a quantum computer can solve all of $\mathtt{PSPACE}$ in polynomial time, while the more restricted CTCs of Lloyd et al.~\cite{LloydEtAl11,LloydEtAl11a} solve all of $\mathtt{PostBQP}$ in polynomial time.
See Pinzani, Gogioso and Coecke~\cite{PinzaniEtAl19} for a diagrammatic treatment of time-travel in terms of traced categories.}
can solve efficiently.
Its classical analog $\mathtt{BPP}$ (bounded-error probabilistic polynomial time) is contained in $\mathtt{BQP}$ because quantum computers can simulate classical ones efficiently, but whether the containment is strict is an open question.
The best we can do is define $\mathtt{BQP}$-complete problems with \emph{circuit approximation} as the canonical example: given the description of a quantum circuit, decide whether measuring the first qubit yields a one, with the promise that the probability for this is bounded away from a half.
Arad and Landau~\cite{AradLandau10} reformulate this in terms of the additive approximation of tensor networks, i.e. the additive approximation of a monoidal functor into $\mathbf{Mat}_\C$.
In one direction, their reduction tells us that we can approximate tensor networks efficiently with a quantum computer.
In the other, it means that if we could approximate such functors with a classical computer then we could also simulate any quantum circuit efficiently.

What do these complexity results imply for evaluating DisCoCat models on quantum computers?
First, that we cannot hope to evaluate them exactly unless we discover (safe and efficient) time travel and prove $\mathtt{PostBQP} = \mathtt{BQP}$.
Second, that we can evaluate them efficiently with a quantum computer, up to additive approximation.
Third, that if we could do the same with classical computers then they would turn out to be as powerful as quantum computers after all.
We would automatically get a classical Shor algorithm that can outperform the best number sieves mathematicans have come up with, and a classical Grover algorithm that can find a needle in a haystack.
Thus, by interpreting pregroup grammars in terms of tensor networks, DisCoCat models provide a way to reformulate quantum computing in terms of natural language processing.
In short, if $\mathtt{BPP} \neq \mathtt{BQP}$ then quantum computers would allow us to (approximately) answer exponentially bigger natural language questions than classical computers can.
The idea of using quantum circuits to evaluate DisCoCat models was first introduced by Zeng and Coecke~\cite{ZengCoecke16}, where they show a quadratic advantage on a more restricted classification task using Grover's algorithm as subroutine.
Given the size of classical NLP models today, empirical evidence of quantum advantage for natural language processing will most probably require fault-tolerant quantum computers with millions of qubits, if not billions.
In the meantime, we are left to explore the possibilities offered by the small noisy quantum computers of today.


\section{DisCoCat models on quantum hardware}\label{section:discocat-qnlp}

We get to the main definition of this thesis: by a QNLP model we mean a monoidal functor $F : \G \to \mathbf{Circ}$ from the category generated by a grammar $\G$ to a category $\mathbf{Circ}$ of quantum circuits.
Chapter~\ref{chapter:discopy} has already given several definitions of $\mathbf{Circ}$:
\begin{itemize}
    \item as a category (example~\ref{example:Circuit}),
    \item as a monoidal category (example~\ref{example:circuit-diagrams}),
    \item as a pivotal category (example~\ref{example:pivotal-circuit}),
    \item as a category with biproducts (example~\ref{example:biproduct-measurement}).
\end{itemize}
So far, these definitions have only covered \emph{pure quantum circuits} with post-selected measurements. Pure quantum gates are interpreted as unitary matrices, preparation (kets) and measurement (bras) as basis vectors.
The evaluation of a closed pure circuit is a complex scalar and the Born rule says its squared amplitude is the probability of a measurement outcome for a given preparation.
We can use a quantum computer to approximate this probability by executing our circuit many times or in quantum computing parlance, taking many \emph{shots}.
At each execution, we measure all the qubits then we divide the number of times the post-selected outcome occurs by the number of shots.
With the interface from the \py{Circuit} class of DisCoPy to that of the t$\vert$ket$\rangle$ compiler~\cite{SivarajahEtAl20}, going from numerical simulation to quantum hardware is as easy as providing an extra argument \py{backend: pytket.Backend} to the method \py{Circuit.eval}.

\begin{example}\label{example:circuit-alive-loves-bob}
Executing the circuit for ``Alice loves Bob'' was the first NLP experiment on a quantum computer, as documented in Coecke et al.~\cite{CoeckeEtAl20b}.

\begin{minted}{python}
F = rigid.Functor(
    dom=Category(Ty, Parsing), cod=Category(Qubits, Circuit),
    ob={s: qubit ** 0, n: qubit ** 1},
    ar={Alice: Ket(0), loves: Ket(0, 0) >> H @ sqrt2 @ X >> CX, Bob: Ket(1)})

drawing.equation(sentence, F(sentence), symbol="$\\mapsto$")
\end{minted}
\ctikzfig{img/qnlp/alice-loves-bob}
\begin{minted}{python}
from pytket.backends.ibm import IBMQBackend
assert F(sentence).eval(
    backend=IBMQBackend('ibmq_singapore', hub='ibmq'), n_shots=2 ** 10) > .5
\end{minted}
\end{example}

Let's unpack what happened when we executed the last line of the example above.
We apply a monoidal functor \py{F} to the diagram for a \py{sentence} to get a circuit diagram.
DisCoPy then translates it into a \py{pytket.Circuit} which gets executed $2^{10}$ times on the \py{'ibmq_singapore'} backend (a 20-qubit quantum device).
We then assert that the probability of measuring all zeros is bigger than a $\frac{1}{2}$ threshold, once multiplied by $(\vert \sqrt{2} \vert^2)^3 = 8$ to account for the three scalars.
What did we achieve?
We have evaluated (the squared amplitude of) the DisCoCat model of example~\ref{example:alice-loves-bob} on a quantum computer!
Indeed, we have chosen our QNLP model so that the quantum states for words encode their interpretation, e.g. we have sent \py{loves} to an (anti-correlated) Bell state scaled by $\sqrt{2}$ so that \py{F(loves).eval().inside} \py{= [[0, 1], [1, 0]]}.
Assuming a universal gate set, every complex vector $u \in \C^{2^n}$ can be encoded as a pure quantum circuit $c : \mathtt{qubit}^0 \to \mathtt{qubit}^n$ scaled by a positive real scalar to account for the normalisation, which we can represent as a quantum gate on zero qubits.

In fact, every choice of encoding will define a faithful monoidal functor $\mathtt{load} : \mathbf{Tensor}_\C \to \mathbf{Circ}$ from complex tensors to pure quantum circuits with real scalars (up to equality of their interpretation) which is an inverse to evaluation, i.e. $\mathtt{load} \fcmp \mathtt{eval} = \id(\mathbf{Tensor}_\C)$.
On objects, it sends a dimension $d \in \N$ to a qudit, i.e. a $d$-dimensional quantum system.
From the rigid structure of $\mathbf{Tensor}_\C$, every arrow $f : x \to y$ can be written as $f = x \otimes u \fcmp \ttcup(x) \otimes y$ for the state $u : 1 \to x y$ given by $u = \ttcap(x) \fcmp x \otimes f$.
Thus, every tensor can be encoded as a scaled circuit followed by a post-selected Bell measurement.
From this isomorphism we can extract an abstract proof of $\mathtt{PostBQP} = \mathtt{PP}$: simulating a post-selected quantum circuit is equivalent to evaluating a monoidal functor into tensors, i.e. contracting a tensor network.

In particular, for any DisCocat model $F : \mathbf{G} \to \mathbf{Tensor}_\C$ we get a QNLP model $F_Q = F \fcmp \mathtt{load} : \mathbf{G} \to \mathbf{Circ}$ such that $F = F_Q \fcmp \mathtt{eval}$.
Evaluating the circuit $F_Q(f) = c$ for a grammatical sentence $f : w_1 \dots w_n \to s$ yields a quantum state that encodes its interpretation $\mathtt{eval}(c) = F(f)$.
When the interpretation is a scalar, i.e. $F(s) = 1$, we can directly evaluate the circuit to compute the squared amplitude $\vert F(f) \vert^2$.
If furthermore we know that the evaluation must be a positive real scalar, e.g. because the vectors for all dictionary entries are positive reals as in our example, we can simply take the square root of this probability to get a truth value for the sentence.
However, we are measuring the probability of an exponentially unlikely event (measuring all qubits to zero) thus in general we will need an exponential number of shots to approximate the result.
If the truth value is an arbitrary complex scalar, we can use a \emph{Hadamard test} to compute the real and imaginary part of the complex scalar $\mathtt{eval}(c) \in \C$.
When the interpretation is an $F(s) = n$-dimensional state for $n \geq 2$, we can perform a \emph{swap test} with the quantum state for another grammatical expression of the same type, which will compute their inner product, i.e. measure their similarity.
From the characterisation of $\mathtt{BQP}$ by Arad and Landau~\cite{AradLandau10}, we know we can get an additive approximation using a polynomial number of shots.
However, in general we have no guarantee that this additive error will be no larger than the value we want to approximate.

Thus, this naive definition of QNLP models is unsatisfactory from a computational point of view: if we post-select on all the qubits we'll have to wait an exponentially long time even to approximate an answer.
It is also unsatisfactory from a category theoretic point of view.
Indeed, so far we have defined the evaluation of pure quantum circuits as a functor into tensor, but what we can actually execute on a quantum device is the Born rule of the result, which is not a functor: the squared amplitude of a composition is not necessarily the composition of the squared amplitudes.
Killing two birds with one stone, we can overcome both limitations (computational and theoretical) if we extend our definition of $\mathbf{Circ}$ from \emph{pure} to \emph{mixed} quantum circuits.
In practice, what our QNLP algorithm is missing is the ability to do nothing, i.e. not to measure a qubit.
In theory, our category $\mathbf{Circ}$ is missing a box for \emph{discard}.
As soon as we leave the realm of pure quantum circuits, quantum states cannot be represented as complex vectors anymore, we need \emph{density matrices}.
Similarly, we cannot interpret quantum circuits as unitary matrices anymore, we need \emph{completely-positive trace-preserving} (CPTP) maps, also called \emph{quantum channels}.
As a bonus, going from pure to mixed and from unitaries to channels gives us enough room to talk about both classical and quantum processes in the same concrete category $\mathbf{Channel}$.


\subsection{Quantum channels and mixed quantum circuits}\label{section:mixed-circuits}

DisCoPy implements a variant of the definition of \emph{classical-quantum maps} (cq-maps) from Coecke and Kissinger~\cite[Chapter~8]{CoeckeKissinger17}.
Abstractly, it is a simplification of the $\mathbf{CP^\star}$ construction from Coecke, Heunen and Kissinger~\cite{CoeckeEtAl12}, which itself generalises the notion of finite-dimensional C$^\star$-algebra to arbitrary dagger compact closed categories.
Concretly, we first define the category $\mathbf{CQMap}$ with objects given by pairs of natural numbers\footnote
{We actually implement an equivalent category where objects are pairs of lists of natural numbers and the arrows are tensors rather than matrices.} $(a, b) \in \N \times \N$ for the classical and quantum dimensions of the system.
The arrows $f : (a, b) \to (c, d)$ are given by $(a b^2) \times (c d^2)$ complex matrices, with composition given by matrix multiplication and tensor given by the following diagram:
\ctikzfig{img/qnlp/channel-tensor}
A \emph{quantum channel} is a cq-map subject to the following two conditions:
\begin{itemize}
\item \emph{complete positivity} (CP), there is a dimension $n \in \N$ and an $(a b) \times (c d n)$ matrix $g$ with element-wise conjugate $g^\star$ such that:
\ctikzfig{img/qnlp/complete-positivity}
\item \emph{trace preservation} (TP) also called \emph{causality}:
\ctikzfig{img/qnlp/trace-preservation}
\end{itemize}
Note that we take the convention to use the \emph{algebraic conjugate} which is the identity on objects, rather than the \emph{diagrammatic conjugate} which reverses the order of wires.
This makes the implementation easier at the cost of breaking the symmetry of the diagram for complete positivity.
We also note that \emph{Picturing quantum processes}~\cite{CoeckeKissinger17} does not distinguish between diagrams and their evaluation as matrices.
Moreover, their definition of cq-map includes the complete positivity condition, thus they do not give a name to the matrices that we call cq-maps.

There is a functor $\mathtt{double} : \mathbf{Mat}_\C \to \mathbf{CQMap}$ which sends a dimension $n$ to the pair $(1, n)$ and a complex matrix $f$ to its \emph{double} $\widehat{f}$, the cq-map given by tensoring with its conjugate $\widehat{f} = f \otimes f^\star$.
The double of a scalar $s : 1 \to 1$ gives the same result as the Born rule $\hat{s} = s \bar{s} = \vert s \vert^2$.
The double of any matrix $f : m \to n$ automatically satisfies the complete-positivity condition; it satisfies causality iff $f$ is an isometry, i.e. $f^\dagger \fcmp f = \id(n)$.
Furthermore the functor $\mathtt{double}$ is faithful up to a global phase~\cite[Proposition~6.6]{CoeckeKissinger17}.
A cq-map is called \emph{pure} if it is in the image of $\mathtt{double}$ and \emph{mixed} otherwise.

There is also a functor $\mathtt{single} : \mathbf{Mat}_\R \to \mathbf{CQMap}$ which sends a dimension $n$ to the pair $(1, n)$ and a real matrix $f : m \to n$ to itself.
Again, this is always completely positive and it satisfies causality iff $f$ is a stochastic matrix, i.e. all its columns sum to one.
For every dimension $x \in \N$, there are channels $\mathtt{measure}(x) : (1, x) \to (x, 1)$ and $\mathtt{encode}(x) : (x, 1) \to (1, x)$ with underlying matrix given by $\mathtt{spider}_{2, 1}(x)$ and $\mathtt{spider}_{2, 1}(x)$.
For every pair of dimensions $a, b \in \N$, we can also define $\mathtt{discard}(a, b) : (a, b) \to (1, 1)$ and $\mathtt{mixed\_state}(a, b) : (1, 1) \to (a, b)$ with underlying matrix given by $\mathtt{spider}_{1, 0}(a) \otimes \mathtt{cup}(b)$ and its dagger.
Thus, the intuition for the causality condition is that the future cannot signal to the past: if we apply a quantum process then discard the output, we might as well have discarded the input.
Abstractly, causality makes the category $\mathbf{Channel}$ \emph{semicartesian}, i.e. the monoidal unit is a terminal object, discarding is the unique arrow from any object into it.

\begin{python}
{\normalfont Implementation of the category $\mathbf{CQMap}$ with \py{CQ} as objects and \py{Channel} as arrows.}

\begin{minted}{python}
@dataclass
class CQ:
    classical: tuple[int, ...] = ()
    quantum: tuple[int, ...] = ()

    def tensor(self, other: CQ) -> CQ:
        return CQ(self.classical + other.classical, self.quantum + other.quantum)

    def downgrade(self) -> tuple[int]:
        return self.classical + 2 * self.quantum

    __matmul__ = tensor

C, Q = lambda x: CQ(classical=x), lambda x: CQ(quantum=x)

@dataclass
class Channel(Composable, Tensorable):
    inside: Tensor[complex]
    dom: CQ
    cod: CQ

    @staticmethod
    def id(x: CQ) -> Channel:
        return Channel(x, x, Tensor[complex].id(x.downgrade()))

    def dagger(self) -> Channel:
        return Channel(self.inside.dagger(), self.cod, self.dom)

    @inductive
    def then(self, other: Channel) -> Channel:
        assert self.cod == other.dom
        return Channel(self.inside >> other.inside, self.dom, other.cod)

    @inductive
    def tensor(self, other: Channel) -> Channel:
        inside = ...  # Given by the diagram above.
        return Channel(inside, self.dom @ other.dom, self.cod @ other.cod)

    @staticmethod
    def double(f: Tensor[complex]) -> Channel:
        return Channel(f @ f.map(lambda x: x.conjugate()), Q(f.dom), Q(f.cod))

    @staticmethod
    def single(f: Tensor[float]) -> Channel:
        inside = Tensor[complex](f.inside, f.dom, f.cod)
        return Channel(inside, C(f.dom), C(f.cod))

    @staticmethod
    def measure(x: tuple[int, ...]) -> Channel:
        return Channel(Tensor[complex].spider(2, 1, x), Q(x), C(x))

    @staticmethod
    def encode(x: tuple[int, ...]) -> Channel:
        return Channel(Tensor[complex].spider(1, 2, x), C(x), Q(x))

    @staticmethod
    def discard(x: CQ) -> Channel:
        inside = Tensor[complex].spider(1, 0, x.classical)\
            @ Tensor[complex].cups(x.quantum, x.quantum[::-1])
        return Channel(inside, x, CQ())
\end{minted}
\end{python}

The category $\mathbf{CQMap}$ inherits a dagger compact closed structure from that of $\mathbf{Mat}_\C$.
The swaps (cups, caps) for classical-quantum systems $(a, b) \in \N \times \N$ are given by tensoring a single swap (cup, cap) for $a$ and a double swap (cup, cap) for $b$.
It is also commutative-monoid-enriched with element-wise addition, note however that the functor $\mathtt{double}$ is not CM-enriched, i.e. $\widehat{f + f'} \neq \widehat{f} + \widehat{f'}$ in general.
In quantum mechanical terms, this corresponds to the distinction between
quantum superposition and probabilistic mixing.
Crucially, the subcategory $\mathbf{Channel}$ of CPTP maps is \emph{not} compact-closed because caps are not causal.
It is not commutative-monoid-enriched: the sum of two channels is not causal (although any convex combination is).

\begin{python}
{\normalfont Implementation of $\mathbf{CQMap}$ as a CM-enriched compact closed category.}

\begin{minted}{python}
CQ.l = CQ.r = property(lambda self: self)

for attr in ("swap", "cups", "caps"):
    def channel_method(left: CQ, right: CQ) -> Channel:
        tensor_method = getattr(Tensor, attr)
        return Channel.single(tensor_method(left.classical, right.classical))\
            @ Channel.double(tensor_method(left.quantum, right.quantum))
    setattr(Channel, attr, channel_method)

def __add__(self, other: Channel) -> Channel:
    assert self.dom == other.dom and self.cod == other.cod
    return Channel(self.inside + other.inside, self.dom, self.cod)

@staticmethod
def zero(dom: CQ, cod: CQ) -> Channel:
    return Channel(Tensor.zero(dom.dowgrade(), cod.downgrade()), dom, cod)

Channel.__add__, Channel.zero = __add__, zero
\end{minted}
\end{python}

We're now ready to define $\mathbf{Circ}$ as a free symmetric monoidal category with a monoidal functor $\mathbf{Circ} \to \mathbf{CQMap}$.
The definition of circuits depends on a choice of \emph{gateset}, i.e. a monoidal signature $\Sigma$ together with a functor $[\![-]\!] : F^S(\Sigma) \to \mathbf{Mat}_\C$ from the free symmetric category of pure circuits as complex matrices.
We assume there is a generating object for each finite-dimensional quantum system $\Sigma_0 = \{ \mathtt{qudit}(n) \}_{n > 1}$ and a box $g \in \Sigma_1$ for each pure quantum process, e.g. unitary gates, preparations (kets) and post-selected measurements (bras).
We say the gateset is universal when the interpretation $[\![-]\!] : F^S(\Sigma) \to \mathbf{Mat}_\C$ is full, i.e. every complex matrix is the interpretation of some pure circuit.

We then define an extended signature $cq(\Sigma) \supset \Sigma$ with objects:
$$cq(\Sigma)_0 = \{ \mathtt{digit}(n) \}_{n > 1} + \{ \mathtt{qudit}(n) \}_{n > 1}$$
for classical and quantum systems of each dimension, and boxes given by:
\begin{align*}
cq(\Sigma)_1 = \{ \hat{g} \}_{g \in \Sigma_1}
&+ \{ \mathtt{measure}(n) : \mathtt{qudit}(n) \to \mathtt{digit}(n) \}_{n > 1}\\
&+ \{ \mathtt{encode}(n) : \mathtt{digit}(n) \to \mathtt{qudit}(n) \}_{n > 1}
\end{align*}
Let $\mathbf{Circ} = F^S(cq(\Sigma))$ be the free symmetric category it generates, where the diagrams are called mixed quantum circuits.
The evaluation $[\![-]\!] : \mathbf{Circ} \to \mathbf{CQMap}$ is given by
$[\![\mathtt{digit}(n)]\!] = (n, 1)$, $[\![\mathtt{qudit}(n)]\!] = (1, n)$ and
$[\![\hat{g}]\!] = \mathtt{double}([\![g]\!])$.
The evaluation of any mixed circuit is always completely-positive~\cite[Corollary~8.6]{CoeckeKissinger17}.
Let $\mathbf{CausalCirc} \injects \mathbf{Circ}$ be the subcategory of causal processes, i.e. $\mathbf{CausalCirc} = F^S(cq(\Sigma'))$ for $\Sigma' \sub \Sigma$ the set of boxes that are interpreted as isometries.
If the gateset is universal, then the interpretation $\mathbf{CausalCirc} \to \mathbf{Channel}$ is full~\cite[Theorem~8.96]{CoeckeKissinger17}.
More explicitly, every quantum channel $f : (a, b) \to (c, d)$ can be written as:
$$f = \mathtt{encode}(a) \otimes b \otimes \mathtt{double}(\vert 0 \rangle_n)
\s \fcmp \s \mathtt{double}(g) \s \fcmp \s
\mathtt{measure}(c) \otimes d \otimes \mathtt{discard}(n)$$
for some $n$-dimensional ancilla $\vert 0 \rangle_n$ and an $(n a b) \times (n c d)$ unitary matrix $g$.
This gives us a general intuition for what it means to take a shot at a quantum circuit: 1) we prepare some qubits in the zero state, 2) we perform classically-controlled unitary gates, 3) we measure some of the qubits and discard the others.
Drawing digits and qudits as thin and thick wires, encode and measure as spiders, discard as three horizontal lines, we get the following diagram for a generic circuit:
\ctikzfig{img/qnlp/stinespring}

DisCoPy implements \py{Circuit} as a subclass of \py{Diagram} with objects generated by two families of objects \py{Digit(n)} and \py{Qudit(n)} indexed by natural numbers \py{n > 1}, where \py{bit = Ty(Digit(2))} and \py{qubit = Ty(Qudit(2))}.
The class \py{Gate} is a subclass of \py{Box} and \py{Circuit} with an attribute \py{array} which will define its interpretation as a pure circuit, i.e. a unitary matrix.
We also have boxes for \py{Ket} and its dagger \py{Bra}, \py{Measure} and its dagger \py{Encode}, \py{Discard} and its dagger \py{MixedState}.

\begin{python}
{\normalfont Implementation of the category $\mathbf{Circ}$ with \py{Digit} and \py{Qudit} as generating objects and \py{Circuit} as arrows.}

\begin{minted}{python}
class Digit(Ob):
    def __init__(self, n: int):
        self.n = n
        super().__init__(name="bit" if n == 2 else "Digit({})".format(n))

class Qudit(Ob):
    def __init__(self, n: int):
        self.n = n
        super().__init__(name="qubit" if n == 2 else "Qudit({})".format(n))

bit, qubit = Ty(Digit(2)), Ty(Qudit(2))

class Circuit(Diagram): pass

class Gate(Box, Circuit):
    def __init__(self, name: str, dom: Ty, cod: Ty,
                 array: list[list[complex]], is_dagger=False):
        self.array = array
        Box.__init__(self, name, dom, cod, is_dagger=is_dagger)

    def dagger(self) -> Gate: return Gate(
        self.name, self.cod, self.dom, self.array, is_dagger=not self.is_dagger)

class Bra(Box, Circuit):
    def __init__(self, *digits: int, base=2):
        self.digits, self.base = digits, base
        name = "Bra({}, base={})".format(', '.join(map(str, digits)), base)
        Box.__init__(self, name, qubit ** len(digits), qubit ** 0)

    def dagger(self) -> Ket: return Ket(*self.digits, base=self.base)

class Ket(Box, Circuit):
    def __init__(self, *digits: int, base=2):
        self.digits, self.base, name = digits, base
        name = "Ket({}, base={})".format(', '.join(map(str, digits)), base)
        Box.__init__(self, name, qubit ** 0, qubit ** len(digits))

    def dagger(self) -> Bra: return Bra(*self.digits, base=self.base)

class Encode(Box, Circuit):
    def __init__(self, dom=bit):
        obj, = dom.inside
        assert isinstance(obj, Digit)
        Box.__init__("Encode({})".format(n), dom, Ty(Qudit(obj.n)))

    def dagger(self) -> Measure: return Measure(self.cod)

class Measure(Box, Circuit):
    def __init__(self, dom=qubit):
        obj, = dom.inside
        assert isinstance(obj, Qudit)
        Box.__init__("Measure({})".format(n), dom, Ty(Digit(obj.n)))

    def dagger(self) -> Encode: return Encode(self.cod)

class Discard(Box, Circuit):
    def __init__(self, x: Ty):
        Box.__init__("Discard({})".format(x), x, Ty())

    def dagger(self) -> MixedState: return MixedState(self.dom)

class MixedState(Box, Circuit):
    def __init__(self, x: Ty):
        Box.__init__("MixedState({})".format(x), Ty(), x)

    def dagger(self) -> Discard: return Discard(self.cod)
\end{minted}
\end{python}

As discussed in example~\ref{example:pivotal-circuit}, the category $\mathbf{Circ}$ also has a dagger compact closed structure where the cups and caps for qudits are given by scaled Bell states and post-selected Bell measurements respectively.
The cups for digits are given by the result of measuring a scaled Bell state, or equivalently as an (unnormalised) correlated probability distributions, the caps can be thought of as classical post-selection on two digits being equal.
We can freely enrich $\mathbf{Circ}$ in commutative monoids and execute (simulate) a formal sum of circuits by executing (simulating) each circuit and adding up the results.
If we can take formal sums, there's no reason not to also take linear combinations of circuits.
Via the Born rule, we can already define positive real scalars as the evaluation of pure quantum gates on zero qubits, such as the $\sqrt 2$ scalars of our first example~\ref{example:circuit-alive-loves-bob}.
What we're missing are \emph{mixed scalars} which get applied after the Born rule.

\begin{python}\label{listing:mixed-scalars}
{\normalfont Implementation of pure and mixed scalars in $\mathbf{Circ}$.}

\begin{minted}{python}
class Sqrt(Gate):
    def __init__(self, x: float):
        super().__init__(
            "$\\sqrt {}$".format(x), Ty(), Ty(), array=[[math.sqrt(x)]])

class Scalar(Box, Circuit):
    def __init__(self, z: complex, is_pure=False):
        self.z, self.is_pure = z, is_pure
        Box.__init__(
            self, "Scalar({}, is_pure={})".format(z, is_pure), Ty(), Ty())
\end{minted}
\end{python}

\begin{python}
{\normalfont Implementation of the subcategory of $\mathbf{Circ}$ spanned by qubits as a compact closed category.}

\begin{minted}{python}
Digit.l = Digit.r = Qudit.l = Qudit.r = property(lambda self: self)

X, Y, Z, H = (Gate(name, qubit, qubit, array) for name, array in zip("XYZH", [
    [[0, 1], [1, 0]], [[0, -1j], [1j, 0]], [[1, 0], [0, -1]],
    [[1 / sqrt(2), 1 / sqrt(2)], [1 / sqrt(2), -1 / sqrt(2)]]]))
CX = Gate('CX', qubit ** 2, qubit ** 2, [[1, 0, 0, 0],
                                         [0, 1, 0, 0],
                                         [0, 0, 0, 1],
                                         [0, 0, 1, 0]])
@staticmethod
@nesting
def cups(left: Ty, right: Ty):
    if left == right == qubit: return Sqrt(2) @ Ket(0, 0) >> H @ qubit >> CX
    raise NotImplementedError

Circuit.cups, Circuit.caps = cups, lambda left, right: cups(left, right).dagger()
\end{minted}
\end{python}

\py{Circuit} comes with a Boolean property \py{is_pure} which defines the subcategory of pure circuits, i.e. that we can interpret as \py{Tensor}.
The \py{eval} method now comes with an optional Boolean argument \py{mixed}: if \py{not mixed} and \py{is_pure} we apply a \py{Tensor}-valued functor, otherwise a \py{Channel}-valued functor.
The optional \py{backend} argument allows to go from numerical simulation to quantum hardware: it translates a circuit diagram with only \py{Digit} outputs (i.e. all qudits have been measured or discarded) into a \py{pytket.Circuit} before executing it on a quantum device.
As of today, DisCoPy does not implement the execution of \py{Encode} on quantum hardware yet, this would require to decide which quantum gate to perform depending on the result of a previous measurement.
Although simulating qudit circuits is no harder than qubit simulation, standard quantum hardware can only execute qubit circuits so far.

\begin{python}
{\normalfont Implementation of the \py{Circuit.eval} method.}

\begin{minted}{python}
for cls in (Gate, Bra, Ket):                  setattr(cls, "is_pure", True)
for cls in (Encode, Measure, Discard, Mixed): setattr(cls, "is_pure", False)

Circuit.is_pure = property(lambda self:
    all(box.is_pure for box in self.boxes)
    and all(isinstance(obj, Qudit) for obj in (self.dom @ self.cod).inside))

class PureEval(Functor):
    ob = ar = {}
    dom, cod = Category(Ty, Circuit), Category(tuple[int, ...], Tensor[complex])

    def __call__(self, other):
        if isinstance(other, Qudit): return [other.n]
        if isinstance(other, Gate) and not other.is_dagger:
            return Tensor[complex](other.array, self(other.dom), self(other.cod))
        if isinstance(other, Bra): return self(other.dagger()).dagger()
        if isinstance(other, Ket):
            if not other.digits: return Tensor.id([])
            if len(other.digits) == 1:
                inside = [[i == other.digits[0] for i in range(other.base)]]
                return Tensor[complex](inside, [], [other.base])
            head, *tail = other.digits
            return self(Ket(head, base=other.base))\
                @ self(Ket(*tail, base=other.base))
        return super().__call__(other)

class MixedEval(Functor):
    ob = ar = {}
    dom, cod = Category(Ty, Circuit), Category(CQ, Channel)

    def __call__(self, other):
        if isinstance(other, Qudit): return Q([other.n])
        if isinstance(other, Digit): return C([other.n])
        if isinstance(other, Scalar): return Channel([[
            abs(other.z) ** 2 if other.is_pure else other.z]], CQ(), CQ())
        if isinstance(box, (Gate, Bra, Ket)):
            return Channel.double(PureEval()(box))
        if isinstance(box, Encode): return Channel.encode(self(box.dom))
        if isinstance(box, Measure): return Channel.measure(self(box.dom))
        if isinstance(box, Discard): return Channel.discard(self(box.dom))
        if isinstance(box, MixtedState): return self(box.dagger()).dagger()
        return super().__call__(other)

def eval(self, mixed=True, backend=None) -> Tensor | Channel:
    if backend is not None: ...  # Interface with pytket.
    return PureEval()(self) if not mixed and self.is_pure else MixedEval()(self)
\end{minted}
\end{python}

\begin{example}
We can simulate a Bell test experiment by applying the evaluation functor $\mathbf{Circ} \to \mathbf{Channel}$.

\begin{minted}{python}
circuit = Ket(0, 0) >> H @ qubit >> CX >> Measure() @ Measure()
circuit.draw()
\end{minted}
\ctikzfig{img/qnlp/bell-test}
\begin{minted}{python}
assert circuit.eval() == Channel([[.5, 0, 0, .5]], CQ(), C([2, 2]))
\end{minted}
\end{example}


\subsection{Simplifying QNLP models with snake removal}\label{subsection:snake-removal}

Previous sections have spelled out a first definition of QNLP models as monoidal functors $F : \G \to \mathbf{Circ}$.
The main limitation of this approach is the need for post-selection: for each cup in our pregroup reduction we need to post-select on the result of a Bell measurement, which requires to double the number of shots in order to measure accurately.
Another related problem is that we need to load all the word vectors on the quantum device at once, requiring a number of qubits proportional to the length of the sentence.
We may solve both issues at once using the \emph{snake removal} algorithm described in listing~\ref{listing:snake-removal}.
Instead of mapping the grammatical structure of the sentence directly onto the architecture of a quantum circuit, we will first rewrite the pregroup diagram to remove unnecessary snakes.
Indeed, from the rigid structure of $\mathbf{Circ}$ we have that the three-qubit circuit for a snake can be simplified to a one-qubit identity circuit.
This is an abstract way to reformulate the correctness of the post-selected teleportation protocol.
Thus, for each snake removed we are effectively reducing the number of required qubits by two, with half the amount of required post-selection.

Pregroup diagrams themselves have no snakes, only boxes for the dictionary entries followed by cups.
The \emph{autonomisation} procedure of Delpeuch~\cite{Delpeuch19} allows to make the snakes manifest by opening up the boxes and filling them with caps.
Abstractly, it is based on the construction of the free rigid category generated by a given monoidal category, i.e. an autonomisation\footnote
{Autonomous is a synonym of rigid, thus autonomisation could have been called \emph{rigidification} but this would be counterintuitive, since it makes a monoidal category more flexible.}
functor $A : \mathbf{MonCat} \to \mathbf{RigidCat}$ which is left adjoint to the functor $U : \mathbf{RigidCat} \to \mathbf{MonCat}$ which forgets adjoints, cups and caps.
The action of the functor $A$ on objects is simple: given a monoidal category $C \simeq F^M(\Sigma) / R$ presented by a monoidal signature $\Sigma$ and relations $R$, we take the quotient $A(C) = F^R(\Sigma) / R$ of the free rigid category generated by $\Sigma$, seen as a rigid signature with no adjoints.
More explicitly, we start from a monoidal category $C$ that has no adjoints, cups or caps, we end up with a larger category $A(C)$ where we have added them freely: the arrows of $A(C)$ are rigid diagrams with boxes and equations coming from $C$.

The embedding functor $E : C \to A(C)$, which maps every arrow in $C$ to itself in this larger context, is monoidal and faithful~\cite[Theorem~1]{Delpeuch19}: if $E(f) = E(g)$ for two arrows $f, g : x \to y$ in $C$, then we must have $f = g$ to begin with, we cannot prove more equalities by introducing snakes.
In fact, this embedding functor is also full~\cite[Theorem~2]{Delpeuch19}, the map $E : C(x, y) \to A(C)(x, y)$ is an isomorphism for all objects $x, y$ in $C$, its inverse is computed by snake removal.
In particular for a free monoidal category $C = F^M(\Sigma)$, we get the following corollary: for a rigid diagram where the domain and codomain of each box, and of the diagram itself, contain no adjoints, the normal form contains no cups or caps, i.e. all snakes have been removed.
Thus, autonomisation allows to define the semantics of a pregroup grammar $\G$ in a monoidal category $C$ that is not rigid, for example in cartesian or semicartesian categories like $\mathbf{Pyth}$ and $\mathbf{CausalCirc}$.
Indeed, we can now define a rigid functor $F : \G \to A(C)$, apply it to a grammatical sentence $f : w_1 \dots w_n \to s$ to get a rigid diagram $F(f) : 1 \to F(s)$ in $A(C)$.
From fullness we know all snakes can be removed until we get a concrete arrow $\mathtt{normal\_form}(F(f))$ in $C$: the meaning of the sentence $f$.
This autonomisation procedure is implemented in three steps:
\begin{enumerate}
    \item we define a rigid functor $\py{wiring} : \G \to A(F^M(\Sigma))$ from our pregroup grammar to the free rigid category generated by a monoidal signature $\Sigma$,
    \item we apply this functor to a \py{sentence: pregroup.Parsing} then take the normal form to get a monoidal diagram, i.e. without cups, caps or adjoints,
    \item we apply a monoidal functor $\py{G} : F^M(\Sigma) \to C$ into our semantic category $C$ to get the meaning of the sentence $\py{G(wiring(sentence).normal_form())}$.
\end{enumerate}

\begin{example}\label{example:autonomisation}
Let's remove the snakes from ``Alice loves Bob'' by applying a functor $\py{wiring} : \G \to A(F^M(\Sigma))$ for the monoidal signature $\Sigma$ given by boxes $1 \to n$ for ``Alice'' and ``Bob'' and $n \otimes n \to s$ for ``loves''.

\begin{minted}{python}
wiring = rigid.Functor(
    dom=Category(rigid.Ty, pregroup.Parsing),
    cod=Category(rigid.Ty, rigid.Diagram),
    ob=lambda x: x,
    ar={Alice: Box("Alice", Ty(), n), Bob: Box("Bob", Ty(), n),
        loves: Cap(n.r, n) @ Cap(n, n.l) >> n.r @ Box("loves", n @ n, s) @ n.l})

steps = sentence, wiring(sentence), wiring(sentence).normal_form()
drawing.equation(*steps, symbol='$\\mapsto$')
\end{minted}

\ctikzfig{img/qnlp/snake-removal}

The resulting diagram is indeed a monoidal diagram, i.e. we have removed the two snakes.
This means we can apply any monoidal functor to it.
For example, the following functor $G : F^M(\Sigma) \to \mathbf{Pyth}$ was introduced in \cite{FeliceEtAl20} in order to define functorial language games, it sends pregroup diagrams to Python functions that evaluate to themselves.

\begin{minted}{python}
G = monoidal.Functor(
    dom=Category(rigid.Ty, rigid.Diagram),
    cod=Category(tuple[type, ...], Function),
    ob={s: pregroup.Parsing, n: pregroup.Parsing},
    ar={Box("Alice", Ty(), n): lambda: Alice,
        Box("Bob", Ty(), n): lambda: Bob,
        Box("loves", n @ n, s): lambda f, g:
            f @ loves @ g >> Cup(n, n.r) @ s @ Cup(n.l, n)})

assert G(wiring(sentence).normal_form())() == sentence
\end{minted}

We can also apply a functor $G : F^M(\Sigma) \to \mathbf{Circ}$ to simplify the circuit of example~\ref{example:circuit-alive-loves-bob}.

\begin{minted}{python}
G = monoidal.Functor(
    dom=Category(rigid.Ty, rigid.Diagram),
    cod=Category(circuit.Ty, Circuit),
    ob={s: qubit ** 0, n: qubit},
    ar={Box("Alice", Ty(), n): Ket(0), Box("Bob", Ty(), n): Ket(1),
        Box("loves", n @ n, s): CX >> H @ sqrt2 @ X >> Bra(0, 0)})

drawing.equation(F(sentence), G(wiring(sentence).normal_form()))
\end{minted}

\ctikzfig{img/qnlp/snake-removed-circuit}

\begin{minted}{python}
assert F(sentence).eval() == G(wiring(sentence).normal_form()).eval()
\end{minted}
\end{example}

\begin{example}\label{example:autonomisation-who}
We can rewrite noun phrases with subject relative pronouns such as ``Alice who loves Bob'' using the following factorisation for ``who''.
We can then use the Frobenius anatomy of subject relative pronouns as given by Sadrzadeh et al.~\cite{SadrzadehEtAl13} to get a simplified quantum circuit.

\begin{minted}{python}
who = Word("who", n.r @ n @ s.l @ n)
phrase = Alice @ who @ loves @ Bob\
    >> Cup(n, n.r) @ n @ s.l @ Cup(n, n.r) @ s @ Cup(n.l, n) >> n @ Cup(s.l, s)

wiring.ar[who] = Cap(n.r, n)\
    >> n.r @ Box("who_1", n, x @ n)\
    >> n.r @ x @ Cap(s, s.l) @ n
    >> n.r @ Box("who_2", x @ s, n) @ s.l @ n

G.ob[x] = qubit
G.ar[Box("who_1", n, x @ n)] = H @ sqrt2 @ Ket(0) >> CX
G.ar[Box("who_2", x @ s, n)] = Circuit.id(qubit)

rewrite_steps = (
    phrase,
    wiring(phrase),
    wiring(phrase).normal_form(),
    G(wiring(phrase).normal_form()))
drawing.equation(rewrite_steps, symbol='$\\mapsto$')
\end{minted}
\ctikzfig{img/qnlp/snake-removal-who}
\end{example}

\begin{example}
We can always construct a trivial \py{wiring} functor which sends a dictionary entry to the disconnected diagram given by tensoring the iterated transpose of boxes with one wire.

\begin{minted}{python}
def trivial_ar(word: Word) -> rigid.Diagram:
    if len(word.cod) == 1:
        obj, = word.cod.inside
        if obj.z == 0: return
        return Box(word.name, word.cod) if obj.z == 0 else\
            trivial_ar(Word(word.name, word.cod.r)).transpose(left=True)\
            if obj.z < 0 else\
            trivial_ar(Word(word.name, word.cod.l)).transpose(left=False)
    return rigid.Diagram.tensor(*[
        trivial_ar(Word("{}_{}".format(word.name, i), x))
        for i, x in enumerate(word.cod)])

trivial = rigid.Functor(
    dom=Category(rigid.Ty, pregroup.Parsing),
    cod=Category(rigid.Ty, rigid.Diagram),
    ob=lambda x: x, ar=trivial_ar)

steps = sentence, trivial(sentence), trivial(sentence).normal_form()
drawing.equation(*steps, symbol='$\\mapsto$')
\end{minted}
\ctikzfig{img/qnlp/snake-removal-trivial}
\end{example}

The \py{wiring} functor is the key ingredient of the autonomisation recipe as given by Delpeuch~\cite{Delpeuch19}.
For each dictionary entry we need to find some monoidal boxes (i.e. without adjoints) from which we can bend the wires to get the desired pregroup type.
We do not know if a non-trivial autonomisation is possible for arbitrary pregroup types, for example object relative pronouns of type $n^r n n^{ll} s^l$.
In example~\ref{example:autonomisation}, wiring was straightforward because all types had the shape $(x^r) y (z^l)$ for a generating object $y \in X$ and some types $x, z \in X^\star$ which do not contain adjoints.
Let us call types of this shape \emph{dependency types}, we claim\footnote
{A sketch of this statement is available on the nLab~\cite{ToumiNLab21}, a detailed proof will appear in de Felice's thesis~\cite{Felice22}.} that a pregroup grammar with only dependency types is in fact a \emph{dependency grammar} (DG) as defined by Gaifman~\cite{Gaifman65}.
From such a dependency grammar $G = (V, X, D, s)$ with dictionary entries $e = (w, (x^r) y (z^l)) \in D$, we construct a monoidal signature $\Sigma_G$ with generating objects $X$ and boxes $f_e : x z \to y$ together with a rigid functor $\py{wiring} : \G \to A(F^M(\Sigma_G))$ which acts as $e \mapsto \ttcap(x^r) \otimes \ttcap(z) \fcmp x^r \otimes f_e \otimes z^l$.
The resulting normal form is a monoidal diagram where all boxes have exactly one output, thus it is isomorphic to a tree which is called a \emph{dependency tree}.

Dependency grammars are weakly equivalent to context-free grammars~\cite[Theorem~3.11]{Gaifman65} and they are in fact strongly equivalent to a subclass of CFGs characterised by a notion of finite degree~\cite[Theorem~3.10]{Gaifman65}.
This strong equivalence can be computed in logarithmic space, thus the parsing problem for DGs reduces to that of CFGs, it can be solved in polynomial time.
Hence, DGs generate the same languages as CFGs and pregroup grammars but they sit somewhere at the intersection of the two frameworks in terms of the grammatical structures they can generate.
We conjecture that a PG is strongly equivalent to a CFG if and only if its dictionary entries all have dependency types, i.e. DGs would be precisely characterised as the intersection of CFGs and PGs.
In addition to their theoretical interest, DGs have also been applied in practice for natural language processing at industrial scale with the spaCy library~\cite{HonnibalMontani17}.
Going from pregroup grammars to the more restricted dependency grammars, all grammatical structures become tree-shaped and can be interpreted in arbitrary monoidal categories.
In particular, we can now define \emph{causal} QNLP models as functors $F : \G \to A(\mathbf{CausalCirc})$ where the meaning of grammatical sentences is given by quantum circuits without post-selection.
Giving experimental support to this approach is an ongoing research project.


\subsection{DisCoCat models via knowledge graph embedding}\label{subsection:kge}

In the previous section, we have showed how to evaluate a given DisCoCat model on a quantum computer.
But where do we get this model from in the first place?
The first two DisCoCat papers~\cite{ClarkEtAl08,ClarkEtAl10} focused on the mathematical foundations for their models.
Assuming that the meaning for words was given, they showed how to compute the meaning for grammatical sentences.
Grefenstette and Sadrzadeh~\cite{GrefenstetteSadrzadeh11} gave a first implementation of a DisCoCat model for a simple pregroup grammar $G = (V, X, D, s)$ made of common nouns and transitive verbs.
Concretely, their vocabulary is given $V = E + R$ for some finite sets $E$ and $R$ which we may call \emph{entities} and (binary) \emph{relations}.
They take basic types $X = \{ s, n \}$ and dictionary $D = \{ (e, n) \}_{e \in E} \cup \{ (r, n^r s n^l) \}_{r \in R}$.
Thus, every grammatical sentence is of the form $f : x r y \to s$ for what we may call a \emph{triple} of subject-verb-object $(x, r, y) \in L(G) \simeq E \times R \times E$.
They define a DisCoCat model $F : \G \to \mathbf{Mat}_\R$ with $F(s) = 1$ and a hyper-parameter $F(n) = d \in \N$ using the following recipe:
\begin{enumerate}
\item extract a co-occurrence matrix from a corpus of text and compute a $d$-dimensional word vector $F(e) \in \R^d$ for each noun $e \in E$,
\item for each transitive verb $r \in R$, find the set $K_r \sub E \times E$ of all pairs of nouns $(x, y) \in K_r$ such that the sentence $x r y \in E \times R \times E$ occurs in the corpus, then define
$$F(r) = \sum_{(x, y) \in K_r} F(x) \otimes F(y)$$
\end{enumerate}
The meaning of a sentence $f : x r y \to s$ would then be given by the inner product $F(f) = \langle F(x) \otimes F(y) \vert F(r) \rangle$ which we can rewrite as
$$F(f) = \sum_{(x', y') \in K_r} \langle F(x) \vert F(x') \rangle \langle F(y) \vert F(y') \rangle$$
i.e. we take the sum of the similarities between our subject-object pair $(x, y)$ and the pairs $(x', y')$ that appeared in the corpus.
However, the task they aim to solve is \emph{word sense disambiguation} which they cast in terms of \emph{sentence similarity}, but if the meaning of sentences are given by scalars there is no meaningful way compare them.
For example, the verb ``draw'' can be synonymous to ``sketch'' but also to ``pull''.
Thus, they want their model to predict that ``Bob draws diagrams'' is more similar to ``Bob sketches diagrams'' than to ``Bob pulls diagrams''.
Using a spider trick that has been later formalised by Kartsaklis et al.~\cite{KartsaklisEtAl12}, they decide to replace inner product by element-wise multiplication.
This amounts to defining a new functor $F' : \G \to \mathbf{Mat}_\R$ by post-composing verb meanings with spiders, i.e.
$$
F'(n) = F(n) = d, \quad F'(s) = d^2, \quad F'(e) = F(e)
$$ $$
\text{and} \quad F'(r) \s = \s F(r) \ \fcmp \ \spider_{1, 2}(d) \otimes \spider_{1, 2}(d)
$$
Indeed, the element-wise multiplication of two vectors $u, v \in \R^d$ can be defined as $u \odot v = u \otimes v \fcmp \spider_{2, 1}(d)$, from which we get the desired meaning for the sentence:
\ctikzfig{img/qnlp/frobenius-trick}
Now that sentence meanings are given by $d^2$-dimensional vectors rather than scalars, we can compute their similarity with inner products and use this to solve the disambiguation task.

The key observation which motivated our previous dissertation~\cite{Toumi18a} as well as the subsequent articles \cite{CoeckeEtAl18a} and \cite{FeliceEtAl19} is that a corpus $K \sub E \times R \times E$ of such subject-verb-object sentences can be seen as the data for a \emph{knowledge graph}.
In this simplified setting, a DisCoCat model can be seen as an instance of \emph{knowledge graph embedding} (KGE) where we want to find a low-dimensional representation of entities and relations.
The interpretation of a sentence can be used for \emph{link prediction}, where we want generalising the corpus to unseen sentences.
By extending the grammar with the words ``who'' and ``whom'', we can use the same model to solve simple instances of question answering.
As discussed in section~\ref{section:anaphora}, extending the grammar with anaphora which we interpret as spiders then allows to answer any conjunctive query.

Taking this DisCoCat-KGE analogy in the other direction, we can interpret knowledge graph embeddings as DisCoCat models.
Take for example\footnote
{We focus on ComplEx, other examples of KGE as functors are treated in \cite[Section~2.6]{Felice22}.}
the ComplEx model of Trouillon et al.~\cite{TrouillonEtAl16,TrouillonEtAl17}, it is defined by a dimension $d \in \N$, a (normalised) complex vector $u_e \in \C^d$ for each entity $e \in E$ and a complex vector $v_r \in \C^d$ for each relation $r \in R$.
Let us pack this into a dependent pair $\theta \in \Theta = \coprod_{d \in \N} \C^{d (\vert E \vert + \vert R \vert)}$.
The interpretation of a triple $f : x r y \to s$ is given by the scoring function $T_\theta(f) = 2 \text{Re}(\langle u_x, v_r, u_y^\star \rangle)$, where $\langle u, v, w \rangle = \sum_{i \leq d} u_i v_i w_i$ is the tri-linear dot product, $u_y^\star$ is the element-wise conjugate, and $2 \text{Re}$ takes twice\footnote
{We scale the original definition by $2$ in order to avoid cluttering the diagrams with $\frac{1}{2}$ scalars.} the real part of a complex number.
We can reformulate this as the complex-valued DisCoCat model $T_\theta : \G \to \mathbf{Mat}_\C$ given by $T_\theta(s) = 1$, $T_\theta(n) = d^2$ and:
\ctikzfig{img/qnlp/trouillon-trick}
where we draw the conjugate of a vector as the horizontal reflection of its box.
We can use the same spider trick as above to rewrite the trilinear product $\langle u_x, v_r, u_y^\star \rangle$ in terms of post-composition with a three-legged spider.
We can also rewrite the real part of a scalar as half the sum with its conjugate $2 \text{Re}(z) = z + \bar z$, from which we get the desired meaning for sentences:
\ctikzfig{img/qnlp/trouillon-functor}
The meaning of a sentence $f : x r y \to s$ is given by a real scalar which we interpret as true when $T_\theta(f) \geq 0$.
In fact, any knowledge graph $K : E \times R \times E \to \{ \pm 1 \}$ can be written as $K = T_\theta \fcmp \mathtt{sign}$ for the function $\mathtt{sign} : \R \to \{ \pm 1 \}$~\cite[Theorem~4]{TrouillonEtAl17}.
Furthermore, the dimension $d \in \N$ of the model can be bounded by the \emph{sign-rank} of the matrices for each relation~\cite[Proposition~2.5.17]{Felice22}, with theoretical guarantees that $d \ll \vert E \vert$ if the problem is learnable efficiently~\cite{AlonEtAl16a}.

Hence, the knowledge graph embedding can be defined as a space of parameters $\Theta$ together with a function $T_- : \Theta \to [\G, \mathbf{Mat}_\C]$ which sends parameters $\theta \in \Theta$ to functors $T_\theta : \G \to \mathbf{Mat}_\C$.
Now if we are given a training set $\Omega \sub E \times R \times E$ annotated by $Y : \Omega \to \{ \pm 1 \}$ for whether each triple belongs to the knowledge graph\footnote
{When the dataset contains only positive triples, it is common use the \emph{local closed world assumption} where we randomly change either the subject or object to generate negative triples.
This can be improved by \emph{adversarial sampling} methods~\cite{CaiWang18}, akin to \emph{generative adversarial networks} where we train a model to generate hard negative examples.
In \cite{FeliceEtAl20} we investigate how this can be formalised in terms of \emph{functorial language games}.},
we want to find the parameters that best approximate the data:
$$
\theta^\star \s = \s \argmin_{\theta \in \Theta} \
\lambda \lVert \theta \rVert + \sum_{f \in \Omega} \mathtt{loss}(T_\theta(f), Y(f))
$$
where $\lVert \theta \rVert$ is a choice of norm (usually $L^2$) scaled by some regularisation hyper-parameter $\lambda \geq 0$ and $\mathtt{loss} : \R \times \R \to \R^+$ is a choice of loss function, usually the negative log-likelihood of the logistic model $\mathtt{loss}(y, y') = \log (1 + \exp(-yy'))$.
If we fix the dimension $d \in \N$ (i.e. we take it as a hyper-parameter) we can use stochastic gradient descent to compute $\theta^\star$ and hence the optimal model $T_{\theta^\star} : \G \to \mathbf{Mat}_\C$.
Using $T_{\theta^\star}$ to predict the value of triples $f \in (E \times R \times E) - \Omega$ not seen during training, Trouillon et al.~\cite{TrouillonEtAl16} obtained state-of-the-art results on standard benchmarks with both fewer parameters and a lower time complexity than competing models.
Again if we extend the grammar with question words and anaphora, we can answer not only Boolean questions but any conjunctive query.

In the Grefenstette-Sadrzadeh implementation of DisCoCat models, we had to compute the meanings for nouns by some other means (e.g. from a co-occurrence matrix) then use a knowledge graph to lift these to the meaning for verbs.
On the other hand, our KGE approach computes the meanings for all words simultaneously, using only samples from the knowledge graph as training data.
Indeed, once reformulated as a DisCoCat model, training a knowledge graph embedding amounts to \emph{learning a functor} $F : \G \to \mathbf{Mat}_\C$ given only access to pairs $(f, a)$ such that $F(f) = a$.
Generalising this from subject-verb-object sentences to arbitrary grammars, Koziell-Pipe and the author~\cite{ToumiKoziell-Pipe21} coined the term \emph{functorial learning} for this approach to structured machine learning where we want to learn structure-preserving functors from data.
We show how the approach can be extended, from a supervised learning task such as link prediction and question answering to unsupervised \emph{functorial language models}, where we use a functor together with a probabilistic grammar to compute the probability of a missing word given its context.
Together with Clark's recent progress in using transformer models for parsing~\cite{Clark21}, functorial language models open the door to training large scale DisCoCat models end-to-end, i.e. from raw text to functor.

Functorial learning is part of the blooming field at the intersection of machine learning and category theory which is surveyed by Shiebler et al.~\cite{ShieblerEtAl21}.
A starting point was the work of Fong et al.~\cite{FongEtAl19} on characterising \emph{backpropagation as functor} into a category of learners, which was later formalised in terms of \emph{parameterised lenses}~\cite[Lemma~2.13]{CruttwellEtAl21}.
The idea of learning functors via gradient descent first appeared in the work of Gavranović~\cite{Gavranovic19,Gavranovic19a}, where the cyclic generative adversarial networks (cycleGAN) of Zhu et al.~\cite{ZhuEtAl20} are reformulated as functors from a finitely presented category into a category of neural networks.
However, there is some ambiguity in what it means exactly to learn a functor: in our approach, arrows in the domain category (i.e. diagrams generated by the grammar) encode the training data, whereas for Gavranović they encode the architecture of a neural network.
In both cases however, we start from existing machine learning algorithms (ComplEx or cycleGAN), reformulate them in terms of functors before we can generalise them.
In some cases, category theory does provide genuinely new learning algorithms, one example is the \emph{reverse derivative ascent} of Wilson and Zanasi~\cite{WilsonZanasi20} where the notion of \emph{reverse derivative category} is used to generalise gradient descent to Boolean functions.


\subsection{Variational quantum question answering}\label{subsection:vqqa}

The previous section introduced the idea of using gradient descent to learn DisCoCat models $F : \G \to \mathbf{Mat}_\C$ from data, we now discuss how to apply the same functorial learning approach in order to learn QNLP models $F : \G \to \mathbf{Circ}$.
There are two main challenges: 1) we need a way to load our model onto a quantum machine, i.e. encode word embeddings as parameterised quantum circuits, 2) we need a way to train the model, i.e. to compute the optimal parameters in some data-driven task.
As we mentioned in section~\ref{section:discocat-qnlp}, we could solve the first challenge by post-composing a classical model $F : \G \to \mathbf{Mat}_\C$ with any choice of encoding $\mathtt{load} : \mathbf{Mat}_\C \to \mathbf{Circ}$ to get a QNLP model $F \fcmp \mathtt{load} : \G \to \mathbf{Circ}$.
However, this would require circuits with depth exponential in the number of qubits, which are out of reach for the NISQ computers available today.
Instead of learning classical vectors of parameters that we then encode as circuits, we introduce a \emph{variational} quantum algorithm where we learn the parameters of quantum circuits directly.
As for the second challenge, we cannot hope to backpropagate gradients through a quantum circuit in the same way as for classical neural networks.
In this section we pick the easiest alternative: we treat our QNLP model as a black box and use a noisy optimisation algorithm such as the stochastic perturbation stochastic approximation (SPSA) of Spall~\cite{Spall98}.
In the next section, we will open the black box and introduce \emph{diagrammatic differentiation} in order to use the quantum circuits to compute their own gradients.

Our variational algorithm may be summarised in the following recipe for a parameterised circuit-valued functor.
\begin{enumerate}
\item Fix a pregroup grammar $G = (V, X, D, s)$ and use it to parse a dataset $\Omega \sub \coprod_{w_1 \dots w_n \in V^\star} \G(w_1 \dots w_n, s)$ of sentences then annotate them with truth values $Y : \Omega \to \R$.
\item Fix a hyper-parameter $F(x) \in \N$ for the number of qubits representing each basic type $x \in X$.
For simplicity, we will assume that $F(s) = 0$ so that sentences are represented as closed circuits.
\item Choose an \emph{ansatz} for each dictionary entry $(w, t) \in D$, i.e. a function $F_-(w, t) : \Theta_{(w, t)} \to \mathbf{Circ}(0, F(t))$ from some space of parameters $\Theta_{(w, t)}$ to the set of $F(t)$-qubit circuits.
Ideally, we want an ansatz that is shallow enough to run on NISQ machines but still hard to approximate classically, such as the instantaneous quantum polynomials (IQP) of Shepherd and Bremner~\cite{ShepherdBremner09}.
A more practical choice is to use a \emph{hardware-efficient} ansatz such as that of Kandala et al.~\cite{KandalaEtAl17}, where we essentially squeeze as many parameters as possible out of whatever hardware we can get our hands on.
\end{enumerate}
As in the previous section, we can abstract these choices away into one big parameter space $\Theta = \prod_{(w, t) \in D} \Theta_{(w, t)}$ and a function $F_- : \Theta \to [\G, \mathbf{Circ}]$ from parameters to functors.
We can now use SPSA (or any noisy optimisation algorithm of our choice) to approximate the optimal parameters:
$$\theta^\star \s = \argmin_{\theta \in \Theta}
\lambda \lVert \theta \rVert + \sum_{f \in \Omega} \mathtt{loss}(\mathtt{eval}(F(f)), Y(f))$$
where $\mathtt{eval} : \mathbf{Circ}(0, 0) \to \R$ takes closed circuits and returns the result of evaluating them either using a classical simulation or a quantum device.

We can now evaluate the optimal functor $F_{\theta^\star} : \G \to \mathbf{Circ}$ on unseen sentences, or equivalently answer Boolean questions.
This variational approach to question answering was first demonstrated in a classical simulation by Ma et al.~\cite{MaEtAl19} in the restricted case of knowledge graphs, i.e. subject-verb-object sentences.
Because the circuits were simulated classically it was possible to use standard gradient descent, with results comparable to the state-of-the-art.
The first NLP experiment on quantum hardware to appear in print was \cite{MeichanetzidisEtAl20}, where we used functorial learning to solve a toy question-answering task on a Shakespeare-inspired dataset.
We used SPSA for the optimisation and the snake removal functor defined in example~\ref{example:autonomisation-who} to simplify the circuits of sentences with relative pronouns like ``Romeo who loves Juliet dies''.
Although this first experiment answered only yes-no questions, the same framework can be applied to \emph{wh}-questions, with Grover's algorithm yielding a quadratic speedup~\cite{CorreiaEtAl22}.
This functorial learning pipeline was then applied on a larger dataset by Lorenz et al.~\cite{LorenzEtAl21}, demonstrating the convergence of the model and its statistical significance over a random baseline.
The pipeline was packaged into its own Python library, \py{lambeq}~\cite{KartsaklisEtAl21}, which builds upon DisCoPy and state-of-the art parsers.
It has also been adapted from question-answering to machine translation~\cite{AbbaszadeEtAl21,VicenteNieto21}, word-sense disambiguation~\cite{Hoffmann21} and even to generative music~\cite{MirandaEtAl21}.


\section{Diagrammatic differentiation}\label{section:diag-diff}

In this section, we introduce \emph{diagrammatic differentiation}: a graphical notation for computing the derivatives of parameterised diagrams.
On the theoretical side, we generalise the \emph{dual number construction} from rigs to monoidal categories (section~\ref{2-dual-diagrams}).
We then apply this construction to the category of ZX diagrams (section~\ref{2b-differentiating-zx}) and of quantum circuits (section~\ref{3-dual-circuits}).
In section~\ref{4-bubbles} we define the gradient of diagrams with bubbles in terms of the chain rule.
We use this to differentiate quantum circuits with neural networks as classical post-processing.
The theory comes with a DisCoPy implementation, gradients of classical-quantum circuits can then
be simplified using the PyZX library~\cite{KissingerVanDeWetering19}, compiled and executed
on quantum hardware with t$\vert$ket$\rangle$~\cite{SivarajahEtAl20}.

\subsection*{Related Work}

Penrose and Rindler~\cite{PenroseRindler84} used string diagrams to describe the geometry of space-time and an extra piece of notation is introduced: the covariant derivative is represented as a bubble around the tensor to be differentiated.
The same bubble notation for vector calculus has been proposed by Kim et al.~\cite{KimEtAl20}, but they have mainly pedagogical motivations and restrict themselves to the case of three-dimensional Euclidean space.
To the best of our knowledge, our definition is the first formal account of string diagrams with bubbles for derivatives.

Blute et al.~\cite{BluteEtAl06} axiomatised the notion of derivative with \emph{differential categories}.
More recently Cockett et al.~\cite{CockettEtAl19} generalised the notion of back-propagation with \emph{reverse derivative categories}, which have been proposed as a categorical foundation for gradient-based learning~\cite{CruttwellEtAl21}.
These frameworks all define the derivative of a morphism with respect to its domain.
In our setup however, we define the derivative of parametrised morphisms with respect to parameters that are in some sense external to the category.
Investigating the relationship between these two definitions is left to future work.

The work presented in this section has its origin in Yeung's dissertation~\cite{Yeung20}, which focused on the diagrammatic differentiation of ansätze used in quantum machine learning.
In joint work with Yeung and de Felice~\cite{ToumiEtAl21}, we gave this approach both its theoretical basis and its first implementation.
Wang and Yeung~\cite{WangYeung22} later generalised it from differentiation to the integration of ZX diagrams.
They also resolve a technical limitation of our approach: they show how to express any formal sum of ZX diagrams in terms of one diagram.
The same idea appeared independently in Jeandel et al.~\cite{JeandelEtAl22}, who also show how to apply our diagrammatic differentiation method to the Ising Hamiltonian, a key ingredient to many variational quantum algorithms~\cite{Hadfield21}.


\subsection{Dual diagrams}\label{2-dual-diagrams}

Dual numbers were first introduced by Clifford~\cite{Clifford73}, they are a fundamental tool for \emph{automatic differentiation}~\cite{Hoffmann16}, i.e. they allow to compute the derivative of a function automatically from its definition.

Given a commutative rig $\S$, the rig of dual numbers $\D[\S]$ extends $\S$ by adjoining a new element $\epsilon$ such that $\epsilon^2 = 0$.
Abstractly, $\D[\S] = \S[x] / x^2$ is a quotient of the rig of polynomials with coefficients in $\S$.
Concretely, elements of $\D[\S]$ are formal sums $s + s' \epsilon$ where $s$ and $s'$ are scalars in $\S$.
We write $\pi_0, \pi_1 : \D[\S] \to \S$
for the projection on the real and epsilon component respectively.
Addition and multiplication of dual numbers are given by:
\begin{align} \begin{split}\label{linearity}
(a + a' \ \epsilon ) + (b + b' \ \epsilon)
\quad &= \quad (a + b) \s + \s (a + b') \ \epsilon
\end{split}\\
\begin{split}\label{product-rule}
(a + a' \ \epsilon ) \times (b + b' \ \epsilon)
\quad &= \quad (a \times b) \s + \s (a \times b' \ + \ a' \times b) \ \epsilon
\end{split}
\end{align}

A related notion is that of \emph{differential rig}: a rig $\S$ equipped with a derivation, i.e. a map $\partial : \S \to \S$ which preserves sums and satisfies the Leibniz product rule
$\partial(f \times g) = f \times \partial(g) + \partial(f) \times g$ for all $f, g \in \S$.
An equivalent condition is that the map $f \mapsto f + (\partial f) \epsilon$ is a homomorphism of rigs $\S \to \D[\S]$.
The correspondance also works the other way around: given a homorphism $\partial : \S \to \D[\S]$ such that $\pi_0 \circ \partial = \id_\S$, projecting on the epsilon component is a derivation $\pi_1 \circ \partial : \S \to \S$.
The motivating example is the rig of smooth functions $\S = \R \to \R$, where differentiation is a derivation.
Concretely, we can extend any smooth function $f : \R \to \R$ to a function $f : \D[\R] \to \D[\R]$ over the dual numbers defined by:
\begin{equation}\label{dual-numbers-eq}
f(a + a' \epsilon) \quad = \quad f(a) \s + \s a' \times (\partial f)(a) \epsilon
\end{equation}

We can use equations~\ref{linearity}, \ref{product-rule} and \ref{dual-numbers-eq} to derive the usual rules for gradients in terms of dual numbers.
For the identity function we have $\id(a + a' \epsilon) = \id(a) + a' \epsilon$, i.e. $\partial \id = 1$.
For the constant functions we have $c(a + a' \epsilon) = c(a) + 0 \epsilon$, i.e. $\partial c = 0$.
For addition, multiplication and composition of functions, we can derive the following \emph{linearity}, \emph{product} and \emph{chain} rules:
\begin{align} \begin{split}
    (f + g)(a + a' \epsilon)
    \s &= \s (f + g)(a) \s + \s a' \times (\partial f + \partial g)(a) \epsilon
\end{split}\\ \begin{split}
    (f \times g)(a + a' \epsilon)
    \s &= \s (f \times g)(a) \s + \s a' \times (f \times \partial g \ + \ \partial f \times g)(a) \epsilon
\end{split}\\ \begin{split}
    (f \circ g)(a + a' \epsilon)
    \s &= \s (f \circ g)(a) \s + \s a' \times (\partial g \ \times \ \partial f \circ g)(a) \epsilon
\end{split} \end{align}

This generalises to smooth functions $\R^n \to \R^m$, where the partial derivative $\partial_i$ is a derivation for each $i < n$.
The functions $\F_2^n \to \F_2^m$ on the two-element field $\F_2$ with elementwise XOR as sum and conjunction as product also forms a differential rig.
The partial derivative is given by $(\partial_i f)(\vec{x}) = f(\vec{x}_{[x_i \mapsto 0]}) \oplus f(\vec{x}_{[x_i \mapsto 1]})$.
Intuitively, the $\F_2$ gradient $\partial_i f(\vec{x}) \in \F_2^m$ encodes which coordinates of $f(\vec{x})$ actually depend on the input $x_i$.
An example of differential rig that isn't also a ring is given by the set $\N[X]$ of polynomials with natural number coefficients, again each partial derivative is a derivation.

A more exotic example is the rig of Boolean functions with elementwise disjunction as sum and conjunction as product.
Boolean functions $\B^n \to \B^m$ can be represented as tuples of $m$ propositional formulae over $n$ variables.
The partial derivative $\partial_i$ for $i < n$ is defined by induction over the formulae:
for variables we have $\partial_i x_j = \delta_{ij}$, for constants $\partial_i 0 = \partial_i 1 = 0$ and for negation $\partial_i \neg \phi = \neg \partial_i \phi$. The derivative of disjunctions and conjunctions are given by the linerarity and product rules.
Equivalently, the gradient of a propositional formula can be given by $\partial_i \phi = \neg \phi_{[x_i \mapsto 0]} \land \phi_{[x_i \mapsto 1]}$.
Concretely, a model satisfies $\partial_i \phi$ if and only if it satisfies $\phi \leftrightarrow x_i$: the derivative is true when the variable and the formula are positively correlated.
Substituting $x_i$ with its negation, we get that a model satisfies $\partial_i \phi_{[x_i \mapsto \neg x_i]}$ if and only if it satisfies $\phi \leftrightarrow \neg x_i$, i.e. iff variable and formula are anti-correlated.
Note that although $\B$ and $\F_2$ are isomorphic as sets, they are distinct rigs.
Their derivations are related however by $\partial^{\F_2}_i f \mapsto \partial^\B_i \phi \lor \partial^\B_i \phi_{[x_i \mapsto \neg x_i]}$ for $\phi : \B^n \to \B$ the formula corresponding to the function $f : \F_2^n \to \F_2$.
That is, a Boolean function depends on an input variable precisely when either the corresponding formula is positively correlated or anti-correlated.

Our main technical contribution is to generalise dual numbers and derivations from rigs to monoidal categories with sums.
Given a monoidal category $C$ with sums (i.e. enriched in commutative monoids), we define the category $\D[C]$ by adjoining a scalar (i.e. an endomorphism of the monoidal unit) $\epsilon$ and quotienting by\footnote{
In the case when $C$ is not braided we also require the axiom $\epsilon \otimes f = f \otimes \epsilon$.
} $\epsilon \otimes \epsilon = 0$.
Concretely, the objects of $\D[C]$ are the same as those of $C$, the arrows
are given by formal sums $f + f' \epsilon$ of parallel arrows $f, f' \in C$.
Composition and tensor are both given by the product rule:
\begin{align}
    (f + f' \epsilon) \ \fcmp \ (g + g' \epsilon)
    &\s = \s f \ \fcmp \ g \s + \s (f' \fcmp g \ + \ f \fcmp g') \ \epsilon\\
    (f + f' \epsilon) \otimes (g + g' \epsilon)
    &\s = \s f \otimes g \s + \s (f' \otimes g \ + \ f \otimes g') \ \epsilon
\end{align}

We say that a unary operator on homsets $\partial : \coprod_{x,y} C(x, y) \to C(x, y)$ is a derivation whenever it satisfies the product rules for both composition
$\partial (f \fcmp g) = (\partial f) \fcmp g + f \fcmp (\partial g)$ and tensor
$\partial (f \otimes g) = (\partial f) \otimes g + f \otimes (\partial g)$.
An equivalent condition is that the map $f \mapsto f + (\partial f) \epsilon$ is a sum-preserving monoidal functor $C \to \D[C]$.
Again, the correspondance between dual numbers and derivations works the other way around: given a sum-preserving monoidal functor $\partial : C \to \D[C]$ such that
$\pi_0 \circ \partial = \id_{C}$, projecting on the epsilon component gives a derivation $\pi_1 \circ \partial : \coprod_{x,y} C(x, y) \to C(x, y)$.
The following propositions characterise the derivations on the category of matrices valued in a commutative rig $\S$.

\begin{proposition}
Dual matrices are matrices of dual numbers, i.e. we have $\D[\mathbf{Mat}_\S] \simeq \mathbf{Mat}_{\D[\S]}$ for all commutative rigs $\S$.
\end{proposition}

\begin{proof}
The isomorphism is given by
$$\big( \sum_{ij} f_{ij} \ket{j} \bra{i} \big)
\ + \ \big( f'_{ij} \sum_{ij} \ket{j}  \bra{i} \big) \epsilon
\s \longleftrightarrow \s
\sum_{ij} (f_{ij} + f'_{ij} \epsilon) \ket{j} \bra{i}$$
\end{proof}

\begin{proposition}
Derivations on $\mathbf{Mat}_\S$ are in one-to-one correspondance with
derivations on $\S$.
\end{proposition}

\begin{proof}
A derivation on $\mathbf{Mat}_\S$ is uniquely determined by its action on
scalars in $\S$. Conversely, applying a derivation $\partial : \S \to \S$
entrywise on matrices yields a derivation on $\mathbf{Mat}_\S$.
\end{proof}

DisCoPy implements parameterised matrices with SymPy~\cite{MeurerEtAl17} expressions as entries.
The method \py{Tensor.grad} takes a SymPy variable and applies element-wise symbolic differentiation.

\begin{python}
{\normalfont Implementation of gradients of parameterised tensors with SymPy.}

\begin{minted}{python}
def grad(self: Tensor[sympy.Expr], var: sympy.Symbol):
    return self.map(lambda x: x.diff(var))

Tensor.grad = grad
\end{minted}
\end{python}

\begin{example}
We can check the product rule for tensor and composition of matrices.

\begin{minted}{python}
x = sympy.Symbol('x')
f = Tensor([[x + 1, 2 * x], [x ** 2, 1 / x     ]], [2], [2])
g = Tensor([[1,     2],     [2 * x, -1 / x ** 2]], [2], [2])

assert f.grad(x) == g
assert (f @ g).grad(x)  == f.grad(x) @ g  +  f @ g.grad(x)
assert (f >> g).grad(x) == f.grad(x) >> g + f >> g.grad(x)
\end{minted}
\end{example}

Fix a monoidal signature $\Sigma$ and let $C_\Sigma^+$ be the free monoidal category with sums that it generates, i.e. arrows are formal sums of diagrams as defined in sections~\ref{subsection:dagger-sums-bubbles} and \ref{subsection:monoidal-daggers-sums-bubbles}.
We assume our diagrams are interpreted as matrices, i.e. we fix a sum-preserving monoidal functor $[\![-]\!]  : C_\Sigma^+ \to \mathbf{Mat}_\S$ for $\S$ a commutative rig with a derivation $\partial : \S \to \S$.
Our main two examples are the ZX-calculus of Coecke and Duncan~\cite{CoeckeDuncan08} with smooth functions $\R^n \to \R$ as phases and the algebraic ZX-calculus over $\S$, introduced by Wang~\cite{Wang20}.
Applying the dual number construction to $C_\Sigma^+$, we get the category of \emph{dual diagrams} $\D[C_\Sigma^+]$ which is where diagrammatic differentiation happens.
By the universal property of $C_\Sigma^+$, every derivation $\partial : C_\Sigma^+ \to \D[C_\Sigma^+]$ is uniquely determined by its image on the generating boxes in $\Sigma_1$.
Intuitively, if we're given the derivative for each box, we can compute the derivative for every sum of diagram using the product rule.

We say that the interpretation $[\![-]\!] : C_\Sigma^+ \to \mathbf{Mat}_\S$ \emph{admits diagrammatic differentiation} if there is a derivation $\partial$ on $C_\Sigma^+$ such that $[\![-]\!] \circ \partial = \partial \circ [\![-]\!]$.
That is, the interpretation of the gradient $[\![\partial d]\!]$ coincides with the gradient of the interpretation $\partial [\![d]\!]$ for all sums of diagrams $d \in C_\Sigma^+$.
We depict the gradient $\partial d$ as a bubble surrounding the diagram $d$, as discussed in sections~\ref{subsection:dagger-sums-bubbles} and \ref{subsection:monoidal-daggers-sums-bubbles}.
Once translated to string diagrams, the axioms for derivations on monoidal
categories with sums become:
\ctikzfig{img/diag-diff/2-1a-product-rule}
\ctikzfig{img/diag-diff/2-1b-product-rule}

We implement dual diagrams with a method \py{Diagram.grad} which takes SymPy variables and returns formal sums of diagrams by applying the product rules for tensor and composition.
By default, we assume that boxes are constant, i.e. their gradient is the empty sum.

\begin{python}
{\normalfont Implementation of dual diagrams.}

\begin{minted}{python}
Box.grad = lambda self, var: Sum([], self.dom, self.cod)

def grad(self: Diagram, var: sympy.Symbol):
    if len(self) == 0: return Sum([], self.dom, self.cod)
    left, box, right = self.layers[0]
    return left @ box.grad(var) @ right >> self[1:]\
        + left @ box @ right >> self[1:].grad(var)

Diagram.grad = grad
\end{minted}
\end{python}

\begin{example}
We can override the default \py{grad} method and check the product rule for diagrams.

\begin{minted}{python}
class DataBox(Box):
    def __init__(self, name: str, dom: Ty, cod: Ty, data: sympy.Expr):
        self.data = data
        super().__init__(name, dom, cod)

    def __eq__(self, other):
        if not isinstance(other, DataBox): return super().__eq__(other)
        return super().__eq__(other) and self.data == other.data

    def grad(self, var):
        return DataBox(self.name, self.dom, self.cod, self.data.diff(var))

phi = sympy.Symbol('\\phi')
x, y, z = map(Ty, "xyz")
f, g = DataBox('f', x, y, phi ** 2), DataBox('g', y, z, 1 / phi)

assert (f @ g).grad(phi) == f.grad(phi) @ g + f @ g.grad(phi)
assert (f >> g).grad(phi) == f.grad(phi) >> g + f >> g.grad(phi)
\end{minted}
\end{example}


\subsection{Differentiating the ZX-calculus}\label{2b-differentiating-zx}

This section applies the dual number construction to the diagrams of the ZX-calculus with smooth functions $\alpha : \R^n \to \R$ as phases.
For each number of variables $n \in \N$, we define $\mathbf{ZX}_n$ as the free symmetric category generated by the signature:
$$\Sigma_0 = \{ x \} \s \text{and} \s \Sigma_1 = \{ H : x \to x \} + \{ Z^{m, n}(\alpha) : x^{\otimes m} \to x^{\otimes n} \ \vert \ m, n \in \N, \alpha : \R^n \to \R \}$$
where $H$ is depicted as a yellow box and $Z^{m, n}(\alpha)$ as a green spider.
The red spider is syntactic sugar for a green spider with yellow boxes connected to each leg.
The interpretation $[\![-]\!]  : \mathbf{ZX}_n \to \mathbf{Mat}_\S$ in matrices over $\S = \R^n \to \C$ is given by on objects by $[\![ x ]\!] = 2$ and on arrows by
$[\![ H ]\!] = \frac{1}{\sqrt{2}} \big(\ket{0}\bra{0} + \ket{0}\bra{1} + \ket{1}\bra{0} - \ket{1}\bra{1}\big)$
and $[\![Z^{m, n}(\alpha)]\!] =
e^{-i \alpha / 2} \ket{0}^{\otimes n} \bra{0}^{\otimes m}
+ e^{i \alpha / 2} \ket{1}^{\otimes n} \bra{1}^{\otimes m}$.
Note that we've scaled the standard interpretation of the green spider by a global phase to match the usual definition of rotation gates in quantum circuits.
For $n = 0$ we get $\mathbf{ZX}_0 = \mathbf{ZX}$ the ZX-calculus with no parameters.
By currying, any ZX diagram $d \in \mathbf{ZX}_n$ can be seen as a function $d : \R^n \to \text{Ar}(\mathbf{ZX})$ such that $[\![-]\!] \circ d : \R^n \to \mathbf{Mat}_\C$ is smooth.

\begin{lemma}\label{lemma-scalars}
A function $s : \R^n \to \C$ can be drawn as a scalar diagram in $\mathbf{ZX}_n$ if and only if it is bounded.
\end{lemma}

\begin{proof}
Generalising \cite[P.~8.101]{CoeckeKissinger17} to parametrised scalars, if there is a $k \in \N$ with $\vert s(\theta) \vert \leq 2^k$ for all $\theta \in \R^n$ then there are parametrised phases $\alpha, \beta : \R^n \to \R$ such that

\ctikzfig{img/diag-diff/2-2-bounded-lemma}

In the other direction, take any scalar diagram $d$ in $\mathbf{ZX}_n$.
Let $k$ be the number of spider in the diagram and $l$ the maximum number
of legs. By decomposing each spider as a sum of two disconnected diagrams,
we can write $d$ as a sum of $2^k$ diagrams. Each term of the sum is a product
of at most $\frac{1}{2} \times k \times l$ bone-shaped scalars. Each bone is
bounded by $2$, thus $[\![d]\!] : \R^n \to \C$ is bounded by $2^{k \times l}$.
\end{proof}

\begin{lemma}\label{lemma-rotations}
In $\mathbf{ZX}_n$, we have $\input{img/diag-diff/2-3a-lemma-rotation.tikzstyles}$
for all affine $\alpha : \R^n \to \R$.
\end{lemma}

\begin{proof}
Because $\alpha$ is affine we have that $\partial \alpha$ is constant, hence bounded and from lemma~\ref{lemma-scalars} we know it can be drawn in $\mathbf{ZX}_n$.
\begin{align*}
\partial [\![ Z(\alpha) ]\!]
&= \partial \big( e^{-i \alpha / 2} \ket{0}+ e^{i \alpha / 2} \ket{1}\big)\\
&= \frac{i\partial\alpha}{2}\big(-e^{-i \alpha / 2} \ket{0} + e^{i\alpha / 2} \ket{1}\big)\\
&= \frac{\partial\alpha}{2}\big(e^{-i\frac{\alpha+\pi}{2}} \ket{0} + e^{i\frac{\alpha+\pi}{2}} \ket{1}\big)
\end{align*}
\end{proof}

\begin{theorem}\label{theorem-zx-diag-diff}
The ZX-calculus with affine maps $\R^n \to \R$ as phases admits diagrammatic
differentiation.
\end{theorem}

\begin{proof}
The Hadamard $H$ has derivative zero.
For the green spiders, we can extend lemma~\ref{lemma-rotations} from
single qubit rotations to arbitrary many legs using spider fusion:
\ctikzfig{img/diag-diff/2-4-zx-theorem}
\end{proof}

Note that there is no diagrammatic differentiation for the ZX-calculus with
smooth maps as phases, even when restricted to bounded functions.
Take for example $\alpha : \R \to \R$ with $\alpha(\theta) = \sin \theta^2$,
it is smooth and bounded by $1$ but its derivative $\partial \alpha$ is
unbounded.
Thus, from lemma~\ref{lemma-scalars} we know it cannot be represented as a
scalar diagram in $\mathbf{ZX}_1$: there can be no diagrammatic
differentiation $\partial : \mathbf{ZX}_1 \to \D[\mathbf{ZX}_1]$.
In such cases, we can always extend the signature by adjoining a new box
for each derivative.

\begin{proposition}
For every interpretation $[\![-]\!] : C_\Sigma^+ \to \mathbf{Mat}_\S$,
there is an extended signature $\Sigma' \supset \Sigma$
and interpretation $[\![-]\!] : C_{\Sigma'}^+ \to \mathbf{Mat}_\S$
such that $C_{\Sigma'}^+$ admits digrammatic differentiation.
\end{proposition}

\begin{proof}
Let $\Sigma' = \cup_{n \in \N} \Sigma^n$ where $\Sigma^0 = \Sigma$
and $\Sigma^{n + 1} = \Sigma^n \cup \{ \partial f \ \vert \ f \in \Sigma^n \}$
with $[\![\partial f]\!] = \partial [\![f]\!]$.
\end{proof}

The issue of being able to represent arbitrary scalars disappears if we work
with the algebraic ZX-calculus instead. Furthermore, we can generalise
from $\S = \R^n \to \C$ to any commutative rig.
We define the category $\mathbf{ZX}_\S$ as the free symmetric category generated by the signature given in \cite[Table~2]{Wang20} and the interpretation $[\![-]\!] : \mathbf{ZX}_\S \to \mathbf{Mat}_\S$ given in \cite[§6]{Wang20}.
In particular, there is a green square $R_Z^{m, n}(a) \in \Sigma_1$ for each $a \in \S$ and $m, n \in \N$ with interpretation
$$[\![R_Z^{m, n}(a)]\!] = \ket{0}^{\otimes n} \bra{0}^{\otimes m} + a \ket{1}^{\otimes n} \bra{1}^{\otimes m}$$
Let $\mathbf{ZX}_\S^+$ be the category of formal sums of algebraic ZX
diagrams over $\S$.

\begin{theorem}
Diagrammatic derivations for the interpretation $[\![-]\!] : \mathbf{ZX}_\S^+ \to \mathbf{Mat}_\S$
are in one-to-one correspondance with rig derivations $\partial : \S \to \S$.
\end{theorem}

\begin{proof}
Given a derivation $\partial$ on $\S$, we have
$\partial [\![R_Z^{m, n}(a)]\!]
= (\partial a) \ket{1}^{\otimes n} \bra{1}^{\otimes m}$
and $\partial a$ can be represented by the scalar diagram
$R_Z^{1, 0}(\partial a) \ket{1}$.
In the other direction, a diagrammatic derivation $\partial$ on
$\mathbf{ZX}_\S^+$ is uniquely determined by its action on scalars
$R_Z^{1, 0}(a) \ket{1}$ for $a \in \S$.
\end{proof}

One application of diagrammatic differentiation is to solve
differential equations between diagrams. As a first step,
we apply Stone's theorem \cite{Stone32} on one-parameter unitary groups
to the ZX-calculus.

\begin{definition}
A one-parameter unitary group is a unitary matrix $U : n \to n$
in $\mathbf{Mat}_{\R \to \C}$ with $U(0) = \id_n$ and $U(\theta) U(\theta') = U(\theta + \theta')$
for all $\theta, \theta' \in \R$. It is strongly continuous when
$\lim_{\theta \to \theta_0} U(\theta) = U(\theta_0)$ for all $\theta_0 \in \R$.

We say a one-parameter diagram $d : x^{\otimes n} \to x^{\otimes n}$
is a unitary group if its interpretation $[\![d]\!]$ is.
\end{definition}

\begin{remark}
The interpretation of diagrams with smooth maps as phases are necessarily strongly continuous.
\end{remark}

\begin{theorem}[Stone]
There is a one-to-one correspondance between strongly continuous one-parameter
unitary groups $U : n \to n$ in $\mathbf{Mat}_{\R \to \C}$ and self-adjoint
matrices $H : n \to n$ in $\mathbf{Mat}_{\C}$. The bijection is given
explicitly by $U(\theta) = \exp(i \theta H)$ and $H = - i (\partial U)(0)$,
translated in terms of diagrams with bubbles we get:
\ctikzfig{img/diag-diff/2-5-stone-theorem}
\end{theorem}

\begin{corollary}
A one-parameter diagram $d : x^{\otimes n} \to x^{\otimes n}$ in
$\mathbf{ZX}_1$ is a unitary group iff there is a constant
self-adjoint diagram
$h : x^{\otimes n} \to x^{\otimes n}$ such that $\partial d = i h \fcmp d$.
\end{corollary}

\begin{proof}
Given the diagram for a unitary group $d$, we compute its diagrammatic
differentiation $\partial d$ and get $h$ by pattern matching.
Conversely given a self-adjoint $h$, the diagram $d = \exp(i \theta h)$
is a unitary group.
\end{proof}

\begin{example}
Let $d = R_z(\alpha) \otimes R_x(\alpha)$ for a smooth $\alpha : \R \to \R$, then the following implies $d(\theta) = \exp(i \theta h)$ for $h = - i \frac{\partial \alpha}{2}(Z \otimes I + I \otimes X)$.
\ctikzfig{img/diag-diff/2-6-simple-example}
\end{example}

\begin{example}
Let $d = P(\alpha, ZX)$ be a Pauli gadget as in \cite[Definition~4.1]{CowtanEtAl20a} then
the following implies
$d(\theta) = \exp(i \theta h)$ for $h = -i \frac{\partial \alpha}{2} Z \otimes X$.
\ctikzfig{img/diag-diff/2-7-pauli-gadget}
\end{example}


\subsection{Differentiating quantum circuits}\label{3-dual-circuits}

In this section, we extend diagrammatic differentiation to the category $\mathbf{Circ}$ of mixed quantum circuits as defined in section~\ref{section:mixed-circuits}.
Recall that the definition of $\mathbf{Circ}$ depends on a choice of gateset, here we assume that this gateset is interpreted as complex matrices parameterised by some number $n \in \N$ of real variables.
That is, we fix an interpretation $[\![-]\!] : \mathbf{Circ} \to \mathbf{Mat}_\S$ for $\S = \R^n \to \C$.
In this context, diagrammatic derivations correspond to the notion of gradient
recipe for parametrised quantum gates introduced by Schuld et al.~\cite{SchuldEtAl19}.

Let $\mathbf{Circ}^+$ be the category with formal sums of mixed quantum circuits as arrows, i.e. the free commutative-monoid-enrichment of $\mathbf{Circ}$.
Again, we want to find a diagrammatic derivation
$\partial : \mathbf{Circ}^+ \to \D[\mathbf{Circ}^+]$
which commutes with the interpretation, i.e. such that
$[\![\partial \hat{f}]\!] = \partial [\![\hat{f}]\!] =
\partial \big( \overline{[\![f]\!]} \otimes [\![f]\!] \big)$
for all circuits $f \in \mathbf{Circ}$.
Note that a diagrammatic derivation for the interpretation of pure quantum circuits does not in general lift to one for mixed quantum circuits.
Indeed, using the product rule we get
$\partial \big( \overline{[\![f]\!]} \otimes [\![f]\!] \big)
\s = \s \partial \overline{[\![f]\!]} \otimes [\![f]\!]
\ + \ \overline{[\![f]\!]} \otimes \partial [\![f]\!]
\s \neq \s \overline{[\![\partial f]\!]} \otimes [\![\partial f]\!]$.

Hence we need equations, called gradient recipes, to rewrite the gradient of a
pure map $\partial [\![\hat{f}]\!]$ as the pure map of a gradient
$[\![\partial \hat{f}]\!]$.
In the special case of Hermitian operators with at most two unique eigenvalues,
gradient recipes are given by the parameter-shift rule. In the general case
where the parameter-shift rule does not apply, gradient recipes require the
introduction of an ancilla qubit.

\begin{theorem}[\cite{SchuldEtAl19}]
For a one-parameter unitary group $f$ with
$[\![f(\theta)]\!] = \exp (i \theta H)$, if $H$ has at most two eigenvalues
$\pm r$, then there is a shift $s \in [0, 2 \pi)$ such that
$[\![r\big(f(\theta + s) - f(\theta - s)\big)]\!] = \partial [\![f(\theta)]\!]$.
\end{theorem}

\begin{corollary}
Mixed quantum circuits with parametrised ZX diagrams as gateset admit diagrammatic differentiation.
\end{corollary}

\begin{proof}
The $Z$ rotation has eigenvalues $\pm 1$, hence the spiders with two legs have
diagrammatic differentiation given by the parameter-shift rule:

\ctikzfig{img/diag-diff/3-2-param-shift}

As for theorem~\ref{theorem-zx-diag-diff}, this extends to
arbitrary-many legs using spider fusion.
\end{proof}

\begin{remark}
In order to encode the subtraction of the parameter shift-rule diagrammatically, we
need either to consider formal sums with minus signs (a.k.a. enrichment in
abelian groups) or simply to extend the signature with the $-1$ scalar.
Such \emph{mixed scalars} are implemented in listing~\ref{listing:mixed-scalars}.

\begin{minted}{python}
Circuit.__sub__ = lambda self, other: self + Scalar(-1) @ other
\end{minted}
\end{remark}

\begin{example}
The quantum enhanced feature spaces of Havlicek et al.~\cite{HavlicekEtAl19} are parametrised
classical-quantum circuits.
The quantum classifier can be drawn as a diagram:

\ctikzfig{img/diag-diff/3-3-quantum-enhanced}

where $U(\vec{x})$ depends on the input, $W(\vec{\theta})$ depends on the
trainable parameters and $f$ is a fixed Boolean function encoded as a linear map.
\end{example}

\begin{example}
We can define a class for a parameterised quantum gate and override the default \py{grad} with its gradient recipe.

\begin{minted}{python}
class Rz(Gate):
    def __init__(self, phase: sympy.Expr):
        self.phase = phase
        half_theta = sympy.pi * phase
        array = [[sympy.exp(-1j * half_theta), 0],
                 [0, sympy.exp(1j * half_theta)]]
        super().__init__("Rz({})".format(phase), qubit, qubit, array)

    def grad(self):
        s = Scalar(sympy.pi * self.phase.diff(var))
        return s @ Rz(self.phase + .25) - s @ Rz(self.phase - .25))

phi = sympy.Symbol('\\phi')
circuit = Ket(0, 0) >> Rz(phi + 1) @ Rz(2 * phi - .5) >> Measure() @ Measure()

assert circuit.grad(phi).eval() == circuit.eval().grad(phi)
\end{minted}
\end{example}


\subsection{Bubbles and the chain rule} \label{4-bubbles}

Sections~\ref{subsection:dagger-sums-bubbles} and \ref{subsection:monoidal-daggers-sums-bubbles} have introduced \emph{bubbles}, diagrammatic gadgets for representing arbitrary operators on the homsets of a monoidal category.
In particular, the element-wise application of any rig-valued function $\beta : \S \to \S$ yields a bubble on the category of matrices $\beta : \mathbf{Mat}_\S(m, n) \to \mathbf{Mat}_\S(m, n)$.
When the bubble has a derivative $\partial \beta$, we may define the gradient of bubbled diagrams with the chain rule $\partial(\beta(f)) = (\partial \beta)(f) \times \partial f$.
In order to make sense of the multiplication, we assume that the homsets of our category have a product on homsets which is compatible with the sum\footnote
{That is, each homset forms a rig. Note that we do not assume that products are compatible with composition, in other words $\mathbf{C}$ need not be rig-enriched.}
and which commutes with the tensor.
For example, each homset $\mathbf{Mat}_\S(m, n)$ is a rig with entrywise sums and products.
More generally, the homsets of any hypergraph category with sums form a rig, where the product is given by pre/post-composition with the co/monoid structure.
We get the following equation:
$$\input{img/diag-diff/4-2-chain-rule.tikzstyles}$$
For scalar diagrams, spiders are empty diagrams and the
equation simplifies to the usual chain rule.

As we have seen in example~\ref{example:neural-net}, we can encode the architecture of any neural network as a diagram with sums and bubbles for the non-linearity.
Thus, we can draw both a parametrised quantum circuit and its classical
post-processing as one bubbled diagram. By applying the
product rule to the quantum circuit and the chain rule to its post-processing,
we can compute a diagram for the overall gradient. This applies to
parametrised quantum circuits seen as machine learning models
\cite{BenedettiEtAl19}, to the patterns of measurement-based quantum
computing seen as ZX-diagrams \cite{DuncanPerdrix10} as well as quantum
natural language processing.

For now, we have only defined gradients of diagrams with respect to one
parameter at a time. In future work, we plan to extend our definition to
compute the Jacobian of a tensor with respect to a vector of variables.
Other promising directions for research include the study of diagrammatic
differential equations, opening the door to studying e.g. the Schrödinger equation with diagrams.


\section{Conclusion} \label{section:conclusion}

This chapter built on the category theory of chapter~\ref{chapter:discopy} to lay the mathemetical foundations of \emph{quantum natural language processing} (QNLP).
We defined QNLP models as \emph{parameterised monoidal functors} $F_\theta : \G \to \mathbf{Circ}$ from a category $\G$ of grammatical derivations to a category $\mathbf{Circ}$ of quantum circuits.
Given the grammatical structure $f : w_1 \dots w_n \to s \in \G$ for a sentence, the QNLP model produces a parameterised quantum circuit $F_\theta(f)$ which computes the meaning of that sentence.
Using a hybrid classical-quantum algorithm, we computed the optimal parameters $\theta \in \Theta$ in a simple supervised learning task: answering yes-no questions from a toy dataset based on Shakespeare's \emph{Romeo and Juliet}.
In short, we performed the first NLP experiment on quantum hardware.

Our QNLP models can be understood as a quantum implementation of \emph{Categorical Compositional Distributional} (DisCoCat) models~\cite{ClarkEtAl08}.
While the fields of NLP and artificial intelligence are torn apart between the symbolic approach of logic-based expert systems and the connectionist approach of black-box neural networks, DisCoCat models offer the best of worlds.
Indeed, they bring together formal grammar and distributional semantics into one NLP model in the form of a monoidal functor from syntax to meaning.
However, the conceptual advantage of DisCoCat models comes at a price: they can be exponentially hard to evaluate on a classical computer.
This is where QNLP comes in: quantum computers may allow to compute approximations of DisCoCat models exponentially faster than any classical computer.
We demonstrated that our framework applies both to the large-scale fault-tolerant regime where we can load our data onto a qRAM, and to the NISQ era where we have only a few noisy qubits.

The field of QNLP is still in its infancy and there remains much work to be done both on the theoretical side (can we prove that QNLP models offer a significant advantage compared to their classical counterpart?) and on the experimental side (can we demonstrate this advantage on real-world data?).
We argue that the biggest challenge is not so much how to evaluate NLP models on quantum hardware, but how to train them.
Indeed, if the architecture of our parameterised quantum circuit is random then the probability of non-zero gradients is exponentially small in the number of qubits and training our model may take an exponential time: this is the so-called \emph{barren plateau} phenomenon.
Thankfully the quantum circuits for QNLP models are not random: we can exploit their internal structure, which comes from the grammatical structure of sentences.
As a first step in that direction, we introduced \emph{diagrammatic differentiation} as a graphical notation for computing the gradient of QNLP models and parameterised diagrams in general.

We conclude with some directions for future work.
An obvious direction would be to generalise QNLP models either in their domain or their codomain, i.e. syntax or semantics.
On the syntax side, we may consider grammatical frameworks other than the pregroup grammars used in this thesis.
Promising candidates include the Lambek calculus with a relevant modality that we used in previous work~\cite{McPheatEtAl21} to model anaphora, and the related \emph{text grammar} proposed by Coecke~\cite{Coecke21} and Wang~\cite{CoeckeWang21}.
On the semantics side, we may generalise our definition of QNLP models beyond the standard quantum circuit model, experimenting with different kinds of quantum hardware such as linear optical quantum computers~\cite{KokEtAl07} or analog quantum computers based on neutral atoms~\cite{HenrietEtAl20}.

After syntax and semantics, it is natural to explore the \emph{pragmatics} of QNLP: how do we train our models and what kind of tasks do we solve?
An important direction would be to remove the need for labeled data with a form of self-supervised learning such as \emph{functorial language models}~\cite{ToumiKoziell-Pipe21}.
This requires our models to predict missing words rather than binary labels, a subroutine that will also be necessary for language generation and automated translation.
Our framework would also be suitable to \emph{generative adversarial} modeling, where we train a pair of models against each other in what can be formalised as a \emph{functorial language game}~\cite{FeliceEtAl20}.
In this game-theoretic setup, it becomes possible to imagine the two players, i.e. the QNLP models for generator and discriminator, sharing an entangled quantum state and communicating with \emph{quantum pseudo-telepathy}~\cite{BrassardEtAl05}.
While most of this thesis focused on the \emph{computational} advantage of QNLP models, these quantum language games would exhibit a form of \emph{communication} advantage: quantum players can solve distributed problems beyond what is possible with classical means.
Hence we may expect quantum machines to have conversations incomprehensible to classical minds.


\begin{acknowledgements}
\addcontentsline{toc}{chapter}{Acknowledgements}

My first words of gratitude go to my supervisor Bob Coecke.
He was the mentor who guided me through to the perks of academia, the architect of the Oxford quantum group in which I learnt so much as well as the guru who indoctrinated me in the science of string diagrams.
After he quit his academic position in the middle of my thesis, he also became the most relaxed and enthusiastic boss in all of quantum industry, creating a new research group from scratch to build upon the work we had started together.
I must also thank my co-supervisor Dan Marsden, who played a crucial role during my masters and the beginning of my DPhil, taming my enthusiasm with his pragmatism and rigour.

Although the words are my own, the work presented in this thesis was of a collective nature.
Thus it is in order to acknowledge my co-authors, first and foremost my academic twin Giovanni de Felice.
I would never have written this thesis without the friendship we have built through more than five years of working, cooking and partying together.
Our QNLP experiments would not have come to life without Konstantinos Meichanetzidis, his physicist mindset and his patience to let us explain our abstract nonsense.
I would also like to thank my academic cousins Richie Yeung and Alex Koziell-Pipe, I trust them to keep the spirit of the quantum group alive.

I am indebted to my examiners Sam Staton and Michael Moortgat for their careful reading of my manuscript, their feedback helped much to improve this thesis.
I am grateful to my college advisor Samson Abramsky for his discreet yet thought-provoking supervision.
Aleks Kissinger and Mehrnoosh Sadrzadeh have also provided me with valuable feedback during my transfer of status.
I want to thank Pawel Sobocinski for inviting me at TalTech where I could present many of the ideas of this thesis, but also for introducing me to string diagrams in the first place in Cali, Columbia over seven years ago.

I am grateful to Simon Harrisson for his generous support through the Wolfson Harrison UK Research Council Quantum Foundation Scholarship.
My DPhil was also supported by the Oxford-DeepMind Graduate Scholarship.
I also thank Cambridge Quantum Computing for taking me as a part-time research scientist.
This brought a new industrial dimension to my research, but also made it possible for me to live with my family in Paris during troubled pandemic times.

On a more personal side, special thanks go to my friends: Tommaso Salvatori who proofread my introduction to NLP, the Wolfson Secret Santa, la Confinerie du 117 and many others who will recognise themselves.
Very special thanks to my family: Astrid, Romeo, my parents Godeleine and Imad, my siblings Raphaël, Maylis and Xavier, my grandfather Etienne (RIP), my grandmothers Véronique and Habiba.
\end{acknowledgements}

\renewcommand*\MakeUppercase[1]{#1}%
\setlength{\baselineskip}{0pt} 
\printbibliography[heading=bibintoc,title=References]
\end{document}